\documentclass[12pt]{article}

\usepackage{amssymb}
\usepackage{amsmath}
\usepackage{amsthm}
\usepackage{makeidx}

\makeatletter
\newfont{\footsc}{cmcsc10 at 8truept}
\newfont{\footbf}{cmbx10 at 8truept}
\newfont{\footrm}{cmr10 at 10truept}
\renewcommand{\ps@plain}{%
\renewcommand{\@oddfoot}{\footsc the electronic journal of combinatorics
  {\footbf 10} (2003), \#R13\hfil\footrm\thepage}}
\makeatother
\theoremstyle{plain}
   \newtheorem {thm}{Theorem}[section]
   \newtheorem {id}[thm]{Identity}

\theoremstyle{definition}
   \newtheorem {obs}[thm]{Observation}
   \newtheorem {conditions}[thm]{Conditions}

\numberwithin{equation}{section}
 
\newcommand{\qbin}[2]{\genfrac{[}{]}{0pt}{}{#1}{#2}}
\newcommand{\gp}[3]{\qbin{#1}{#2}_{#3}}
\newcommand{\Tzero}[3]{\mathrm T_0 {\textstyle(#1,#2;#3)}}
\newcommand{\Tone}[3]{\mathrm T_1 {\textstyle(#1,#2;#3)}}
\newcommand{\U}[3]{\mathrm U{\textstyle(#1,#2;#3)}}
\newcommand{\tzero}[3]{\mathrm t_0 {\textstyle(#1,#2;#3)} }
\newcommand{\tone}[3]{\mathrm t_1 {\textstyle(#1,#2;#3)} }
\newcommand{\tauzero}[3]{\tau_0 {\textstyle(#1,#2;#3)} }
\newcommand{\Trb}[4]{\binom {#1,#2;#4}{#3}_2 }
\newcommand{\V}[3]{\mathrm V {\textstyle(#1,#2;#3)}}

\newenvironment{id*}[2][\quad]
  {\textbf{Identity #2} \textbf{(#1).} \itshape }

\newenvironment{obs*}[1]
   {\textbf{Observation #1.} }

\newenvironment{oldresult}[1]
  {\textbf{#1} \itshape } 

\makeindex

\allowdisplaybreaks[1]

\title{Finite Rogers-Ramanujan Type Identities}

\author{Andrew V. Sills\thanks{The research contained herein comprises a substantial portion of
the author's doctoral dissertation, submitted in partial fulfillment of the
requirements for the Ph.D. degree at the University of Kentucky.  The
doctoral dissertation was completed under the supervision of George E.
Andrews, Evan Pugh Professor of Mathematics at the Pennsylvania State
University.  This research was partially supported by a grant provided to
the author by Professor Andrews.}\\
\small through August 2003:\\
\small Department of Mathematics\\[-0.8ex]
\small The Pennsylvania State University, University Park, PA, USA \\[-0.8ex]
\small since August 2007:\\
\small Department of Mathematical Sciences\\[-0.8ex]
\small Georgia Southern University, Statesboro, GA, USA\\[-0.8ex]
\small \texttt{asills@georgiasouthern.edu}\\
\small \texttt{http://home.dimacs.rutgers.edu/\~{}asills}}

\date{\small Submitted: May 14, 2002; Revised: August 27, 2002; \\ 
Accepted: April 10, 2003; Published: April 23, 2003.\\
\small MR Subject Classifications: 05A10, 11B65}

\begin{document}
\maketitle

\begin{abstract}
  Polynomial generalizations
of all 130 of the identities in Slater's list of identities
of the Rogers-Ramanujan type are presented.  Furthermore, duality relationships 
among many of the identities are derived.  
Some of the these polynomial identities 
were previously known but many are new.  The author has implemented 
much of the finitization process in a Maple package which is available
for free download from the author's website. 
\end{abstract}


\setcounter{section}{-1}
\pagenumbering{arabic}

\section{Introduction}
\subsection{Three approaches to finitization}
There are at least three avenues of approach that lead to finite 
Rogers-Ramanujan type identities.
\begin{enumerate}
    \item \emph{Combinatorics and models from statistical mechanics.}  This 
approach has been studied extensively by Andrews, Baxter, Berkovich, 
Forrester, 
McCoy, Schilling, Warnaar and others; see, e.g., ~\cite{hhm}, \cite{lg3}, 
\cite{scratch}, \cite{abf:8v}, \cite{6guys}, \cite{bm:cf}, \cite{bmo:poly}, 
\cite{bms},
\cite{fb}, \cite{sw}, \cite{sow:gbc}, \cite{sow:qti}, \cite{sow:rqtc}.
\index{Warnaar, S. Ole}\index{McCoy, Barry M.}\index{Andrews, George E.}
\index{Baxter, Rodney M.} \index{Berkovich, Alexander} \index{Forrester, P. J.}
\index{Schilling, Anne}

    \item \emph{The Strong Bailey Lemma.}  This method is discussed in
chapter 3 of Andrews' $q$-series monograph~\cite{qs}.
    \item \emph{The method of nonhomogeneous $q$-difference equations}.
This method is introduced in~\cite[Chapter 9]{qs} and studied extensively
herein.  
\end{enumerate}

  While these three methods sometimes lead to similar results,
often the results are different.  Even in the cases where the different
methods lead to the same finitization, each method has its own inherent
interest.  For instance, from the 
statistical mechanics point of view, finitization makes it 
possible to consider $q\rightarrow q^{-1}$ duality, which 
in the case
of Baxter's hard hexagon model, allows one to neatly pass from one regime to
another~\cite{hhm}.  Finitizations arising as a result of the application
of the strong form of Bailey's Lemma give rise to important questions
in computer algebra as in Paule~(\cite{pp:rrt} and \cite{pp:rr}).
Finally, the method of $q$-difference equations has been studied
combinatorially in~\cite{comb}.  It is this method that will be 
studied in depth in this present work.
   
  Granting the intrinsic merit of all of these approaches, a particularly
interesting aspect
of the third method stems from the fact that there is no known 
overarching theory which guarantees a given attempt at finitization will
be successful.  The fact that \emph{all} of Slater's list succumbed to this 
method is evidence in favor of the existence of such an overarching 
theory.

  Let us now begin to study this third method in detail.
\subsection{Overview of this work}
In his monograph on $q$-series~\cite[Chapter 9]{qs}, Andrews 
\index{Andrews, George E.}
indicated a method (referred to herein as the ``method of first order 
nonhomogeneous $q$-difference equations," or more briefly as the ``method 
of $q$-difference equations") to produce sequences of polynomials
which converge to the Rogers-Ramanujan identities and 
identities of similar type.  By appropriate application of the $q$-binomial
theorem, formulas for the polynomials can easily be produced for what
the physicists call ``fermionic representations" of the polynomials.  
The identities explored in ~\cite{hhm} and \cite{qs} relate 
to Baxter's solution \index{Baxter, Rodney} 
of the hard hexagon model in statistical 
mechanics~\cite{rjb:rr}.  In~\cite{scratch}, Andrews and Baxter suggest
some ideas for how a computer algebra system can be employed to find
what the physicists call ``bosonic representations" of polynomials
which converge to Rogers-Ramanujan type products. When we have both a 
fermionic and bosonic representation of a polynomial sequence which 
converges to a series-product identity, the series-product identity is
said to have been \emph{finitized}. 

  In his Ph.D. thesis~\cite{jpos}, Santos conjectured \index{Santos, J. P. O.} 
bosonic (but no fermionic) representations for polynomial sequences 
which converge to many of the 
identities in Lucy Slater's paper on 
Rogers-Ramanujan Type \index{Slater, Lucy J.}
Identities~\cite{ljs}.

  This present work extends and unifies the results found in~\cite{hhm}, 
\cite[Chapter 9]{qs},
\cite{scratch} and \cite{jpos}.  
Background material is presented in \S 1.

  In \S 2, it is proved that the method of $q$-difference equations 
can be used to algorithmically produce polynomial generalizations of
Rogers-Ramanujan type series, and find fermionic representations of
them.
As in \cite{scratch} and \cite{jpos}, bosonic 
representations need to be conjectured, but the methods and computer 
algebra tools discussed in \S 2 indicate how appropriate conjectures 
can be found efficiently.  

   In \S 3,  at least one finitization is 
presented for each of the 130 identities in Slater's list~\cite{ljs}.
In the case 
of some of the simpler identities in Slater's list, the finitization found 
corresponds to a previously known polynomial identity, but in many of the 
cases, the identities found are new.  Considerable care was taken to
provide appropriate references for the previously known, and 
previously conjectured identities or pieces of identities. 
In each case, the bosonic representations can best 
be understood in terms of either Gaussian polynomials or $q$-trinomial 
co\"efficients~\cite{lg3}.  Particularly noteworthy is the discovery that 
bosonic representations of a number of the finitized Slater identities used
a weighted combination of two different $q$-trinomial co\"efficients,  
referred to herein as $\V{L}{A}{q}$~(see (\ref{Vdef})).  It turns 
out that this ``$V$" function enters naturally into the theory of 
$q$-trinomial co\"efficients due to certain internal 
symmetries of the $T_0$ and $T_1$ $q$-trinomial co\"efficients~(\ref{Vsym}),
although its existence had previously gone unnoticed. 

   Section 4 contains a discussion of various methods for proving the 
polynomial identities conjectured by the method of $q$-difference 
equations.  Particular emphasis is placed upon the algorithmic proof theory 
\index{Wilf, Herbert S.} \index{Zeilberger, Doron}
of Wilf and Zeilberger~(\cite{pwz:a=b}, \cite{wz:rf}, 
\cite{wz:cp}, \cite{wz:multiq}, 
\cite{dz:fa}, \cite{dz:ct}).  It is to be noted that the author has proved 
every identity in \S 3 using the ``method of recurrence proof" 
discussed in Section 4, including the 1991 Santos conjectures, as well as 
new polynomial identities.  Thus, all of the identities in Slater's 
list~\cite{ljs} may now be viewed as corollaries of the polynomial 
identities presented in \S 3.

   Once a series-product identity is finitized, a $q\to q^{-1}$ duality 
theory can be discussed.  In \cite{hhm}, Andrews describes the duality 
between various identities associated with the four regimes of the hard 
hexagon model.  An extensive study of the duality relationships among 
the identities presented in \S 3 is undertaken in \S 5.  A number 
of previously unknown multisum identities arise as a result of this 
duality study.
  
  In \S 6, a relaxed version of the finitization method of \S 2 
is considered wherein we drop the requirement that the 
two-variable generalization 
of the Rogers-Ramanujan type series satisfy a first order nonhomogeneous 
$q$-difference equation.  It is then demonstrated that this method can be 
used to find several identities due to Bressoud~\cite{dmb}, as well 
\index{Bressoud, David M.}
as to find additional new finitizations of Rogers-Ramanujan type identities,
at least one of which arises in the work of Warnaar~\cite{sow:rqtc}.
\index{Warnaar, S. Ole}

  Finally, the appendix is an annotated and cross-referenced version of 
Slater's list of identities from~\cite{ljs}.  Since Slater's list of identities 
has been the source for 
further research for many mathematicians, my hope is that others will find 
this version of Slater's list useful.
\index{Slater, Lucy J.}

\section{Background Material}
\label{backgr}

\subsection{$q$-Binomial co\"efficients} \label{gausspoly}
\index{q-factorial}
We define the {\em infinite rising q-factorial} $(a;q)_\infty$ as follows:
 \[ (a;q)_\infty := \prod_{m=0}^\infty (1-aq^m), \] 
where $a$ and $q$ may be thought of as complex numbers, 
and then the {\em finite rising q-factorial} $(a;q)_n$ by
 \[ (a;q)_n := \frac{(a;q)_\infty}{(aq^n;q)_\infty} \]
for all complex $n$, $a$, and $q$.
Thus, if $n$ is a positive integer, \[(a;q)_n = \prod_{m=0}^{n-1} 
(1-aq^m).\] 
In the $q$-factorials $(a;q)_n$ and $(a;q)_\infty$, the ``$q$'' is referred to
as the ``base" of the factorial.
It will often be convenient to abbreviate a product of rising 
$q$-factorials with a common base 
\[ (a_1;q)_\infty ({a_2};q)_\infty ({a_3};q)_\infty
\dots(a_r;q)_\infty \]
by the more compact notation
\[ ({a_1}, {a_2}, {a_3}, \dots , {a_r};q)_\infty.\]

\index{Gaussian polynomial|(}
The {\em Gaussian polynomial} $\gp{A}{B}{q}$ may be defined
\footnote{Variations of this definition are possible for $B<0$ or $B>A$; see, e.g. 
Berkovich, McCoy, and Orrick~\cite[p. 797, eqn. (1.7)]{bmo:poly} for a variation
frequently used in statistical mechanics.}:
\index{Berkovich, Alexander}\index{McCoy, Barry M.} \index{Orrick, William P.}
  \[ \gp{A}{B}{q} := 
     \left\{  \begin{array}{ll}
      (q;q)_A (q;q)_B^{-1} (q;q)_{A-B}^{-1}, &\mbox{if $0\leqq B\leqq A$} \\
      0,                                     &\mbox{otherwise.}
              \end{array} \right. \]

Note that even though the Gaussian polynomial $\gp{A}{B}{q}$
is defined as a rational function, it 
does, in fact, reduce to a polynomial for all integers $A$, $B$, just as the
fraction
$\frac{A!}{B!(A-B)!}$ simplifies to an integer. 
Notice that in the case where $A$ and $B$ are positive integers with $B\leqq A$, 
  \begin{equation}\label{gppos}
   \gp{A}{B}{q} = \frac{(1-q^A)(1-q^{A-1})(1-q^{A-2})\cdots (1-q^{A-B+1})}
{(1-q)(1-q^2)(1-q^3)\cdots(1-q^B)}, 
\end{equation}
and so 
\begin{equation} \label{gpdeg}
   \mathrm{deg}\left( \gp{A}{B}{q} \right) = B(A-B).
\end{equation} 

\begin{equation}
  \lim_{q\to 1}\gp{A}{B}{q} = \binom{A}{B}, \label{qanbin}
\end{equation}
where $\binom{A}{B}$ is the ordinary binomial co\"efficient, thus 
Gaussian polynomials are also called {\em $q$-binomial co\"efficients}.

Just as ordinary binomial co\"efficients satsify the symmetry relationship
\[ \binom{A}{B} = \binom{A}{A-B}, \]
so do Gaussian polynomials satisfy the symmetry relationship
\begin{equation}
  \gp{A}{B}{q} = \gp{A}{A-B}{q} \label{gpsym}.
\end{equation}
Likewise, the Pascal triangle recurrence 
\[ \binom{A}{B} = \binom{A-1}{B-1} + \binom{A-1}{B} \] has two $q$-analogs:
\begin{eqnarray}
   \gp{A}{B}{q} = \gp{A-1}{B}{q} + q^{A-B}\gp{A-1}{B-1}{q} \label{qpt1} \\
   \gp{A}{B}{q} = \gp{A-1}{B-1}{q} + q^B \gp{A-1}{B}{q} \label{qpt2},
\end{eqnarray}
for $A>0$ and $0\leqq B\leqq A$.
For a complete discussion and proofs of (\ref{gpsym}) -- (\ref{qpt2}), see 
Andrews~\cite[pp. 305 ff]{top}.

We also record the easily established identity
 \begin{equation}  
    \gp{A}{B}{1/q} = q^{B(B-A)}\gp{A}{B}{q} \label{gpinv}
 \end{equation}
and the asymptotic result  
\begin{equation}
  \lim_{n\to\infty} \gp{2n+a}{n+b}{q} = \frac{1}{(q;q)_\infty}. \label{gplim}
\end{equation}
\index{Gaussian polynomial|)}

\index{binomial theorem}
The binomial theorem may be stated as
\[  \sum_{j=0}^\infty \binom{L}{j} t^j = (1+t)^L. \]
\index{q-binomial theorem}
The $q$-binomial theorem, which seems to have been discovered
independently by Cauchy~\cite{cauchy}, Heine~\cite{heine}, and 
Gauss~\cite{gauss},
\index{Cauchy, A.-L.} \index{Heine, E.} \index{Gauss, K. F.}  
follows:

\begin{oldresult}{$q$-Binomial Theorem.}
 \cite[p. 488, Thm. 10.2.1]{sf} or \cite[p. 17, Thm. 2.1]{top}.
  If $|t|<1$ and $|q|<1$, 
   \begin{equation} \label{qbinthm}
     \sum_{k=0}^\infty \frac{(a;q)_k}{(q;q)_k} t^k = 
      \frac{(at;q)_\infty}{(t;q)_\infty}.
   \end{equation}
\end{oldresult}

 We will make use of the following two corollaries of 
(\ref{qbinthm}):  The first corollary, which appears to be due to
H. A. Rothe~\cite{har}, \index{Rothe, H. A.}
  \begin{equation} \label{qbc1}
    \sum_{k=0}^j \gp{j}{k}{q} (-1)^k q^{\binom{k}{2}} t^k = (t;q)_j.
  \end{equation}
may be obtained from (\ref{qbinthm}) by setting $a=q^{-j}$. 
The second corollary,
  \begin{equation} \label{qbc2}
    \sum_{k=0}^\infty \gp{j+k-1}{k}{q} t^k = \frac{1}{(t;q)_j},
  \end{equation}
is the case $a=q^j$ of (\ref{qbinthm}). 
If in (\ref{qbc1}), we replace $q$ by $q^{2r}$, set $t=-q^{r+s}$ and let 
$j\to\infty$, we obtain
   \begin{equation} \label{qbc3}
    \sum_{k=0}^\infty \frac{q^{rk^2 + sk}}{(q^{2r};q^{2r})_k} = 
    (-q^{r+s};q^{2r})_\infty,
  \end{equation}
a formula useful for simplifying certain multisums.

\subsection{$q$-Trinomial co\"efficients}
\index{q-trinomial co\"efficients|(}
\subsubsection{Definitions}
Consider the Laurent polynomial $(1+x+x^{-1})^L$.  Analogous to the binomial theorem,
we find
\index{trinomial co\"efficient!ordinary}
\begin{equation}
  (1+x+x^{-1})^L = \sum_{j=-L}^L \binom{L}{j}_2 x^{j} \label{tt}
\end{equation}
where
\begin{eqnarray}
  \binom{L}{A}_2 & = & \sum_{r\geqq 0} \frac{L!}{ r!(r+A)! (L-2r-A)!} 
  \label{tri1}\\
                 & = & \sum_{r=0}^L (-1)^r \binom{L}{r} 
                 \binom{2L-2r}{L-A-r}. \label{tri2}
\end{eqnarray}
These $\binom{L}{A}_2$ are called {\em trinomial co\"efficients}, (not to be 
confused with the co\"efficients which arise in the expansion of
$(x+y+z)^L$, which are also often called trinomial co\"efficients).

The two representations (\ref{tri1}) and (\ref{tri2}) of $\binom{L}{A}_2$ 
give rise to different $q$-analogs
due to Andrews and 
Baxter~\cite[p. 299, eqns. (2.7)--(2.12)]{lg3}:~\footnote{Note: 
Occasionally in the literature 
(e.g. Andrews and Berkovich~\cite{ab:ta}, or Warnaar~\cite{sow:qti}),
\index{Andrews, George E.}\index{Berkovich, Alexander} \index{Warnaar, S. Ole}
superficially different definitions of the $T_0$ and
$T_1$ functions are used.}  
\index{q-trinomial co\"efficients!definitions}
\begin{gather}
  \Trb{L}{B}{A}{q} := \sum_{r\geqq 0} \frac{q^{r(r+B)}(q;q)_L} {(q;q)_r 
  (q;q)_{r+A} (q;q)_{L-2r-A}}
    =\sum_{r= 0}^L q^{r(r+B)} \gp{L}{r}{q} \gp{L-r}{r+A}{q} \label{Trbdef} \\
  \Tzero{L}{A}{q} := \sum_{r=0}^L (-1)^r \gp{L}{r}{q^2} \gp{2L-2r}{L-A-r}{q}
    \label{Tzerodef} \\
  \Tone{L}{A}{q} := \sum_{r=0}^L (-q)^r \gp{L}{r}{q^2} \gp{2L-2r}{L-A-r}{q}
     \label{Tonedef} \\
  \tauzero{L}{A}{q} := \sum_{r=0}^L (-1)^r q^{Lr - \binom r2} \gp{L}{r}{q}
    \gp{2L-2r}{L-A-r}{q} \label{tau0def}\\
  \tzero{L}{A}{q} := \sum_{r=0}^L (-1)^r q^{r^2} \gp{L}{r}{q^2} 
\gp{2L-2r}{L-A-r}{q} \label{t0def}\\
  \tone{L}{A}{q} := \sum_{r=0}^L (-1)^j q^{r(r-1)} \gp{L}{r}{q^2} 
\gp{2L-2r}{L-A-r}{q} \label{t1def} 
\end{gather}

It is convenient to follow Andrews~\cite{eulers} \index{Andrews, George E.}
and define
  \begin{equation}\label{Udef}
    \U{L}{A}{q} := \Tzero{L}{A}{q} + \Tzero{L}{A+1}{q}.
  \end{equation}

Further, I will define  
  \begin{equation}\label{Vdef}
    \V{L}{A}{q} := \Tone{L-1}{A}{q} + q^{L-A} \Tzero{L-1}{A-1}{q}.
  \end{equation}

\subsubsection{Recurrences} \label{qtr}
\index{q-trinomial co\"efficients!recurrences}
The following Pascal triangle type relationship is easily deduced from
(\ref{tt}):
\begin{equation} \label{tpt}
   \binom{L}{A}_2 = \binom{L-1}{A-1}_2 + \binom{L-1}{A}_2 + 
   \binom{L-1}{A+1}_2.
\end{equation}
We will require the following $q$-analogs of (\ref{tpt}), which are due to
Andrews and Baxter~\cite[pp. 300--1, eqns. (2.16), (2.19), (2.25) 
(2.26), (2.28), and (2.29)]{lg3}:
\index{Andrews, George E.}\index{Baxter, Rodney}
For $L\geqq 1$,
\begin{gather} 
  \Tone{L}{A}{q}  =  \Tone{L-1}{A}{q}  + q^{L+A} 
\Tzero{L-1}{A+1}{q} + q^{L-A} \Tzero{L-1}{A-1}{q}\label{ET1} \\
  \Tzero{L}{A}{q} =  \Tzero{L-1}{A-1}{q} + q^{L+A} 
\Tone{L-1}{A}{q} \nonumber \\ + q^{2L+2A} \Tzero{L-1}{A+1}{q} 
\label{ET0} \\
\Trb{L}{A-1}{A}{q} = q^{L-1} \Trb{L-1}{A-1}{A}{q} + q^A \Trb{L-1}{A+1}{A+1}{q}
   \nonumber \\
 + \Trb{L-1}{A-1}{A-1}{q} \label{ETrb1}\\
\Trb{L}{A}{A}{q} = q^{L-A} \Trb{L-1}{A-1}{A-1}{q} + q^{L-A-1} \Trb{L-1}{A-1}{A}{q}
   \nonumber \\
 + \Trb{L-1}{A+1}{A+1}{q} \label{ETrb0}\\
\Trb{L}{B}{A}{q} = \Trb{L-1}{B}{A}{q} + q^{L-A-1+B}\Trb{L-1}{B}{A+1}{q}
  + q^{L-A}\Trb{L-1}{B-1}{A-1}{q} \label{ETrb28} \\
\Trb{L}{B}{A}{q} = \Trb{L-1}{B}{A}{q} + q^{L-A}\Trb{L-1}{B-2}{A-1}{q}
  + q^{L+B}\Trb{L-1}{B+1}{A+1}{q} \label{ETrb29} 
\end{gather}

The following identities of Andrews and Baxter~\cite[p. 301, eqns. (2.20)
and (2.27 corrected)]{lg3}, 
\index{Andrews, George E.} \index{Baxter, Rodney}
which reduce to the tautology ``$0=0$" in the
case where $q=1$ are also useful:
\begin{gather} 
\Tone{L}{A}{q} - q^{L-A}\Tzero{L}{A}{q} - \Tone{L}{A+1}{q}+ 
q^{L+A+1}\Tzero{L}{A+1}{q}=0,\label{E0}\\
\Trb{L}{A}{A}{q} + q^L\Trb{L}{A}{A+1}{q} - \Trb{L}{A+1}{A+1}{q} -
   q^{L-A}\Trb{L}{A-1}{A}{q}=0. \label{ETrb00}
\end{gather}
Observe that (\ref{E0}) is equivalent to
\begin{equation} \label{Vsym}
\V{L+1}{A+1}{q} = \V{L+1}{-A}{q}.
\end{equation}

The following recurrences appear in Andrews~\cite[p. 661, Lemmas 4.1 
and 4.2]{eulers}:  \index{Andrews, George E.}
For $L\geqq 1$,
\begin{eqnarray}
  \U{L}{A}{q} & = & (1 + q^{2L-1})\U{L-1}{A}{q} + q^{L-A}
   \Tone{L-1}{A-1}{q}\nonumber \\
           &   & + q^{L+A+1} \Tone{L-1}{A+2}{q}. \label{U1} \\
  \U{L}{A}{q} &=& (1+q+q^{2L-1})\U{L-1}{A}{q} -q\U{L-2}{A}{q} \nonumber\\
      & & +q^{2L-2A}\Tzero{L-2}{A-2}{q} + q^{2L+2A+2}\Tzero{L-2}{A+3}{q}.
    \label{U0}
\end{eqnarray} 

An analogous recurrence for the ``V" function is
\begin{eqnarray}
  \V{L}{A}{q} &=& (1+q^{2L-2})\V{L-1}{A}{q} + q^{L-A} \Tzero{L-2}{A-2}{q}
  \nonumber \\
    &  & + q^{L+A-1} \Tzero{L-2}{A+1}{q}. \label{V0}
\end{eqnarray}
\begin{proof}
 \begin{eqnarray*}
   \V{L}{A}{q} &=& \Tone{L-1}{A}{q} + q^{L-A} \Tzero{L-1}{A-1}{q} 
       \qquad\mbox{\ \ (by (\ref{Vdef}))} \\ \\
   &=& \Tone{L-2}{A}{q}  + q^{L+A-1}\Tzero{L-2}{A+1}{q} \\
   & & + q^{L-A-1} \Tzero{L-2}{A-1}{q} + \Tzero{L-2}{A-2}{q} \\
   & &+ q^{L+A-2} \Tone{L-2}{A-1}{q} + q^{2L+2A-4} \Tzero{L-2}{A}{q} 
     \mbox{\ \ (by (\ref{ET0} and \ref{ET1}))} \\ \\
   &=& \V{L-1}{A}{q} + q^{L+A-1}\Tzero{L-2}{A+1}{q}  \\
   & &+ \Tzero{L-2}{A-2}{q}  + q^{L+A-2} \Tone{L-2}{A-1}{q}\\
   & &+ q^{2L+2A-4} \Tzero{L-2}{A}{q} \qquad \mbox{\ \ (by (\ref{Vdef}))} \\ \\
   &=&  (1+q^{2L-2})\V{L-1}{A}{q} + q^{L-A} \Tzero{L-2}{A-2}{q}\\
   & & + q^{L+A-1} \Tzero{L-2}{A+1}{q} \qquad\mbox{\ \ (by (\ref{E0}) and (\ref{Vdef}))}.
 \end{eqnarray*}
\end{proof}

\subsubsection{Identities}\label{qti}

From (\ref{tt}), it is easy to deduce the symmetry relationship
\begin{equation}
   \binom{L}{A}_2 = \binom{L}{-A}_2. \label{tsym}
\end{equation} 
Two $q$-analogs of (\ref{tsym}) are
\begin{equation} \Tzero{L}{A}{q} = \Tzero{L}{-A}{q} \label{T0sym} \end{equation} 
and
\begin{equation}  \Tone{L}{A}{q} = \Tone{L}{-A}{q}.  \label{T1sym} \end{equation}
The analogous relationship for the ``round bracket" $q$-trinomial 
co\"efficient (\cite[p. 299, eqn. (2.15)]{lg3}) is
\begin{equation} \Trb{L}{B}{-A}{q} = q^{A(A+B)} \Trb{L}{B+2A}{A}{q}.
 \label{Trbsym} \end{equation}
Other fundamental relations among the various $q$-trinomial co\"efficients
include the following (see Andrews and Baxter~\cite[\S 2.4, pp. 
305--306]{lg3}):
\begin{gather}
  \Trb{L}{A}{A}{q} = \tauzero{L}{A}{q} \label{tau0eq} \\
  \Tzero{L}{A}{q^{-1}} = q^{A^2 - L^2} \tzero{L}{A}{q} = q^{A^2-L^2} 
     \tauzero{L}{A}{q^2} \label{T0inv} \\
  \Tone{L}{A}{q^{-1}} = q^{A^2 - L^2} \tone{L}{A}{q} \label{T1inv} \\
  \tauzero{L}{A}{q^{2}} = \Trb{L}{A}{A}{q^2} = \tzero{L}{A}{q} 
     \label{tau0id} \\
  \Trb{L}{A-1}{A}{q^2} = q^{A-L} \tone{L}{A}{q} \label {t1id}
\end{gather}

\subsubsection{Asymptotics}
\index{q-trinomial co\"efficients!asymptotics}
The following asymptotic results for $q$-trinomial co\"efficients are 
proved in, or are direct consequences of, 
Andrews and Baxter~\cite[\S 2.5, pp. 309--312]{lg3}: 
\index{Andrews, George E.}\index{Baxter, Rodney}
 \begin{gather}
  \lim_{L\to\infty} \Trb{L}{A}{A}{q} = \lim_{L\to\infty} \tauzero{L}{A}{q} =
     \frac{1}{(q;q)_\infty}  \label{tau0lim} \\
 \lim_{L\to\infty} \Trb{L}{A-1}{A}{q} = \frac{1+q^A}{(q;q)_\infty}  \\
  \underset{L-A\ \rm{even}}{\lim_{L\to\infty}} \Tzero{L}{A}{q} = 
     \frac{ (-q;q^2)_\infty
     + (q,q^2)_\infty}{2(q^2;q^2)_\infty} \label{T0elim} \\
  \underset{L-A\ \rm{odd}}{\lim_{L\to\infty}} \Tzero{L}{A}{q} = 
    \frac{ (-q;q^2)_\infty 
    - (q;q^2)_\infty }{2(q^2;q^2)_\infty} \label{T0olim} \\
  \lim_{L\to\infty} \Tone{L}{A}{q} = \frac{(-q^2;q^2)_\infty}{(q^2;q^2)_\infty}
 \label{T1lim}
\\
  \lim_{L\to\infty} \V{L}{A}{q} = \frac{(-q^2;q^2)_\infty}{(q^2;q^2)_\infty}
\\
  \lim_{L\to\infty} \tzero{L}{A}{q} = \frac{1}{(q^2;q^2)_\infty} \label{t0lim} \\
  \lim_{L\to\infty} q^{-L} \tone{L}{A}{q} = \frac{ q^{-A} + q^A}{(q^2;q^2)_\infty}
     \label{t1lim} \\
  \lim_{L\to\infty} \U{L}{A}{q} = \frac{(-q;q^2)_\infty}{(q^2;q^2)_\infty}   
     \label{Ulim}    
  \end{gather}
\index{q-trinomial co\"efficients|)}

\subsection{Miscellaneous Results}
\index{triple product identity}
The following result, found by Jacobi in 1829, is fundamental:

  \begin{oldresult}{Jacobi's Triple Product Identity.}
\index{Jacobi, K.} \cite[p. 21, Theorem 2.8]{top} or \cite{mi:1psi1}. 
 For $z\neq 0$ and $|q|<1$,
 \begin{eqnarray} 
  \sum_{j=-\infty}^\infty z^j q^{j^2} & = & 
  \prod_{j=1}^\infty (1+zq^{2j-1})(1+z^{-1}q^{2j-1})(1-q^{2j}) \label{jtp} \\
         &=& (-zq, -z^{-1} q, q^2; q^2)_\infty \nonumber  
 \end{eqnarray}
\end{oldresult}

Note that the asymptotics of the Gaussian polynomials and the
$q$-trinomial co\"efficients, beside being instances of Jacobi's 
Triple Product identity, are 
reciprocals of particular values of $\vartheta$-functions from the 
classical theory of elliptic functions.  If we follow 
Slater\cite[p. 197 ff.]{ljs:book} \index{Slater, Lucy J.}
and define
\begin{gather*}
\vartheta_2(z,q) := \sum_{j=-\infty}^\infty q^{j+\frac 12} e^{(2j+1)iz}
\mbox{\ and}\\
\vartheta_4(z,q) := \sum_{j=-\infty}^\infty (-1)^j q^{j^2} e^{2jiz},
\end{gather*}
where $i=\sqrt{-1}$, \index{theta functions}
we see that 
\begin{gather*}
\vartheta_4(0,q) = \sum_{j=-\infty}^\infty (-1)^j q^{j^2} = 
\frac{(q;q)_\infty}{(-q;q)_\infty} = (q;q)_\infty (q;q^2)_\infty  
= \frac{1}{\lim_{L\to\infty} \Tone{L}{A}{\sqrt q}}, \\
\vartheta_2\left( \pi(1-\frac{\tau}{2}), q^{3/2}\right) = (q;q)_\infty
= \frac{1}{\lim_{L\to\infty} \tauzero{L}{A}{q}} = \frac{1}{\lim_{L\to\infty} \gp{2L+a}{L+b}{q}},\\
\vartheta_2\left(\pi(1-\frac{\tau}{2}), q^2\right) = \frac{(q^2;q^2)_\infty}
{(-q;q^2)_\infty} = (q;q^2)_\infty (q^4;q^4)_\infty
= \frac{1}{\lim_{L\to\infty} \U{L}{A}{q}}.
\end{gather*} 
where $q= e^{\pi i \tau}$.

The next result, due to Cauchy, is a finite form of (\ref{jtp}): 
 \begin{equation} \label{fjtp}
  \sum_{j=-n}^n z^j q^{j^2} \gp{2n}{n+j}{q^2} = (-z^{-1}q;q^2)_n (-zq;q^2)_n.
 \end{equation}
The proof of (\ref{fjtp}) follows from (\ref{qbc1}).  
See also Andrews~\cite[p. 49, 
Example 1]{top}.
\index{Andrews, George E.}

   The following two results can be used to simplify 
certain sums of two instances of Jacobi's triple product identity:
 
  \begin{oldresult}{Quintuple Product Identity.}
The quintuple product identity was perhaps
first stated in recognizable form by G. N. Watson~\cite{gnw}, and 
independently rediscovered by  
W. N. Bailey~\cite[p. 219, eqn. (2.1)]{wnb}.  However, as demonstrated by
Slater~\cite[pp. 204-205]{ljs:book}, it can be derived from a general theorem on
Weierstra\ss's $\sigma$-functions~\cite[p. 451, ex. 3]{ww}.  Weierstra\ss's
results on $\sigma$-functions can be found in~\cite{kw}.
\index{Watson, G. N.} \index{Bailey, W. N.} \index{Weierstra\ss, K.}
  \begin{gather}
    \prod_{j=1}^\infty (1+z^{-1} q^j)(1+z q^{j-1})(1- z^{-2} q^{2j-1})
       (1- z^2 q^{2j-1}) (1-q^j) \label{qpi}\\ =
    \prod_{j=1}^\infty (1-z^3 q^{3j-2}) (1-z^{-3} q^{3j-1}) (1-q^{3j})
  +z \prod_{j=1}^\infty (1-z^{-3} q^{3j-2}) (1-z^3 q^{3j-1}) (1-q^{3j}) 
\nonumber
  \end{gather} 
  or, in abbreviated notation, 
  \begin{gather*}
    (z^3 q, z^{-3} q^2, q^3;q^3)_\infty +
    z (z^{-3} q, z^3 q^2, q^3; q^3)_\infty =
    (-z^{-1} q, -z, q; q)_\infty (z^{-2} q, z^2 q; q^2)_\infty. 
  \end{gather*}
\end{oldresult} 

 \index{Bailey, W. N.}
Next, an identity of 
W. N. Bailey~\cite[p. 220, eqn. (4.1)]{wnb}:
 \begin{gather}
  \prod_{j=1}^\infty (1+z^2 q^{4j-3})(1+z^{-2} q^{4j-1}) (1-q^{4j})
+z\prod_{j=1}^\infty (1+z^2 q^{4j-1})(1+z^{-2} q^{4j-3}) (1-q^{4j}) 
\label{btp}\\ =
  \prod_{j=1}^\infty (1+z   q^{ j-1})(1+z^{-1} q^{ j  }) (1-q^j  )
\nonumber
  \end{gather}
  or, in abbreviated notation,
  \begin{gather*}
    (-z^2 q, -z^{-2} q^3, q^4; q^4)_\infty 
  +z(-z^2 q^3, -z^{-2} q , q^4;q^4)_\infty  =
  (-z, -z^{-1}q,q;q)_\infty. 
 \end{gather*}  

We will also require the following result: \index{Abel's lemma}

  \begin{oldresult}{Abel's Lemma.} 
\cite[p. 57]{ww} or \cite[p. 190]{nt}.
 If $\lim_{n\to\infty} a_n = L$, then 
\begin{equation} \label{abel}
  \lim_{t\to 1^{-}} (1-t)\sum_{n=0}^\infty a_n t^n = L.
\end{equation}
\end{oldresult}

And finally, the various forms of the Heine transformation of
$_2\phi_1$ basic hypergeometric series are given
below. \index{Heine's transformation}

  \begin{oldresult}{Heine's Transformations.}\index{Heine, E.} 
See Gasper and Rahman~\cite[p. 241, eqns. (III.1), (III.2), (III.3)]{gr}
For $|q|<1$, $|z|<1$, and $|b|<1$, 
\begin{eqnarray}
  \sum_{j=0}^\infty \frac {(a;q)_j (b;q)_j}{(c;q)_j (q;q)_j} z^j
  &=& \frac{(b,az;q)_\infty}{(c,z;q)_\infty}  
    \sum_{j=0}^\infty \frac{ (c/b;q)_j (z;q)_j} {(az;q)_j (q;q)_j} b^j 
     \label{heine1}\\
  &=& \frac{ (c/b,bz;q)_\infty}{ (c,z;q)_\infty}
    \sum_{j=0}^\infty \frac{ (abz/c;q)_j (b;q)_j}{(bz;q)_j (q;q)_j} (c/b)^j
     \label{heine2}\\
  &=& \frac{(abz/c;q)_\infty}{(z;q)_\infty}       
    \sum_{j=0}^\infty \frac{(c/a;q)_j (c/b;q)_j} {(c;q)_j (q;q)_j} (abz/c)^j
      \label{heine3}
\end{eqnarray}      
\end{oldresult}   

\subsection{Rogers-Ramanujan Type Identities}

The Rogers-Ramanujan identities (in their analytic form) 
may be stated as follows:
\index{Rogers-Ramanujan identities|(}

 \begin{oldresult}{Rogers-Ramanujan Identities---analytic.}
Due to L. J. Rogers, 1894. \index{Rogers, L. J.}\index{Ramanujan, S.}
If $|q|<1$, then
\begin{equation}
  \sum_{j=0}^\infty \frac{q^{j^2}}{(q;q)_j} 
= \prod_{j=0}^\infty \frac{1}{(1-q^{5j+1})(1-q^{5j+4})} \label{RR1}
\end{equation}
and
\begin{equation} 
  \sum_{j=0}^\infty \frac{q^{j(j+1)}}{(q;q)_j} 
= \prod_{j=0}^\infty \frac{1}{(1-q^{5j+2})(1-q^{5j+3})}. \label{RR2}
\end{equation}
\end{oldresult}

The Rogers-Ramanujan identities are also of interest in combinatorics.  
The series and products in the above theorem are generating functions for 
certain classes of integer partitions.  (A {\em partition} of a 
nonnegative integer $n$ is an unordered representation of $n$ into 
positive integral summands.  Each summand is called a {\em part} of the 
partition.) \index{partition (of an integer)}
Indeed, MacMahon~\cite[Chapter 3]{pam} \index{MacMahon, P. A.}
realized by 1918 that
the Rogers-Ramanujan identities may be stated combinatorially as 
follows:

 \begin{oldresult}{First Rogers-Ramanujan Identity--combinatorial.}
The number of partitions of an integer $n$ into distinct, nonconsecutive parts 
equals the number of partitions of $n$ into parts congruent to 
$1$ or $4 \pmod{5}$.
\end{oldresult}

 \begin{oldresult}{Second Rogers-Ramanujan Identity--combinatorial.}
 The number of partitions of an integer $n$ into distinct, 
nonconsecutive parts, all 
of which are at least $2$
equals the number of partitions of $n$ into parts congruent to $2$ or $3
\pmod{5}$.
\end{oldresult}

 By 1980, physicist Rodney Baxter\index{Baxter, Rodney} had discovered that the 
Rogers-Ramanujan identities were intimately linked to his solution of the 
hard hexagon model in statistical mechanics.   His results appear
in \cite{rjb:hh}, \cite{rjb:rr} and \cite{rjb:esm}. The version of the 
Rogers-Ramanujan identities preferred by physicists is given next.

 \begin{oldresult}{Rogers-Ramanujan Identities---fermionic/bosonic.}
If $|q|<1$, then
\begin{equation} \label{RR1sm}
  \sum_{j=0}^\infty \frac{q^{j^2}}{(q;q)_j} 
= \frac{1}{(q;q)_\infty} \sum_{j=-\infty}^\infty 
   \left( q^{j(10j+1)} - q^{(5j+2)(2j+1)} \right). 
\end{equation}
and
\begin{equation} \label{RR2sm}
  \sum_{j=0}^\infty \frac{q^{j(j+1)}}{(q;q)_j} 
  = \frac{1}{(q;q)_\infty} \sum_{j=-\infty}^\infty 
   \left( q^{j(10j+3)} - q^{(5j+1)(2j+1)} \right). 
\end{equation}
\end{oldresult}

In the language of the physicists, the left hand sides of (\ref{RR1sm}) and
(\ref{RR2sm}) are called ``fermionic" representations,
\index{fermionic representation} and the 
right-hand sides are called ``bosonic" 
\index{bosonic representation} representations.  For convenience, 
I will adopt this terminology, and use it througout this 
paper.  
The equality of the right-hand sides of (\ref{RR1}) and (\ref{RR1sm}) and 
the equality of the right-hand sides of (\ref{RR2}) and (\ref{RR2sm}) are
direct consequences of (\ref{jtp}) and (\ref{btp}).

Around the same time as Baxter was working on the hard hexagon model, 
it was discovered by Lepowski and Wilson~\cite{lw} 
\index{Wilson, Robert L.}\index{Lepowski, J.}
that Rogers-Ramanujan 
identities have a Lie theoretic interpretation and proof.

There are many series-product identites which resemble the Rogers-Ramanujan 
identities in form, and are thus called ``identities of the 
Rogers-Ramanujan type."  The seminal papers in the subject from an 
analytic viewpoint include L. J. Rogers' papers from 1894~\cite{ljr:mem2} and 
1917~\cite{ljr:1917}, 
\index{Rogers, L. J.} 
F. H. Jackson's 1928 paper~\cite{fhj}, 
\index{Jackson, F. H.}
and W. N. Bailey's papers of 1947~\cite{wnb:comb} and 1949~\cite{wnb:rrtype}.
\index{Bailey, W. N.}
Around 1950, Lucy J. Slater,
\index{Slater, Lucy J.}
a student of W. N. Bailey, produced a
list of 130 identities of the Rogers-Ramanujan type as a part of her Ph.D. 
thesis and published them in~\cite{ljs}.
\index{Slater's list of Rogers-Ramanujan type identities}  
An annotated 
version of Slater's list is included as Appendix~1.  
Much of the early 
history of the Rogers-Ramanujan identities is discussed by Hardy 
in~\cite{ghh}.  \index{Hardy, G. H.}
Andrews outlines much of the history through 1970
in~\cite{scripta}.\index{Andrews, George E.}

The ideas (now known as the ``Bailey pair," ``Bailey's Lemma," and 
the ``Bailey transform," see~\cite[Chapter 3]{qs}) that proved central to the 
discovery of large numbers of 
Rogers-Ramanujan type identities is due to Bailey~\cite{wnb:comb}  
and was exploited extensively by Slater in~\cite{ljs}.  The full iterative potential of 
Bailey's Lemma (dubbed ``Bailey chains" by Andrews), 
\index{Andrews, George E.}
was explored by Peter Paule \index{Paule, Peter}
in~\cite{pp:thesis} and \cite{pp:rrt} and by Andrews~\cite{multiRR}.
\index{Andrews, George E.}
\index{Bailey, W. N.}
\index{Slater, L. J.}

Seminal contributions to the combinatorial aspect of Rogers-Ramanujan type 
identites were made by 
\index{Schur, I.} 
I. Schur~(\cite{is:rr}
and \cite{is:mod6}), P. A. 
MacMahon~\cite{pam}, \index{MacMahon, P. A.}
H. G{\"o}llnitz~\cite{hg}, \index{G\"ollnitz, H.} 
and  B. Gordon~(\cite{bg:rr} \index{Gordon, Basil} and 
\cite{bg:gg}).
H. L. Alder~\cite{hla} provided a nice survey article of Rogers-Ramanujan 
history from the partition theoretic 
viewpoint. \index{Alder, H. L.}

Besides the contributions of Baxter listed above, other seminal 
contributions to the entry of the Rogers-Ramanujan identites into physics 
were made by Andrews, Baxter and Forrester~\cite{abf:8v} and \cite{fb}, and 
by the Kyoto group~\cite{kyoto}. 
\index{Andrews, George E.}
\index{Baxter, Rodney}
\index{Forrester, P. J.}
Starting in the 1990's, Alexander Berkovich and
Barry McCoy (\cite{bm:cf} and \cite{bm:ABgen}), sometimes
jointly with William Orrick \cite{bmo:poly}
or Anne Schilling \cite{bms}, along with Ole Warnaar~(\cite{sow:gbc},
\cite{sow:qti}, \cite{sow:rqtc}),
made significant contributions to the
study of Rogers-Ramanujan type identities via various models from statistical
mechanics.
In~\cite{bm:rri}, Berkovich and McCoy present a history 
from the 
viewpoint of physics.
\index{Rogers-Ramanujan identities|)} 
\index{Berkovich, Alexander}
\index{McCoy, Barry M.}
\index{Warnaar, S. Ole}

\section{Finitization of Rogers-Ramanujan Type Identities}\label{finit}

\subsection{The Method of $q$-Difference Equations} \label{qdiffeqn}
\index{q-difference equations|(}
We now turn our attention to a method for discovering finite analogs of 
Rogers-Ramanujan type identities via $q$-difference equations.
I have automated much of the process on the computer algebra
system Maple, in a package entitled ``\texttt{RRtools}," 
which is documented in~\cite{avs:RRtools}.
\index{RRtools (Maple package)}
\index{Sills, Andrew V.}

 The method of $q$-difference equations was pioneered by Andrews 
in \cite[\S 9.2, p. 88]{qs}:\index{Andrews, George E.}
We begin with an identity of the Rogers-Ramanujan type
\[ \phi(q) = \Pi(q) \]
where $\phi(q)$ is the series and $\Pi(q)$ is an infinite product or sum of 
several infinite products.  We consider a two variable generalization 
$f(q,t)$ which satisfies the following three conditions:
\begin{conditions}\label{cond}
\hfil\break
 \begin{enumerate}
  \item $f(q,t) = \sum_{n=0}^\infty P_n(q) t^n$ where the $P_n(q)$ are 
  polynomials,
  \item $\phi(q) = \lim_{t\to 1^-} (1-t)f(q,t) = 
     \lim_{n\to\infty} P_n(q) = \Pi(q) $, and
  \item $f(q,t)$ satisfies a nonhomogeneous $q$-difference 
  equation of the form
    \[ f(q,t) = R_1(q,t) + R_2(q,t) f(q,tq^k) \]
where $R_i(q,t)$ are rational functions of $q$ and $t$ for $i = 1,2$ and
$k\in\mathbb Z_+$.
\end{enumerate}
\end{conditions}

\begin{thm}
If $\phi(q)$ is written in the form
\[ \sum_{j=0}^\infty \frac{(-1)^{aj} q^{bj^2 + cj} 
          \prod_{i=1}^r (d_i q^{e_i}; q^{k_i})_{j+l_i} }
       {(q^m;q^m)_j 
          \prod_{i=1}^s (\delta_i q^{\epsilon_i}; 
          q^{\kappa_i})_{j+\lambda_i}},
\]
where $a=0$ or $1$; $b,m\in\mathbb{Z}_+$;  $c\in\mathbb{Z}$; \\
$d_i = \pm 1$; $e_i, k_i\in\mathbb{Z}_+$, $l_i\in\mathbb{Z}$ 
  for $1\leqq i \leqq r$;\\ 
$\delta_i = \pm 1$; $\epsilon_i, \kappa_i\in\mathbb{Z}_+$;
$\lambda_i \in\mathbb{Z}$  for $1\leqq i\leqq s$;
then 
\[ f(q,t) = \sum_{j=0}^\infty \frac{(-1)^{aj} t^{2bj/g} q^{bj^2 + cj} 
          \prod_{i=1}^r (d_i t^{k_i/g} q^{e_i}; q^{k_i})_{j+l_i} }
       {(t;q^m)_{j+1} 
          \prod_{i=1}^s (\delta_i t^{\kappa_i/g} q^{\epsilon_i}; 
          q^{\kappa_i})_{j+\lambda_i}}, \]
where $g = gcd(m, k_1, k_2, \dots, k_r, \kappa_1, \kappa_2, \dots, \kappa_s)$ 
is a two variable generalization of $\phi(q)$ which
satisfies Conditions~\ref{cond}.
\end{thm}
\begin{proof}  First, we will demonstrate that $f(q,t)$ satisfies 
condition (3):
\begin{eqnarray*}
f(q,t)
&=& \frac{\prod_{i=1}^r (d_i t^{k_i/g} q^{e_i};q^{k_i})_{l_i}}
     {(1-t)\prod_{i=1}^s (\delta_i t^{\kappa_i/g} q^{\epsilon_i}; 
          q^{\kappa_i})_{\lambda_i}}  \\ & & \quad +
  \sum_{j=1}^\infty \frac{(-1)^{aj} t^{2bj/g} q^{bj^2 + cj} 
          \prod_{i=1}^r (d_i t^{k_i/g} q^{e_i}; q^{k_i})_{j+l_i} }
       {(t;q^m)_{j+1} 
          \prod_{i=1}^s (\delta_i t^{\kappa_i/g} q^{\epsilon_i}; 
          q^{\kappa_i})_{j+\lambda_i}} \\
&=&\frac{\prod_{i=1}^r (d_i t^{k_i/g} q^{e_i};q^{k_i})_{l_i}}
     {(1-t)\prod_{i=1}^s (\delta_i t^{\kappa_i/g} q^{\epsilon_i}; 
          q^{\kappa_i})_{\lambda_i}}  \\ & & \quad + 
  \sum_{j=0}^\infty \frac{(-1)^{aj+a} t^{2bj+2b/g} q^{bj^2 + 2bj + b + cj+c} 
          \prod_{i=1}^r (d_i t^{k_i/g} q^{e_i}; q^{k_i})_{j+l_i+1} }
       {(t;q^m)_{j+2} 
          \prod_{i=1}^s (\delta_i t^{\kappa_i/g} q^{\epsilon_i}; 
          q^{\kappa_i})_{j+\lambda_i+1}}\\
&=&\frac{\prod_{i=1}^r (d_i t^{k_i/g} q^{e_i};q^{k_i})_{l_i}}
     {(1-t)\prod_{i=1}^s (\delta_i t^{\kappa_i/g} q^{\epsilon_i}; 
          q^{\kappa_i})_{\lambda_i}} \\ & & \quad + 
 \frac{(-1)^a t^{2b/g} q^{b+c} \prod_{i=1}^r (1-d_i t^{k_i/g} q^{e_i})} 
   {(1-t)\prod_{i=1}^s (1-\delta_i t^{\kappa_i/g}q^{\epsilon_i})} 
     \\ & & \quad\quad \times
 \sum_{j=0}^\infty 
 \frac{(-1)^{aj} t^{2bj/g} q^{bj^2 + 2bj + cj} 
          \prod_{i=1}^r (d_i t^{k_i/g} q^{e_i + k_i}; q^{k_i})_{j+l_i} }
       {(tq^m;q^m)_{j+1} 
          \prod_{i=1}^s (\delta_i t^{\kappa_i/g} q^{\epsilon_i + \kappa_i}; 
          q^{\kappa_i})_{j+\lambda_i}}\\
&=& \frac{\prod_{i=1}^r (d_i t^{k_i/g} q^{e_i};q^{k_i})_{l_i}}
     {(1-t)\prod_{i=1}^s (\delta_i t^{\kappa_i/g} q^{\epsilon_i}; 
          q^{\kappa_i})_{\lambda_i}} \\ & & \quad + 
 \frac{(-1)^a t^{2b/g} q^{b+c} \prod_{i=1}^r (1-d_i t^{k_i/g} q^{e_i})} 
   {(1-t)\prod_{i=1}^s (1-\delta_i t^{\kappa_i/g}q^{\epsilon_i})} f(q,tq^g)
\end{eqnarray*}

 Now that we have an explicit formula for $f(q,t)$ and a non-homogeneous
$q$-difference equation (which is first order in $q^g$) satisfied by
$f(q,t)$, one can, after clearing out denominators, collect powers of $t$,
and see that the co\"efficient of $t^n$ is a polynomial $P_n(q)$, 
and thus condition (1) is satisfied. 

 Finally, condition (2) is satisfied as a direct consequence of
(\ref{abel}).
\end{proof}

The nonhomogeneous $q$-difference equation can be used to find 
a recurrence which the $P_n(q)$ satisfy, and thus a list of $P_0(q), 
P_1(q), \dots, P_N(q)$ can be produced for any $N$. 

 The fermionic representation of the 
finitization is obtained by expanding the rising $q$-factorials which 
appear in $f(q,t)$ using (\ref{qbc1}) and (\ref{qbc2}), 
changing variables so that the resulting power of $t$ is $n$, so that 
$P_n(q)$ can be seen as the co\"efficient of $t^n$.

 Obtaining the bosonic representation for $P_n(q)$ is trickier, and requires 
conjecturing the correct form.  Note that the \texttt{RRtools} 
\index{RRtools (Maple package)} Maple package contains a number of 
tools to aid the user in making an appropriate conjecture; 
see~\cite{avs:RRtools}. \index{Sills, Andrew V.}

 After the proposed polynomial identity is (correctly) conjectured, 
it can be proved by one of the 
techniques discussed in \S~\ref{proof}.

\subsection{A Detailed Example}
To serve as a prototypical example, we will examine the finitization
process on identity
(\ref{t-ljslist}.7) from Slater's list,\index{Slater, Lucy J.} an 
identity due to Euler:\index{Euler, L.}
\begin{equation} \label{ls7}
  \sum_{j=0}^\infty \frac{q^{j^2+j}}{(q^2;q^2)_j} 
  = \prod_{j=1}^\infty {(1+q^{2j})}
\end{equation}
Due to its extreme simplicity, it is hoped that the general method will
be made transparent.  There are various ``short cuts" which
will be not be exploited since such short cuts are not applicable in
more general settings.

The two variable generalization of the LHS of~(\ref{ls7}) is
\[ f(q,t) = \sum_{j=0}^\infty \frac{t^j q^{j^2 + j}} {(1-t)(tq^2;q^2)_j}. \]
Next, we produce the non-homogenous $q$-difference equation.  
The details of the calculation are 
written below.  
\begin{eqnarray*}
  f(q,t) & = &  \sum_{j=0}^\infty \frac{t^j q^{j^2 + j}} {(t;q^2)_{j+1}} \\
         & = & \frac{1}{1-t} + \sum_{j=1}^\infty \frac{t^j q^{j^2 + j}} 
         {(t;q^2)_{j+1}} \\
         & = & \frac{1}{1-t} + \sum_{j=0}^\infty \frac{t^{j+1} q^{(j+1)^2 + 
         (j+1)}} {(t;q^2)_{j+2}} \\
         & = & \frac{1}{1-t} + \frac{tq^2}{1-t}\sum_{j=0}^\infty 
         \frac{(tq^2)^{j} q^{j^2 + j}} {(tq^2;q^2)_{j+1}} \\
         & = & \frac{1}{1-t} + \frac{tq^2}{1-t} f(q,tq^2)
\end{eqnarray*}

Thus, the non-homogeneous $q$-difference equation satisfied by 
$f(q,t)$ is 
\begin{equation} \label{qd7}
 f(q,t) =  \frac{1}{1-t} + \frac{tq^2}{1-t} f(q,tq^2).
\end{equation}

Next, we find the sequence of polynomials $\{ P_n(q) \}_{n=0}^\infty$ as
follows:

 Clearing denominators in (\ref{qd7}) gives
\[ (1-t) f(q,t) =  1 + {tq^2} f(q,tq^2), \]
which is equivalent to 
\[ f(q,t) = 1 + t f(q,t) + tq^2 f(q,tq^2) .\]
Thus,
\begin{eqnarray*}
 \sum_{n=0}^\infty P_n(q) t^n & = & 
       1 +    t\sum_{n=0}^\infty P_n(q) t^n 
         + tq^2\sum_{n=0}^\infty P_n(q) (tq^2)^n \\
 & = & 1 +     \sum_{n=0}^\infty P_n(q) t^{n+1} 
         +     \sum_{n=0}^\infty P_n(q) t^{n+1} q^{2n+2} \\
 &=  & 1 +     \sum_{n=1}^\infty P_{n-1}(q)  t^n
         +     \sum_{n=1}^\infty q^{2n} P_{n-1}(q) t^n\\
 & = & 1 +     \sum_{n=1}^\infty (1+q^{2n}) P_{n-1}(q) t^n.
\end{eqnarray*} \index{q-difference equations|)}
We can read off from the last line that the polynomial sequence
$\{ P_n(q) \}_{n=0}^\infty$ satisfies the following recurrence relation:
\begin{gather*} P_0 (q) = 1 \\ 
P_n (q) = (1+q^{2n}) P_{n-1} \mbox{,\ \ if $n\geqq 1$.}
\end{gather*}
Note that for this example, since a first order recurrence was obtained,
$P_n(q)$ is expressible as a finite product, and thus in some sense, the
problem is done.  However, the overwhelming majority of the identities 
from Slater's list yield finitizations whose minimimal recurrence order
is greater than one, and thus \emph{not} expressible as a finite product.
In such cases, we must work harder to find a representation for $P_n(q)$
which can be seen to converge in a direct fashion to the RHS of the 
original identity.   Thus we continue the demonstration:

Now that a recurrence for the $P_n(q)$ is known, a finite 
list $\{ P_n(q) \}_{n=0}^N $ can be produced:  
\begin{eqnarray*}
P_0(q) &=& 1\\
P_1(q) &=& q^2 + 1\\
P_2(q) &=& q^6 + q^4 + q^2 + 1 \\
P_3(q) &=& q^{12} + q^{10} + q^{8} + 2q^6 + q^4 + q^2 + 1 \\
P_4(q) &=& q^{20} + q^{18} + q^{16} + 2q^{14} + 2q^{12} + 2q^{10} + 2q^8 + 
2q^6 + q^4 + q^2 + 1 
\end{eqnarray*}
Notice that the degree of $P_n(q)$ appears to be $n(n+1)$. Being familiar 
with Gaussian polynomials, \index{Gaussian polynomial}
we recall that the degree of $\gp{2n+1}{n+1}{q}$ is 
also $n(n+1)$ (by (\ref{gpdeg})), and wonder if Gaussian polynomials might play a fundamental 
r\^ole in the bosonic representation of $P_n(q)$.  Also, since 
\[ \Pi(q) = \frac{(q,q^3,q^4;q^4)_\infty}{(q;q)_\infty}\] 
(an instance of Jacobi's triple product identity
multiplied by $1/(q;q)_\infty$) and
\[ \lim_{n\to\infty} \gp{2n+1}{n+1}{q} = \frac{1}{(q;q)_\infty} 
\mbox{\quad(by \ref{gplim})}, \] we have 
further evidence in favor of the Gaussian polynomial $\gp{2n+1}{n+1}{q}$
playing a central r\^ole.  Using the method of successive approximations
by Gaussian polynomials discussed by Andrews and Baxter 
in~\cite{scratch}\footnote{Once again, this method is implemented as a 
procedure in
\texttt{RRtools}.},
one can conjecture that, at
least for small $n$, it is true that
\[ P_n(q) = \gp{2n+1}{n+1}{q} - q\gp{2n+1}{n+2}{q} - q^3\gp{2n+1}{n+3}{q} + 
q^6\gp{2n+1}{n+4}{q} + q^{10}\gp{2n+1}{n+5}{q} - \dots, \] 
which is a good start, but the bosonic representation must be a 
{\em bi}lateral series, i.e. a series where the index of summation runs over
{\em all} integers, not just the nonnegative integers.  Thus we employ 
(\ref{gpsym}) to rewrite the above as
\[ P_n(q) = \gp{2n+1}{n+1}{q} - q\gp{2n+1}{n-1}{q} - q^3\gp{2n+1}{n+3}{q} + 
q^6\gp{2n+1}{n-3}{q} + q^{10}\gp{2n+1}{n+5}{q} - \dots, \] which is 
equivalent to
\begin{equation} \label{ls7fR}
  P_n(q) = \sum_{j=-\infty}^\infty (-1)^j q^{2j^2 + j} \gp{2n+1}{n+2j+1}{q},
\end{equation}
which is in the desired (bosonic) form. 

  Obtaining the fermionic representation for $P_n(q)$ is more 
straightforward and does not involve any guesswork:
 \begin{eqnarray*}
   \sum_{n=0}^\infty P_n(q) t^n &=& f(q,t) \\
   & = &  \sum_{j=0}^\infty \frac{t^j q^{j^2 + j}} {(1-t)(tq^2;q^2)_j} \\ 
   & = &  \sum_{j=0}^\infty t^j q^{j^2 + j} 
      \sum_{k=0}^\infty \gp{j+k}{k}{q^2} t^k \mbox{\ by (\ref{qbc2})} \\
   & = & \sum_{j=0}^\infty \sum_{k=0}^\infty t^{j+k} q^{j^2 + j}
            \gp{j+k}{j}{q^2} \mbox{\ by (\ref{gpsym})} \\
   & = & \sum_{n=0}^\infty t^n \sum_{j=0}^\infty q^{j^2 + j} \gp{n}{j}{q^2}
           \mbox{\ (by taking $n=j+k$) }
 \end{eqnarray*}
By comparing co\"efficients of $t^n$ in the extremes, we find
\begin{equation} \label{ls7fL}
   P_n(q) = \sum_{j=0}^\infty q^{j^2 + j} \gp{n}{j}{q^2}.
\end{equation}

Combining (\ref{ls7fL}) and (\ref{ls7fR}), we obtain the 
conjectured polynomial identity

\begin{equation} \label{ls7f}
\sum_{j=0}^\infty q^{j^2 + j} \gp{n}{j}{q^2} = 
\sum_{j=-\infty}^\infty (-1)^j q^{2j^2 + j} \gp{2n+1}{n+2j+1}{q}
\end{equation}
To see that (\ref{ls7f}) is indeed a finitization of (\ref{ls7}), more 
calculations are needed:
\begin{eqnarray*}
 &   & \lim_{n\to\infty} \sum_{j=0}^\infty q^{j^2 + j} \gp{n}{j}{q^2} \\
 & = & \lim_{n\to\infty} \sum_{j=0}^\infty q^{j^2 + j} 
          \frac{(q^2;q^2)_n}{ (q^2;q^2)_j (q^2;q^2)_{n-j} } \\
 & = & \sum_{j=0}^\infty \frac{ q^{j^2 + j}}{(q^2;q^2)_j},
\end{eqnarray*}
and so the LHS of (\ref{ls7f}) converges to the LHS of (\ref{ls7}).
\begin{eqnarray*}
 &   & \lim_{n\to\infty} \sum_{j=-\infty}^\infty (-1)^j q^{2j^2 + j} 
          \gp{2n+1}{n+2j+1}{q} \\
 & = & \frac{1}{(q;q)_\infty} \sum_{j=-\infty}^\infty (-1)^j q^{2j^2 + j} 
       \mbox{\quad (by (\ref{gplim})) }\\
 & = & \frac{1}{(q;q)_\infty} \cdot (q,q^3,q^4;q^4)_\infty 
    \mbox{\quad (by ~\ref{jtp}) } \\
 & = & \prod_{j=1}^\infty (1+q^{2j}),
\end{eqnarray*}
and so the RHS of (\ref{ls7f}) converges to the RHS of (\ref{ls7}).   

\subsection{Another Example}
Next, let us consider Identity (\ref{t-ljslist}.81) from Slater's list,\index{Slater, Lucy J.}

\begin{equation} \label{ls81}
  \sum_{j=0}^\infty \frac{q^{j(j+1)/2}}{(q;q^2)_j(q;q)_j} 
  = \frac{(q,q^6,q^7;q^7)_\infty (q^5,q^9;q^{14})_\infty}
    {(q;q)_\infty/(-q;q)_\infty}
\end{equation}

The two variable generalization of the LHS of~(\ref{ls81}) is
\begin{equation} \label{f81}
 f(q,t) = \sum_{j=0}^\infty \frac{t^j q^{j(j+1)/2}} {(t^2q;q^2)_j
(t;q)_{j+1}}. 
\end{equation}
Thus, the first order non-homogeneous $q$-difference equations satisfied
by $f(q,t)$ is
\begin{equation} \label{qd81}
  f(q,t)  =  \frac{1}{1-t} + \frac{tq}{(1-t)(1-t^2q)} f(q,tq).
\end{equation}
As before, we clear denominators, solve for $f(q,t)$, and read off the
recurrence satisfied by the polynomials $P_n(q)$ where 
$f(q,t) = \sum_{n=0}^\infty P_n(q) t^n$:
\begin{gather*} P_0 (q) = 1 \\ 
P_1 (q) = q + 1\\
P_2 (q) = q^3 + q^2 + q + 1 \\
P_n (q) = (1+q^n) P_{n-1} + qP_{n-2} - qP_{n-3} \mbox{,\ \ if $n\geqq 3$.}
\end{gather*}
Now that a recurrence for the $P_n(q)$ is known, a finite 
list $\{ P_n(q) \}_{n=0}^N $ can be produced:  
\begin{eqnarray*}
P_0(q) &=& 1\\
P_1(q) &=& q + 1\\
P_2(q) &=& q^3 + q^2 + q + 1 \\
P_3(q) &=& q^6 + q^5 + q^4 + 2q^3 + 2q^2 + q + 1 \\
P_4(q) &=& q^{10} + q^9 + q^8 + 2q^7 + 3q^6 + 2q^5 + 3q^4 + 
3q^3 + 2q^2 + q + 1 
\end{eqnarray*}
Notice that the degree of $P_n(q)$ appears to be $n(n+1)/2$, so we may be
tempted to guess that the Gaussian polynomial $\gp{2n+1}{n+1}{\sqrt{q}}$,
plays the key r\^ole in the bosonic representation of $P_n(q)$, but let us
look further before jumping to conclusions.  Notice that the denominator of
the infinite product side of (\ref{ls81}) contains ${(q;q)_\infty / 
(-q;q)_\infty}$, which is the reciprocal of the limit of the $T_1$ 
trinomial co\"efficient (\ref{T1lim}), and \emph{not} that of the proposed Gaussian
polynomial.   Successive approximations of
$P_n(q)$ by the $T_1$ function for small $n$ leads to the conjecture that 
the bosonic representation of $P_n(q)$ is
  \[  P_n(q) = \sum_{k=-\infty}^\infty q^{(21k + 1)k/2} 
  \Tone{n+1}{7k}{\sqrt{q}} - q^{(21k + 13)k/2 + 1} 
  \Tone{n+1}{7k+2}{\sqrt{q}}. \]
As in the previous example, the fermionic representation is easily obtained 
from (\ref{f81}) by expanding each of the rising $q$-factorials by 
(\ref{qbc2}), and so we arrive at the conjectured identity
\begin{gather}\label{ls81f}
\sum_{j\geqq 0}\sum_{k\geqq 0} q^{j(j+1)/2 + k} \gp{j+k-1}{k}{q^2} 
\gp{n-2k}{j}{q} \\
= \sum_{k=-\infty}^\infty q^{(21k + 1)k/2} 
  \Tone{n+1}{7k}{\sqrt{q}} - q^{(21k + 13)k/2 + 1} 
  \Tone{n+1}{7k+2}{\sqrt{q}}. \nonumber
\end{gather}
To see that the RHS of (\ref{ls81f}) is indeed a finitization of the RHS 
of (\ref{ls81}), take the limit as $n\to\infty$ of the RHS of (\ref{ls81f}),
apply (\ref{T1lim}), followed by Jacobi's triple product identity 
(\ref{jtp}), and then the quintuple product identity (\ref{qpi}).  The 
analogous calculation for the LHS is straight forward.

\section{Polynomial Generalizations of the Identities\\ in Slater's List}
\label{t-SLfin}
\index{Slater's list of Rogers-Ramanujan type identities!polynomial 
generalizations of|(}
\index{Slater, Lucy J.|(}
\subsection{Introduction}
Listed below are finite analogs of the identities on Slater's 
list~\cite{ljs}, along 
with recurrences satisfied by the polynomials.  For easy reference the 
numbering scheme corresponds to that of Slater's list. In some cases, more 
than one bosonic representation was found, such as one that uses a 
$q$-binomial co\"efficient and the other a $q$-trinomial co\"efficient.  In 
these instances the equations numbers are suffixed with a ``b" and ``t", 
respectively.  

Many of these identities had been discovered previously.  
The identities related to Baxter's solution of the hard hexagon model from
statistical mechanics (\ref{t-SLfin}.14b, \ref{t-SLfin}.18b, 
\ref{t-SLfin}.79b, \ref{t-SLfin}.94b, \ref{t-SLfin}.96b, 
\ref{t-SLfin}.99b) were known to Andrews by 1981. 
\index{Andrews, George E.}
\index{Baxter, Rodney} By 1990, Andrews had 
bosonic $q$-trinomial representations for (\ref{t-SLfin}.14t), 
(\ref{t-SLfin}.18t), (\ref{t-SLfin}.34U), and (\ref{t-SLfin}.36U).
A number of the identities are special cases of identities due
Berkovich and McCoy, sometimes jointly with Orrick, which may be found
in~\cite{bm:cf}, \cite{bm:ABgen}, and
\cite{bmo:poly}.
\index{Berkovich, Alexander}
\index{McCoy, Barry M.}
\index{Orrick, William P.}
Note that the Berkovich-McCoy
identities first stated in~\cite{bm:cf} are proved in
a later paper written jointly with Anne 
Schilling~\cite{bms}.
\index{Schilling, Anne}
It is interesting to note that the bosonic forms of (3.2-b), (3.3-t), (3.11-b), (3.19),
(3.28), (3.29), (3.31), (3.32), (3.33), (3.49), (3.50-b), (3.54), (3.91), (3.92), (3.93), (3.120),
(3.121), (3.122), and (3.123) are specializations of bosonic forms from 
\cite{bm:cf}, \cite{bm:ABgen}, \cite{bmo:poly}, or \cite{bms}, 
where a hypothesis is violated, e.g. in $M(p,p')$ models, $p$ and $p'$ 
must be relatively
prime.

  In his Ph.D. thesis~\cite{jpos}, Santos 
\index{Santos, J. P. O.}
conjectured complete bosonic forms for 36 of the identities listed here, 
and provided proofs for three of them.  For an additional 16 of the 
identities, Santos provided a conjecture for the bosonic form for either 
even or odd values of $n$, but not both.  Santos did not provide fermionic 
representations.  Detailed references are provided with each of the 
previously known identities; the others are believed to be new.

\subsection{Identities} 
\begin{obs} Identity \ref{t-ljslist}.1 is Euler's pentagonal
number theorem, which is an instance of Jacobi's Triple Product
(Theorem~\ref{jtp} with $q$ replaced by $\sqrt{q}$ and $z=-q^{3/2}$).
The series is bilateral in its most natural form and thus the method
of $q$-difference equations is not applicable.  Nevertheless, there
are several known finitizations: e.g. 
Schur\cite[eqn. (30)]{is:rr}.
See also Paule and Riese~\cite[p. 22, eqn. (15)]{pr:qzeil} 
for a different finitization
due to Rogers. \index{Pentagonal Number Theorem!finite}
\index{Euler, L.}
\end{obs}

\begin{id}[Finite forms of \ref{t-ljslist}.2/7]
This identity is equivalent to an instance of the $q$-binomial
theorem: eqn. (\ref{qbc1}) with $q$ replaced by $q^2$ and $t=-q^2$.
 \begin{gather*}
   \sum_{j\geqq 0} q^{j^2+j} \gp{n}{j}{q^2} =
   \sum_{j=-\infty}^\infty (-1)^j q^{2j^2+j} \gp{2n+1}{n+2j+1}{q} 
   \tag{\ref{t-SLfin}.2-b} \\
 = \sum_{k=-\infty}^\infty q^{12k^2 + 2k} \Tone{n+1}{6k+1}{q} -
    q^{12k^2 + 10k + 2} \Tone{n+1}{6k+3}{q} \tag{\ref{t-SLfin}.2-t} \\
    = (-q^2;q^2)_n \tag{\ref{t-SLfin}.2-p}\\
  P_0 = 1,\\ 
  P_n = (1+q^{2n}) P_{n-1} \mbox{\ if $n\geqq 1$.} \\
 \end{gather*}
\end{id}

\begin{id}[Finite forms of \ref{t-ljslist}.3/23 (with $q$ replaced by $-q$)]
Identity \ref{t-SLfin}.3-b is 
an instance of the $q$-binomial theorem: (\ref{qbc1}) with $q$ 
replaced by $q^2$ and $t=-q$.
\index{Berkovich, Alexander}
\index{McCoy, Barry M.}
\index{Orrick, William P.}
 \begin{gather*}
    \sum_{j\geqq 0} q^{j^2} \gp{n}{j}{q^2} =
    \sum_{j=-\infty}^\infty (-1)^j q^{2j^2} \gp{2n}{n+2j}{q} 
    \tag{\ref{t-SLfin}.3-b} \\
   = \sum_{j=-\infty}^\infty (-1)^j q^{3j^2 + j}
       \U{n}{3j}{q}\tag{\ref{t-SLfin}.3-t}\\
   = (-q;q^2)_n  \tag{\ref{t-SLfin}.3-p}\\
  P_0 = 1,\\ 
  P_n = (1+q^{2n-1}) P_{n-1} \mbox{\ if $n\geqq 1$.} \\
 \end{gather*}
\end{id}

\begin{id}[Finite form of \ref{t-ljslist}.4]
\begin{gather*}
   \sum_{i\geqq 0} \sum_{j\geqq 0} \sum_{k\geqq 0} 
     (-1)^{j+k} q^{i^2 + j^2 + 2k} 
     \gp{j}{i}{q^2} \gp{j+ k -1}{k}{q^2}  \gp{n-i-k}{j}{q^2} \\=
   \sum_{j=-\infty}^\infty (-1)^j q^{j^2} \U{n-1}{j}{q} \tag{\ref{t-SLfin}.4} \\
    P_0 = 1,\\
    P_1 = -q + 1,\\ 
    P_n = (1-q^2 - q^{2n-1}) P_{n-1} +(q^2-q^{2n-2}) P_{n-2} 
      \mbox{\ if $n\geqq 2$.}
 \end{gather*}
\end{id}

\begin{id}[Finite forms of \ref{t-ljslist}.5/9/84]
Identity \ref{t-SLfin}.5-b is a special case of an identity due to 
Berkovich and McCoy~\cite[p. 59, eqn. (3.14) with $L=2n$, $p=3$, $p'=4$,
$r=3/2$, $s=4/3$, $a=b=2$]{bm:cf}.
\index{Berkovich, Alexander}
\index{McCoy, Barry M.}
  \begin{gather*}
   \sum_{j\geqq 0} q^{2j^2 + j} \gp{n+1}{2j+1}{q} =
   \sum_{j=-\infty}^\infty (-1)^j q^{3j^2 + j} \gp{2n}{n+2j}{q} 
   \tag{\ref{t-SLfin}.5-b} \\
 = \sum_{k=-\infty}^\infty q^{6k^2 + k} \Tone{n+1}{6k+1}{\sqrt{q}} -
   \sum_{k=-\infty}^\infty q^{6k^2 + 5k + 1} \Tone{n+1}{6k+3}{\sqrt{q}}
   \tag{\ref{t-SLfin}.5-t} \\
   = (-q;q)_n \tag{\ref{t-SLfin}.5-p} \\
   P_0 = 1,\\
   P_n = (1+q^n) P_{n-1}  
      \mbox{\ if $n\geqq 1$.}\\
 \end{gather*}
\end{id}

\begin{id}[Finite form of \ref{t-ljslist}.6] 
 \begin{gather*}   \sum_{i\geqq 0}\sum_{j\geqq 0} \sum_{k\geqq 0}
   q^{j^2 + i(i-1)/2 + k} 
   \gp{j-1}{i}{q} 
   \gp{j+k-1}{k}{q^2} \gp{n-j-i-2k}{j}{q}  \tag{\ref{t-SLfin}.6}
 \end{gather*}
\[ = \left\{ \begin{array}{ll}
 \sum_k q^{6 k^2 +  k     } \gp{2m}  {m + 3k    }{q}
     +  q^{6 k^2 + 5k + 1} \gp{2m-1}{m + 3k + 1}{q}
 & \mbox{if $n=2m$,}\\
  \sum_k q^{6 k^2 +  k    } \gp{2m}  {m + 3k    }{q}
       + q^{6 k^2 + 5k + 1} \gp{2m+1}{m + 3k + 2}{q}
 & \mbox{if $n=2m+1$}
\end{array}
\right.
\]
\begin{gather*}
 P_0 = P_1 = 1\\
 P_n = P_{n-1} + (q + q^{n-1}) P_{n-2} + (q^{n-2} - q) P_{n-3} \mbox{\ if 
 $n\geqq 3$}
\end{gather*} 
\end{id}

\begin{obs}
Identity $(7)$ is $(2)$ with $q$ replaced by $q^2$.
\end{obs}

\begin{id}[Finite form of \ref{t-ljslist}.8]  
Due to Berkovich, McCoy and Orrick, 
\cite[p. 805, eqn. (2.34) with $L=n+1$, $\nu=2$, $s'=1$, 
$r'=0$]{bmo:poly} 
\index{Berkovich, Alexander}
\index{McCoy, Barry M.}
\index{Orrick, William P.}
 \begin{gather*}
   \sum_{j\geqq 0} \sum_{k\geqq 0} q^{j(j+1)/2 + k(k+1)/2}
     \gp{j}{k}{q} \gp{n-k}{j}{q} 
  =\sum_{j=-\infty}^\infty (-1)^j q^{2j^2 + j} \Tone{n+1}{4j+1}{\sqrt{q}} 
  \tag{\ref{t-SLfin}.8} \\
  P_0 = 1\\
  P_1 = q + 1\\
  P_n = (1 + q^n) P_{n-1} + q^n P_{n-2} \mbox{\ if $n\geqq 2$.} 
 \end{gather*}
\end{id}

\begin{obs}
Identity $(9)$ is $(5)$ with $q$ replaced by $-q$.
\end{obs}

\begin{id}[Finite form of \ref{t-ljslist}.10/47] 
 \begin{gather*}
   \sum_{i\geqq 0}\sum_{j\geqq 0}\sum_{k\geqq 0}
       q^{j^2 + i^2 - i + k}
       \gp{j}{i}{q^2}\gp{j+k-1}{k}{q^2} \gp{n-i-k}{j}{q^2} \\
    = \sum_{j=-\infty}^\infty q^{2j^2 + j} 
          \Big[ \Tzero{n}{2j}{q} + \Tzero{n-1}{2j}{q} \Big] 
   \tag{\ref{t-SLfin}.10} \\
  P_0 = 1 \\
  P_1 = q + 1\\
  P_n = (1+q+q^{2n-1}) P_{n-1} + (q^{2n-3} - q) P_{n-2} \mbox{\ if $n\geqq 2$.} 
\end{gather*}
\end{id}

\begin{id}[Finite forms of \ref{t-ljslist}.11/51/64] 
Bosonic $q$-binomial representation conjectured by 
Santos~\cite[p. 74, eqn. 6.30]{jpos}.
Identity \ref{t-SLfin}.11-b is a special case of an identity due to 
Berkovich and McCoy~\cite[p. 59, eqn. (3.14) with $L=2n+1$, $p=4$, $p'=6$,
$r=s=2$, $a=4$, $b=2$]{bm:cf}.
\index{Santos, J. P. O.}
\index{Berkovich, Alexander}
\index{McCoy, Barry M.}
  \begin{gather*}
     \sum_{j\geqq 0} \sum_{k\geqq 0} q^{j^2 + j + k^2} \gp{j}{k}{q^2}
    \gp{n+j-k+1}{2j+1}{q}
  = \sum_{j=-\infty}^\infty (-1)^j q^{6j^2 + 2j} \gp{2n+1}{n + 3j + 1}{q}
  \tag{\ref{t-SLfin}.11-b}\\
  = \sum_{j=-\infty}^\infty q^{4j^2 + 2j} \U{n}{2j}{q} \tag{\ref{t-SLfin}.11-t} \\
  P_0 = 1\\
  P_1 = q^2 + q + 1\\
  P_n = (1 + q + q^{2n}) P_{n-1} + (q^{2n-1} - q) P_{n-2} 
\mbox{\ if $n\geqq 2$.}\\
  \end{gather*}
\end{id}

\begin{id}[Finite form of \ref{t-ljslist}.12]
Bosonic representation conjectured by 
Santos~\cite[p. 64, eqn. 6.2]{jpos}.  
\index{Santos, J. P. O.}
The identity is a special case of
Berkovich, McCoy and Orrick~\cite[p. 805, eqn. (2.34) with $L=n+1$,
$\nu=2$, $r'=s'=0$]{bmo:poly}.
\index{Berkovich, Alexander}
\index{McCoy, Barry M.}
\index{Orrick, William P.}
 \begin{gather*}
   \sum_{j\geqq 0} \sum_{k\geqq 0} q^{j(j+1)/2 + k(k-1)/2} 
     \gp{j}{k}{q} \gp{n-k}{j}{q} 
  = \sum_{j=-\infty}^\infty (-1)^j q^{2j^2} \Tone{n+1}{4j+1}{\sqrt{q}}
\tag{\ref{t-SLfin}.12} \\
  P_0 = 1\\
  P_1 = q + 1\\
  P_n = (1 + q^n) P_{n-1} + q^{n-1} P_{n-2} \mbox{\ if $n\geqq 2$.}
 \end{gather*}
\end{id}

\begin{id}[Finite form of \ref{t-ljslist}.13]
Note:  If $P_{i,n}$ represents the polynomial sequence for identity
$\ref{t-SLfin}.i$ for $i=8, 12, 13$, then
\[ P_{13,n} = P_{12,n-1} + P_{8,n-1} + q^{n-1} P_{12,n-2}. \]
\begin{gather*}
   \sum_{j\geqq 0} \sum_{k\geqq 0} q^{j(j - 1)/2 + k(k+1)/2} 
     \gp{j}{k}{q} \gp{n-k}{j}{q} \\
  = \sum_{j=-\infty}^\infty 
     (-1)^j q^{2j^2} (1+q^j) \Tone{n}{4j+1}{\sqrt{q}}
    +(-1)^j q^{2j^2+n-1} \Tone{n-1}{4j+1}{\sqrt{q}}
\tag{\ref{t-SLfin}.13} \\
  P_0 = 1\\
  P_1 = 2\\
  P_n = (1 + q^{n-1}) P_{n-1} + q^{n-1} P_{n-2} \mbox{\ if $n\geqq 2$.}
 \end{gather*}
\end{id}

\begin{id}[Finite forms of the 2nd Rogers-Ramanujan Identity]
\index{Rogers-Ramanujan identities!MacMahon-Schur finitization}
\index{MacMahon, P. A.}
\index{Schur, I.}
\index{Andrews, George E.}
\index{Berkovich, Alexander}
\index{McCoy, Barry M.}
Fermionic representation due to MacMahon~\cite{pam} 
Bosonic $q$-binomial representation $(14$-b$)$ due to Schur~\cite{is:rr}.
Bosonic $q$-trinomial representation $(14$-t$)$ due to 
Andrews~\cite[p. 5. eqn. 1.12]{qtrr}.   Identity \ref{t-SLfin}.14-b is 
subsumed as a special case of Berkovich and McCoy~\cite[p. 59, eqn. (3.14) 
with $p=2$, $p'=5$,
$r=s=1$, $a=4$; $b=2,3$]{bm:cf}.
 \begin{gather*}
   \sum_{j\geqq 0} q^{j(j+1)} \gp{n-j}{j}{q} =
   \sum_{j=-\infty}^\infty (-1)^j q^{j(5j+3)/2} 
     \gp{n+1}{\lfloor\frac{n+5j+3}{2}\rfloor}{q} \tag{\ref{t-SLfin}.14-b}\\
  = \sum_{k=-\infty}^\infty q^{10k^2 +  3k    } \binom{n+1, 5k+1; q}{5k+1}_2
                           -q^{10k^2 +  7k + 1} \binom{n+1, 5k+2; q}{5k+2}_2
  \tag{\ref{t-SLfin}.14-t} \\
  P_0 = P_1 = 1\\
  P_n = P_{n-1} + q^n P_{n-2} \mbox{\ if $n\geqq 2$.}
  \end{gather*}
\end{id}

\begin{id}[Finite form of \ref{t-ljslist}.15 (with $q$ replaced by $-q$)]
\begin{gather*}
  \sum_{j\geqq 0}\sum_{k\geqq 0}\sum_{l\geqq 0} (-1)^l q^{3j^2 - 2j + k + 2l}
    \gp{j+k-1}{k}{q^2} \gp{j+l-1}{l}{q^2} \gp{n-2j-k-l}{j}{q^2} \nonumber \\
= \sum_{j=-\infty}^\infty 
  (-1)^j q^{10 j^2 + 3j    } \gp{n-2}{\lfloor \frac{n+5j-1}{2}\rfloor}{q^2}
+ (-1)^j q^{10 j^2 + 7j + 1} \gp{n-2}{\lfloor \frac{n + 5j}{2}\rfloor}{q^2}
\tag{\ref{t-SLfin}.15} \\
  P_0 = P_1 = P_2 = 1\\
  P_n = (1 + q - q^2) P_{n-1} + (q^3 + q^2 - q) P_{n-2} + (q^{2n-5} - q^3) 
  P_{n-3} \mbox{\ if $n\geqq 3$.}
\end{gather*}
\end{id}

\begin{id}[Finite form of \ref{t-ljslist}.16] 
Bosonic representation for odd $n$ conjectured by 
Santos~\cite[p. 65, eqn. 6.4]{jpos}
\index{Santos, J. P. O.}
  \begin{gather*}
    \sum_{j\geqq 0} \sum_{k\geqq 0} (-1)^k q^{j^2 + 2j + 2k} \gp{j+k-1}{k}{q^2} 
    \gp{n-k}{j}{q^2} \\  = 
    \sum_{k=-\infty}^\infty q^{10k^2 + 3k}    \U{n+1}{5k}{q} -
                            q^{10k^2 + 7k + 1}\U{n+1}{5k+1}{q} 
\tag{\ref{t-SLfin}.16} \\
  P_0 = 1\\
  P_1 = q^3 + 1\\
  P_n = (1 - q^2 + q^{2n+1}) P_{n-1} + q^{2} P_{n-2} \mbox{\ if $n\geqq 2$.}
 \end{gather*}
\end{id}

\begin{id}[Finite form of \ref{t-ljslist}.17] 
Bosonic representation conjectured for odd $n$ by 
Santos~\cite[p. 65, eqn. 6.6]{jpos}
\index{Santos, J. P. O.}
 \begin{gather*}
    \sum_{j\geqq 0} \sum_{k\geqq 0} (-1)^k q^{j^2+j+k}
     \gp{j+k}{j}{q^2} \gp{n-k}{j}{q^2} =
   \sum_{j=-\infty}^\infty (-1)^j q^{j(5j+3)/2} \Tone{n+1}{\lfloor       
      \frac{5j+2}{2} \rfloor}{q} \tag{\ref{t-SLfin}.17} \\
  P_0 = 1\\
  P_1 = q^2 - q + 1\\
  P_n = (1 - q + q^{2n}) P_{n-1} + q P_{n-2} \mbox{\ if $n\geqq 2$.}
 \end{gather*}
\end{id}

\begin{id}[Finite forms of the 1st Rogers-Ramanujan Identity]
\index{Rogers-Ramanujan identities!MacMahon-Schur finitization}
Fermionic representation due to MacMahon~\cite{pam}.  
Bosonic $q$-binomial representation $(18$-b$)$ due to Schur~\cite{is:rr}.
Bosonic $q$-trinomial representation 
$(18$-t$)$ due to Andrews~\cite[p. 5, eqn. 
1.11]{qtrr}.   Identity \ref{t-SLfin}.18-b is
subsumed as a special case of Berkovich and McCoy~\cite[p. 59, eqn. (3.14)
with $p=2$, $p'=5$,
$r=1$, $s=2$, $a=3$; $b=2,3$]{bm:cf}.

\index{MacMahon, P. A.}
\index{Schur, I.}
\index{Andrews, George E.}
\index{Berkovich, Alexander}
\index{McCoy, Barry M.}
  \begin{gather*}
       \sum_{j\geqq 0} q^{j^2} \gp{n-j}{j}{q}
     = \sum_{j=-\infty}^\infty (-1)^j q^{j(5j+1)/2} \gp{n}{\lfloor 
       \frac{n+5j+1}{2} \rfloor}{q} \tag{\ref{t-SLfin}.18-b}\\ 
     = \sum_{k=-\infty}^\infty  q^{10k^2 +  k}      \binom{n,5k;q}{5k}_2 
                               -q^{10k^2 + 9k + 2} \binom{n,5k+2;q}{5k+2}_2 
      \tag{\ref{t-SLfin}.18-t} \\
  P_0 = P_1 = 1\\
  P_n = P_{n-1} + q^{n-1} P_{n-2} \mbox{\ if $n\geqq 2$.}
  \end{gather*}
\end{id}

\begin{id}[Finite form of \ref{t-ljslist}.19] 
Bosonic representation conjectured for even $n$ by 
Santos~\cite[p. 66, eqn. 6.7]{jpos}.
\index{Santos, J. P. O.}
\begin{gather*} 
  \sum_{j\geqq 0}\sum_{k\geqq 0}\sum_{l\geqq 0}
    (-1)^{j+k+l} q^{3j^2 + k + 2l} \gp{j+k-1}{k}{q^2}
    \gp{j+l-1}{l}{q^2} \gp{n-2j-k-l}{j}{q^2} \tag{\ref{t-SLfin}.19}\\
 = \sum_{j=-\infty}^\infty (-1)^j q^{j(5j+1)/2} 
     \gp{n-1}{\lfloor \frac{2n+5}{4}\rfloor}{q^2}
 \end{gather*}

 \begin{gather*}
 P_0 = P_1 = P_2 = 1 \\
 P_n = (1-q-q^2) P_{n-1} + (q+q^2 -q^3) P_{n-2} + (q^3 - q^{2n-3}) P_{n-3} 
\mbox{\  if $n\geqq 3$}
\end{gather*}
\end{id}

\begin{id}[Finite Form of \ref{t-ljslist}.20] 
Bosonic representation due Santos~\cite[p. 23ff, Lemmas 2.7 and 2.8]{jpos}.
The proof also appears in Santos~\cite{jpos:rr}.
\index{Santos, J. P. O.}  
\begin{gather*}
    \sum_{j\geqq 0} \sum_{k\geqq 0}
    (-1)^k q^{j^2 + 2k} \gp{j+k-1}{k}{q^2} \gp{n-k}{j}{q^2}  \\ =
  \sum_{k=-\infty}^\infty q^{10k^2+k} \U{n}{5k}{q} - q^{10k^2 + 9k +2}
    \U{n}{5k+1}{q} \tag{\ref{t-SLfin}.20}  \\
  P_0 = 1\\
  P_1 = q + 1\\
  P_n = (1 -q^2  + q^{2n-1}) P_{n-1} + q^2 P_{n-2} \mbox{\ if $n\geqq 2$.}
  \end{gather*} 
\end{id}

\begin{id}[Finite Form of \ref{t-ljslist}.21 (with $q$ replaced by $-q$)] 
 \begin{gather*}
  \sum_{i\geqq 0} \sum_{j\geqq 0}\sum_{k\geqq 0}
  \sum_{l\geqq 0} (-1)^l q^{i^2 + j^2 + k + 2l} \gp{j}{i}{q^2}
     \gp{j+k-1}{k}{q^2} \gp{j+l-1}{l}{q^2} \gp{n-i-k-l}{j}{q^2} 
     \nonumber\\=
   \sum_{j=-\infty}^\infty (-1)^j q^{10j^2 + j}      \U{n}{5j}{q} +
                           (-1)^j q^{10j^2 + 9j + 2} \U{n}{5j+2}{q}
   \tag{\ref{t-SLfin}.21}\\
   P_0 = 1\\
   P_1 = q+1\\
   P_2 = q^4 + 2q^2 + q + 1\\
   P_n = (1+q-q^2+q^{2n-1}) P_{n-1} + (-q+q^2+q^3 + q^{2n-2})P_{n-2} 
    -q^3 P_{n-3} \mbox{\ if $n\geqq 3$.} 
 \end{gather*}
\end{id}

\begin{id}[Finite Form of \ref{t-ljslist}.22] 
Bosonic representation conjectured by 
Santos~\cite[p. 66, eqn. 6.9]{jpos}
\index{Santos, J. P. O.}
 \begin{gather*}
   \sum_{j\geqq 0} \sum_{k\geqq 0} \sum_{l\geqq 0} q^{j^2 + j + k(k+1)/2 + l} 
    \gp{j}{k}{q} \gp{j+l}{l}{q^2} \gp{n-j-k-2l}{j}{q}  \\ =
   \sum_{j=-\infty}^\infty (-1)^j q^{3j^2 + 2j} \Tone{n+1}{3j+1}{\sqrt{q}}
   \tag{\ref{t-SLfin}.22} \\
  P_0 = P_1 =1\\
  P_2 = q^2 + q + 1 \\
  P_n = P_{n-1} + (q + q^{n}) P_{n-2} + (q^n - q)P_{n-3} \mbox{\ if $n\geqq 
  3$.}
\end{gather*} 
\end{id}

\begin{obs}
Identity $(23)$ is the same as $(3)$.              
\end{obs}

\begin{obs}
The series side of (\ref{t-ljslist}.24) 
does not contain a quadratic power of $q$, and so 
the method of $q$-difference equations is not applicable.  Note, however,
that (\ref{t-ljslist}.24)
can be transformed to an infinite product times (\ref{t-ljslist}.25) via a 
transformation of Heine:  \index{Heine's transformation}
\index{Heine, E.}
\begin{eqnarray*}
 & & \sum_{j=0}^\infty \frac{ (-1;q)_{2j} q^j}{ (q^2;q^2)_j} \\
 &=& \lim_{c\to 0} 
     \sum_{j=0}^\infty \frac{ (-1;q^2)_j (-q;q^2)_j q^j}{(c;q^2)_j (q^2;q^2)_j} \\
 &=& \frac{(-q^2;q^2)_\infty}{(q;q^2)_\infty} \lim_{c\to 0}
     \sum_{j=0}^\infty \frac{ (-q;q^2)_j (q^2/c;q^2)_j (c/-q)^j}
    { (-q^2;q^2)_j (q^2;q^2)_j}\qquad\mbox{ (by \ref{heine2})} \\
  &=& \frac{(-q^2;q^2)_\infty}{(q;q^2)_\infty} \lim_{c\to 0}
     \sum_{j=0}^\infty \frac{ (-q;q^2)_j (1-q^2/c)(1-q^4/c)\dots(1-q^{2j}/c)
     (c/-q)^j}
    {(q^4;q^4)_j} \\ 
  &=& \frac{(-q^2;q^2)_\infty}{(q;q^2)_\infty} \lim_{c\to 0}
     \sum_{j=0}^\infty \frac{ (-q;q^2)_j (c-q^2)(c-q^4)\dots(c-q^{2j})
     (-q)^{-j}}
    {(q^4;q^4)_j} \\ 
  &=& \frac{(-q^2;q^2)_\infty}{(q;q^2)_\infty} \lim_{c\to 0}
     \sum_{j=0}^\infty \frac{ (-q;q^2)_j (-1)^j (c-q)(c-q^3)\dots(c-q^{2j-1})}
    {(q^4;q^4)_j} \\
   &=& \frac{(-q^2;q^2)_\infty}{(q;q^2)_\infty} 
     \sum_{j=0}^\infty \frac{ (-q;q^2)_j  (-1)^j (-q)(-q^3)\dots(-q^{2j-1})}
    {(q^4;q^4)_j} \\
    &=& \frac{(-q^2;q^2)_\infty}{(q;q^2)_\infty} 
     \sum_{j=0}^\infty \frac{ (-q;q^2)_j q^{j^2}} {(q^4;q^4)_j}
\end{eqnarray*}
Also note that the infinite product representation of (\ref{t-ljslist}.26)
is the same as that of (\ref{t-ljslist}.24), thus the polynomial sequence
indicated in (\ref{t-SLfin}.26) is, in some sense, a finitization of
(\ref{t-ljslist}.24) as well as (\ref{t-ljslist}.26).
\end{obs}
  
\begin{id}[Finite Form of \ref{t-ljslist}.25] 
Bosonic representation for even $n$ conjectured by
Santos~\cite[p. 67, eqn. 6.11]{jpos}
\index{Santos, J. P. O.}
 \begin{gather*}
   \sum_{i\geqq 0} \sum_{j\geqq 0} \sum_{k\geqq 0} (-1)^k q^{i^2 + j^2 + 2k} 
     \gp{j}{i}{q^2} \gp{j+k-1}{k}{q^2} \gp{n-i-k}{j}{q^2} \\ =
   \sum_{j=-\infty}^\infty (-1)^j q^{3j^2} \U{n}{3j}{q} \tag{\ref{t-SLfin}.25} \\
  P_0 = 1\\
  P_1 = q + 1\\
  P_n = (1 - q^2 + q^{2n-1}) P_{n-1} + (q^2 + q^{2n-2}) P_{n-2} \mbox{\ if $n\geqq 2$.}
 \end{gather*}
\end{id}

\begin{id}[Finite Form of \ref{t-ljslist}.26] 
Bosonic representation conjectured by 
Santos~\cite[p. 67, eqn. 6.12]{jpos}.
\index{Santos, J. P. O.}
 \begin{gather*}
   \sum_{j\geqq 0} \sum_{k\geqq 0} \sum_{l\geqq 0} q^{j^2  + k(k+1)/2 + l} 
    \gp{j}{k}{q} \gp{j+l}{l}{q^2} \gp{n-j-k-2l}{j}{q} \\ =
   \sum_{j=-\infty}^\infty (-1)^j q^{3j^2} \Tone{n}{3j}{\sqrt{q}} 
\tag{\ref{t-SLfin}.26}\\
  P_0 = P_1 = 1\\
  P_2 = 2q + 1\\
  P_n = P_{n-1} + (q + q^{n-1}) P_{n-2} + (q^{n-1} - q) P_{n-3}
 \mbox{\ if $n\geqq 3$.}
 \end{gather*}
\end{id}

\begin{id}[Finite Form of \ref{t-ljslist}.27/87] 
Bosonic representation for even $n$ conjectured by 
Santos~\cite[p. 82, eqn. 6.50]{jpos}.
Bosonic representation for odd $n$ conjectured by 
Santos~\cite[p. 68, eqn. 6.13]{jpos}.

\begin{gather*}
\sum_{i\geqq 0}\sum_{j\geqq 0}\sum_{k\geqq 0}\sum_{l\geqq 0}
   (-1)^l q^{2j^2 + 2j + i^2 + k + 2l}
  \gp{j}{i}{q^2} \gp{j+k}{k}{q^2} \gp{j+l-1}{l}{q^2}
  \gp{n-j-i-k-l}{j}{q^2} \tag{\ref{t-SLfin}.27}
\end{gather*}

\[ = \left\{ \begin{array}{ll}
                 \sum_k q^{12k^2 + 4k}      \gp{2m+1}{m+3k+1}{q^2} +
                        q^{12k^2 + 8k + 1}  \gp{2m  }{m+3k+1}{q^2}
   & \mbox{if $n=2m$,} \\
                 \sum_k q^{12k^2 + 4k}      \gp{2m+1}{m+3k+1}{q^2} +
                        q^{12k^2 + 8k + 1}  \gp{2m+2}{m+3k+2}{q^2}
   & \mbox{if $n=2m+1$.}
\end{array} \right. \]
\begin{gather*}
 P_0 = 1\\
 P_1 = q+1\\
 P_2 = q^4 + q^2 + q + 1\\
 P_n = (1+q-q^2) P_{n-1} + (q^{2n}+q^3 + q^2 -q) P_{n-2} + (q^{2n-1}-q^3) 
 P_{n-3} \mbox{\ if $n\geqq 3$}
\end{gather*}
\end{id}

\begin{id}[Finite form of \ref{t-ljslist}.28] 
Bosonic representation for even $n$ conjectured by 
Santos~\cite[p. 68, eqn. 6.14]{jpos}. 
\index{Santos, J. P. O.}
\index{Berkovich, Alexander}
\index{McCoy, Barry M.}
 \begin{gather*}
   \sum_{j\geqq 0} \sum_{k\geqq 0} q^{j^2 + j +k^2 +k} 
      \gp{j}{k}{q^2} \gp{n+j-k+1}{2j+1}{q} =
   \sum_{j=-\infty}^\infty q^{3j^2 + 2j} \Tone{n+1}{3j+1}{q} 
\tag{\ref{t-SLfin}.28}\\
  P_0 = 1\\
  P_1 = q^2 + q + 1\\
  P_n = (1 + q + q^{2n}) P_{n-1} + (q^{2n}-q) P_{n-2} \mbox{\ if $n\geqq 2$.}
 \end{gather*}
\end{id}

\begin{id}[Finite form of \ref{t-ljslist}.29] 
Bosonic representation proved by 
Santos~\cite[p. 17, Lemma 2.3]{jpos}.
\index{Santos, J. P. O.}
 \begin{gather*}
   \sum_{j\geqq 0} \sum_{k\geqq 0} q^{k^2 +j^2} 
     \gp{j}{k}{q^2} \gp{n+j-k}{2j}{q} =
  \sum_{j=-\infty}^\infty q^{3j^2+j} \U{n}{3j}{q} \tag{\ref{t-SLfin}.29}\\
  P_0 = 1\\
  P_1 = q + 1\\
  P_n = (1 + q + q^{2n-1}) P_{n-1} + (q^{2n-2}-q) P_{n-2} \mbox{\ if $n\geqq 2$.}
 \end{gather*}
\end{id}

\begin{obs}
Identity $(30)$ is $(24)$ with $q$ replaced by $-q$.
\end{obs}

\begin{id}[Finite form of the third Rogers-Selberg Identity]
\index{Rogers-Selberg identities!finite|(}  
Bosonic representation for even $n$ conjectured by 
Santos~\cite[p. 69, eqn. 6.16]{jpos}.
\index{Santos, J. P. O.}
\begin{gather*}
   \sum_{j\geqq 0} \sum_{k\geqq 0} 
    (-1)^k q^{2j^2 + 2j + k} \gp{2j+k}{k}{q} 
    \gp{n-j-k}{j}{q^2}  =
   \sum_{j=-\infty}^\infty (-1)^j q^{j(7j+5)/2} 
      \gp{n+1}{\lfloor \frac{2n + 7j + 6}{4} \rfloor}{q^2}  
\tag{\ref{t-SLfin}.31}\\
  P_0 = 1\\
  P_1 = -q + 1\\
  P_2 = q^4 + q^2 - q + 1\\
  P_n = (1 - q - q^{2}) P_{n-1} + (q^{2n} - q^3 + q^2 + q) P_{n-2} + q^3 
  P_{n-3} \mbox{\ if $n\geqq 3$.}
\end{gather*}
\end{id} 

\begin{id}[Finite form of second Rogers-Selberg Identity] 
Bosonic representation conjectured for even $n$ by 
Santos~\cite[p. 69, eqn. 6.17]{jpos}.
\index{Santos, J. P. O.}
\begin{gather*}
   \sum_{j\geqq 0} \sum_{k\geqq 0} 
    (-1)^{k} q^{2j^2 + 2j + k} \gp{2j+k-1}{k}{q} 
    \gp{n-j-k}{j}{q^2}    \\ =
   \sum_{j=-\infty}^\infty (-1)^j q^{j(7j+3)/2}
      \gp{n+1}{\lfloor \frac{2n + 7j + 5}{4} \rfloor}{q^2}
 \tag{\ref{t-SLfin}.32}\\
  P_0 = P_1 = 1\\
  P_2 = q^4  + 1\\
  P_n = (1 - q - q^{2}) P_{n-1} + (q^{2n} - q^3 + q^2 + q) P_{n-2} + q^3 
  P_{n-3} \mbox{\ if $n\geqq 3$.}
\end{gather*}
\end{id}

\begin{id}[Finite Form of the first Rogers-Selberg Identity] 
Bosonic representation conjectured for even $n$ by 
Santos~\cite[p. 70, eqn. 6.18]{jpos}.
\index{Santos, J. P. O.}
\begin{gather*}
   \sum_{j\geqq 0} \sum_{k\geqq 0} 
   (-1)^{k} q^{2j^2 + k} \gp{2j+k-1}{k}{q} 
    \gp{n-j-k}{j}{q^2}  =
   \sum_{j=-\infty}^\infty (-1)^j q^{j(7j+1)/2}
      \gp{n}{\lfloor \frac{2n + 7j + 2}{4} \rfloor}{q^2} 
\tag{\ref{t-SLfin}.33}\\
  P_0 = P_1 = 1\\
  P_2 = q^2 + 1\\
  P_n = (1 - q - q^{2}) P_{n-1} + (q^{2n-2} - q^3 + q^2 + q) P_{n-2} + q^3 
  P_{n-3} \mbox{\ if $n\geqq 3$.}
\end{gather*}
\end{id}
\index{Rogers-Selberg identities!finite|)}

\begin{id}[Finite form of the second G\"ollnitz-Gordon identity]
\index{G\"ollnitz-Gordon identities!finite|(}
Bosonic $U$- representation given by 
Andrews~\cite[p. 13, eqn. 4.10]{qtrr}.  
Fermionic/U identity given by Berkovich and McCoy~\cite[p. 42, eqn 2.23,
with $L=n+1$, $\nu=2$, $s'=1$, $r'=0$, and $q$ replaced by $q^2$.]{bm:ABgen}.
\index{Andrews, George E.}
\index{Berkovich, Alexander}
\index{McCoy, Barry M.}
 \begin{gather*}
   \sum_{j\geqq 0} \sum_{k\geqq 0} q^{j^2 + 2j +k^2} 
      \gp{j}{k}{q^2} \gp{n-k}{j}{q^2} =
   \sum_{j=-\infty}^\infty (-1)^j q^{4j^2 + 3j} \U{n+1}{4j+1}{q}
 \tag{\ref{t-SLfin}.34-U}\\ =
   \sum_{j=-\infty}^\infty (-1)^j q^{12j^2 + 5j} \binom{n+1,4j+1;q^2}{4j+1}_2
                         + (-1)^j q^{12j^2 + 11j + 2}
                            \binom{n+1, 4j+2; q^2}{4j+2}_2 
\tag{\ref{t-SLfin}.34-t}\\
  P_0 = 1\\
  P_1 = q^3 + 1\\
  P_n = (1 + q^{2n}) P_{n-1} + q^{2n} P_{n-2}  \mbox{\ if $n\geqq 2$.}
 \end{gather*}
\end{id}

\begin{id}[Finite form of \ref{t-ljslist}.35/106]
\begin{gather*}
\sum_{i\geqq 0}\sum_{j\geqq 0}\sum_{k\geqq 0} q^{j(j+3)/2 + i^2 + k}
   \gp{j}{i}{q^2} \gp{j+k}{k}{q^2} \gp{n-2i-2k}{j}{q}  \\ =
\sum_{j=-\infty}^\infty (-1)^j q^{4j^2 + 3j} \V{n+2}{4j+2}{\sqrt{q}}
  \tag{\ref{t-SLfin}.35}\\
  P_0 = 1\\
  P_1 = q^2 + 1\\
  P_2 = q^5 + q^3 + q^2 + q + 1\\
  P_n = (1+q^{n+1}) P_{n-1} + qP_{n-2} + (q^n - q) P_{n-3}
\end{gather*}
\end{id}

\begin{id}[Finite form of the first G\"ollnitz-Gordon identity] 
Bosonic $U$-representation given by 
Andrews~\cite[p. 11, eqn. 4.4]{qtrr}.
Fermionic/$U$ identity given by Berkovich and McCoy~\cite[p. 42, eqn. 2.23,
with $L=n$, $\nu=2$, $s'=r'=0$ and $q$ replaced by $q^2$]{bm:ABgen}.
\index{Andrews, George E.}
\index{Berkovich, Alexander}
\index{McCoy, Barry M.}
 \begin{gather*}
   \sum_{j\geqq 0} \sum_{k\geqq 0} q^{j^2 + k^2} 
      \gp{j}{k}{q^2} \gp{n-k}{j}{q^2} =
   \sum_{j=-\infty}^\infty (-1)^j q^{4j^2 + j} \U{n}{4j}{q} 
\tag{\ref{t-SLfin}.36-U}\\ =
   \sum_{j=-\infty}^\infty (-1)^j q^{12j^2 -j} \binom{n,4j;q^2}{4j}_2
                         + (-1)^j q^{12j^2 + 7j + 1} 
                               \binom{n,4j+1;q^2}{4j+1}_2 
\tag{\ref{t-SLfin}.36-t}\\
  P_0 = 1\\
  P_1 = q + 1\\
  P_n = (1 + q^{2n-1}) P_{n-1} + q^{2n-2} P_{n-2}  \mbox{\ if $n\geqq 2$.}
 \end{gather*}
\end{id}\index{G\"ollnitz-Gordon identities!finite|)}

\begin{id}[Finite form of \ref{t-ljslist}.37/105]
\begin{gather*}
\sum_{i\geqq 0}\sum_{j\geqq 0}\sum_{k\geqq 0} q^{j(j+1)/2 + i^2 + k}
   \gp{j}{i}{q^2} \gp{j+k}{k}{q^2} \gp{n-2i-2k}{j}{q}  \\ =
\sum_{j=-\infty}^\infty (-1)^j  q^{4j^2 + j} \V{n+1}{4j+1}{\sqrt{q}}
   \tag{\ref{t-SLfin}.37}\\
  P_0 = 1\\
  P_1 = q+ 1\\
  P_2 = q^3 + q^2 + 2q + 1\\
  P_n = (1 + q^{n}) P_{n-1} + qP_{n-2} + (q^{n-1}-q) P_{n-3}  
\mbox{\ if $n\geqq 3$.}
\end{gather*}
\end{id}

\begin{id}[Finite form of \ref{t-ljslist}.38/86] 
Bosonic representation \textup{(\ref{t-SLfin}.38-b)} due to 
Santos~\cite{jpos}.  
Identity \textup{(\ref{t-SLfin}.38-b)} stated by
Andrews and Santos~\cite[p. 94, eqn 3.2]{attached}.
Identity \textup{(\ref{t-SLfin}.38-t)} due to Andrews~\cite[p. 663, eqn. 
(5.3)]{eulers}.
\index{Santos, J. P. O.}
\index{Andrews, George E.}
  \begin{gather*}
    \sum_{j\geqq 0} q^{2j(j+1)} \gp{n+1}{2j+1}{q} =
    \sum_{j=-\infty}^\infty q^{4j^2 + 3j} \gp{n+1}{ 
    \lfloor\frac{n+4j+3}{2}\rfloor}{q^2} \tag{\ref{t-SLfin}.38-b} \\
  =  \sum_{k=-\infty}^\infty q^{12k^2 + 5k} \binom{n+1, 6k+1; q}{6k+1}_2 -  
       q^{12k^2 + 11k + 2} \binom{n+1, 6k+3; q}{6k + 3}_2 
\tag{\ref{t-SLfin}.38-t}\\
  P_0 = 1\\
  P_1 = q + 1\\
  P_n = (1 + q) P_{n-1} + (q^{2n}-q) P_{n-2}  \mbox{\ if $n\geqq 2$.}
  \end{gather*}
\end{id}

\begin{id}[Finite Form of \ref{t-ljslist}.39/83] 
Bosonic representation \textup{(\ref{t-SLfin}.39-b)} conjectured by 
Santos~\cite[p. 71, eqn. 6.22]{jpos}.  
Identity \textup{(\ref{t-SLfin}.39-b)} stated by 
Andrews and Santos~\cite[p. 94, eqn. 3.1]{attached}.
Identity \textup{(\ref{t-SLfin}.39-t)} due to Andrews~\cite[p. 663, eqn. 
(5.2)]{eulers}.
\index{Santos, J. P. O.}
\index{Andrews, George E.}
 \begin{gather*}
   \sum_{j\geqq 0} q^{2j^2} \gp{n}{2j}{q} =
   \sum_{j=-\infty}^\infty q^{4j^2 + j} \gp{n}{\lfloor \frac{n+4j+1}{2} 
   \rfloor}{q^2} \tag{\ref{t-SLfin}.39-b} \\
 = \sum_{k=-\infty}^\infty q^{12k^2 + k}     \binom{n,6k  ;q}{6k  }_2 - 
   \sum_{k=-\infty}^\infty q^{12k^2 + 7k +1} \binom{n,6k+2;q}{6k+2}_2
     \tag{\ref{t-SLfin}.39-t}\\
  P_0 = P_1 = 1\\
  P_n = (1 + q) P_{n-1} + (q^{2n-2}-q) P_{n-2}  \mbox{\ if $n\geqq 2$.}
 \end{gather*}
\end{id}

\begin{id}[Finite Form of \ref{t-ljslist}.40]
\index{Bailey's mod 9 identities!finite|(}
\begin{gather*} 
  \sum_{i\geqq 0} \sum_{I\geqq 0} \sum_{j\geqq 0}\sum_{k\geqq 0}\sum_{K\geqq 0}
    (-1)^{i+I+K}
    q^{3j^2 + 3j + i(3i-1)/2 + I(3I+1)/2 + 3k + 3K}
  \gp{j+1}{i}{q^3} \gp{j}{I}{q^3} \\ \times \gp{j+k}{k}{q^6} \gp{j+K-1}{K}{q^3}
  \gp{n - j - i - I - 2k - K}{j}{q^3} \tag{\ref{t-SLfin}.40}
\end{gather*}

\[ = \left\{ \begin{array}{ll} 
  \sum_k  q^{18j^2 +  7j    } \gp{2m + 1}{m + 3j + 1}{q^3}
        - q^{18j^2 + 11j + 1} \gp{2m    }{m + 3j + 1}{q^3} 
  &\mbox{if $n=2m$,} \\
  \sum_k  q^{18j^2 + 7j    } \gp{2m + 1}{m + 3j + 1}{q^3}
        - q^{18j^2 +11j + 1} \gp{2m + 2}{m + 3j + 2}{q^3}
  &\mbox{if $n=2m+1$.}
\end{array} \right. \]
 \begin{gather*}
   P_0 = 1 \\
   P_1 = -q + 1\\
   P_2 = q^6 + q^3 - q + 1\\
   P_3 = -q^{10} - q^8 - q^7 + q^6 -q^4 + q^3 - q + 1\\
   P_n = (1-q^3) P_{n-1} + (q^{3n} + 2q^3) P_{n-2} + (q^6 - q^3 - 
   q^{3n-1} - q^{3n-2}) P_{n-3} \\
    +(q^{3n-3} - q^6) P_{n-4} \mbox{\ if $n\geqq 4$}
 \end{gather*}
\end{id}

\begin{id}[Finite Form of \ref{t-ljslist}.41]
 \begin{gather*} 
    \sum_{i\geqq 0} \sum_{I\geqq 0} \sum_{j\geqq 0}\sum_{k\geqq 0}\sum_{K\geqq 0}
   (-1)^{i+I+K}
    q^{3j^2 + 3j + i(3i-1)/2 + I(3I+1)/2 + 3k + 3K}
  \gp{j}{i}{q^3} \gp{j+1}{I}{q^3}\\ \times \gp{j+k}{k}{q^6} \gp{j+K-1}{K}{q^3}
  \gp{n - j - i - I - 2k - K}{j}{q^3}\tag{\ref{t-SLfin}.41}
 \end{gather*}

\[ = \left\{ \begin{array}{ll} 
  \sum_k q^{18 k^2 +  5k    } \gp{2m+1}{m + 3k + 1}{q^3}
        -q^{18 k^2 + 13k + 2} \gp{2m  }{m + 3k + 1}{q^3}
  &\mbox{if $n=2m$,}\\
  \sum_k q^{18 k^2 +  5k    } \gp{2m+1}{m + 3k + 1}{q^3}
        -q^{18 k^2 + 13k + 2} \gp{2m+2}{m + 3k + 2}{q^3}
  &\mbox{if $n=2m+1$.}
\end{array} \right. \]
\begin{gather*}
   P_0 = 1 \\
   P_1 = -q^2 + 1\\
   P_2 = q^6 + q^3 - q^2 + 1\\
   P_3 = -q^{11} - q^8 - q^7 + q^6 -q^5 + q^3 - q^2 + 1\\
   P_n = (1-q^3) P_{n-1} + (q^{3n} + 2q^3) P_{n-2} + (q^6 - q^3 - 
   q^{3n-1} - q^{3n-2}) P_{n-3} \\
    +(q^{3n-3} - q^6) P_{n-4} \mbox{\ if $n\geqq 4$}
 \end{gather*}
\end{id}

\begin{id}[Finite Form of \ref{t-ljslist}.42]
\begin{gather*}
    \sum_{i\geqq 0} \sum_{I\geqq 0} \sum_{j\geqq 0}\sum_{k\geqq 0}\sum_{K\geqq 0}
    (-1)^{i+I+K}
    q^{3j^2 + i(3i-1)/2 + I(3I+1)/2 + 3k + 3K}
  \gp{j}{i}{q^3} \gp{j}{I}{q^3} \\ \times \gp{j+k-1}{k}{q^6} \gp{j+K-1}{K}{q^3}
  \gp{n - j - i - I - 2k - K}{j}{q^3} \tag{\ref{t-SLfin}.42} 
\end{gather*}

\[ = \left\{ \begin{array}{ll} 
  \sum_k q^{18 k^2 +   k    } \gp{2m  }{m + 3k    }{q^3}
        -q^{18 k^2 + 17k + 4} \gp{2m-1}{m + 3k + 1}{q^3}
 &\mbox{if $n=2m$,}\\
  \sum_k q^{18 k^2 +   k    } \gp{2m  }{m + 3k    }{q^3}
        -q^{18 k^2 + 17k + 4} \gp{2m+1}{m + 3k + 2}{q^3}
 &\mbox{if $n=2m+1$.}
\end{array} \right. \]
\begin{gather*}
   P_0 = P_1 = 1 \\
   P_2 = q^3 + 1\\
   P_3 = -q^5 -q^4 + q^3+ 1\\
   P_n = (1-q^3) P_{n-1} + (q^{3n} + 2q^3) P_{n-2} + (q^6 - q^3 - 
   q^{3n-4} - q^{3n-5}) P_{n-3} \\
    +(q^{3n-6} - q^6) P_{n-4} \mbox{\ if $n\geqq 4$}
 \end{gather*}
\end{id}
\index{Bailey's mod 9 identities!finite|)}

\begin{id}[Finite Form of \ref{t-ljslist}.43] 
Bosonic form for odd $n$ conjectured by 
Santos~\cite[p. 71, eqn. 6.23]{jpos}.
\index{Santos, J. P. O.}
 \begin{gather*}
   \sum_{j\geqq 0} \sum_{k\geqq 0} \sum_{l\geqq 0} q^{(j^2 + 3j + k^2 + k + 
   2l)/2} \gp{j}{k}{q} \gp{j+l}{l}{q^2} \gp{n-k-2l}{j}{q}  \\ =
   \sum_{j=-\infty}^\infty (-1)^j q^{5j^2 + 4j} \Tone{n+2}{5j+2}{\sqrt{q}}
   \tag{\ref{t-SLfin}.43} \\
  P_0 = 1\\
  P_1 = q^2 + 1\\
  P_2 = q^5 + 2q^3 + q^2 + q + 1\\
  P_n = (1+q^{n+1}) P_{n-1} + (q+q^{n+1}) P_{n-2} - qP_{n-3}
   \mbox{\ if $n\geqq 3$}
 \end{gather*}
\end{id}

\begin{id}[Finite form of \ref{t-ljslist}.44/63] 
Identity \textup{(\ref{t-SLfin}.44-b)} due to
Andrews~\cite[p. 5291, eqn. 3.2]{hhm}; 
corrected here.
Bosonic $q$-trinomial representation conjectured by 
Santos~\cite[p. 77, eqn. 6.38]{jpos}.
\index{Santos, J. P. O.}
\index{Andrews, George E.}
\begin{gather*}
   \sum_{j\geqq 0} \sum_{k\geqq 0}
      q^{3j(j+1)/2 + k} \gp{j+k}{k}{q^2} \gp{n-2j-2k-1}{j}{q}  =
   \sum_{j=-\infty}^\infty (-1)^j q^{5j^2 + 3j} \gp{n+1}{\lfloor 
   \frac{n+5j+3}{2} \rfloor }{q} \tag{\ref{t-SLfin}.44-b} \\
 = \sum_{k=-\infty}^\infty q^{(15k^2 +  7k)/2}     \Tone{n+1}{5k+1}{\sqrt{q}}
                          -q^{(15k + 13)k/2 + 1} \Tone{n+1}{5k+2}{\sqrt{q}}
   \tag{\ref{t-SLfin}.44-t} \\
  P_0 = P_1 = 1\\
  P_2 = q + 1\\
  P_n =  P_{n-1} + q P_{n-2} +(q^n - q) P_{n-3}
   \mbox{\ if $n\geqq 3$}
  \end{gather*}
\end{id}

\begin{id}[Finite form of \ref{t-ljslist}.45] 
Bosonic representation for even $n$ conjectured by 
Santos~\cite[p. 72, eqn 6.25]{jpos}.
\index{Santos, J. P. O.}
 \begin{gather*}
   \sum_{j\geqq 0} \sum_{k\geqq 0} \sum_{l\geqq 0} q^{(j^2 + j + k^2 + k + 
   2l)/2} \gp{j}{k}{q} \gp{j+l}{l}{q^2} \gp{n-k-2l}{j}{q}  \\ =
   \sum_{j=-\infty}^\infty (-1)^j q^{5j^2 + 2j} \Tone{n+1}{5j+1}{\sqrt{q}}
   \tag{\ref{t-SLfin}.45}\\
  P_0 = 1\\
  P_1 = q + 1\\
  P_2 = q^3 + 2q^2 + 2q + 1\\
  P_n = (1+q^{n}) P_{n-1} + (q+q^{n}) P_{n-2} - qP_{n-3}
   \mbox{\ if $n\geqq 3$}
 \end{gather*}
\end{id}

\begin{id}[Finite form of \ref{t-ljslist}.46/62] 
Identity \textup{(\ref{t-SLfin}.46-b)}
due to Andrews~\cite[p. 5291, eqn. 3.1]{hhm}. 
Bosonic $q$-trinomial representation conjectured by 
Santos~\cite[p. 77, eqn. 6.37]{jpos}.
\index{Santos, J. P. O.}
\index{Andrews, George E.}
\begin{gather*}
   \sum_{j\geqq 0} \sum_{k\geqq 0}
      q^{j(3j+1)/2 + k}  \gp{j+k}{k}{q^2}\gp{n-2j-2k}{j}{q} =
   \sum_{j=-\infty}^\infty (-1)^j q^{j(5j+1)} \gp{n}{\lfloor 
   \frac{n-5j}{2} \rfloor}{q} \tag{\ref{t-SLfin}.46-b}\\
 = \sum_{k=-\infty}^\infty q^{(15k^2 +   k)/2}     \Tone{n}{5k}{\sqrt{q}}
                          -q^{(15k+11)k/2+1} \Tone{n}{5k+2}{\sqrt{q}}
   \tag{\ref{t-SLfin}.46-t}\\
  P_0 = P_1 = P_2 = 1\\
  P_n = P_{n-1} + q P_{n-2} + (q^{n-2} - q) P_{n-3}
   \mbox{\ if $n\geqq 3$}
  \end{gather*}
\end{id}

\begin{obs} Identity $(47)$ equivalent to $(10)$ and to $(54) + q\times(49).$ 
\end{obs}

\begin{obs} Identity $(48)$ is $(54) - q\times(49).$
\end{obs}

\begin{id}[Finite form of \ref{t-ljslist}.49]
   \begin{gather*} 
   \sum_{i\geqq 0} \sum_{I\geqq 0} \sum_{j\geqq 0}\sum_{k\geqq 0}\sum_{K\geqq 0}
    (-1)^K q^{j^2 + 2j + i^2 + i + I^2/2 + I/2 + k + 2K}
   \gp{j}{i}{q^2} \gp{j+1}{I}{q} \gp{j+k}{k}{q^2}\\ \times \gp{j+K}{K}{q^2}
   \gp{n-j- 2i - I - 2k - 2K}{j}{q} \\ =
     \sum_{j=-\infty}^\infty (-1)^j q^{6j^2 + 5j} \gp{n+2}{\lfloor \frac{n + 6j + 5}{2}
        \rfloor}{q} \tag{\ref{t-SLfin}.49}\\
     P_0 = 1\\
     P_1 = -q+1\\
     P_2 = q^3 + q^2 + 1\\
     P_3 = -q^5 + 1\\
     P_4 = q^8 + q^7 + q^6 + q^5 + 2q^4 + q^3 + q^2 + 1\\
     P_n = P_{n-1} + (q+q^2+q^{n+1})P_{n-2} + (-q-q^2-q^{n+1})P_{n-3}
       + (-q^3 + q^{n+1}) P_{n-4} \\ + (q^3 - q^{n+1}) P_{n-5}
      \mbox{\ if $n\geqq 5$.}
   \end{gather*}
\end{id}

\begin{id}[Finite form of \ref{t-ljslist}.50] 
Bosonic representation \textup{(\ref{t-SLfin}.50-b)} conjectured by 
Santos~\cite[p. 73, eqn. 6.29]{jpos}.
\index{Santos, J. P. O.}
  \begin{gather*}
  \sum_{j\geqq 0} \sum_{k\geqq 0} q^{j^2 + 2j + k^2} \gp{j}{k}{q^2} 
    \gp{n-j-k+1}{2j+1}{q} =
    \sum_{j=-\infty}^\infty (-1)^j q^{6j^2 + 4j} \gp{2n+2}{n + 3j + 
    2}{q} \tag{\ref{t-SLfin}.50-b} \\
   =\sum_{j=-\infty}^\infty q^{12j^2 + 6j} \U{n+1}{6j+1}{q}
 \tag{\ref{t-SLfin}.50-t}\\
  P_0 = 1\\
  P_1 = q^3 + q + 1\\
  P_n = (1+q + q^{2n+1}) P_{n-1} + (q^{2n}-q) P_{n-2} 
   \mbox{\ if $n\geqq 2$}.\\
  \end{gather*}
\end{id}

\begin{obs} Identity $(51)$ is the same as $(11)$ and $(64)$. \end{obs}

\begin{id}[Finite form of \ref{t-ljslist}.52]
\index{Santos, J. P. O.}
  \begin{gather*}  \sum_{j\geqq 0}  q^{2j^2 - j} \gp{n}{2j}{q}  =
    \sum_{k=-\infty}^\infty q^{4k^2+k} \gp{2n}{n+4k}{\sqrt{q}} -      
     q^{4k^2 + 5k + 3/2} \gp{2n}{n+4k+3}{\sqrt{q}} 
     \tag{\ref{t-SLfin}.52-b} \\ =
    \sum_{k=-\infty}^\infty  q^{6k^2 + k} \Tone{n}{6k}{\sqrt{q}} 
                      - q^{6k^2 + 5k + 1} \Tone{n}{6k + 2}{\sqrt{q}}
\tag{\ref{t-SLfin}.52-T} \\ =
   \sum_{j=-\infty}^\infty (-1)^j q^{6j^2 + 2j} \U{n-1}{3j}{q} 
\tag{\ref{t-SLfin}.52-U}  \\
  = (-q;q)_{n-1} \mbox{\ if $n\geqq 2$.}\\
   P_0 = P_1 = 1\\
   P_n = (1+q^{n-1}) P_{n-1} \mbox{\ if $n\geqq 2$}
  \end{gather*}
\end{id}

\begin{id}[Finite form of \ref{t-ljslist}.53] 
Bosonic representation for even $n$ conjectured by 
Santos~\cite[p. 74, eqn. 6.32]{jpos}.
\index{Santos, J. P. O.}
\begin{gather*}
 \sum_{i\geqq 0}\sum_{j\geqq 0}\sum_{k\geqq 0}\sum_{l\geqq 0}  
 (-1)^{i+l} q^{4j^2 + i^2 +4k +4l}
    \gp{2j}{i}{q^2}  \gp{j+k-1}{k}{q^8} \gp{j+l-1}{l}{q^4} \\ \times
    \gp{n -2j - i - 2k - l}{j}{q^4} \tag{\ref{t-SLfin}.53}
\end{gather*}
\[ = \left\{ \begin{array}{ll} 
  \sum_k q^{24 k^2 +  2k    } \gp{2n  }{n + 3k}{q^4}
       - q^{24 k^2 + 22k + 5} \gp{2n-1}{n+3k+1}{q^4}
  &\mbox{if $n=2m$,} \\
  \sum_k q^{24 k^2 +  2k    } \gp{2n  }{n + 3k}{q^4}
       - q^{24 k^2 + 22k + 5} \gp{2n+1}{n+3k+2}{q^4}
   &\mbox{if $n=2m+1$.} 
\end{array} \right. \]
\begin{gather*}
   P_0 = P_1 = 1 \\
   P_2 = q^4 + 1\\
   P_3 = -q^7 -q^5 + q^4+ 1\\
   P_n = (1-q^4) P_{n-1} + (q^{4n-4} + 2q^4) P_{n-2} + (q^8 - q^4 - 
   q^{4n-5} - q^{4n-7}) P_{n-3} \\
    +(q^{4n-8} - q^8) P_{n-4} \mbox{\ if $n\geqq 4$}
 \end{gather*}
\end{id}

\begin{id}[Finite form of \ref{t-ljslist}.54]
\begin{gather*} 
  1 + \sum_{j\geqq 1}\sum_{i\geqq 0}\sum_{k\geqq 0} \sum_{l\geqq 0}
    (-1)^l q^{j^2 + i^2 + i + k + l}  \gp{j-1}{i}{q^2} \gp{j+k-1}{k}{q^2}
     \gp{j+l-2}{l}{q} \\ \times \gp{n-j-2k-2i-l}{j}{q} \tag{\ref{t-SLfin}.54} \\ =
  \sum_{j=-\infty}^\infty (-1)^j q^{6j^2 + j}
      \gp{n}{\lfloor \frac{n+6j+1}{2} \rfloor}{q} \\
   P_0 = P_1 = 1 \\
   P_2 = q + 1\\
   P_3 = q^2 + q + 1\\
   P_4 = q^4 + q^3 + 2q^2 + q + 1\\
   P_n = P_{n-1} + (q + q^2 + q^{n-1}) P_{n-2} - (q+q^2 + q^{n-1}) P_{n-3} 
    +(q^{n-1} - q^3) P_{n-4} \\ (q^3 - q^{n-1}) P_{n-5} \mbox{\ if $n\geqq 5$}
\end{gather*}
\end{id}

\begin{id}[Finite form of \ref{t-ljslist}.55] 
Bosonic representation for even $n$ conjectured by 
Santos~\cite[p. 75, eqn. 6.33]{jpos}.
\index{Santos, J. P. O.}
\begin{gather*}
  \sum_{i\geqq 0} \sum_{I\geqq 0} \sum_{j\geqq 0}\sum_{k\geqq 0}\sum_{K\geqq 0}
   (-1)^{i+I+K} q^{4j^2 + 4j + 2i^2 - i + 2I^2 + I +4k +4K}
    \gp{j+1}{i}{q^4} \gp{j}{I}{q^4} \gp{j+k}{k}{q^8}\\ \times\gp{j+K-1}{K}{q^4}
    \gp{n - j - i - I - 2k - K}{j}{q^4} \tag{\ref{t-SLfin}.55}
\end{gather*}
\[ = \left\{ \begin{array}{ll} 
  \sum_k q^{24 k^2 + 10k    } \gp{2n+1}{n+3k+1}{q^4}
       - q^{24 k^2 + 14k + 1} \gp{2n  }{n+3k+1}{q^4}
     &\mbox{if $n=2m$,} \\
  \sum_k q^{24 k^2 + 10k    } \gp{2n+1}{n+3k+1}{q^4}
       - q^{24 k^2 + 14k + 1} \gp{2n+2}{n+3k+2}{q^4}
     &\mbox{if $n=2m+1$.}
\end{array} \right. \]
\begin{gather*}
   P_0 =  1 \\
   P_1 = -q + 1\\
   P_2 = q^8 + q^4 - q + 1\\
   P_3 = -q^{13} - q^{11} - q^9 + q^8 -q^5 + q^4 - q+ 1\\
   P_n = (1-q^4) P_{n-1} + (q^{4n} + 2q^4) P_{n-2} + (q^8 - q^4 - 
   q^{4n-1} - q^{4n-3}) P_{n-3} \\
    +(q^{4n-4} - q^8) P_{n-4} \mbox{\ if $n\geqq 4$}
 \end{gather*}
\end{id}

\begin{id}[Finite form of \ref{t-ljslist}.56] 
Bosonic representation conjectured by Santos~\cite[p. 75, eqn. 6.34]{jpos}.
\index{Santos, J. P. O.}
  \begin{gather*} 
   \sum_{i,j,k,l\geqq 0} (-1)^i q^{j^2 + 2j + i^2 + i + k + l}
     \gp{j}{i}{q^2} \gp{j+k}{k}{q^2} \gp{j+l}{l}{q} 
     \gp{n-k-2i-2k-l}{j}{q} \\ =
   \sum_{j=-\infty}^\infty q^{6j^2 + 5j} \gp{n+2}{\lfloor \frac{n+6j+2}{2}
\rfloor}{q} \tag{\ref{t-SLfin}.56} \\
 P_0 = 1 \\
 P_1 = q + 1\\
 P_2 = q^3 + q^2 + 2q + 1 \\
 P_3 = q^5 + 2q^4 + 2q^3 + 2q^2 + 2q + 1 \\
 P_n = (1+q) P_{n-1} + q^{n+1} P_{n-2} + (-q-q^2) P_{n-3} + (q^2 - 
 q^{n-1}) P_{n-4} \mbox{\ if $n\geqq 4$} 
   \end{gather*}
\end{id}

\begin{obs}[Finite form of \ref{t-ljslist}.57] 
This identity is (55) with $q$ replaced by $-q$.
\end{obs}

\begin{id}[Finite form of \ref{t-ljslist}.58] 
Bosonic representation for even $n$ conjectured by 
Santos~\cite[p. 76, eqn. 6.36]{jpos}.
\index{Santos, J. P. O.}
\begin{gather*}
  1 + \sum_{j\geqq 1}\sum_{k\geqq 0}\sum_{l\geqq 0} 
    q^{j^2 + i(i+1)/2 + k}
    \gp{j-1}{i}{q} \gp{j+k-1}{k}{q^2} \gp{n-i-j-2k}{j}{q}  \\ =
  \sum_{j=-\infty}^\infty q^{6j^2 + j} 
    \gp{n}{\lfloor \frac{n + 6j + 1}{2} \rfloor}{q} \tag{\ref{t-SLfin}.58}\\
 P_0 = P_1 = 1 \\
 P_2 = q + 1 \\
 P_3 = q^2 + q + 1 \\
 P_n = (1+q) P_{n-1} + q^{n+1} P_{n-2} + (-q-q^2) P_{n-3} + (q^2 - 
 q^{n-1}) P_{n-4} \mbox{\ if $n\geqq 4$} 
\end{gather*}
\end{id}

\begin{id}[Finite form of \ref{t-ljslist}.59]
Identity \ref{t-SLfin}.59 is a special case of an identity due to 
Berkovich and McCoy~\cite[p. 59, eqn. (3.14) with $p=4$, $p'=7$,
$r=s=2$, $a=6$, $b=3,4$]{bm:cf}.
\index{Berkovich, Alexander}
\index{McCoy, Barry M.}
\begin{gather*}
 \sum_{j\geqq 0}\sum_{k\geqq 0} q^{j^2 + 2j + k} \gp{j+k}{k}{q^2} 
   \gp{n-j-2k}{j}{q} =
 \sum_{j=-\infty}^\infty (-1)^j q^{7j^2 + 5j} \gp{n+2}{\lfloor 
   \frac{n+7j+5}{2} \rfloor}{q} \tag{\ref{t-SLfin}.59} \\
  P_0 = P_1 = 1 \\
  P_2 = q^3 + q + 1\\
  P_n = P_{n-1} + (q+q^{n+1}) P_{n-2} - q P_{n-3} \mbox{\ if $n\geqq 3$}
\end{gather*}
\end{id}

\begin{id}[Finite form of \ref{t-ljslist}.60]
Identity \ref{t-SLfin}.60 is a special case of an identity due to 
Berkovich and McCoy~\cite[p. 59, eqn. (3.14) with $p=4$, $p'=7$,
$r=s=2$, $s=3$, $a=4$; $b=4,5$]{bm:cf}.
\index{Berkovich, Alexander}
\index{McCoy, Barry M.}
 \begin{gather*}
    \sum_{j\geqq 0} \sum_{k\geqq 0} q^{j^2 + j + k} \gp{j+k}{k}{q^2}
   \gp{n-j-2k}{j}{q} =
   \sum_{j=-\infty}^\infty (-1)^j q^{7j^2 + 3j} \gp{n+1}{ \lfloor 
\frac{n+7j+3}{2} \rfloor}{q} \tag{\ref{t-SLfin}.60}\\
  P_0 = P_1 = 1 \\
  P_2 = q^2 + q + 1\\
  P_n = P_{n-1} + (q+q^{n}) P_{n-2} - q P_{n-3} \mbox{\ if $n\geqq 3$}
\end{gather*} 
\end{id}

\begin{id}[Finite form of \ref{t-ljslist}.61]
Identity \ref{t-SLfin}.61 is a special case of an identity due to 
Berkovich and McCoy~\cite[p. 59, eqn. (3.14) with $p=4$, $p'=7$,
$r=2$, $s=3$, $a=3,4$; $b=3$]{bm:cf}.
\index{Berkovich, Alexander}
\index{McCoy, Barry M.}
  \begin{gather*}
 \sum_{j\geqq 0} \sum_{k\geqq 0} q^{j^2 + k} \gp{j+k-1}{k}{q^2}
   \gp{n-j-2k}{j}{q} =
  \sum_{j=-\infty}^\infty (-1)^j q^{7j^2 + j} \gp{n}{\lfloor \frac{n+7j+1}{2}
\rfloor}{q} \tag{\ref{t-SLfin}.61}\\
  P_0 = P_1 = 1 \\
  P_2 = q + 1\\
  P_n = P_{n-1} + (q+q^{n-1}) P_{n-2} - q P_{n-3} \mbox{\ if $n\geqq 3$}
  \end{gather*}
\end{id}

\begin{obs} Identity $(62)$ is equivalent to $(46)$. (Andrews~\cite[p. 20, 
eqns. (8.5) and (8.6)]{comb})\end{obs}

\begin{obs} Identity $(63)$ is the same as $(44)$. \end{obs}

\begin{obs} Identity $(64)$ is the same as $(11)$. \end{obs}

\begin{obs} Identity $(65)$ is the equivalent to $(37) + {\sqrt q}\times 
(35)$. \end{obs}

\begin{obs} Identity $(66)$ is  equivalent to $(71)+ q\times(68)$. \end{obs}

\begin{obs} Identity $(67)$ is equivalent to $(71) - q\times(68)$. \end{obs}

\begin{id}[Finite form of \ref{t-ljslist}.68] 
Bosonic form conjectured by Santos~\cite[p. 79, eqn. 6.42]{jpos}.
\index{Santos, J. P. O.}
\begin{gather*}
   \sum_{i\geqq 0} \sum_{j\geqq 0}\sum_{k\geqq 0}\sum_{l\geqq 0}
    (-1)^k q^{j^2 + 2j + 2i^2 + 2k + l}
     \gp{j}{i}{q^4} \gp{j+k}{k}{q^2} \gp{j+l}{l}{q^2}
     \gp{n - 2i - k - l}{j}{q^2} \\ =
   \sum_{j=-\infty}^\infty (-1)^j q^{8j^2+6j} \U{n+1}{4j+1}{q} 
\tag{\ref{t-SLfin}.68}\\
  P_0 = 1 \\
  P_1 = q^3 - q^2 + q + 1\\
  P_2 = q^8 - q^7 + q^6 + 2q^4 + q + 1\\
  P_n = (1 + q -q^2 + q^{2n+1}) P_{n-1} + (q^3 + q^2 -q) P_{n-2} 
   +(q^{2n+1} - q^3) P_{n-3} \mbox{\ if $n\geqq 
   3$}
\end{gather*}
\end{id}

\begin{id}[Finite form of \ref{t-ljslist}.69] 
Bosonic form conjectured by Santos~\cite[p. 79, eqn. 6.43]{jpos}.
\index{Santos, J. P. O.}    
\begin{gather*}
   \sum_{i\geqq 0} \sum_{j\geqq 0}\sum_{k\geqq 0}\sum_{l\geqq 0}
    (-1)^i q^{j^2 + 2j + 2i^2 + 2i + k + 2l}
     \gp{j}{i}{q^4} \gp{j+k}{k}{q^2} \gp{j+l}{l}{q^2} 
     \gp{n - 2i - k - l}{j}{q^2} \\ = 
   \sum_{j=-\infty}^\infty q^{8j^2+6j} \U{n+1}{4j+1}{q} \tag{\ref{t-SLfin}.69}
\\
  P_0 = 1 \\
  P_1 = q^3 + q^2 + q + 1\\
  P_2 = q^8 + q^7 + q^6 + 2q^5 + 2q^4 + 2q^3 + 2q^2 + q + 1\\
  P_n = (1+ q +q^2 + q^{2n+1}) P_{n-1} - ( q + q^2 + q^{3}) P_{n-2} 
   +(q^3 - q^{2n+1}) P_{n-3} \mbox{\ if $n\geqq 3$}
\end{gather*}
\end{id}

\begin{id}[Finite form of \ref{t-ljslist}.70]
\begin{gather*} 
  \sum_{i\geqq 0} \sum_{j\geqq 0}\sum_{k\geqq 0}\sum_{l\geqq 0}
    (-1)^l q^{j^2 + 2j + 2i^2 + k + 2l}
     \gp{j}{i}{q^4} \gp{j+k}{k}{q^2} \gp{j+l-1}{l}{q^2} 
     \gp{n-2i-k-l}{j}{q^2} \\ =
   \sum_{j=-\infty}^\infty (-1)^j q^{8j^2 + 4j} 
     \Big[ \Tzero{n}{4j+1}{q} + \Tzero{n+1}{4j+1}{q} \Big] 
\tag{\ref{t-SLfin}.70}\\
 P_0 = 1\\
 P_1 = q^3 + q + 1\\
 P_2 = q^8 + q^6 + q^4 + q^3 + q^2 + q + 1\\
 P_n = (1+q-q^2+q^{2n+1}) P_{n-1} + (q^3 + q^2 - q) P_{n-2} +
  (q^{2n-1} - q^3) P_{n-3} \mbox{\ if $n\geqq 3$}
\end{gather*}
\end{id}

\begin{id}[Finite form of \ref{t-ljslist}.71]  
Bosonic representation conjectured by Santos~\cite[p. 80, eqn. 6.44]{jpos}.
\index{Santos, J. P. O.}
\begin{gather*} 
    1 + \sum_{j\geqq 1}\sum_{i\geqq 0}\sum_{k\geqq 0}\sum_{l\geqq 0}
     (-1)^l q^{j^2 + 2i^2 + 2i + k + 2l} \gp{j-1}{i}{q^4} \gp{j+k-1}{k}{q^2}
      \\ \times \gp{j+l-2}{l}{q^2} \gp{n-2i-k-l}{j}{q^2} \\ =
   \sum_{j=-\infty}^\infty (-1)^j q^{8j^2 + 2j} \U{n}{4j}{q} 
 \tag{\ref{t-SLfin}.71}\\
 P_0 = 1\\
 P_1 = q + 1\\
 P_2 = q^4 + q^3 + q^2 + q + 1\\
 P_3 = q^9 + q^8 + q^7 + 2q^5 + 2q^4 + 2q^3 + q^2 + q + 1\\
 P_n = (1+q+q^{2n-1}) P_{n-1} + (q^4 - q - q^{2n-1}) P_{n-2} 
  -(q^4 + q^5 -q^{2n-1}) P_{n-3}  \\ + (q^5 - q^{2n-1}) P_{n-4}
   \mbox{\ if $n\geqq 4$}
\end{gather*}
\end{id}

\begin{id}[Finite form of \ref{t-ljslist}.72] 
Bosonic representation conjectured by Santos~\cite[p. 80, eqn. 6.45]{jpos}.
\index{Santos, J. P. O.}
\begin{gather*}
    1 + \sum_{j\geqq 1}\sum_{i\geqq 0}\sum_{k\geqq 0}\sum_{l\geqq 0}
     (-1)^i q^{j^2 + 2i^2 + 2i + k + 2l} \gp{j-1}{i}{q^4} \gp{j+k-1}{k}{q^2}
      \\ \times \gp{j+l-2}{l}{q^2} \gp{n-2i-k-l}{j}{q^2} \\=
   \sum_{j=-\infty}^\infty q^{8j^2 + 2j} \U{n}{4j}{q} 
\tag{\ref{t-SLfin}.72}\\
  P_0 = 1 \\
  P_1 = q + 1\\
  P_2 = q^4 + q^3 + q^2 + q + 1\\
  P_3 = q^9 + q^8 + q^7 + 2q^6 + 2q^5
  +2q^4 + 2q^3 + q^2 + q + 1\\
  P_n = (1 + q^2 + q^{2n-1}) P_{n-1} + q^{2n-2} P_{n-2} 
   -(q^{2n-1} + q^4 + q^2) P_{n-3} \\+(q^4 - q^{2n-2}) P_{n-4}\mbox{\ if $n\geqq 
   4$}
\end{gather*}
\end{id}

\begin{obs} Identity $(73)$ is equivalent to $(77) + (78)$  and to
 $(77) + (75) + q\times(76)$. \end{obs}

\begin{obs} Identity $(74)$ is equivalent to $(77) + (78) - q\times(76)$ 
and to$(77) + (75)$. \end{obs}

\begin{id}[Finite form of \ref{t-ljslist}.75]  
Note: $(75)$ is $(78) - q\times(76)$.   
Bosonic representation conjectured by Santos~\cite[p. 81, eqn. (6.46)]{jpos}.
\index{Santos, J. P. O.}
 \begin{gather*}    
   1 + \sum_{j\geqq 1} \sum_{i\geqq 0} \sum_{I\geqq 0} \sum_{k\geqq 0}
    \sum_{K\geqq 0} q^{ j(j+1)/2 + i(3i+1)/2 + I(I+1)/2 + 2k + K}
    \gp{j-1}{i}{q^3} \gp{j}{I}{q} \gp{j+k-1}{k}{q^2} 
    \\ \times \gp{j+K-1}{K}{q^2}
     \gp{n-1-3i-I-2k-2K}{j-1}{q} \\ 
   \sum_{j=-\infty}^\infty (-1)^j q^{9j^2         } \Tone{n+1}{6j}{\sqrt{q}}
                          -(-1)^j q^{9j^2 + 6j + 1} \Tone{n+1}{6j+2}{\sqrt{q}}
   \tag{\ref{t-SLfin}.75}\\
  P_0 = 1 \\
  P_1 = q + 1\\
  P_2 = q^3 + q^2 + q + 1\\
  P_3 = q^6 + q^5 +2q^4 + 2q^3 + 2q^2 + q + 1\\
  P_n = (1 + q + q^{n}) P_{n-1} - (q +q^{2}) P_{n-3} +(q^2 - q^n) P_{n-4}
   \mbox{\ if $n\geqq 4$}
\end{gather*}
\end{id}

\begin{id}[Finite form of \ref{t-ljslist}.76]
  \begin{gather*}
   \sum_{i\geqq 0} \sum_{I\geqq 0} \sum_{j\geqq 0}\sum_{k\geqq 0}\sum_{K\geqq 0} 
    (-1)^i q^{j(j+3)/2 + 3i(i+1)/2
       I(I+1)/2 + k + 2K}
      \gp{j}{i}{q^3} \gp{j+1}{I}{q}  \gp{j+k}{k}{q^2} \\ \times \gp{j+K}{K}{q^2}
      \gp{n - 3i - I - 2k - 2K}{j}{q} \\ = 
    \sum_{j=-\infty}^\infty (-1)^j q^{9j^2 + 6j} \Tone{n+2}{6j+2}{\sqrt{q}} 
    \tag{\ref{t-SLfin}.76}\\
  P_0 = 1 \\
  P_1 = q^2 + q + 1\\
  P_2 = q^5 + q^4 + 2q^3 + 2q^2 + 2q + 1\\
  P_3 = q^9 + q^8 + 2q^7 + 3q^6 + 4q^5 + 4q^4 +4q^3 + 3q^2 + 2q + 1\\
  P_4 = q^{14} + q^{13} + 2q^{12} + 3q^{11} + 5q^{10} + 6q^9 + 7q^8 + 8q^7 
    + 9q^6 + 8q^5 + 7q^4 + 5q^3 \\+ 4q^2 + 2q + 1\\
  P_n = (1 + q^{n+1}) P_{n-1} + (q + q^2 + q^{n+1}) P_{n-2} 
   -(q^2+q) P_{n-3} \\-(q^3 + q^{n+1}) P_{n-4} + (q^3-q^{n+1}) P_{n-5} 
   \mbox{\ if $n\geqq 5$}
  \end{gather*}
\end{id}

\begin{id}[Finite form of \ref{t-ljslist}.77]
 \begin{gather*}  
   \sum_{i\geqq 0} \sum_{j\geqq 0}\sum_{k\geqq 0}\sum_{l\geqq 0}
    (-1)^i q^{ (j^2 + j + 3i^2 + 3i + 2k + 2l)/2}
     \gp{j}{i}{q^3} \gp{j+k}{k}{q^2} \gp{j+l-1}{l}{q} 
     \gp{n - 3i - 2k - l}{j}{q}  \\ =
   \sum_{j=-\infty}^\infty (-1)^j q^{9j^2 + 3j}  \Tone{n+1}{6j+1}{\sqrt{q}}
   \tag{\ref{t-SLfin}.77}\\
  P_0 = 1 \\
  P_1 = q + 1\\
  P_2 = q^3 + 2q^2 + 2q + 1\\
  P_n = (1 + q^{n}) P_{n-1} + (q +q^{n}) P_{n-2} +(q^n - q) P_{n-3}
   \mbox{\ if $n\geqq 3$} 
\end{gather*}
\end{id}

\begin{obs} Identity $(78)$ is equivalent to $(75) + q\times(76)$. \end{obs}

\begin{id}[Finite form of \ref{t-ljslist}.79/98]
Eqn. \textup{(\ref{t-SLfin}.79-b)} due to Andrews~\cite[p. 80, eqn. 8.42]{qs}.
\index{Andrews, George E.}
   \begin{gather*}
   \sum_{j\geqq 0}  q^{j^2 } \gp{n+j}{2j}{q}  = 
   \sum_{k=-\infty}^\infty q^{15k^2 + k} \gp{2n}{n+5k}{q} - q^{15k^2 + 
   11k + 2} \gp{2n}{n + 5k+2}{q} \tag{\ref{t-SLfin}.79-b} \\ =
   \sum_{j=-\infty}^\infty (-1)^j q^{10j^2 + 2j} \U{n}{5j}{q} 
\tag{\ref{t-SLfin}.79-t}\\
  P_0 = 1 \\
  P_1 = q + 1\\
  P_n = (1 + q + q^{2n-1}) P_{n-1} - q P_{n-2} 
   \mbox{\ if $n\geqq 2$}
  \end{gather*}
\end{id}

\begin{id}[Finite form of \ref{t-ljslist}.80] 
Bosonic representation conjectured by Santos~\cite[p. 77, eqn. 6.39]{jpos}.
\index{Santos, J. P. O.}
 \begin{gather*}
  \sum_{j\geqq 0}\sum_{k\geqq 0} q^{ j(j+1)/2 + k} \gp{j+k}{k}{q^2} 
     \gp{n-2k}{j}{q} \nonumber \\ =
  \sum_{k=-\infty}^\infty q^{(21k^2 + 5k)/2} \Tone{n+1}{7k+1}{\sqrt{q}} -
       q^{(21k+ 19)k/2 +2} \Tone{n+1}{7k+3}{\sqrt{q}}  
\tag{\ref{t-SLfin}.80}\\
  P_0 = 1 \\
  P_1 = q + 1\\
  P_2 = q^3 + q^2 + 2q + 1\\
  P_n = (1 + q^n) P_{n-1} + q P_{n-2} - q P_{n-3} 
   \mbox{\ if $n\geqq 3$} 
  \end{gather*}
\end{id}

\begin{id}[Finite form of \ref{t-ljslist}.81] 
Bosonic representation conjectured by Santos~\cite[p. 81, eqn. 6.47]{jpos}.
\index{Santos, J. P. O.}
   \begin{gather*} 
     \sum_{j\geqq 0}\sum_{k\geqq 0} q^{ j(j+1)/2 + k} \gp{j+k-1}{k}{q^2} 
     \gp{n-2k}{j}{q} \nonumber \\ =
     \sum_{k=-\infty}^\infty 
       q^{(21k^2 + k)/2} \Tone{n+1}{7k}{\sqrt{q}} -
       q^{(21k^2 + 13k + 2)/2} \Tone{n+1}{7k+2}{\sqrt{q}} 
\tag{\ref{t-SLfin}.81}\\
  P_0 = 1 \\
  P_1 = q + 1\\
  P_2 = q^3 + q^2 + q + 1\\
  P_n = (1 + q^n) P_{n-1} + q P_{n-2} - q P_{n-3} 
   \mbox{\ if $n\geqq 3$} 
   \end{gather*}
\end{id}

\begin{id}[Finite form of \ref{t-ljslist}.82] 
Bosonic representation conjectured by Santos~\cite[p. 82, eqn. 6.48]{jpos}.
\index{Santos, J. P. O.}
\begin{gather*}
     \sum_{j\geqq 0}\sum_{k\geqq 0} q^{ j(j+3)/2 + k} \gp{j+k}{k}{q^2} 
     \gp{n-2k}{j}{q} \nonumber \\ = 
     \sum_{k=-\infty}^\infty q^{(21k^2 + 11k)/2} \Tone{n+2}{7k+2}{\sqrt{q}} -
       q^{(21k^2 + 17k + 2)/2} \Tone{n+2}{7k+3}{\sqrt{q}} 
\tag{\ref{t-SLfin}.82} \\
  P_0 = 1 \\
  P_1 = q^2 + 1\\
  P_2 = q^5 + q^3 + q^2 + q + 1\\
  P_n = (1 + q^{n+1}) P_{n-1} + q P_{n-2} - q P_{n-3} 
   \mbox{\ if $n\geqq 3$} 
   \end{gather*}
\end{id}

\begin{obs} Identity $(83)$ is the same as $(39)$. \end{obs}
\begin{obs} Identity $(84)$ is the same as $(9)$. \end{obs}
\begin{obs} Identity $(85)$ is the same as $(52)$. \end{obs}
\begin{obs} Identity $(86)$ is the same as $(38)$. \end{obs}
\begin{obs} Identity $(87)$ is the same as $(27)$. \end{obs}
\begin{obs} Identity $(88)$ is equivalent to $(91) - q^2\times(90)$. \end{obs}
\begin{obs} Identity $(89)$ is equivalent to $(93) - q\times(91)$. \end{obs}

\begin{id}[Finite form of \ref{t-ljslist}.90] 
Bosonic representation conjectured by Santos~\cite[p. 83, eqn. 6.52]{jpos}.
\index{Santos, J. P. O.}
\begin{gather*}
  \sum_{i\geqq 0} \sum_{j\geqq 0}\sum_{k\geqq 0}
   (-1)^i q^{j^2 + 3j + 3i(i+1)/2 + k}
  \gp{j}{i}{q^3} \gp{2j+k+1}{k}{q} \gp{n-j- 3i- 2k}{j}{q}
  \\ =
  \sum_{j=-\infty}^\infty (-1)^j q^{(27j^2 + 21j)/2} 
    \gp{n+3}{\lfloor \frac{n + 9j + 7}{2} \rfloor}{q} 
\tag{\ref{t-SLfin}.90}\\
  P_0 = 1 \\
  P_1 = 1\\
  P_2 = q^4 + q^2 + q + 1\\
  P_3 = q^5 + q^4 + q^2 + q + 1\\
  P_4 =q^{10} + q^8 + q^7  + 2q^6 + 2q^5 + 2q^4 + q^3 + 2q^2 + q + 1\\
  P_n =  P_{n-1} + (q + q^2 + q^{n+2}) P_{n-2} 
   -(q^2+q) P_{n-3} -q^3 P_{n-4} \\+ (q^3-q^{n+2}) P_{n-5} 
   \mbox{\ if $n\geqq 5$}
\end{gather*}
\end{id}

\begin{id}[Finite form of \ref{t-ljslist}.91] 
Bosonic representation conjectured for even $n$ by 
Santos~\cite[p. 84, eqn. 6.53]{jpos}.
\index{Santos, J. P. O.}
\begin{gather*}
   \sum_{i\geqq 0} \sum_{j\geqq 0}\sum_{k\geqq 0}
  (-1)^i q^{j^2 + 2j + 3i(i+1)/2 + k}
  \gp{j}{i}{q^3} \gp{2j+k+1}{k}{q} \gp{n-j- 3i- 2k}{j}{q}
   \nonumber  \\ =
  \sum_{j=-\infty}^\infty (-1)^j q^{(27 j^2 + 15j)/2} 
      \gp{n+2}{\lfloor \frac{n + 9j + 5}{2} \rfloor}{q} 
\tag{\ref{t-SLfin}.91}\\
  P_0 = 1 \\
  P_1 = 1\\
  P_2 = q^3 + q^2 + q + 1\\
  P_3 = q^4 + q^3 + q^2 + q + 1\\
  P_4 =q^8 + q^7  + q^6 + 2q^5 + 3q^4 + 2q^3 + 2q^2 + q + 1\\
  P_n =  P_{n-1} + (q + q^2 + q^{n+1}) P_{n-2} 
   -(q^2+q) P_{n-3} -q^3 P_{n-4} \\+ (q^3-q^{n+1}) P_{n-5} 
   \mbox{\ if $n\geqq 5$}
\end{gather*}
\end{id}

\begin{id}[Finite form of \ref{t-ljslist}.92] 
Bosonic representation conjectured by Santos~\cite[p. 84, eqn 6.54]{jpos}.
\index{Santos, J. P. O.}
\begin{gather*}
   \sum_{i\geqq 0} \sum_{j\geqq 0} \sum_{k\geqq 0}
    (-1)^i q^{j^2 + j + 3i(i+1)/2  + k }
  \gp{j}{i}{q^3} \gp{2j+k}{k}{q} 
   \gp{n-j- 3i- 2k}{j}{q} \\ =
  \sum_{j=-\infty}^\infty (-1)^j q^{(27 j^2 + 9j)/2} 
      \gp{n+1}{\lfloor \frac{n + 9j + 3}{2} \rfloor}{q} 
\tag{\ref{t-SLfin}.92}\\
  P_0 = 1 \\
  P_1 = 1\\
  P_2 = q^2 + q + 1\\
  P_3 = q^3 + q^2 + q + 1\\
  P_n =  (1-q)P_{n-1} + (2q + q^n) P_{n-2} 
   +(q-q^2+q^n) P_{n-3} +(q^n-q^2) P_{n-4} 
   \mbox{\ if $n\geqq 4$}
\end{gather*}
\end{id}

\begin{id}[Finite form of \ref{t-ljslist}.93] 
Bosonic representation for even $n$ conjectured by 
Santos~\cite[p. 85, eqn 6.55]{jpos}.
\index{Santos, J. P. O.}
\begin{gather*}
  \sum_{j\geqq 0} \sum_{i\geqq 0} \sum_{k\geqq 0} 
  (-1)^{i} q^{j^2 + 3i(i+1)/2 + k } \gp{j-1}{i}{q^3}
    \gp{2j+k-2}{k}{q}  
    \gp{n-3i-j-2k}{j}{q} \\ = 
  \sum_{j=-\infty}^\infty (-1)^j q^{(27 j^2 + 3j)/2} 
      \gp{n}{\lfloor \frac{n + 9j + 1}{2} \rfloor}{q} \tag{\ref{t-SLfin}.93}\\
  P_0 = 1 \\
  P_1 = 1\\
  P_2 = q + 1\\
  P_3 = q^2 + q + 1\\
  P_4 = q^4 + q^3 + 2q^2 + q + 1\\
  P_n =  P_{n-1} + (q + q^2 + q^{n-1}) P_{n-2} 
   -(q^2+q) P_{n-3} -q^3 P_{n-4} \\+ (q^3-q^{n-1}) P_{n-5} 
   \mbox{\ if $n\geqq 5$}
\end{gather*}
\end{id}

\begin{id}[Finite form of \ref{t-ljslist}.94]
Eqn. \textup{(\ref{t-SLfin}.94-b)}  was not stated 
explicilty but was indicated indirectly 
by Andrews~\cite[p. 5291]{hhm} as the dual of eqn. \textup{(4.1)} for odd $N$.
\index{Andrews, George E.}
  \begin{gather*} 
     \sum_{j\geqq 0}  q^{j(j+1)}  \gp{n+j+1}{2j+1}{q} \\ =
     \sum_{k=-\infty}^\infty q^{15k^2 + 4k} \gp{2n+1}{n+5k+1}{q} - 
       q^{15k^2 + 14k + 3} \gp{2n+1}{n+5k+3}{q} \tag{\ref{t-SLfin}.94-b} \\ =
    \sum_{j=-\infty}^\infty (-1)^j q^{10j^2 + 3j    } \Tone{n+1}{5j+1}{q} +
    \sum_{j=-\infty}^\infty (-1)^j q^{10j^2 + 7j + 1} \Tone{n+1}{5j+2}{q}
     \tag{\ref{t-SLfin}.94-t}\\
  P_0 = 1 \\
  P_1 = q^2 + q + 1\\
  P_n = (1 + q + q^{2n}) P_{n-1} - q P_{n-2} 
   \mbox{\ if $n\geqq 2$}
  \end{gather*}
\end{id}

\begin{id}[Finite form of \ref{t-ljslist}.95] 
Bosonic representation conjectured by Santos~\cite[p. 86, eqn. 6.57]{jpos}.
\index{Santos, J. P. O.}
\begin{gather*}
  \sum_{j\geqq 0}\sum_{k\geqq 0}\sum_{l\geqq 0}
   (-1)^l q^{3j^2 - 2j + k + 2l} \gp{j+k-1}{k}{q^2}
   \gp{j+l-1}{l}{q^2} \gp{n - 2j - k - l}{j}{q^2} \\ = 
  \sum_{k=-\infty}^\infty q^{15 k^2 +  4k    } \U{n-2}{5k  }{q} -
                          q^{15 k^2 + 14k + 3} \U{n-2}{5k+2}{q} 
\tag{\ref{t-SLfin}.95}\\
  P_0 = P_1 = P_2 = 1 \\
  P_n = (1 + q -q^2) P_{n-1} + (q^3 + q^2 -q) P_{n-2}  +(q^{2n-5}-q^3) 
  P_{n-3}
   \mbox{\ if $n\geqq 3$}
\end{gather*}
\end{id}

\begin{id}[Finite form of \ref{t-ljslist}.96]
Eqn. \textup{(\ref{t-SLfin}.96-b)}  was not stated 
explicilty but was indicated indirectly 
by Andrews~\cite[p. 5291]{hhm} as the dual of eqn. \textup{(4.2)} for odd $N$.
\index{Andrews, George E.}
\begin{gather*}
   \sum_{j\geqq 0} q^{j(j+2)} \gp{n+j+1}{2j+1}{q} \\ = 
  \sum_{k=-\infty}^\infty  q^{15k^2 + 7k   }\gp{2n+2}{n+5k+2}{q} -
                           q^{15k^2 +13k +2}\gp{2n+2}{n+5k+3}{q} 
                           \tag{\ref{t-SLfin}.96-b}\\
= \sum_{j=-\infty}^\infty (-1)^j q^{10j^2 + 6j} \U{n+1}{5j+1}{q} 
\tag{\ref{t-SLfin}.96-t}\\
  P_0 = 1 \\
  P_1 = q^3 + q + 1\\
  P_n = (1 + q + q^{2n+1}) P_{n-1} - q P_{n-2} 
   \mbox{\ if $n\geqq 2$}
\end{gather*} 
\end{id}

\begin{obs}
Identity (97) is equivalent to (95).
\end{obs} 

\begin{obs} Identity $(98)$ is the same as $(79)$. \end{obs}

\begin{id}[Finite forms of \ref{t-ljslist}.99]
Eqn. \textup{(\ref{t-SLfin}.99-b)}  was not stated 
explicilty but was indicated 
indirectly 
by Andrews~\cite[p. 5291]{hhm} as the dual of eqn. \textup{(4.2)} for even $N$.
\index{Andrews, George E.}
  \begin{gather*} 
   \sum_{j\geqq 0}  q^{j(j+1)} \gp{n+j}{2j}{q}  \\ =
     \sum_{k=-\infty}^\infty q^{15k^2 + 2k    } \gp{2n+1}{n+5k+1}{q} -
                             q^{15k^2 + 8k + 1} \gp{2n+1}{n+5k+2}{q}
    \tag{\ref{t-SLfin}.99-b}\\
   =\sum_{j=-\infty}^\infty (-1)^j q^{10j^2 + j  }\Tone{n+1}{5j}{q}   
                           -(-1)^j q^{10j^2 +9j+2}\Tone{n+1}{5j+2}{q}
   \tag{\ref{t-SLfin}.99-t}\\
  P_0 = 1 \\
  P_1 = q^2 + 1\\
  P_n = (1 + q + q^{2n}) P_{n-1} - q P_{n-2} 
   \mbox{\ if $n\geqq 2$}
  \end{gather*}
\end{id}

\begin{id}[Finite form of \ref{t-ljslist}.100]
  \begin{gather*} 
  \sum_{j\geqq 0}\sum_{k\geqq 0}\sum_{l\geqq 0}
   (-1)^l q^{3j^2 + k + 2l}
  \gp{j+k-1}{k}{q^2} \gp{j+l-1}{l}{q^2}
  \gp{n - 2j - k - l}{j}{q^2}\\ = 
  \sum_{k=-\infty}^\infty q^{15k^2 +  2k    } \U{n-1}{5k    }{q}
                        - q^{15k^2 +  8k + 1} \U{n-1}{5k + 1}{q}
    \tag{\ref{t-SLfin}.100}\\
  P_0 = P_1 = P_2 =1 \\
  P_n = (1+q-q^2) P_{n-1} + (q^3+q^2-q) P_{n-2} 
    + (q^{2n-3} - q^3) P_{n-3} 
   \mbox{\ if $n\geqq 3$}
  \end{gather*}
\end{id}

\begin{obs} Identity $(101)$ is the sum of $(105$-a$)$ and $(104)$.  \end{obs}

\begin{obs}[Finite form of \ref{t-ljslist}.102]
Identity $(102)$ is $(105$-a$) + q\times (103)$.
\end{obs}

\begin{id}[Finite form of \ref{t-ljslist}.103]
  \begin{gather*} 
  \sum_{i\geqq 0} \sum_{I\geqq0} \sum_{j\geqq 0} 
     \sum_{k\geqq 0}\sum_{K\geqq 0}  
     (-1)^I
      q^{j(j+3)/2 + i^2 + i  + I(I+1)/2 + k +K}  \\
      \times \gp{j}{i}{q^2} \gp{j}{I}{q} \gp{j+k}{k}{q^2} \gp{j+K}{K}{q} 
      \gp{n-2i-I-2k-K}{j}{q} \\ =
    \sum_{k=-\infty}^\infty q^{16k^2 +  8k    } \Tone{n+2}{8k+2}{\sqrt{q}} 
                          - q^{16k^2 + 16k + 3} \Tone{n+2}{8k+4}{\sqrt{q}} 
                          \tag{\ref{t-SLfin}.103}\\
  P_0 = 1 \\
  P_1 = q^2 + q + 1\\
  P_2 = q^5 + q^4 + q^3 + 2q^2 + q + 1\\
  P_3 = q^9 + q^8 + q^7 + q^6 + 3q^5 + 3q^4 + 3q^3 + 3q^2 + 2q + 1\\
  P_n =  (1+q+q^{n+1}) P_{n-1}  -q P_{n-2} 
   +(q^{n+1}-q^2-q) P_{n-3} + (q^2- q^{n+1}) P_{n-4}  
   \mbox{\ if $n\geqq 4$}
  \end{gather*}
\end{id}
 
\begin{id}[Finite form of \ref{t-ljslist}.104]
   \begin{gather*}
     1+\sum_{j\geqq 1} \sum_{i\geqq 0} \sum_{k\geqq 0}
       q^{j(j+1)/2 + i^2 + i + k} \gp{j-1}{i}{q^2} \gp{j+k-1}{k}{q^2}
       \gp{n-2i-2k}{k}{q}\\ =
     \sum_{k=-\infty}^\infty q^{16k^2}
      \V{n+1}{8k+1}{\sqrt{q}} 
        - q^{16k^2 + 8k +1}
     \V{n+1}{8k+3}{\sqrt{q}}
     \tag{\ref{t-SLfin}.104}\\
     P_0 = 1\\
     P_1 = q + 1\\
     P_2 = q^3 + q^2 + q + 1\\
     P_3 = q^6 + q^5 + q^4 + 2q^3 + 2q^2 + q + 1\\
     P_n = (1+q+q^n) P_{n-1} - q^n P_{n-2} + (q^n-q^2-q) P_{n-3}  + 
     (q^2-q^n) P_{n-4}\mbox{\ if $n\geqq 4$ }
    \end{gather*}
\end{id}

\begin{obs} Identity $(105)$ is the same as $(37)$. \end{obs}

\noindent\begin{id*}[Finite form of Identity 
\ref{t-ljslist}.105-a]{\ref{t-SLfin}.105-a}
 \begin{gather*}
    \sum_{i\geqq 0} \sum_{j\geqq 0} \sum_{k\geqq 0}  
    q^{ j(j+1)/2 + i^2 + i +k}
     \gp{j}{i}{q^2} \gp{j+k}{k}{q^2}  \gp{n-2i-2k}{j}{q} 
   \tag{\ref{t-SLfin}.105-a}\\
 = \sum_{j=-\infty}^\infty (-1)^j q^{4j^2 + 2j} \Tone{n+1}{4j+1}{\sqrt{q}}\\
  P_0 = 1 \\
  P_1 = q + 1\\
  P_2 = q^3 + q^2 + 2q + 1\\
  P_n =  (1+q^n) P_{n-1} + q P_{n-2} + (q^n - q) P_{n-3}
   \mbox{\ if $n\geqq 3$}
  \end{gather*}
\end{id*}

\begin{obs} Identity $(106)$ is the same as $(35)$. \end{obs}

\begin{id}[Finite form of \ref{t-ljslist}.107]
  \begin{gather*}
     \sum_{i\geqq 0} \sum_{j\geqq 0}\sum_{k\geqq 0}\sum_{l\geqq 0}
     (-1)^i q^{j^2 + j + 3i^2 + 2k + l}
       \gp{j}{i}{q^6} \gp{j+k}{k}{q^4} \gp{j+l-1}{l}{q^2}
       \gp{n-3i-2k-l}{j}{q^2} \\ =
     \sum_{j=-\infty}^\infty 
  (-1)^j q^{18j^2 + 3j}    
      \V{n+1}{6j+1}{q}  
 + (-1)^j q^{18j^2 +15j + 3} \V{n+1}{6j+3}{q} 
    \tag{\ref{t-SLfin}.107}\\
  P_0 = 1 \\
  P_1 = q^2 + 1\\
  P_2 = q^6 + q^4 + q^3 + 2q^2 + 1\\
  P_n =  (1+ q^{2n}) P_{n-1} + (q^2 + q^{2n-1}) P_{n-2} 
   + (q^{2n-2} - q^2) P_{n-3}
   \mbox{\ if $n\geqq 3$}
   \end{gather*}
\end{id}

\begin{obs} Identity $(108)$ is equivalent to $(115) - q^2\times(116)$. 
\end{obs}

\begin{obs} Identity $(109)$ is equivalent to $(109$-a$) + q\times (110)$.
\end{obs}

\noindent\begin{id*}[Finite form of Identity 
\ref{t-ljslist}.109-a]{\ref{t-SLfin}.109-a}
 \begin{gather*}
   \sum_{i\geqq 0} \sum_{j\geqq 0}\sum_{k\geqq 0}\sum_{K\geqq 0}\sum_{l\geqq 0} 
    (-1)^{i+L}
     q^{j^2 + 3i^2 + k + K + 2L}
     \gp{j}{i}{q^6} \gp{j+k-1}{k}{q^2} 
     \gp{j+K-1}{K}{q^2} \\ \times
     \gp{j+L-1}{L}{q^2}  \gp{n-3i-k-K-L}{j}{q^2}  \\ = 
   \sum_{j=-\infty}^\infty 
 (-1)^j q^{18 j^2         } \Big[\Tzero{n}{6j  }{q} + \Tzero{n-1}{6j  }{q}\Big]\\
+(-1)^j q^{18j^2 + 12j + 2} \Big[\Tzero{n}{6j+2}{q} + \Tzero{n-1}{6j+2}{q}\Big] 
     \tag{\ref{t-SLfin}.109-a} \\
  P_0 = 1\\
  P_1 = q + 1\\
  P_2 = q^4 + 2q^2 + q + 1\\
  P_n = (1+q-q^2+q^{2n-1})P_{n-1} + (q^3+q^2-q+q^{2n-2})P_{n-2}  + 
    (q^{2n-3}-q^3) P_{n-3} 
    \mbox{\quad if $n\geqq 3$}
 \end{gather*}
\end{id*}

\begin{id}[Finite form of \ref{t-ljslist}.110]
 \begin{gather*}
   \sum_{i\geqq 0} \sum_{j\geqq 0}\sum_{k\geqq 0}\sum_{K\geqq 0} 
   \sum_{L\geqq 0} (-1)^{i+L} 
      q^{j^2 + 2j + 3i^2 +k + K + 2L}
      \gp{j}{i}{q^6} \gp{j+k}{k}{q^2} \gp{j+K-1}{K}{q^2}
     \nonumber \\ \times
     \gp{j+L-1}{L}{q^2} \gp{n-3i-k-K-L}{j}{q^2} \nonumber \\ =
\sum_{j=-\infty}^\infty (-1)^j q^{18j^2 + 6j} 
  \Big[\Tzero{n+1}{6j+1}{q} + \Tzero{n}{6j+1}{q}\Big] \tag{\ref{t-SLfin}.110} \\
  P_0 = 1\\
  P_1 = q^3 + q + 1\\
  P_2 = q^8 + q^6 + 2q^4 + q^3 + q^2 + q + 1\\
  P_n = (1+q-q^2+q^{2n+1})P_{n-1} + (q^3 + q^2 -q + q^{2n} )P_{n-2}  + 
    (q^{2n-1} - q^3 ) P_{n-3} 
    \mbox{\quad if $n\geqq 3$}
 \end{gather*}
\end{id}

\begin{obs} Identity $(111)$ is equivalent to $(114) -q\times(115)$. \end{obs}

\begin{obs} Identity $(112)$ is equivalent to $(115) +q^3\times(116)$. \end{obs}

\begin{obs} Identity $(113)$ is equivalent to $(114) -q^3\times(115)$. \end{obs}

\begin{id}[Finite form of \ref{t-ljslist}.114] 
Bosonic representation conjectured by Santos~\cite[p. 89, eqn. 6.63]{jpos}.
\index{Santos, J. P. O.}
\begin{gather*}
 1 + \sum_{j\geqq 1} \sum_{i\geqq 0} \sum_{k\geqq 0}\sum_{K\geqq 0} 
 \sum_{l\geqq 0}
 (-1)^{i+K}
   q^{j^2 + 3i^2 + 3i + k + 2K + 2l} \gp{j-1}{i}{q^6} \gp{j+k-1}{k}{q^2}
   \\ \times \gp{j+K-2}{K}{q^2}
   \gp{j+l-1}{l}{q^2} \gp{n-1-3i-k-K-k}{j-1}{q^2} \\ = 
 \sum_{j=-\infty}^\infty (-1)^j q^{18j^2 + 3j} \U{n}{6j}{q} \tag{\ref{t-SLfin}.114}
\\ P_0 = 1\\
  P_1 = q + 1\\
  P_2 = q^4 + q^3 + q^2 + q + 1\\
  P_3 = q^9 + q^8 + q^7 + q^6 + 2q^5 + 2q^4 + 
   2q^3 + q^2 + q + 1\\
  P_n = (1+q+q^{2n+1})P_{n-1} + (q^4-q)P_{n-2} -(q^4+q^5) P_{n-3} \\ 
    +(q^5 - q^{2n-1}) P_{n-4} \mbox{ if $n\geqq 3$}
\end{gather*}
\end{id}

\begin{id}[Finite form of \ref{t-ljslist}.115] 
Bosonic representation conjectured by Santos~\cite[p. 89, eqn. 6.64]{jpos}.
\index{Santos, J. P. O.}
\begin{gather*}
 \sum_{i\geqq 0} \sum_{j\geqq 0}\sum_{k\geqq 0}\sum_{K\geqq 0} 
  (-1)^i q^{j^2 + 2j + 3i^2 + 3i + k + 4K}
  \gp{j}{i}{q^6} \gp{j+k}{k}{q^4} \gp{j+K}{K}{q^4} \\ \times
  \gp{n - 3i  - k - 2K}{j}{q^2} \\= 
 \sum_{j=-\infty}^\infty (-1)^j q^{18j^2 + 9j} \U{n+1}{6j+1}{q} 
  \tag{\ref{t-SLfin}.115} \\
  P_0 = 1\\
  P_1 = q^3 + q + 1\\
  P_2 = q^8 + q^6 + q^5 + 2q^4 + q^3 + q^2 + q + 1\\
  P_3 = q^{15} + q^{13} + q^{12} + 2q^{11} + q^{10} + 2q^9 + 2q^8 + 3q^7 \\ + 
  2q^6 + 3q^5 + 2q^4 +  2q^3 + q^2 + q + 1\\
  P_n = (1+q+q^{2n+1})P_{n-1} + (q^4 - q +)P_{n-2} 
   - (q^4 + q^5) P_{n-3}  +(q^5 - q^{2n+1}) P_{n-4} 
  \mbox{\quad if $n\geqq 3$}
\end{gather*}
\end{id}

\begin{id}[Finite form of \ref{t-ljslist}.116] 
Bosonic representation conjectured by  Santos~\cite[p. 90, eqn 6.65]{jpos}.
\index{Santos, J. P. O.}
\begin{gather*}
  \sum_{i\geqq 0}  \sum_{j\geqq 0}\sum_{k\geqq 0}\sum_{K\geqq 0} 
  (-1)^i q^{j^2 + 4j + 3i^2 + 3i + k + 4K}
  \gp{j}{i}{q^6} \gp{j+k}{k}{q^2} \gp{j+K}{K}{q^4} 
  \gp{n - 3i -k - 2K}{j}{q^2} \\=
 \sum_{j=-\infty}^\infty (-1)^j q^{18j^2 + 15j} \U{n+2}{6j+2}{q} 
 \tag{\ref{t-SLfin}.116}\\
  P_0 = 1\\
  P_1 = q^5 + q + 1\\
  P_2 = q^{12} + q^8 + q^7 + q^6 + q^5 + q^4 + q^2 + q + 1\\
  P_3 = q^{21} + q^{17} + q^{16} + q^{15} + q^{14} + 2q^{13} + q^{12} + q^{11} + 
  q^{10} + 3q^9 \\
    + 2q^8 + 2q^7 + q^6 + 2q^5 + q^4 +  q^3 + q^2 + q + 1\\
  P_n = (1+q+q^{2n+3})P_{n-1} + (q^4 -q)P_{n-2}  
    - (q^5 + q^4) P_{n-3}  +(q^5 - q^{2n+3}) P_{n-4} 
    \mbox{\quad if $n\geqq 4$}
\end{gather*}
\end{id}

\begin{id}[Finite form of \ref{t-ljslist}.117] 
Bosonic representation conjectured by Santos~\cite[p. 88, eqn. 6.62]{jpos}.
\index{Santos, J. P. O.}
\begin{gather*}
  \sum_{j\geqq 0}\sum_{k\geqq 0}\sum_{K\geqq 0} 
  (-1)^K q^{j^2 + k + 2K} \gp{j+k-1}{k}{q^2}
  \gp{j+K-1}{K}{q^2} \gp{n-k-K}{j}{q^2} \\  = 
 \sum_{k=-\infty}^\infty q^{21k^2 +  2k    } \U{n}{7k    }{q} -
                         q^{21k^2 + 16k + 3} \U{n}{7k + 2}{q}
 \tag{\ref{t-SLfin}.117}\\
  P_0 = 1\\
  P_1 = q + 1\\
  P_2 = q^4 + q^2 + q + 1\\
  P_n = (1+ q -q^2 + q^{2n-1})P_{n-1} + (q^3 + q^2 -q)P_{n-2}  
    - q^3 P_{n-3} \mbox{\quad if $n\geqq 3$}
\end{gather*}
\end{id}

\begin{id}[Finite form of \ref{t-ljslist}.118] 
Bosonic representation conjectured by Santos~\cite[p. 90, eqn. 6.65]{jpos}
\index{Santos, J. P. O.}
\begin{gather*}
  \sum_{j\geqq 0}\sum_{k\geqq 0}\sum_{K\geqq 0} 
 (-1)^K q^{j^2 + 2j + k + 2K} \gp{j+k-1}{k}{q^2}
  \gp{j+K-1}{K}{q^2} \gp{n-k-K}{j}{q^2} \\=
  \sum_{k=-\infty}^\infty q^{21k^2 +  4k    } \U{n+1}{7k  }{q} -
                          q^{21k^2 + 10k + 1} \U{n+1}{7k+1}{q}
  \tag{\ref{t-SLfin}.118}\\
  P_0 = 1\\
  P_1 = q^3 + 1\\
  P_2 = q^8 + q^4 + q^3 + 1\\
  P_n = (1+ q -q^2 + q^{2n+1})P_{n-1} + (q^3 + q^2 -q)P_{n-2}  
    - q^3 P_{n-3} \mbox{\quad if $n\geqq 3$}
\end{gather*}
\end{id}

\begin{id}[Finite form of \ref{t-ljslist}.119] 
Bosonic representation conjectured by Santos~\cite[p. 91, eqn. 6.67]{jpos}.
\index{Santos, J. P. O.}
\begin{gather*}
  \sum_{j\geqq 0}\sum_{k\geqq 0}\sum_{K\geqq 0} 
  (-1)^K q^{j^2 + 2j + k + 2K} \gp{j+k}{k}{q^2}
  \gp{j+K-1}{K}{q^2} \gp{n-k-K}{j}{q^2} \\=
  \sum_{k=-\infty}^\infty q^{21k^2 +  8k     } \U{n+1}{7k+1}{q} -
                          q^{21k^2 + 20k +  4} \U{n+1}{7k+3}{q}
 \tag{\ref{t-SLfin}.119}\\
  P_0 = 1\\
  P_1 = q^3 + q + 1\\
  P_2 = q^8 + q^6 + q^4 + q^3 + q^2 + q + 1\\
  P_n = (1+ q -q^2 + q^{2n+1})P_{n-1} + (q^3 + q^2 -q)P_{n-2}  
    - q^3 P_{n-3} \mbox{\quad if $n\geqq 3$}
\end{gather*}
\end{id}

\begin{id}[Finite form of \ref{t-ljslist}.120] 
Bosonic representation conjectured by Santos~\cite[p. 91, eqn. 6.68]{jpos}.
\index{Santos, J. P. O.}
\begin{gather*}
  1+\sum_{j\geqq 1}\sum_{i\geqq 0}\sum_{k\geqq 0}
     q^{j^2 + j + i^2 + i + k} \gp{j-1}{i}{q^2}
      \gp{j+k-1}{k}{q^2} \gp{n-i-k}{j}{q^2} \nonumber \\ = 
  \sum_{k=-\infty}^\infty q^{24k^2 +  2k     } \gp{2n+1}{n+6k+1}{q} -
                          q^{24k^2 + 10k +  1} \gp{2n+1}{n+6k+2}{q}
  \tag{\ref{t-SLfin}.120}\\
  P_0 = 1\\
  P_1 = q^2 + 1\\
  P_2 = q^6 + q^4 + q^3 + q^2 + 1\\
  P_n = (1+ q + q^2 + q^{2n})P_{n-1} - (q^3 + q^2 +q)P_{n-2}  
    + (q^3 - q^{2n}) P_{n-3} \mbox{\quad if $n\geqq 3$}
\end{gather*}
\end{id}

\begin{id}[Finite form of \ref{t-ljslist}.121] 
Bosonic representation conjectured by Santos~\cite[p. 92, eqn. 6.69]{jpos}.
\index{Santos, J. P. O.}
\begin{gather*}
  1+\sum_{j\geqq 1}\sum_{i\geqq 0}\sum_{k\geqq 0} 
    q^{j^2 +  i^2 + i + k} \gp{j-1}{i}{q^2}
      \gp{j+k-1}{k}{q^2} \gp{n-i-k}{j}{q^2} \nonumber \\ =
  \sum_{k=-\infty}^\infty q^{24k^2 +  2k     } \gp{2n+1}{n+6k+1}{q} -
                          q^{24k^2 + 14k +  2} \gp{2n+1}{n+6k+2}{q}
  \tag{\ref{t-SLfin}.121}\\
  P_0 = 1\\
  P_1 = q + 1\\
  P_2 = q^4 + q^3 + q^2 + q + 1\\
  P_n = (1+ q + q^2 + q^{2n-1})P_{n-1} - (q^3 + q^2 +q)P_{n-2}  
    + (q^3 - q^{2n-1}) P_{n-3} \mbox{\quad if $n\geqq 3$}
\end{gather*}
\end{id}

\begin{id}[Finite form of \ref{t-ljslist}.122] 
Bosonic representation conjectured by Santos~\cite[p. 92, eqn. 6.70]{jpos};
corrected below.
\index{Santos, J. P. O.}
\begin{gather*}
  \sum_{i,j,k\geqq 0} q^{j^2 + 3j + i^2 + i + k} \gp{j}{i}{q^2}
    \gp{j+k}{k}{q^2} \gp{n+1-i-k}{j+1}{q^2} \\ = 
  \sum_{k=-\infty}^\infty q^{24k^2 + 14k     } \gp{2n+3}{n+6k+3}{q} -
                          q^{24k^2 + 22k +  3} \gp{2n+3}{n+6k+4}{q}
  \tag{\ref{t-SLfin}.122}\\
  P_0 = 1\\
  P_1 = q^4 + q^2 + q + 1\\
  P_2 = q^{10} + q^8 + q^7 + 2q^6 + q^6 + 2q^4 + q^3 + 2q^2 + q + 1\\
  P_n = (1+ q + q^2 + q^{2n+2})P_{n-1} - (q^3 + q^2 + q) P_{n-2} +
    (q^3 - q^{2n+2})P_{n-3}  \mbox{\quad if $n\geqq 3$}
\end{gather*}
\end{id}

\begin{id}[Finite form of \ref{t-ljslist}.123] 
Bosonic representation conjectured by Santos~\cite[p. 93, eqn. 6.71]{jpos};
corrected below.
\index{Santos, J. P. O.}
\begin{gather*}
  \sum_{i\geqq 0} \sum_{j\geqq 0}\sum_{k\geqq 0}\sum_{l\geqq 0}  
  (-1)^i q^{2j^2 + 2j + 2i^2 + 2i+k + 2l}
  \gp{j}{i}{q^4} \gp{j+k}{k}{q^2} \gp{j+l}{l}{q^2}
  \gp{n - 2i - k - l}{j}{q^2} \nonumber \\ =
  \sum_{k=-\infty}^\infty q^{24k^2 + 10k    } \gp{2n+2}{n + 6k + 2}{q} -
                          q^{24k^2 + 22k + 4} \gp{2n+2}{n + 6k + 4}{q}
  \tag{\ref{t-SLfin}.123}\\
  P_0 = 1\\
  P_1 = q^3 + q^2 + q + 1\\
  P_2 = q^8 + q^7 + q^6 + 2q^5 + 2q^4 + 2q^3 + 2q^2 + q + 1\\
  P_n = (1+ q + q^2 + q^{2n+1})P_{n-1} - (q^3 + q^2 +q)P_{n-2}  
    + (q^3 - q^{2n+1}) P_{n-3} \mbox{\quad if $n\geqq 3$}
\end{gather*}
\end{id}

\begin{id}[Finite form of \ref{t-ljslist}.124] 
Bosonic representation for even $n$ conjectured by 
Santos~\cite[p. 93, eqn. 6.72]{jpos}.
\index{Santos, J. P. O.}
 \begin{gather*}
  \sum_{i\geqq 0} \sum_{j\geqq 0}\sum_{k\geqq 0}\sum_{K\geqq 0}
  \sum_{l\geqq 0} \sum_{L\geqq 0}
    (-1)^{i+K+l} q^{2j^2 + 2j + 3i^2 + k + K + 
    2l + L} \gp{j}{i}{q^6} \gp{j+k}{k}{q^2} \gp{j+K}{K}{q^2} \\ \times
    \gp{j+l-1}{l}{q^2} \gp{j+L-1}{L}{q^2} 
    \gp{n-j-3i-k-K-l-L}{j}{q^2}  \tag{\ref{t-SLfin}.124}  
\end{gather*}
 \[ =\left\{ \begin{array}{l}
   \sum_k  q^{108 k^2 +  12 k     } \gp{2m+1}{m + 9k + 1}{q^2} 
        -  q^{108 k^2 +  60 k +  8} \gp{2m+1}{m + 9k + 3}{q^2}
        +  q^{108 k^2 +  48 k +  5} \gp{2m  }{m + 9k + 2}{q^2} \\ \qquad
        -  q^{108 k^2 +  96 k + 21} \gp{2m  }{m + 9k + 4}{q^2}
    \mbox{\ \ if $n = 2m$,} \\
  \sum_k  q^{108 k^2 +  12 k     } \gp{2m+1}{m + 9k + 1}{q^2}
        - q^{108 k^2 +  60 k +  8} \gp{2m+1}{m + 9k + 3}{q^2}
        + q^{108 k^2 +  48 k +  5} \gp{2m+2}{m + 9k + 3}{q^2} \\ \qquad
        - q^{108 k^2 +  96 k + 21} \gp{2m+2}{m + 9k + 5}{q^2}
     \mbox{\ \ if $n= 2m+1$.}
   \end{array} \right. \]
  \begin{gather*}
  P_0 = 1\\
  P_1 = 1\\
  P_2 = q^4 + q^2  + 1\\
  P_3 = q^5 + q^4 + q^2 + 1\\
  P_n = (1- q^2)P_{n-1} + (2q^2 +q^{2n})P_{n-2}  
    + (q^4 - q^2 + q^{2n-1}) P_{n-3} + (q^{2n-2}-q^4)P_{n-4} 
    \mbox{\quad if $n\geqq 3$}
  \end{gather*}
\end{id}

\begin{id}[Finite form of \ref{t-ljslist}.125] 
Bosonic representation for for even $n$ conjectured by 
Santos~\cite[p. 94, eqn. 6.73]{jpos}.
\index{Santos, J. P. O.}
 \begin{gather*}
  \sum_{i\geqq 0} \sum_{j\geqq 0}\sum_{k\geqq 0}\sum_{K\geqq 0}
  \sum_{l\geqq 0} \sum_{L\geqq 0} (-1)^{i+K+l} q^{2j^2 + 4j + 3i^2 + k + K +
    2l + L} \gp{j}{i}{q^6} \gp{j+k}{k}{q^2} \gp{j+K}{K}{q^2} \\ \times
    \gp{j+l-1}{l}{q^2}\gp{j+L-1}{L}{q^2} 
    \gp{n-j-3i-k-K-l-L}{j}{q^2}  \tag{\ref{t-SLfin}.125}
\end{gather*}
 \[ =\left\{ \begin{array}{l}
   \sum_k q^{108 k^2 +  24 k     } \gp{2m+2}{m + 9k + 2}{q^2} -
          q^{108 k^2 +  48 k +  4} \gp{2m+2}{m + 9k + 3}{q^2}
        + q^{108 k^2 +  60 k +  7} \gp{2m+1}{m + 9k + 3}{q^2} \\ \qquad -
          q^{108 k^2 +  84 k + 15} \gp{2m+1}{m + 9k + 4}{q^2}
   \mbox {\ \  if $n=2m$,} \\
   \sum_k q^{108 k^2 +  24 k     } \gp{2m+2}{m + 9k + 2}{q^2} -
          q^{108 k^2 +  48 k +  4} \gp{2m+2}{m + 9k + 3}{q^2}
        + q^{108 k^2 +  60 k +  7} \gp{2m+3}{m + 9k + 4}{q^2} \\ \qquad  -
          q^{108 k^2 +  84 k + 15} \gp{2m+3}{m + 9k + 5}{q^2}
    \mbox{\ \  if $n=2m+1$.}
    \end{array} \right. \]
 \begin{gather*}
  P_0 = 1\\
  P_1 = 1\\
  P_2 = q^6 + q^2  + 1\\
  P_3 = q^7 + q^6 + q^2 + 1\\
  P_n = (1- q^2)P_{n-1} + (2q^2 + q^{2n+2})P_{n-2}  
    + (-q^2 + q^4 + q^{2n+1}) P_{n-3} \\ + (q^{2n}-q^4)P_{n-4} 
     \mbox{\quad if $n\geqq 4$}
  \end{gather*}
\end{id}

\begin{obs}
Identity (126) is equivalent to 
$(71) + q\times(68) -q\times(128)$.
\end{obs}

\begin{obs}
Identity (127) is equivalent to $(71) - q\times (128).$
\end{obs}

\begin{id}[Finite form of \ref{t-ljslist}.128]
\begin{gather*}
   \sum_{j\geqq 0}\sum_{i\geqq 0}\sum_{I\geqq 0}\sum_{k\geqq 0} \sum_{K\geqq 0}
   \sum_{L\geqq 0} 
 (-1)^{i+L} q^{j^2 + 2j + 2i^2 + 2i + 2I^2 + 2I + 4k + K+2L} \gp{j}{i}{q^4}  
   \gp{j}{I}{q^4} \gp{j+k}{k}{q^4}\\ \times \gp{j+K}{K}{q^2} 
   \gp{j+L-1}{L}{q^2}
   \gp{n-2i-2I-2k-K+L}{j}{q^2} \\ =
  \sum_{k=-\infty}^\infty q^{32k^2 + 12k}     \U{n+1}{8k+1}{q} -
                          q^{32k^2 + 28k + 5} \U{n+1}{8k+3}{q}
   \tag{\ref{t-SLfin}.128}\\
  P_0 = 1\\
  P_1 = q^3 + q + 1\\
  P_2 = q^8 + q^6 + 2q^4 + q^3 + q^2 + q + 1\\
  P_3 = q^{15} + q^{13} + 2q^{11} + 2q^9 + q^8 + 3q^7 + q^6 + 2q^5 + 2q^4 + 
  2q^3 + q^2 + q + 1\\
   P_n = (1 + q + q^{2n+1}) P_{n-1} + (q^4 -q - q^{2n+1}) 
    P_{n-2} + (-q^5 - q^4 + q^{2n+1}) P_{n-3} \\
      +(q^5  - q^{2n+1})P_{n-4}  \mbox{\ if $n\geqq 4$.}
\end{gather*}
\end{id}

\begin{obs}
Identity (129) is equivalent to $q^{-2} \times \Big( (128) - (68) \Big)$.
\end{obs}

\begin{id}[Finite form of \ref{t-ljslist}.130]
   \begin{gather*}
     \sum_{i\geqq 0} \sum_{j\geqq 0}\sum_{k\geqq 0}\sum_{l\geqq 0}
  (-1)^l q^{j^2 + 2i^2 + k + 2l}
        \gp{j}{i}{q^4} \gp{j+k}{k}{q^2} \gp{j+l-1}{l}{q^2}
        \gp{n-2i-k-l}{j}{q^2} \\ =
   \sum_{j=-\infty}^\infty 
      (-1)^j q^{8j^2}         \Big[\Tzero{n}{4j}{q}   + \Tzero{n-1}{4j}{q} \Big] \\
   +  (-1)^j q^{8j^2 + 4j +1} \Big[\Tzero{n}{4j+1}{q} + \Tzero{n-1}{4j+1}{q} \Big] 
   \tag{\ref{t-SLfin}.130}\\
  P_0 = 1\\
  P_1 = 2q + 1\\
  P_2 = 2q^4 + 2q^2 + 2q + 1\\
  P_n = (1+q- q^2 + q^{2n-1})P_{n-1}  + (q^3 + q^2 -q)P_{n-2}  
    + (q^{2n-3} - q^3) P_{n-3} \mbox{\quad if $n\geqq 3$}
  \end{gather*}
\end{id}     
\index{Slater's list of Rogers-Ramanujan type identities!polynomial 
generalizations of|)}
\index{Slater, Lucy J.|)}

\section{Proving Conjectured Polynomial Identities}\label{proof}

We shall now discuss several methods by which one can prove the identities 
stated in the previous section.
\subsection{Direct Proof}

The following finitization of identity 39 on Slater's list 
 \begin{gather*}
   \sum_{j\geqq 0} q^{2j^2} \gp{n}{2j}{q} =
   \sum_{j=-\infty}^\infty q^{4j^2 + j} \gp{n}{\lfloor \frac{n+4j+1}{2} 
   \rfloor}{q^2}  \tag{\ref{t-SLfin}.39-b}
 \end{gather*}
was stated by Andrews and Santos in~\cite[p. 94, eqn. 3.1]{attached}, but 
only hint of the proof was given, since it is a finite analog of
the proof of (\ref{t-ljslist}.39), which was written out in full.
The details of the proof hinted at are 
presented below: \index{Andrews, George E.}\index{Santos, J. P. O.}

\begin{proof} 
\begin{eqnarray*}
 &   & \sum_{j\geqq 0} q^{2j^2} \gp{n}{2j}{q} \\
 & = & \sum_{j\geqq 0} q^{(2j)^2/2} \gp{n}{2j}{q} \\
 & = & \underset{k\mathrm{\ even}} {\sum_{k\geqq 0}} q^{k^2/2} \gp{n}{k}{q}\\
 & = & \frac12 \left( \sum_{k\geqq 0}        q^{k^2/2} \gp{n}{k}{q}
                     +\sum_{k\geqq 0} (-1)^k q^{k^2/2} \gp{n}{k}{q} \right) \\
 & = & \frac12 \Big[ (-\sqrt{q};q)_n + (\sqrt{q};q)_n \Big]
       \mbox{\qquad (by (\ref{qbc1}))}
\end{eqnarray*}

 CASE 1: (Even $n$)
\begin{eqnarray*}
 & = & \frac12 \Big[ (-\sqrt{q};q^2)_{n/2} (-\sqrt{q^3};q^2)_{n/2} 
                    +( \sqrt{q};q^2)_{n/2} ( \sqrt{q^3};q^2)_{n/2} \Big] \\
 & = & \frac12 \left(
       \sum_{k=-\infty}^\infty        q^{k^2 + k/2} \gp{n}{n/2 + k}{q^2}
      +\sum_{k=-\infty}^\infty (-1)^k q^{k^2 + k/2} \gp{n}{n/2 + k}{q^2} 
      \right)\\
 &   & \hspace*{1 in}\mbox{\quad (by Theorem~\ref{fjtp})} \\
 & = & \underset{k\mathrm{\ even}}{\sum_{k=-\infty}^\infty}
         q^{k^2 + k/2} \gp{n}{n/2 + k}{q^2} \\
 & = & \sum_{j=-\infty}^\infty q^{4j^2 + j} \gp{n}{\lfloor \frac{n+4j}{2} 
   \rfloor}{q^2}  \\
 & = & \sum_{j=-\infty}^\infty q^{4j^2 + j} \gp{n}{\lfloor \frac{n+4j+1}{2} 
   \rfloor}{q^2} 
\end{eqnarray*}

 CASE 2: (Odd $n$)
\begin{eqnarray*}
 & = & \frac12 \Big[ (-\sqrt{q};q^2)_{(n+1)/2} (-\sqrt{q^3};q^2)_{(n-1)/2} 
                    +( \sqrt{q};q^2)_{(n+1)/2} ( \sqrt{q^3};q^2)_{(n-1)/2} 
\Big] \\
 & = & \frac 12 \Big[ 
    (1+q^{(2n-1)/2})(-\sqrt{q};q^2)_{(n-1)/2} (-\sqrt{q^3};q^2)_{(n-1)/2} \\
 &   &\hskip 2cm
   +(1-q^{(2n-1)/2})( \sqrt{q};q^2)_{(n-1)/2} ( \sqrt{q^3};q^2)_{(n-1)/2} 
\Big] \\
 & = & \frac 12 \left(
    \sum_{k=-\infty}^\infty        q^{k^2+k/2} \gp{n-1}{(n-1)/2 + k}{q^2} +
    \sum_{k=-\infty}^\infty (-1)^k q^{k^2+k/2} \gp{n-1}{(n-1)/2 + k}{q^2} \right)
\\ & & \hskip 0.5cm
   +\frac{q^{(2n-1)/2}}{2} \Bigg(
    \sum_{k=-\infty}^\infty        q^{k^2+k/2} \gp{n-1}{(n-1)/2 + 
    k}{q^2} \\
  & &\hskip 3.2cm  +\sum_{k=-\infty}^\infty (-1)^k q^{k^2+k/2} \gp{n-1}{(n-1)/2 + k}{q^2} 
    \Bigg)
\\
  & = & \underset{k\mathrm{\ even}}{\sum_{k=-\infty}^\infty} 
     q^{k^2+k/2} \gp{n-1}{(n-1)/2 + k}{q^2}
      +\underset{k\mathrm{\ odd}}{\sum_{k=-\infty}^\infty} 
     q^{k^2+k/2} \gp{n-1}{(n-1)/2 + k}{q^2} \\
  & = & \sum_{j=-\infty}^\infty q^{4j^2-j} \gp{n-1}{(n-1)/2 - 2j}{q^2}\\
 & &\hskip 5mm +q^{n-1/2} \sum_{j=-\infty}^\infty q^{(1-2j)^2 +(1-2j)/2} 
     \gp{n-1}{(n-1)/2 - 2j+1}{q^2} \\
  & = & \sum_{j=-\infty}^\infty q^{4j^2-j} \gp{n-1}{(n-1)/2 - 2j}{q^2} +
   \sum_{j=-\infty}^\infty q^{n+1-4j}q^{4j^2-j} 
     \gp{n-1}{\frac{n+1-4j}{2}}{q^2} \\
  & = & \sum_{j=-\infty}^\infty q^{4j^2-j} \left(
       \gp{n-1}{\frac{n+1-4j}{2} - 1}{q^2}
     + q^{n+1-4j}\gp{n-1}{\frac{n+1-4j}{2} - 1}{q^2} \right) \\
  & = & \sum_{j=-\infty}^\infty q^{4j^2 - j} 
  \gp{n}{\frac{n+1-4j}{2}}{q^2} \mbox{\quad (by (\ref{qpt2}))}\\
  & = & \sum_{j=-\infty}^\infty q^{4j^2 + j} 
  \gp{n}{\frac{n+1+4j}{2}}{q^2} \\
   & = & \sum_{j=-\infty}^\infty q^{4j^2 + j} 
  \gp{n}{\lfloor \frac{n+1+4j}{2} \rfloor}{q^2}
\end{eqnarray*} 
\end{proof}

While the preceeding proof is nice in that it is a finite analog of the 
proof of (A.39) given by Andrews and Santos~\cite[pp. 93-94]{attached}, the 
method is not applicable to all the polynomial identities stated in 
\S~\ref{t-SLfin}. \index{Andrews, George E.}\index{Santos, J. P. O.}

\subsection{WZ proofs} \label{WZproofs}
We now turn our attention to a method which, in theory, will provide proofs 
for all of the polynomial identities in \S~\ref{t-SLfin}.   In 1990, 
Wilf and Zeilberger published their ground-breaking paper ``Rational 
Functions Certify Combinatorial Identities"~\cite{wz:rf}, for which they 
later won the prestigious Leroy P. Steele Prize for Seminal Contribution to 
Research~\cite{ams:steele}.  Their original methods, which were designed 
for single sum hypergeometric type identities, were successfully extended to 
multisum identities and $q$-analogs (see, e.g. Wilf and 
Zeilberger~\cite{wz:multiq} and Koornwinder~\cite{thk}). 
\index{Zeilberger, Doron}
\index{Wilf, Herbert S.}
\index{Koornwinder, T.}
\index{Steele Prize}

\subsubsection{A Brief Introduction to the WZ-theory}
\index{WZ method|(}
The following is a brief exposition of the part of Wilf and Zeilberger's
WZ-theory\footnote{named in honor of two famous complex variables; see~\cite[p. 
148, footnote 1]{wz:rf}.}  which will be required for our present purposes. 
For a detailed 
introduction to WZ theory, including the Sister Celine's algorithm,
Gosper's Algorithm, Zeilberger's algorithm, and WZ pairs, 
see Petkov{\v s}ek, Wilf and Zeilberger~\cite{pwz:a=b}.
\index{Zeilberger, Doron}
\index{Wilf, Herbert S.}
\index{Petkov{\v s}ek, Marko}
\index{Fasenmyer, Sister Mary Celine}

Recall that the identities conjectured in \S~\ref{t-SLfin} are of the
form 
\[ LHS(n) = RHS(n) \] where the LHS is a polynomial known to satisfy a
certain recurrence relation of order, say $r$.   
Thus, to prove the identity is true for all
$n$, it is sufficient to show that $RHS(n)$ satisfies the same $r$th order
recurrence
relation  and the initial conditions
\[ LHS(0) = RHS(0), \quad LHS(1)=RHS(1), \dots, LHS(r-1)=RHS(r-1). \]
  We need to introduce some notation at this point.  Let us say that
the $r$th order recurrence relation satisfied by $LHS(n)$ is
\begin{equation} \label{genrec}
  \sum_{i=0}^r p_i(q) LHS(n-i) = 0,
\end{equation}
where the $p_i(q)$ are polynomials in $q$ depending only on $n$, 
and note that $RHS(n)$ is of the form
$\sum_{j=-\infty}^\infty F(n,j)$.
We obtain a function $G(n,j)$ which satisfies the following conditions:
\begin{equation} \label{wzcond1}
  \sum_{i=0}^r p_i(q) F(n-i,j)= G(n,j) - G(n,j-1)
\end{equation} and
\begin{equation} \label{wzcond2}
 \lim_{j=\pm\infty} G(n,j) = 0.
\end{equation}
Then, by summing (\ref{wzcond1}) over all $j$, we automatically 
obtain 
\[ \sum_{i=0}^r p_i(q) RHS(n-i) = 0,\] and this observation together with the 
checking of the appropriate initial 
conditions, completes the proof.  
Of course, the big question is, ``How does 
one obtain this $G(n,j)$ function?"  The process by which one finds
$G$ is known as either\index{Zeilberger's algorithm} \index{creative 
telescoping} \index{Zeilberger, Doron}
``Zeilberger's Algorithm" or 
``creative telescoping"; see Zeilberger~\cite{dz:ct} and \cite{dz:fa}. 
For a detailed explaination of  
why such a $G$ is guaranteed to exist, and how to find one, the reader is 
referred to 
Petkov{\v s}ek, Wilf and Zeilberger~\cite{pwz:a=b}.  
\index{Zeilberger, Doron}
\index{Wilf, Herbert S.}
For our present 
purposes, we merely note that there are several packages available which 
will find $G(n,j)$ for a given $F(n,j)$.  Zeilberger's \texttt{qEKHAD} 
package\footnote{named 
in honor of Shalosh B. Ekhad, \index{Ekhad, Shalosh, B.}
Zeilberger's computer and prolific 
author in its 
own right.} 
for Maple is available for free download from his web site
\texttt{http://www.math.rutgers.edu/\~{ }zeilberg}.  Axel Riese 
has written Mathematica packages for proving both single and multisum 
$q$-hypergeometric identities.  They are available for download, free of 
charge to 
researchers and non-commercial users, at \\
\texttt{http://www.risc.uni-linz.ac.at/research/combinat/risc/software/}\\ 
and have accompanying documentation in
\index{Paule, Peter}
\index{Riese, Axel}
\index{qMultiSum (Mathematica package)}
Paule and Riese~\cite{pr:qzeil} and Riese~\cite{ar:qMultiSum} respectively.
  We note that the $G(n,j)$ function is a rational function multiple of
the summand function $F(n,j)$, i.e. that 
\[ G(n,j) = F(n,j) R(n,j) \] for some rational function $R(n,j)$.  
This function $R(n,j)$ is called the {\em WZ certificate} or simply {\em 
certificate} function, and it is actually this certificate function rather 
than the $G(n,j)$ function that is 
produced for a given summand $F(n,j)$ by the Maple and Mathematica 
packages of Zeilberger and Riese respectively.
\index{Zeilberger, Doron}
\index{Riese, Axel}
\vskip 3mm
\noindent\textbf{A Detailed Example.}\\
Suppose we are interested in proving the identity conjectured in 
\S~\ref{finit} using Zeilberger's Algorithm.  Recall that the identity is
\begin{equation} \label{ls7ff}
\sum_{j=0}^n q^{j^2+j} \gp{n}{j}{q^2} 
=\sum_{j=-\infty}^\infty (-1)^j q^{2j^2+j} \gp{2n+1}{n+2j+1}{q}. 
\end{equation}
\index{Zeilberger, Doron}
We saw in \S~\ref{finit} that the idenitity (\ref{ls7ff}) 
satisfies the recurrence \[ P_n(q) = (1+q^{2n}) P_{n-1}(q)\] and the initial
condition $P_0 (q) = 1$.  Thus to prove (\ref{ls7ff}), it suffices to show that
the righthand side of (\ref{ls7ff}) satisfies, the same recurrence and 
initial conditions, i.e. that
\begin{equation} 
  \sum_{j=-\infty}^\infty F(n,j) 
   = (1+q^{2n}) \sum_{j=-\infty}^\infty F(n-1,j),\label{1} 
\end{equation}
for $n\geqq 0$ and 
\begin{equation} 
   \sum_{j=-\infty}^\infty F(0,j)=1 \label{2} 
\end{equation}
where $F(n,j) = (-1)^j q^{2j^2 + j} \gp{2n+1}{n+2j+1}{q}$. 
As mentioned in the previous section, one can use the \texttt{qZeil} function 
in Riese's \texttt{qZeil} \index{Riese, Axel}\index{qZeil (Mathematica 
package)} 
Mathematica package to produce a function $G(n,j)$ such that 
\begin{equation} 
   F(n,j) - (1+q^{2n})F(n-1,j) = G(n,j) - G(n,j-1) \label{WZeq}
\end{equation}
and 
\begin{equation}   
  \lim_{j\to\pm\infty} G(n,j) = 0. \label{bc}
\end{equation}
We then sum both sides of equation (\ref{WZeq}) over all integral $j$, 
and observe that the righthand side of (\ref{WZeq}) telescopes to 0. 
This will then guarantee that 
equation (\ref{1}) holds.
Let us now begin a Mathematica session (for a detailed explanation of the
use of the \texttt{qZeil} package, see Paule and Riese~\cite{pr:qzeil}):
\index{Paule, Peter}\index{Riese, Axel}\index{qZeil (Mathematica
package)}
\begin{verbatim}
In[1]:=  << qZeil.m

Out[1]= Axel Riese's q-Zeilberger
          implementation version 2.01 (04/12/01) loaded

In[2]:= F[n_,j_]:= (-1)^j q^(2j^2 + j) qBinomial[2n+1, n+2j+1, q]

In[3]:= qZeil[F[n,j], {j, -Infinity, Infinity}, n, 1]

                       
Out[3]=\end{verbatim} \[{SUM}(n) = 
  \left( 1 + q^{2\,n} \right) \,{SUM}(-1 + n)\]

I interrupt the Mathematica session here to note that {\em if we trust the 
computer}, the proof of the identity is complete, 
since the \texttt{qZeil} procedure output the recurrence we are trying to
demonstrate that $\sum F(n,j)$ satisfies.  (Remember that the initial
conditions were already verified in an earlier Maple session.  However, if
we do not implicitly trust the computer, we can use it to produce the WZ 
certificate, which in turn allows us to verify a rational function identity 
algebraically equivalent to (\ref{WZeq}).  Let us therefore
continue the Mathematica
session by producing the WZ certificate\footnote{Note that Zeilberger's
\index{Zeilberger, Doron}
algorithm is much stronger than we actually need here.  Zeilberger's 
algorithm not only finds the certificate, but the {\em recurrence} satisfied by
the summand.  Thus, even though the method of $q$-difference equations
allows us to know the recurrence in advance, Zeilberger's algorithm by
no means requires this.}:
\begin{verbatim}
In[4]:= Cert[ ]
  
Out[4]=\end{verbatim} \[\frac{q^{-2\,j + n}\,\left( -q^{2\,j} + q^n \right) \,
    \left( -q^{1 + 2\,j} + q^n \right) }{\left( -1 + q^n
      \right) \,\left( 1 + q^n \right) \,
    \left( -1 + q^{1 + 2\,n} \right) }\] 

Thus, the \texttt{qZeil} package claims that
\[ G(n,j) = F(n,j) R(n,j)\]
where \[ R(n,j) =\frac{q^{n-2j}(q^n-q^{2j})(q^n-q^{2j+1})}
    {(q^n-1)(q^n+1)(q^{2n+1}-1)} \]
is the desired function $G$.
It is easy to see that $G(n,j)$ satifies equation (\ref{bc}) since the 
Gaussian polynomial $\gp{A}{B}{q}$ is 0 when $B>A$ and when $B<0$.  We now
must verify that it satisfies equation (\ref{WZeq}).  For this purpose,
it is easiest to divide (\ref{WZeq}) through by $F(n,j)$ to obtain the
equivalent formulation
\begin{equation} 
  1 - \frac{(1+q^{2n}) F(n-1,j)}{F(n,j)} =R(n,j) - \frac{F(n,j-1) 
  R(n,j-1)}{F(n,j)}. 
  \label{WZeq2}
\end{equation}
Continuing with Mathematica:
\begin{verbatim}
In[5]:= R[n_,j_] := (q^(n-2j) (-q^(2j)+q^n) (q^n-q^(2j+1)))/
          ((q^n-1)(q^n+1) (q^(2n+1) - 1))

In[6]:= lhs = qSimplify[1 - (1+q^(2n)) F[n-1,j]/F[n,j]]

Out[6]= \end{verbatim} \[ 1 + \frac{\left( -1 - q^{2\,n} \right) \,
     \left( 1 - q^{-2\,j + n} \right) \,
     \left( 1 - q^{1 + 2\,j + n} \right) }{\left( 1 - 
       q^{2\,n} \right) \,\left( 1 - q^{1 + 2\,n} \right) }\]
\begin{verbatim}
In[7]:= rhs = qSimplify[R[n,j] - F[n,j-1] R[n,j-1]/F[n,j]]

Out[7]=\end{verbatim} \begin{gather*}
\frac{q^{-2\,j + n}\,\left( -q^{2\,j} + q^n \right) \,
     \left( -q^{1 + 2\,j} + q^n \right) }{\left( 1 - 
       q^{2\,n} \right) \,\left( 1 - q^{1 + 2\,n} \right) } +\\
   \frac{q^{3 - 6\,j + n}\,
     \left( -q^{-2 + 2\,j} + q^n \right) \,
     \left( -q^{-1 + 2\,j} + q^n \right) \,
     \left( 1 - q^{2\,j + n} \right) \,
     \left( 1 - q^{1 + 2\,j + n} \right) }{\left( 1 - 
       q^{2\,n} \right) \,
     \left( 1 - q^{1 - 2\,j + n} \right) \,
     \left( 1 - q^{2 - 2\,j + n} \right) \,
     \left( 1 - q^{1 + 2\,n} \right) }\end{gather*}
\begin{verbatim}
In[8]:= Simplify[lhs - rhs]

Out[8]= 0
\end{verbatim}
And thus using the Mathematica \texttt{Simplify} procedure together with 
the \texttt{qSimplify} procedure available in the \texttt{qZeil} package, 
we see that (\ref{WZeq2}), and therefore (\ref{WZeq}), and therefore 
(\ref{1}) is true.  Of course, the complete computer skeptic can verify
(\ref{WZeq2}) directly by hand.

\subsubsection{Dealing with Non-minimal Recurrences}
\textbf{Paule's Creative Symmetrization.}
\index{creative symmetrization|(}
It is well known (see e.g. Paule~\cite{pp:rr}) that creative
\index{Paule, Peter} 
telescoping does not always output the {\em minimal} recurrence that is 
satisfied by the summand.  Let us examine an example where this is the case,
Identity~\ref{t-SLfin}.5-b:
\begin{gather*}
 \sum_{j\geqq 0} q^{2j^2 + j} \gp{n+1}{2j+1}{q} =
   \sum_{j=-\infty}^\infty (-1)^j q^{3j^2 + j} \gp{2n}{n+2j}{q} 
   \tag{\ref{t-SLfin}.5-b}
\end{gather*}
From the Maple session in which this identity was conjectured we know that 
the LHS satisfies the recurrence
\begin{gather*}
   P_0 = 1,\\
   P_n = (1+q^n) P_{n-1}  
      \mbox{\ if $n\geqq 1$.}\\
 \end{gather*}
We return to Mathematica in the hopes of demonstrating that the RHS 
satisfies the same recurrence:
\begin{verbatim}
In[9]:= qZeil[ (-1)^j q^(3j^2 + j) qBinomial[2n, n + 2j, q], 
          {j, -Infinity, Infinity}, n, 3] 
                  
Out[9]= \end{verbatim} \begin{gather*}{SUM}(n) = 
  q^3\,\left( 1 - q^{-5 + 2\,n} \right) \,
    \left( 1 - q^{-4 + 2\,n} \right) \,
    {SUM}(-3 + n) - \\
   \frac{\left( q^5 + q^6 + q^7 + q^{3\,n} - q^{2 + 2\,n} + 
        q^{2 + 3\,n} \right) \,{SUM}(-2 + n)}
      {q^4} - \\ \frac{\left( -q - q^2 - q^3 - q^{2\,n} \right)
        \,{SUM}(-1 + n)}{q}\end{gather*}
but alas the creative telescoping algorithm finds a {\em third order} 
recurrence, when we were attempting to show that the RHS satisfies a 
certain {\em first order} recurrence.  But thanks to the following useful 
observation of Peter Paule, \index{Paule, Peter}all hope is not lost. 

  Notice that any function $F(n,j)$ can be written as the sum of its even 
part and its odd part, i.e.
\begin{equation} \label{evenodd}
  F(n,j) = \frac{F(n,j)+F(n,-j)}{2} + \frac{F(n,j) - F(n,-j)}{2}.
\end{equation}
If we sum both sides of (\ref{evenodd}) over all $j\in\mathbb{Z}$, the odd 
portion of $F(n,j)$ vanishes, so we obtain
\begin{equation} \label{crsym}
  \sum_{j\in\mathbb{Z}} F(n,j) 
= \sum_{j\in\mathbb{Z}} \frac{F(n,j) + F(n,-j)}{2}
\end{equation}
Applying (\ref{crsym}) to the RHS of (\ref{t-SLfin}.5), we obtain
\begin{equation}
  \sum_{j=-\infty}^\infty 2F(n,j) = 
  \sum_{j=-\infty}^\infty (-1)^j q^{3j^2 - j} \gp{2n}{n+2j}{q} (1+q^{2j})
\end{equation}
\begin{verbatim}
In[10]:= qZeil[ (-1)^j q^(3j^2 - j)(1 + q^(2j)) qBinomial[2n, n + 2j, 
      q], {j, -Infinity, Infinity}, n, 1]
                    
Out[10]= \end{verbatim} \[{SUM}(n) = 
  \left( 1 + q^n \right) \,{SUM}(-1 + n)\]
which is the desired (minimal) recurrence.
\index{creative symmetrization|)}

\textbf{Operator Algebra.}
\index{operator algebra}
  The technique of creative symmetrization works perfectly on the example 
under consideration and in fact works on many well known examples (see 
Paule~\cite{pr:qzeil}).  However, creative symmetrization by no means {\em 
guarantees} that the minimal, or even a lower order recurrence, will be 
produced using Zeilberger's algorithm.  It would therefore be wise to have 
a way of dealing with non-minimal recurrences in the event that creative 
symmetrizing fails.  For comparison, I will continue to work with the same 
example as from the previous section, even though the creative symmetrizing 
technique produces a minimal recurrence for it.
  But first, the forward shift operator must be introduced. Define
\[ NF(n,j) := F(n+1,j) \] and by extension,
\[ N^s F(n,j) := F(n+s,j) \] for any integer $s$.
For example, (\ref{WZeq}) re-written in operator notation is
\begin{equation}
   (1 - (1+q^{2n})N^{-1}) F(n,j) = (1 - J^{-1}) G(n,j),
\end{equation}
where, of course, $J$ is the forward shift operator in $j$, i.e.
\[ JF(n,j) := F(n,j+1) \] and by extension,
\[ J^s F(n,j) := F(n,j+s) \] for any integer $s$.

  Now the (unsymmetrized) RHS of (\ref{t-SLfin}.5-b) was shown by 
Mathematica to satisfy the recurrence
\begin{gather*}
  \sum_{j} F(n) = 
  q^3 (1 - q^{2n-5}) (1-q^{2n-4}) \sum_{j} F(n-3,j) \\ 
- (q + q^2 + q^3 + q^{3n} - q^{2n-2} + q^{3n-2}) \sum_{j} F(n-2,j) \\
+ (1 + q + q^2 + q^{2n-1}) \sum_{j} F(n-1,j),
\end{gather*} which is easily seen to be equivalent to
\begin{gather*}
  -q^3(1-q^{2n+1})(1-q^{2n+2}) \sum_j F(n,j) \\ + 
  (q + q^2 + q^3 - q^{2n+4} + q^{3n+5} + q^{3n+7}) \sum_j F(n+1,j) \\ -
  (1 + q + q^2 + q^{2n+5}) \sum_j F(n+2,j) \\ +
  \sum_j F(n+3,j) = 0,
\end{gather*}
which in turn can be written in operator notation as
\begin{gather} \label{RHS5op}
  \Big[ -q^3(1-q^{2n+1})(1-q^{2n+2}) +
  (q + q^2 + q^3 - q^{2n+4} + q^{3n+5} + q^{3n+7}) N \\ -
  (1 + q + q^2 + q^{2n+5}) N^2 + N^3 \Big] \sum_j F(n,j) = 0. \nonumber
\end{gather}
\index{operator!annihilating}
We say that an operator $B$ {\em annihilates} a sequence $S(n)$ if
\[ [B] S(n) = 0, \] so it is clear that the operator in (\ref{RHS5op}) 
annihilates the RHS of (\ref{t-SLfin}.5-b). 
  Similarly, the LHS of (\ref{t-SLfin}.5-b) was shown by Mathemtaica to 
satisfy the recurrence
  \begin{equation*}
     \sum_j F(n,j) = (1 + q^n) \sum_j F(n-1,j)
  \end{equation*} which is equivalent to the recurrence
  \begin{equation*}
     \sum_j F(n+1,j) - (1+q^{n+1}) \sum_j F(n,j) = 0
  \end{equation*} which in turn can be written in operator notation as
  \begin{equation} \label{LHS5op}
      \Big[ N - (1+q^{n+1}) \Big] \sum_j F(n,j) = 0.
  \end{equation} 
Now it is clear that if an operator $B$ annihilates a sequence $S(n)$, 
then any left multiple of $B$ will also annihilate $S(n)$, i.e.
\[ [B]S(n) = 0 \implies [AB]S(n) = 0, \] where $A$ is any $n$-shift operator.
Thus, if it can be shown that the operator in (\ref{RHS5op}) is a left 
multiple of (\ref{LHS5op}) (along with the appropriate initial conditions),
then the identity will be proved.   Dividing the third order recurrence
on the right by the first order recurrence, we find 
that
\begin{gather*}
\Big[ N^2 - q(1+q -q^{n+2} + q^{2n+4})N + q^3(q^{2n+1} - 1)(q^{n+1}-1)\Big]
\Big[ N - (1+q^{n+1}) \Big] \\ 
=\Big[ -q^3(1-q^{2n+1})(1-q^{2n+2}) +
  (q + q^2 + q^3 - q^{2n+4} + q^{3n+5} + q^{3n+7}) N \\ -
  (1 + q + q^2 + q^{2n+5}) N^2 + N^3 \Big],
\end{gather*}
and so both the LHS and RHS of (\ref{t-SLfin}.5-b) are annhilated by the
third order recurrence found by Mathematica, and together with the 
appropriate initial conditions, the identity is established.

\subsubsection{Multisum techniques}
The Bosonic forms for each of the identities in \S~\ref{t-SLfin} are 
single sums, but many are single sums involving $q$-trinomial 
co\"efficients.  There are currently no computer implementations of 
Zeilberger's algorithm which allows direct inputting of $q$-trinomial 
co\"efficients, but recalling the definitions of the $q$-trinomial 
co\"efficients (\ref{Trbdef}--\ref{Udef}), we can express any single sum 
$q$-trinomial expressions as a double sum $q$-binomial expression.
Thus the situation we are now in is wanting to prove an identity
\[ LHS(n) = RHS(n) \] where
$RHS(n)$ is of the form
\[ RHS(n) = \sum_{j=-\infty}^\infty \sum_{s=0}^\infty F(n,j,s).\]
for all nonnegative integers $n$.
The conditions analogous to (\ref{wzcond1}) and (\ref{wzcond2})
in the two-fold sum case is
that we find functions $G(n,j,s)$ and $H(n,j,s)$ such that
\begin{equation} \label{2wzcond1}
  \sum_{i=0}^r p_i(q) F(n+i,j,s) = G(n,j,s) - G(n,j-1,s) + H(n,j,s) - H(n,j,s-1)
\end{equation}  
and
\begin{equation} \label{2wzcond2}
  \lim_{j\to\pm\infty} G(n,j,s) = \lim_{s\to\pm\infty} H(n,j,s) = 0
\end{equation}
Then, by summing (\ref{2wzcond1}) over all $j$ and $s$, we immediately obtain
  \[ \sum_{i=0}^r p_i(q) RHS(n-i) = 0, \]
which, together with the appropriate initial conditions, will complete the
proof.  The WZ certificate is now a pair of functions $R_j(n,j,s)$ and  
$R_s(n,j,s)$ such that
  \[ G(n,j,s) = F(n,j,s) R_j (n,j,s) \] and
  \[ H(n,j,s) = F(n,j,s) R_s (n,j,s). \]
(Of course, these ideas extend from two-fold summation to $k$-fold summation for
any fixed positive integer $k$.)

\noindent\textbf{A Detailed Example.}\\
Let us consider the identity (\ref{t-SLfin}.4).
\begin{gather*}
   \sum_{i\geqq 0} \sum_{j\geqq 0} \sum_{k\geqq 0} 
     (-1)^{j+k} q^{i^2 + j^2 + 2k} 
     \gp{j}{i}{q^2} \gp{j+ k -1}{k}{q^2}  \gp{n-i-k}{j}{q^2} \\=
   \sum_{j=-\infty}^\infty (-1)^j q^{j^2} \U{n-1}{j}{q} \tag{\ref{t-SLfin}.4}
\end{gather*}
where during the Maple session in which it was conjectured, the LHS was found to 
satisfy the recurrence
\begin{gather*}
    P_0 = 1,\\
    P_1 = -q + 1,\\ 
    P_n = (1-q^2 - q^{2n-1}) P_{n-1} +(q^2-q^{2n-2}) P_{n-2} 
      \mbox{\ if $n\geqq 2$.}
 \end{gather*}
Translating the above into forward shifts and operator notation, we obtain
 \begin{equation} \label{LHS4op}
   \Big[ -q^2(-1+q^n)(1+q^n) + (1-q^2 - q^{3n+3})N - N^2 \Big] \sum\sum 
   F(n,j,r) = 0
 \end{equation}
Thus, it needs to be shown that the RHS satisfies the same recurrence (or 
an operator algebraic multiple of it) and the initial conditions.
Now as stated previously, the Riese Mathematica packages can not deal with 
the RHS of (\ref{t-SLfin}.4) directly.  Instead, we must use (\ref{Udef}) 
and (\ref{Tzerodef}) to find that
\begin{gather*} \sum_{j=-\infty}^\infty (-1)^j q^{j^2} \U{n-1}{j}{q}
  =  \sum_{j=-\infty}^\infty (-1)^j q^{j^2} \Big\{ \Tzero{n-1}{j}{q}
     + \Tzero{n-1}{j+1}{q} \Big\} \\
  =   \sum_{j=-\infty}^\infty (-1)^j q^{j^2} \Big\{
        \sum_{r=0}^{n-1} (-1)^r \gp{n-1}{r}{q^2} \gp{2n-2-2r}{n-1-j-r}{q} \\+
        \sum_{r=0}^n     (-1)^r \gp{n-1}{r}{q^2} \gp{2n-2-2r}{n  -j-r}{q}
     \Big\} \\
  =   \sum_{j=-\infty}^\infty \sum_{r=0}^\infty
       (-1)^{j+r} q^{j^2} \gp{n-1}{r}{q^2} \Big( 
           \gp{2n-2r-2}{n-j-r-1}{q} + \gp{2n-2r-2}{n-j-r}{q} \Big) \\
  =  \sum_{j=-\infty}^\infty \sum_{r=0}^\infty
       (-1)^{j+r} q^{j^2} \gp{n-1}{r}{q^2}  \gp{2n-2r-2}{n-j-r-1}{q}
          \Big(  1 +  \frac{1-q^{n+j-r-1}}{1-q^{n-j-r}} \Big).
\end{gather*}
We now return to Mathematica, this time using Riese's \texttt{qMultiSum} 
package.\index{Riese, Axel}\index{qMultiSum (Mathematica
package)}
\begin{verbatim}
In[11]:= <<qMultiSum.m

Out[11]= Axel Riese's qMultiSum implementation 
           version 2.15 (07/26/01) loaded

In[12]:= qFindRecurrence[ (-1)^(j + r) q^(j^2) 
          qBinomial[n - 1, r, q^2] 
          qBinomial[2n - 2r - 2, n - 1 - j - r, q]
          (1 + (1 - q^(n + j - r - 1))/(1 - q^(n - j - r))),
           n, {j, r}, 1, {0, 2}]
           
Out[12] = \end{verbatim}
\begin{gather*} 
     -( \left( q^2 - q^n \right) \,\left( q^3 - q^n \right) \,
      \left( q^4 - q^n \right) \,\left( q^2 + q^n \right) \,
      \left( q^3 + q^n \right) \,\left( q^4 + q^n \right) \, \\
      F[-4 + n,-1 + j,-1 + r]  \\ + 
   q^6\,\left( -1 + q + q^2 \right) \,\left( q^2 - q^n \right) \,
    \left( q^3 - q^n \right) \,\left( q^2 + q^n \right) \,
    \left( q^3 + q^n \right) \,\\
     F[-3 + n,-1 + j,-1 + r] \\ - 
   q\,\left( q^2 - q^n \right) \,\left( q^3 - q^n \right) \,
    \left( q^2 + q^n \right) \,\left( q^3 + q^n \right) \,
    \left( q^7 - q^{2\,n} \right) \,F[-3 + n,-1 + j,r] \\ - 
   q^{8 + 2\,n}\,\left( q^2 - q^n \right) \,
    \left( q^2 + q^n \right) \,F[-2 + n,-2 + j,r] \\ - 
   q^{11}\,\left( -1 - q + q^2 \right) \,\left( q^2 - q^n \right) \,
    \left( q^2 + q^n \right) \,F[-2 + n,-1 + j,-1 + r] \\ + 
   q^9\,\left( q^2 - q^n \right) \,\left( q^2 + q^n \right) \,
    \left( -q^3 + q^4 + q^5 + q^{2\,n} \right) \,F[-2 + n,-1 + j,r] \\ -
    q^{6 + 2\,n}\,\left( q^2 - q^n \right) \,
    \left( q^2 + q^n \right) \,F[-2 + n,j,r] \\ - 
   q^{14 + 2\,n}\,F[-1 + n,-2 + j,r] - 
   q^{15}\,F[-1 + n,-1 + j,-1 + r] \\ - 
   q^{11}\,\left( -q^4 - q^5 + q^6 - q^{2\,n} \right) \,
    F[-1 + n,-1 + j,r] - q^{12 + 2\,n}\,F[-1 + n,j,r] \\ - 
   q^{15}\,F[n,-1 + j,r])  == 0
\end{gather*}
Thus we have a recurrence satisfied by the RHS summand.  Notice, however, 
that the recurrence involves shifts in the summation variables as well as 
in $n$.  To remedy this, we use the \texttt{qRecurrenceToCertificate} and 
\texttt{qSumCertificate} procedures.

\begin{verbatim}          
In[13]:= qRecurrenceToCertificate[%]

Out[13]= \end{verbatim}
\begin{gather*}
  \Delta_j[q^{18}\,\left( -1 + q^n \right) \,
      \left( 1 + q^n \right) \,\left( -1 + q^{1 + n} \right) \,
      \left( 1 + q^{1 + n} \right) \,\left( -1 + q^{2 + n} \right) \,
      \left( 1 + q^{2 + n} \right) \,F[n,j,r] \\ + 
     q^{16}\,\left( -1 + q + q^2 \right) \,
      \left( -1 + q^{1 + n} \right) \,\left( 1 + q^{1 + n} \right) \,
      \left( -1 + q^{2 + n} \right) \,\left( 1 + q^{2 + n} \right) \,
      F[1 + n,j,r] \\
     + q^{18}\,\left( -1 + q^{1 + n} \right) \,
      \left( 1 + q^{1 + n} \right) \,\left( -1 + q^{2 + n} \right) \,
      \left( 1 + q^{2 + n} \right) \,
      \left( -1 + q^{1 + 2\,n} \right) \,F[1 + n,j,1 + r] \\ + 
     q^{15}\,\left( -1 - q + q^2 \right) \,
      \left( -1 + q^{2 + n} \right) \,\left( 1 + q^{2 + n} \right) \,
      F[2 + n,j,r] \\
      - q^{16}\,\left( -1 + q^{2 + n} \right) \,
      \left( 1 + q^{2 + n} \right) \,
      \left( -1 + q + q^2 - q^{2 + 2\,n} + q^{5 + 2\,n} \right) \,
      F[2 + n,j,1 + r] \\
     + q^{18 + 2\,n}\,
      \left( -1 + q^{2 + n} \right) \,\left( 1 + q^{2 + n} \right) \,
      F[2 + n,1 + j,1 + r] \\- q^{15}\,F[3 + n,j,r] - 
     q^{15}\,\left( -1 - q + q^2 - q^{4 + 2\,n} + q^{5 + 2\,n}
        \right) \,F[3 + n,j,1 + r]\\ - 
     q^{20 + 2\,n}\,F[3 + n,1 + j,1 + r] - q^{15}\,F[4 + n,j,1 + r]] \\
    + \Delta_r[
    q^{18}\,\left( -1 + q^{1 + n} \right) \,
      \left( 1 + q^{1 + n} \right) \,\left( -1 + q^{2 + n} \right) \,
      \left( 1 + q^{2 + n} \right) \,
      \left( -1 + q^{1 + 2\,n} \right) \,F[1 + n,j,r] \\ - 
     q^{16}\,\left( -1 + q^{2 + n} \right) \,
      \left( 1 + q^{2 + n} \right) \,
      \left( -1 + q + q^2 - q^{2 + 2\,n} - q^{4 + 2\,n} + 
        q^{5 + 2\,n} \right) \,F[2 + n,j,r] \\ - 
     q^{15}\,\left( -1 - q + q^2 - q^{4 + 2\,n} + q^{5 + 2\,n} + 
        q^{7 + 2\,n} \right) \,F(3 + n,j,r) - q^{15}\,F[4 + n,j,r]] \\ +
    q^{18}\,\left( -1 + q^n \right) \,\left( 1 + q^n \right) \,
    \left( -1 + q^{1 + n} \right) \,\left( 1 + q^{1 + n} \right) \,
    \left( -1 + q^{2 + n} \right) \,\left( 1 + q^{2 + n} \right) \,
    F[n,j,r] \\
    + q^6\,\left( -1 + q + q^2 \right) \,
    \left( q^2 - q^{4 + n} \right) \,
    \left( q^3 - q^{4 + n} \right) \,
    \left( q^2 + q^{4 + n} \right) \,
    \left( q^3 + q^{4 + n} \right) \,F[1 + n,j,r]\\ - 
   q\,\left( q^2 - q^{4 + n} \right) \,
    \left( q^3 - q^{4 + n} \right) \,
    \left( q^2 + q^{4 + n} \right) \,
    \left( q^3 + q^{4 + n} \right) \,
    \left( q^7 - q^{8 + 2\,n} \right) \,F[1 + n,j,r] \\ - 
   q^{14 + 2\,n}\,\left( q^2 - q^{4 + n} \right) \,
    \left( q^2 + q^{4 + n} \right) \,F[2 + n,j,r] \\ - 
   q^{16 + 2\,n}\,\left( q^2 - q^{4 + n} \right) \,
    \left( q^2 + q^{4 + n} \right) \,F[2 + n,j,r] \\ - 
   q^{11}\,\left( -1 - q + q^2 \right) \,
    \left( q^2 - q^{4 + n} \right) \,
    \left( q^2 + q^{4 + n} \right) \,F[2 + n,j,r] \\ + 
   q^9\,\left( q^2 - q^{4 + n} \right) \,
    \left( q^2 + q^{4 + n} \right) \,
    \left( -q^3 + q^4 + q^5 + q^{8 + 2\,n} \right) \,F[2 + n,j,r] \\ - 
   q^{15}\,F(3 + n,j,r) - q^{20 + 2\,n}\,F[3 + n,j,r] \\ - 
   q^{22 + 2\,n}\,F[3 + n,j,r] - 
   q^{11}\,\left( -q^4 - q^5 + q^6 - q^{8 + 2\,n} \right) \,
    F[3 + n,j,r]\\
    - q^{15}\,F[4 + n,j,r] == 0
\end{gather*}

Note that the $\Delta_j$ and $\Delta_r$ are the shift operators $J-1$ and 
$R-1$ respectively.  Thus, the functions they act on are the certificate 
function pair $R_j(n,j,r)$ and $R_r(n,j,r)$ respectively.  We now find the 
recurrence (with shifts only in the $n$) for the RHS sum:

\begin{verbatim}
In[14]:=qSumCertificate[%]

Out[14]=
\end{verbatim}
\begin{gather*}
  q^3\,\left( -1 + q^n \right) \,\left( 1 + q^n \right) \,
    \left( -1 + q^{1 + n} \right) \,\left( 1 + q^{1 + n} \right) \,
    \left( -1 + q^{2 + n} \right) \,\left( 1 + q^{2 + n} \right) \,
    {SUM}[n] \\+ 
   q\,\left( -1 + q^{1 + n} \right) \,\left( 1 + q^{1 + n} \right) \,
    \left( -1 + q^{2 + n} \right) \,\left( 1 + q^{2 + n} \right) \,
    \left( -1 + q + q^{3 + 2\,n} \right) \,{SUM}[1 + n] \\
    - \left( -1 + q^{2 + n} \right) \,\left( 1 + q^{2 + n} \right) \,
    \left( 1 + q^3 - q^{3 + 2\,n} - q^{5 + 2\,n} + q^{6 + 2\,n}
      \right) \,{SUM}[2 + n] \\- 
   q\,\left( -1 + q - q^{3 + 2\,n} + q^{4 + 2\,n} + q^{6 + 2\,n}
      \right) \,{SUM}[3 + n] - 
   SUM[4 + n] == 0
\end{gather*}
So, once again the computer finds a nonminimal recurrence.  So, we perform
another right division to demonstrate that 
the operator represented above is a left multiple of the operator in 
(\ref{LHS4op}):
\begin{gather*}
  \Big[  -N^2 - (1-q)(-1+q^{n+2})(1+q^{n+2}) N + \\
 q(1+q^{n+1}) (-1 + q^{n+2})(1+q^{n+2})(-1+q^{n+1}) \Big] \\
   \Big[ -N^2 + (1-q^2 - q^{2n+3}) N + q^2(1-q^n)(1+q^n) \Big] \\ =
  N^4 + q(-1+q-q^{3n+3}+q^{2n+4}+q^{2n+6})N^3 \\
  +(-1+q^{n+2})(1+q^{n+2})(1+q^3-q^{2n+3}-q^{2n+5}+q^{2n+6})N^2 \\
  -q(-1+q^{n+1})(1+q^{n+2})(-1+q^{n+2})(1+q^{n+2})(-1+q+q^{2n+3})N \\
  -q^3(q^n-1)(1+q^n)(-1+q^{n+1})(1+q^{n+1})(-1+q^{n+2})(1+q^{n+2})
\end{gather*}

\subsubsection{WZ Certificates for Selected Identities}

\indent The RHS of Identity~\ref{t-SLfin}.2-b is certified by
  \[ \frac{q^{-2\,j + n}\,
    \left( -q^{2\,j} + q^n \right)
      \,\left( -q^{1 + 2\,j} + 
      q^n \right) }{\left( -1 + 
      q^n \right) \,
    \left( 1 + q^n \right) \,
    \left( -1 + q^{1 + 2\,n}
      \right) }.\]

The RHS of Identity~\ref{t-SLfin}.3-b is certified by 
  \[ \frac{q^{-2\,j + n}\,
    \left( -q^{2\,j} + q^n \right)
      \,\left( -q^{1 + 2\,j} + 
      q^n \right) }{\left( -1 + 
      q^n \right) \,
    \left( 1 + q^n \right) \,
    \left( -q + q^{2\,n} \right) }.\]

After creative symmetrizing with the factor $(1 + q^{2j})$, 
the RHS of Identitiy~\ref{t-SLfin}.5-b is certified by
   \[ \frac{q^{-2\,j + n}\,
    \left( -q^{2\,j} + q^n \right)
      \,\left( -q^{1 + 2\,j} + 
      q^n \right) }{\left( -1 + 
      q^n \right) \,
    \left( 1 + q^n \right) \,
    \left( -q + q^{2\,n} \right) }.\]

The RHS of Identity~\ref{t-SLfin}.11-b is certified by 
  \[ \frac{ q^{2n-3j} (q^n-q^{3j}) (q^n-q^{3j+1}) (q^n - q^{3j+2}) }
  {(q^n-1) (q^n+1) (q^{2n}-q) (q^{2n+1} - 1) }. \]

The RHS of Identity~\ref{t-SLfin}.50-b is certified by 
  \[ \frac{q^{1 - 3\,j + 2\,n}\,
    \left( -q^{3\,j} + q^n \right)
      \,\left( -q^{1 + 3\,j} + 
      q^n \right) \,
    \left( -q^{2 + 3\,j} + 
      q^n \right) }{\left( -1 + 
      q^n \right) \,
    \left( 1 + q^n \right) \,
    \left( -1 + q^{1 + n} \right)
      \,\left( 1 + q^{1 + n}
      \right) \,
    \left( -1 + q^{1 + 2\,n}
      \right) }.\]

The RHS of Identity~\ref{t-SLfin}.96-b is certified by 
   \begin{gather*}
      -\frac{q^{2 - 5\,k + 2\,n}\,
      \left( -q^{5\,k} + q^n
        \right) \,
      \left( -q^{1 + 5\,k} + 
        q^n \right) \,
      \left( -q^{2 + 5\,k} + 
        q^n \right)}
      {\left( -1 + 
        q^{2 + 5\,k} \right) \,
      \left( 1 + q^{2 + 5\,k}
        \right) \,
      \left( -1 + q^n \right) \,
      \left( 1 + q^n \right) \,
      \left( -1 + q^{1 + n}
        \right)} \\
   \times\frac{\left( -q^{3 + 5\,k} + 
        q^n \right) \,
      \left( -q^{4 + 5\,k} + 
        q^n \right) } {\left( 1 + q^{1 + n} \right)
        \,\left( -1 + 
        q^{1 + 2\,n} \right) \,
      \left( -1 + 
        q^{3 + 2\,n} \right)}. 
     \end{gather*}

The RHS of Identity~\ref{t-SLfin}.99-b is certified by
 \begin{gather*} 
   - \frac{q^{-5\,k + 2\,n}\,
      \left( -q^{5\,k} + q^n
        \right) \,
      \left( -q^{1 + 5\,k} + 
        q^n \right) \,
      \left( -q^{2 + 5\,k} + 
        q^n \right) \,
      \left( -q^{3 + 5\,k} + 
        q^n \right) }
     {\left( -1 + 
        q^{2 + 5\,k} \right) \,
      \left( 1 + q^{2 + 5\,k}
        \right) \,
      \left( -1 + q^n \right) \,
      \left( 1 + q^n \right) \,
      \left( -q + q^n \right) \,
      \left( q + q^n \right) \,
      \left( -q + q^{2\,n} \right)} \\ 
    \times\frac{\left( -q^{4 + 5\,k} + 
        q^n \right) \,
      \left( 1 + q^{9 + 10\,k} - 
        q^{3 + 5\,k + n} - 
        2\,q^{4 + 5\,k + n} + 
        q^{2 + 5\,k + 3\,n}
        \right) } {\,\left( -1 + 
        q^{1 + 2\,n} \right) \,
      \left( 1 + q^{4 + 10\,k} - 
        q^{2 + 5\,k + n} - 
        q^{3 + 5\,k + n} - 
        q^{4 + 5\,k + n} + 
        q^{5 + 5\,k + 3\,n}
        \right)}.
\end{gather*}\index{WZ method|)}

\subsection{Recurrence Proof} \label{rp}
\index{recurrence proof|(}
Depending on the complexity of the summand and recurrence, 
and the limits of the memory 
available on a given computer system, it may or may not be feasible at the 
present time to obtain a computer generated proof of a given finite 
Rogers-Ramanujan type identity.  However, guided by the philosophy of the 
WZ method, one can show that the 
conjectured bosonic form satisfies the same recurrence relation and initial 
conditions as the known fermionic form.  The main differences between 
this ``recurrence proof" and the WZ method is that the entire summation is dealt with as a whole, rather than just the 
summand, and no proof certificate is produced in the process.  Let us work 
through some examples:

\vskip 5mm
\noindent\begin{id*}[$q$-Trinomial Finite form of \ref{t-ljslist}.3/23 
(with $q$ replaced by $-q$)]{\ref{t-SLfin}.3-T} 
 \begin{gather*}
    \sum_{j\geqq 0} q^{j^2} \gp{n}{j}{q^2} 
   = \sum_{j=-\infty}^\infty (-1)^j q^{3j^2 + j} \U{n}{3j}{q}\tag{\ref{t-SLfin}.3-t}\\
  P_0 = 1,\\ 
  P_n = (1+q^{2n-1}) P_{n-1} \mbox{\ if $n\geqq 1$.} \\
 \end{gather*}
\end{id*}

\begin{proof}
Let \[ P_n = \sum_{j=-\infty}^\infty (-1)^j q^{3j^2 + j} \U{n}{3j}{q}. \]
We need to show that $P_n - (1+q^{2n-1}) P_{n-1} = 0 $ and $P_0 = 1$.

\begin{eqnarray*}
& & P_n - P_{n-1} + q^{2n-1} P_{n-1} \\
&=& \hskip 4mm
    \sum_{j=-\infty}^\infty (-1)^j q^{3j^2 + j         } \U{n  }{3j  }{q} 
   -\sum_{j=-\infty}^\infty (-1)^j q^{3j^2 + j         } \U{n-1}{3j  }{q}\\
& &-\sum_{j=-\infty}^\infty (-1)^j q^{3j^2 + j + 2n - 1} \U{n-1}{3j  }{q}
\end{eqnarray*}
By applying (\ref{U1}) to the first term, we obtain
\begin{eqnarray*}
&=& \hskip 4mm
    \sum_{j=-\infty}^\infty (-1)^j q^{3j^2 + j         }\U   {n-1}{3j  }{q}
   +\sum_{j=-\infty}^\infty (-1)^j q^{3j^2 + j + 2n - 1}\U   {n-1}{3j  }{q}\\
& &+\sum_{j=-\infty}^\infty (-1)^j q^{3j^2 -2j +  n    }\Tone{n-1}{3j-1}{q}
   +\sum_{j=-\infty}^\infty (-1)^j q^{3j^2 +4j +  n + 1}\Tone{n-1}{3j+2}{q}\\
& &-\sum_{j=-\infty}^\infty (-1)^j q^{3j^2 + j         }\U   {n-1}{3j  }{q}
   -\sum_{j=-\infty}^\infty (-1)^j q^{3j^2 + j + 2n - 1}\U   {n-1}{3j  }{q}
\end{eqnarray*}
whereupon the first term cancels the fifth and the second cancels the sixth
leaving
\begin{eqnarray*}
&=& \hskip 4mm
    \sum_{j=-\infty}^\infty (-1)^j q^{3j^2 - 2j + n} \Tone{n-1}{3j-1}{q}
   +\sum_{j=-\infty}^\infty (-1)^j q^{3j^2 + 4j +n+1}\Tone{n-1}{3j+2}{q}\\
&=& 0 \mbox{\quad (by shifting $j$ to $j-1$ in the second term).}
\end{eqnarray*}
Since $P_0 = \sum_{j=-\infty}^\infty (-1)^j q^{3j^2 + j} \U{0}{3j}{q} = 1$,
the proof is complete.
\end{proof} 

The preceeding example was particularly simple.  In general, more effort 
is required in recurrence proofs.  Let us now look at a somewhat more 
intricate example.
\vskip 5mm 
\begin{id*}[Finite form of \ref{t-ljslist}.17]{\ref{t-SLfin}.17}  
 \begin{gather*}
    \sum_{j\geqq 0} \sum_{k\geqq 0} (-1)^k q^{j^2+j+k}
     \gp{j+k}{j}{q^2} \gp{n-k}{j}{q^2} =
   \sum_{j=-\infty}^\infty (-1)^j q^{j(5j+3)/2} \Tone{n+1}{\lfloor       
      \frac{5j+2}{2} \rfloor}{q} \\
  P_0 = 1\\
  P_1 = q^2 - q + 1\\
  P_n = (1 - q + q^{2n}) P_{n-1} + q P_{n-2} \mbox{\ if $n\geqq 2$.}
 \end{gather*}
\end{id*}

\begin{proof}
Let \begin{eqnarray*}
 P_n & = & \sum_{j=-\infty}^\infty (-1)^j q^{j(5j+3)/2} \Tone{n+1}{\lfloor       
      \frac{5j+2}{2} \rfloor}{q}  \\
     & = & \sum_{k=-\infty}^\infty q^{10k^2 + 3k  }\Tone{n+1}{5k+1}{q}
                                  -q^{10k^2 +13k+4}\Tone{n+1}{5k+3}{q}.
\end{eqnarray*}
We need to show that $P_n - (1 - q + q^n) P_{n-1} - q P_{n-2} = 0$
if $n\geqq 2$ and that the appropriate initial conditions hold.

\begin{eqnarray*}
& & \hskip 4mm P_n - P_{n-1} +q P_{n-1} - q^{2n} P_{n-1} - q P_{n-2} \\
&=& \hskip 4mm
    \sum_{k=-\infty}^\infty q^{10k^2 + 3k      }\Tone{n+1}{5k+1}{q}
   -\sum_{k=-\infty}^\infty q^{10k^2 + 3k      }\Tone{n  }{5k+1}{q}\\
& &+\sum_{k=-\infty}^\infty q^{10k^2 + 3k + 1  }\Tone{n  }{5k+1}{q}
   -\sum_{k=-\infty}^\infty q^{10k^2 + 3k +2n  }\Tone{n  }{5k+1}{q}\\
& &-\sum_{k=-\infty}^\infty q^{10k^2 + 3k + 1  }\Tone{n-1}{5k+1}{q}
   -\sum_{k=-\infty}^\infty q^{10k^2 +13k + 4  }\Tone{n+1}{5k+3}{q}\\
& &+\sum_{k=-\infty}^\infty q^{10k^2 +13k + 4  }\Tone{n  }{5k+3}{q}
   -\sum_{k=-\infty}^\infty q^{10k^2 +13k + 5  }\Tone{n  }{5k+3}{q}\\
& &+\sum_{k=-\infty}^\infty q^{10k^2 +13k +2n+4}\Tone{n  }{5k+3}{q}
   +\sum_{k=-\infty}^\infty q^{10k^2 +13k + 5  }\Tone{n-1}{5k+3}{q}
\end{eqnarray*}
Apply (\ref{ET1}) to the first, third, sixth, and ninth terms to obtain:
\begin{eqnarray*}
&=& \hskip 4mm
    \sum_{k=-\infty}^\infty q^{10k^2 + 3k      }\Tone {n  }{5k+1}{q}
   +\sum_{k=-\infty}^\infty q^{10k^2 + 8k+  n+2}\Tzero{n  }{5k+2}{q}\\
& &+\sum_{k=-\infty}^\infty q^{10k^2 - 2k + n  }\Tzero{n  }{5k  }{q}
   -\sum_{k=-\infty}^\infty q^{10k^2 + 3k      }\Tone {n  }{5k+1}{q}\\
& &+\sum_{k=-\infty}^\infty q^{10k^2 + 3k + 1  }\Tone {n-1}{5k+1}{q}
   +\sum_{k=-\infty}^\infty q^{10k^2 + 8k + n+2}\Tzero{n-1}{5k+2}{q}\\
& &+\sum_{k=-\infty}^\infty q^{10k^2 - 2k + n  }\Tzero{n-1}{5k  }{q}
   -\sum_{k=-\infty}^\infty q^{10k^2 + 3k +2n  }\Tone {n  }{5k+1}{q}\\
& &-\sum_{k=-\infty}^\infty q^{10k^2 + 3k + 1  }\Tone {n-1}{5k+1}{q}
   -\sum_{k=-\infty}^\infty q^{10k^2 +13k + 4  }\Tone {n  }{5k+3}{q}\\
& &-\sum_{k=-\infty}^\infty q^{10k^2 +18k + n+8}\Tzero{n  }{5k+4}{q}
   -\sum_{k=-\infty}^\infty q^{10k^2 + 8k + n+2}\Tzero{n  }{5k+2}{q}\\
& &+\sum_{k=-\infty}^\infty q^{10k^2 +13k + 4  }\Tone {n  }{5k+3}{q}
   -\sum_{k=-\infty}^\infty q^{10k^2 +13k + 5  }\Tone {n-1}{5k+3}{q}\\
& &-\sum_{k=-\infty}^\infty q^{10k^2 +18k + n+8}\Tzero{n-1}{5k+4}{q}
   -\sum_{k=-\infty}^\infty q^{10k^2 + 8k + n+2}\Tzero{n-1}{5k+2}{q}\\
& &+\sum_{k=-\infty}^\infty q^{10k^2 +13k +2n+4}\Tone{n  }{5k+3}{q}
   +\sum_{k=-\infty}^\infty q^{10k^2 +13k + 5  }\Tone{n-1}{5k+3}{q}
\end{eqnarray*}
whereupon the first terms cancels the fourth, the second cancels the 
twelfth, the fifth cancels the ninth, the sixth cancels the sixteenth,
the tenth cancels the thirteenth, and the fourteenth cancels the eighteenth 
leaving
\begin{eqnarray*}
&=& \hskip 4mm
    \sum_{k=-\infty}^\infty q^{10k^2 - 2k + n   }\Tzero{n  }{5k  }{q} 
   +\sum_{k=-\infty}^\infty q^{10k^2 - 2k + n   }\Tzero{n-1}{5k  }{q}\\
& &-\sum_{k=-\infty}^\infty q^{10k^2 + 3k +2n   }\Tone {n  }{5k+1}{q}
   -\sum_{k=-\infty}^\infty q^{10k^2 +18k + n +8}\Tzero{n  }{5k+4}{q}\\
& &-\sum_{k=-\infty}^\infty q^{10k^2 +18k + n +8}\Tzero{n-1}{5k+4}{q}
   +\sum_{k=-\infty}^\infty q^{10k^2 +13k +2n +4}\Tone {n  }{5k+3}{q}
\end{eqnarray*}
Next, we apply (\ref{ET0}) to the first and fourth terms, and (\ref{ET1}) 
to the third and sixth terms:
\begin{eqnarray*}
&=& \hskip 4mm
    \sum_{k=-\infty}^\infty q^{10k^2 - 2k + n   }\Tzero{n-1}{5k-1}{q}
   +\sum_{k=-\infty}^\infty q^{10k^2 + 3k +2n   }\Tone {n-1}{5k  }{q}\\
& &+\sum_{k=-\infty}^\infty q^{10k^2 + 8k +3n   }\Tzero{n-1}{5k+1}{q} 
   +\sum_{k=-\infty}^\infty q^{10k^2 - 2k + n   }\Tzero{n-1}{5k  }{q}\\
& &-\sum_{k=-\infty}^\infty q^{10k^2 + 3k +2n   }\Tone {n-1}{5k+1}{q}
   -\sum_{k=-\infty}^\infty q^{10k^2 + 8k +3n +1}\Tzero{n-1}{5k+2}{q}\\
& &-\sum_{k=-\infty}^\infty q^{10k^2 - 2k +3n -1}\Tzero{n-1}{5k  }{q}
   -\sum_{k=-\infty}^\infty q^{10k^2 -18k + n +8}\Tzero{n-1}{5k-5}{q}\\
& &-\sum_{k=-\infty}^\infty q^{10k^2 -13k +2n +4}\Tone {n-1}{5k-4}{q}
   -\sum_{k=-\infty}^\infty q^{10k^2 - 8k +3n   }\Tzero{n-1}{5k-3}{q}\\
& &-\sum_{k=-\infty}^\infty q^{10k^2 -18k + n +8}\Tzero{n-1}{5k-4}{q}
   +\sum_{k=-\infty}^\infty q^{10k^2 +13k +2n +4}\Tone {n-1}{5k+3}{q}\\
& &+\sum_{k=-\infty}^\infty q^{10k^2 +18k +3n +7}\Tzero{n-1}{5k+4}{q}
   +\sum_{k=-\infty}^\infty q^{10k^2 + 8k +3n +1}\Tzero{n-1}{5k+2}{q}
\end{eqnarray*}

In the above, the sixth and fourteenth terms cancel each other.
The second, third, fifth, and seventh terms sum to zero by (\ref{E0}).
Next, replace $k$ by $-k$ and apply (\ref{T1sym}) to the ninth term. 
Replace $k$ by $-k$ and apply (\ref{T0sym}) to the tenth term.  Then, the 
ninth, tenth, twelfth, and thirteen terms sum to zero by (\ref{E0}).  Now 
we have
\begin{eqnarray*}
& & \hskip 4mm
    \sum_{k=-\infty}^\infty q^{10k^2 - 2k + n   } \Tzero{n-1}{5k-1}{q}
   +\sum_{k=-\infty}^\infty q^{10k^2 - 2k + n   } \Tzero{n-1}{5k  }{q}\\
& &-\sum_{k=-\infty}^\infty q^{10k^2 -18k + n + 8}\Tzero{n-1}{5k-5}{q}
   -\sum_{k=-\infty}^\infty q^{10k^2 -18k + n + 8}\Tzero{n-1}{5k-4}{q}
\end{eqnarray*}
In the third and fourth terms, replace $k$ by $1-k$ and apply 
(\ref{T0sym}).  Then the first term will cancel the fourth and the second 
will cancel the third:
\begin{eqnarray*}
&=& 0.
\end{eqnarray*}
Upon verifying the easily checked initial conditions, the proof is 
complete.
\end{proof}

  Further examples of recurrence proofs may be found elsewhere in the 
literature:
Santos~\cite[Chapter 2]{jpos} 
contains proofs of Identities~\ref{t-SLfin}.29 and 
\ref{t-SLfin}.38-b.  Santos proves Identity~~\ref{t-SLfin}.20 in
\cite{jpos:rr}.  Proofs of Identities 3.8 and 3.12 are given in Santos and 
Sills~\cite[Theorems 1 and 3]{ss:qpell}.
\index{Sills, Andrew V.}
\index{Santos, J. P. O.}

As one can imagine, once higher order recurrences are encountered, the
process of giving a recurrence proof becomes increasingly tedious. 
Merely transcribing one line to the next without ever inadvertantly changing
a ``0'' to a ``1" or a ``+" to ``$-$" is more than one could reasonbly 
expect a human to do.  Accordingly, I have written the \texttt{recpf}
Maple package to assist the human mathematician in successfully carrying
out recurrence proofs in an efficient and minimally tedious manner.  The
\texttt{recpf} package is documented
in~\cite{avs:RRtools}, 
and Maple worksheets
containing recurrence proofs for the identities listed in 
\S~\ref{t-SLfin} are available from my web site:
\texttt{http://www.math.psu.edu/sills} (through August 2003),
\texttt{http://www.math.rutgers.edu/{\~{}}sills} (starting September 2003). 
\index{recurrence proof|)}

\section{Rogers-Ramanujan Reciprocal Duality} \label{dual}
\subsection{Introduction}
\index{duality, reciprocal|(}
In \cite{hhm}, Andrews demonstrated a type of duality relationship that 
exists among a few sets of Rogers-Ramanujan type identities.
\index{Andrews, George E.}  
The {\em (reciprocal) dual} of a polynomial 
  \[ a_n q^n + a_{n-1} q^{n-1} + a_{n-2} q^{n-2} + \cdots +a_1 q + a_0 \] is
  \[ a_0 q^n + a_{1} q^{n-1} + a_{2} q^{n-2} + \cdots +a_{n-1} q + a_n. \]
Equivalently,  the reciprocal of $P(q)$ is 
\begin{equation}\label{recip} 
  q^{\mathrm{deg}(P(q))} P(q^{-1}).
\end{equation} 
If $q^{\mathrm{deg}(P(q))} P(q^{-1})= P(q)$, the associated identity is 
called {\em self-dual}.
Let us work through an example of calculating the dual of an identity:
Consider, say, identity 10 from Slater's list.  A finite form 
(\ref{t-SLfin}.10) is

 If $n$ is a nonnegative integer, then
 \begin{gather*}
   \sum_{i\geqq 0}\sum_{j\geqq 0}\sum_{k\geqq 0}
       q^{j^2 + i^2 - i + k}
       \gp{j}{i}{q^2}\gp{j+k-1}{k}{q^2} \gp{n-i-k}{j}{q^2} \\
    = \sum_{j=-\infty}^\infty q^{2j^2 + j} 
          \Big[ \Tzero{n}{2j}{q} + \Tzero{n-1}{2j}{q}. \Big] 
    \tag{\ref{t-SLfin}.10}
 \end{gather*}
An examination of the first few cases $n=0, 1, 2,$ etc., 
convinces one that the 
degree of the polynomial is $n^2$.  Thus, by (\ref{t-SLfin}.10) and 
(\ref{recip}), the dual polynomials of (\ref{t-SLfin}.10) can be 
represented by either of the two forms
 \begin{gather} \label{dual10f}
   \sum_{i\geqq 0}\sum_{j\geqq 0}\sum_{k\geqq 0}
       q^{n^2 - (j^2 + i^2 - i + k)}
       \gp{j}{i}{1/q^2}\gp{j+k-1}{k}{1/q^2} \gp{n-i-k}{j}{1/q^2} \\
    = \sum_{j=-\infty}^\infty q^{n^2 - (2j^2 + j)} 
          \Big[ \Tzero{n}{2j}{1/q} + \Tzero{n-1}{2j}{1/q} \Big] 
 \end{gather}
Applying  (\ref{T0inv}) and (\ref{tau0eq}) on the righthand side
and (\ref{gpinv}) on the left hand side, we obtain
\begin{gather*}
  \sum_{h\geqq 0}\sum_{i\geqq 0}\sum_{k\geqq 0}  
   q^{k + i + 2i(h+i+k) + (h+k)^2 }
    \gp{n-i-h-k}{i}{q^2} \gp{n-i-h-1}{k}{q^2} \gp{n-i-k}{h}{q^2} \\ =
  \sum_{j=-\infty}^\infty q^{2j^2 - j} \Trb{n}{2j}{2j}{q^2}
                    + q^{2j^2 - j + 2n -1} \Trb{n-1}{2j}{2j}{q^2}.
   \tag{\ref{dual}.10f}
\end{gather*}

We can suppose that $|q|<1$ and let $n\to\infty$ in ({\ref{dual}.10f}) 
to obtain a Rogers-Ramanujan 
type identity.  First, we consider the lefthand side:
\begin{eqnarray*}
 & & \lim_{n\to\infty} \sum_{h,i,k\geqq 0} 
   q^{k + i + 2i(h+i+k) + (h+k)^2 }
    \gp{n-i-h-k}{i}{q^2} \gp{n-i-h-1}{k}{q^2} \gp{n-i-k}{h}{q^2} \\
 &=& \hspace*{-0.1 in} \lim_{n\to\infty} \sum_{h,i,k\geqq 0} 
   \frac{ q^{k + i + 2i(h+i+k) + (h+k)^2 } 
     (q^2;q^2)_{n-i-h-k} (q^2;q^2)_{n-i-h-1} (q^2;q^2)_{n-i-k} }
    {(q^2;q^2)_i (q^2;q^2)_{n-2i-j-k} (q^2;q^2)_k (q^2;q^2)_{n-i-k-h-1}
     (q^2;q^2)_h (q^2;q^2)_{n-i-k-h} } \\
 &=& \sum_{h=0}^\infty\sum_{i=0}^\infty\sum_{k=0}^\infty  
   \frac{ q^{k + i + 2i(h+i+k) + (h+k)^2 }}
        { (q^2;q^2)_h (q^2;q^2)_i (q^2;q^2)_k}
\end{eqnarray*}

Next, we consider the righthand side:
\begin{eqnarray*}
 & & \lim_{n\to\infty} \sum_{j=-\infty}^\infty q^{2j^2 - j} \Trb{n}{2j}{2j}{q^2}
                    + q^{2j^2 - j + 2n -1} \Trb{n-1}{2j}{2j}{q^2}\\
 &=& \frac{1}{(q^2;q^2)_\infty} \left[
   \sum_{j=-\infty}^\infty q^{2j^2 - j} + 0 \right]
    \mbox{\qquad (by (\ref{tau0lim}) and since $|q|<1$)} \\
 &=& \frac{1}{(q^2;q^2)_\infty} {(-q,-q^3,q^4;q^4)_\infty}
     \mbox{\qquad (by Theorem~\ref{jtp})} \\
 &=& (-q;q)_\infty \\
 &=& \prod_{j=1}^\infty (1+q^j)
\end{eqnarray*}

Thus, for $|q|<1$, 
\begin{gather} \label{dual10}
\sum_{h=0}^\infty\sum_{i=0}^\infty\sum_{k=0}^\infty  
   \frac{ q^{k + i + 2i(h+i+k) + (h+k)^2 }}
        { (q^2;q^2)_h (q^2;q^2)_i (q^2;q^2)_k} = \prod_{j=1}^\infty (1+q^j).
\end{gather}

Note that the triple sum in the preceeding equation can be simplified:
\begin{eqnarray*}
 & & \sum_{h=0}^\infty\sum_{i=0}^\infty\sum_{k=0}^\infty  
   \frac{ q^{k + i + 2i(h+i+k) + (h+k)^2 }}
        { (q^2;q^2)_h (q^2;q^2)_i (q^2;q^2)_k} \\
 &=& \sum_{i=0}^\infty\sum_{k=0}^\infty  
   \frac{ q^{k + i + i^2 + (i+k)^2 }}
        { (q^2;q^2)_i (q^2;q^2)_k}  
   \sum_{h=0}^\infty \frac{ q^{h^2 + (2i+2k)h }}{(q^2;q^2)_h }\\
 &=& \sum_{i=0}^\infty\sum_{k=0}^\infty  
   \frac{ q^{k + i + i^2 + (i+k)^2 }}
        { (q^2;q^2)_i (q^2;q^2)_k} (-q^{2i+2k+1};q^2)_\infty
    \mbox{\quad (by (\ref{qbc3}))}\\
 &=& (-q;q^2)_\infty \sum_{i=0}^\infty \sum_{k=0}^\infty  
   \frac{ q^{k^2 + 2ik + 2i^2 + k + i }}
        { (q^2;q^2)_i (q^2;q^2)_k (-q;q^2)_{i+k} } \\
 &=& (-q;q^2)_\infty \sum_{i=0}^\infty \sum_{K=i}^\infty  
   \frac{ q^{K^2 + K + i^2 }}
        { (q^2;q^2)_i (q^2;q^2)_{K-i} (-q;q^2)_{K} }
   \mbox{\quad (by taking $K=k+i$)} \\
 &=& (-q;q^2)_\infty \sum_{K=0}^\infty \frac{ q^{K^2 + K} }{(-q;q^2)_K} 
   \sum_{i=0}^K  \frac{ q^{i^2}}
   { (q^2;q^2)_i (q^2;q^2)_{K-i} } \\
 &=& (-q;q^2)_\infty \sum_{K=0}^\infty \frac{ q^{K^2 + K} }
   {(-q;q^2)_K (q^2;q^2)_K} 
   \sum_{i=0}^K  \frac{ q^{i^2} (q^2;q^2)_K}
   { (q^2;q^2)_i (q^2;q^2)_{K-i} } \\
 &=& (-q;q^2)_\infty \sum_{K=0}^\infty 
  \frac{ q^{K^2 + K} } {(-q;q^2)_K (q^2;q^2)_K} 
   \sum_{i=0}^K  q^{i^2} \gp{K}{i}{q^2}  \\
 &=& (-q;q^2)_\infty \sum_{K=0}^\infty \frac{ q^{K^2 + K} (-q;q^2)_K }
  {(-q;q^2)_K (q^2;q^2)_K } \mbox{\qquad(by (\ref{qbc1}))} \\
  &=& (-q;q^2)_\infty \sum_{K=0}^\infty \frac{ q^{K^2 + K} }{(q^2;q^2)_K}.  
\end{eqnarray*}
Thus, (\ref{dual10}) is equivalent to
\begin{gather*} (-q;q^2)_\infty
\sum_{j=0}^\infty \frac{ q^{j(j+1)}} {(q^2;q^2)_j} = \prod_{j=1}^\infty (1+q^j),
\end{gather*}
or, after dividing through by $(-q;q^2)_\infty$,
\begin{gather*} 
\sum_{j=0}^\infty \frac{ q^{j(j+1)}} {(q^2;q^2)_j} = \prod_{j=1}^\infty 
(1+q^{2j}),
\end{gather*}
which is Identity~\ref{t-ljslist}.7.
I will not go so far as to say that ``Identities~\ref{t-ljslist}.7
and \ref{t-ljslist}.10 are dual" since it was necessary to take the limit as
$n\to\infty$ and to divide out an 
infinite product in order to obtain Identity~\ref{t-ljslist}.7 from the 
dual of (\ref{t-SLfin}.10).
In fact, it is not hard to show that Identity~\ref{t-ljslist}.7 is 
{\em self}-dual.  
Nonetheless, it is clear that (\ref{t-ljslist}.7) and 
(\ref{t-ljslist}.10) are very closely related identities.

In some cases, the sequence of polynomials $\{ P_n (q) \}_{n=0}^\infty$ does 
not converge, but the subsequence $\{ P_{2m} (q) \}_{m=0}^\infty$ converges to
one series, and the subsequence $\{ P_{2m+1} (q) \}_{m=0}^\infty$ converges to
a different series.  One immediate clue that this may be occuring is when
the formula for the degree of the polynomial varies with the 
parity of $n$.  For example, with the finite First Rogers-Ramanujan Identity
(\ref{t-SLfin}.18), 
\index{Rogers-Ramanujan identities!MacMahon-Schur finitization}
the degree of the polynomial is $n^2/4$ if $n$ is 
even, and $(n^2-1)/4$ if $n$ is odd.  In cases such as this, we consider 
the dual of (\ref{t-SLfin}.18) to be a pair of identities.  The 
appropriate calculation shows that the dual of ``18 even" is identity 79, 
and the dual of ``18 odd" is identity 99.

In \S~\ref{t-SLfin}, I presented a finitization of each of the 
Rogers-Ramanujan type identities on Slater's list.  In this section, I
present the dual polynomial identities for some of the identities appearing 
in 
\S~\ref{t-SLfin}, as well as their limiting cases.

\subsection{Identities}
\hskip 6.25mm\begin{obs*}{\ref{dual}.2} Identity (2) is self-dual. \end{obs*}

\begin{obs*}{\ref{dual}.3} Identity (3) is self dual. \end{obs*}

\begin{id*}[Dual of \ref{t-SLfin}.4]{\ref{dual}.4f}\nopagebreak
\begin{gather*}
   \sum_{h\geqq 0}\sum_{i\geqq 0}\sum_{k\geqq 0}
  (-1)^{n-i-h} q^{2i(i+h+k) + (h+k)^2} \gp{n-i-h-k}{i}{q^2} 
  \gp{n-i-h-1}{k}{q^2}\\ \times \gp{n-i-k}{h}{q^2} \\ =
  \sum_{j=-\infty}^\infty (-1)^j \tzero njq +
  \sum_{j=-\infty}^\infty (-1)^j q^{2n-1} \tzero{n-1}{j}{q}
\tag{\ref{dual}.4f}
\end{gather*}
\end{id*}

\begin{obs*}{\ref{dual}.5} Identity (5) (with $q$ replaced by $-q$) is
self dual. \end{obs*}

\begin{id*}[Dual of \ref{t-SLfin}.6 even]{\ref{dual}.6f-even}
 \begin{gather*}
    \sum_{i\geqq 0}\sum_{j\geqq 0}\sum_{k\geqq 0} 
    q^{j^2 + i(i+1)/2 + k} \gp{m-j}{i}{q} 
    \gp{m-j+k-1}{k}{q^2} \gp{j+m-i-2k}{m-j}{q} \\ =
    \sum_{j=-\infty}^\infty  q^{j(3j-1)} \gp{2m}{m+3j}{q} 
                           + q^{j(3j+4) + m + 1} \gp{2m-1}{m+3j+1}{q} 
    \tag{\ref{dual}.6f-even}
 \end{gather*}
\end{id*}

\begin{id*}[Dual of \ref{t-ljslist}.6 even]{\ref{dual}.6-even}
 \begin{gather*}
 \sum_{i=0}^\infty \sum_{j=0}^\infty \sum_{k=0}^\infty
   \frac{q^{j^2 + i(i+1)/2 + k}}{(q;q)_i (q^2;q^2)_k (q;q)_{2j - i - 2k}} =
   \prod_{n=1}^\infty \frac{(1+q^{6n-2})(1+q^{6n-4})(1-q^{6n})}{1-q^n} 
 \tag{\ref{dual}.6-even}
 \end{gather*}
\end{id*}

\begin{id*}[Dual of \ref{t-SLfin}.6 odd]{\ref{dual}.6f-odd}
 \begin{gather*}
    \sum_{i\geqq 0}\sum_{j\geqq 0}\sum_{k\geqq 0} 
    q^{j^2 + j + i(i+1)/2 + k} \gp{m-j}{i}{q} 
    \gp{m-j+k-1}{k}{q^2} \gp{j+m-i-2k+1}{m-j}{q} \\ =
    \sum_{j=-\infty}^\infty  q^{j(3j-1) + n} \gp{2m}{m+3j}{q} 
                           + q^{j(3j+4)  1} \gp{2m+1}{m+3j+2}{q} 
    \tag{\ref{dual}.6f-odd} 
 \end{gather*}
\end{id*}

\begin{id*}[Dual of \ref{t-ljslist}.6 odd]{\ref{dual}.6-odd}
 \begin{gather*}
 \sum_{i=0}^\infty \sum_{j=0}^\infty \sum_{k=0}^\infty
   \frac{q^{j^2 + j + i(i+1)/2 + k}}{(q;q)_i (q^2;q^2)_k (q;q)_{2j -i -2k +1}} =
   \prod_{n=1}^\infty \frac{(1+q^{6n-1})(1+q^{6n-5})(1-q^{6n})}{1-q^n} 
 \tag{\ref{dual}.6-odd}
 \end{gather*}
\end{id*}

\begin{obs*}{\ref{dual}.8} The dual of $(8)$ is $(12)$. \end{obs*}

\begin{id*}[Dual of \ref{t-SLfin}.10]{\ref{dual}.10f}
\begin{gather*}
  \sum_{h\geqq 0}\sum_{i\geqq 0}\sum_{k\geqq 0}  
   q^{k + i + 2i(h+i+k) + (h+k)^2 }
    \gp{n-i-h-k}{i}{q^2} \gp{n-i-h-1}{k}{q^2} \gp{n-i-k}{h}{q^2} \\ =
  \sum_{j=-\infty}^\infty q^{2j^2 - j} \Trb{n}{2j}{2j}{q^2}
                    + q^{2j^2 - j + 2n -1} \Trb{n-1}{2j}{2j}{q^2}
   \tag{\ref{dual}.10f}
\end{gather*}
\end{id*}

\begin{id*}[Dual of \ref{t-ljslist}.10]{\ref{dual}.10}
\begin{gather*}
  \sum_{h=0}^\infty\sum_{i=0}^\infty\sum_{k=0}^\infty  
   \frac{ q^{k + i + 2i(h+i+k) + (h+k)^2 }}
        { (q^2;q^2)_h (q^2;q^2)_i (q^2;q^2)_k} =
  \prod_{j=1}^\infty (1+q^j)
   \tag{\ref{dual}.10} \\
\prod_{j=0}^\infty (1+q^{2j-1})
\sum_{j=0}^\infty \frac{ q^{j(j+1)}} {(q^2;q^2)_j} = \prod_{j=1}^\infty (1+q^j)
   \tag{\ref{dual}.$10'$}
\end{gather*}
\end{id*}

\begin{id*}[Dual of \ref{t-SLfin}.11]{\ref{dual}.11f}
\begin{gather*}
  \sum_{h\geqq 0}\sum_{i\geqq 0} 
   q^{(h+i)^2 + i(i+1) }
    \gp{n-i-h}{i}{q^2} \gp{2n-2i-h+1}{h}{q^2} \\ =
  \sum_{j=-\infty}^\infty (-1)^j q^{3j^2 + j} \gp{2n+1}{n+3j+1}{q}
   \tag{\ref{dual}.11f}
\end{gather*}
\end{id*}

\begin{id*}[Dual of \ref{t-ljslist}.11]{\ref{dual}.11}
\begin{gather*}
  \sum_{h=0}^\infty\sum_{i=0}^\infty  
  \frac { q^{(h+i)^2 + i(i+1)} }
    { (q^2;q^2)_i (q;q)_h } = 
  \prod_{j=1}^\infty (1+q^j)
   \tag{\ref{dual}.11}
\end{gather*}
\end{id*}

\begin{obs*}{\ref{dual}.12}The dual of (12) is (8).\end{obs*}

\begin{obs*}{\ref{dual}.14-even} The dual of (14 even) is (99). \end{obs*}

\begin{obs*}{\ref{dual}.14-odd}  The dual of (14 odd) is (96).  \end{obs*}

\begin{id*}[Dual of \ref{t-SLfin}.16]{\ref{dual}.16f} 
 \begin{gather*}
   \sum_{h\geqq 0} \sum_{k\geqq 0}  (-1)^k q^{2(h+k) + (h+k)^2}
    \gp{n-h-1}{k}{q^2} \gp{n-k}{h}{q^2} \\=
   \sum_{k=-\infty}^\infty 
        q^{15k^2 +  7k        } \Trb{n+1}{5k+1}{5k+1}{q^2}
       -q^{15k^2 + 17k + 4    } \Trb{n+1}{5k+3}{5k+3}{q^2} \\
       +q^{15k^2 +  7k + 1 +2n} \Trb{n+1}{5k+1}{5k+1}{q^2} 
       -q^{15k^2 + 17k + 5 +2n} \Trb{n+1}{5k+3}{5k+3}{q^2}            
  \tag{\ref{dual}.16f} 
 \end{gather*}
\end{id*}

\begin{id*}[Dual of \ref{t-ljslist}.16]{\ref{dual}.16} 
Note:  The dual of (\ref{t-ljslist}.16) with $q$ replaced by $-q$
is equivalent to  $(q;q^2)_\infty \times (\ref{t-ljslist}.96)$.
 \begin{gather*}
  \sum_{h=0}^\infty \sum_{k=0}^\infty  \frac{(-1)^k q^{2(h+k) + (h+k)^2}}
    { (q^2;q^2)_\infty (q^2;q^2)_h }\\ =
  \prod_{j=1}^\infty \frac{ (1-q^{10j-4}) (1-q^{10j-6}) (1-q^{10j})
    (1-q^{20j-2})(1-q^{20j-18})} {(1-q^{2j})} \tag{\ref{dual}.16} 
 \end{gather*} 
\end{id*}

\begin{id*}[Dual of \ref{t-SLfin}.17]{\ref{dual}.17f}
 \begin{gather*}
   \sum_{h\geqq 0} \sum_{k\geqq 0} (-1)^k q^{(h+k)^2 + h} \gp{n-h}{k}{q^2}
     \gp{n-k}{h}{q^2} \\=
   \sum_{k=-\infty}^\infty q^{15k^2 +  2k    }\Trb{n+1}{5k  }{5k+1}{q} 
                          -q^{15k^2 + 12k + 2}\Trb{n+1}{5k+2}{5k+3}{q}
   \tag{\ref{dual}.17f}
 \end{gather*}
\end{id*}

\begin{id*}[Dual of \ref{t-ljslist}.17]{\ref{dual}.17}
Note: The dual of (\ref{t-ljslist}.17) with $q$ replaced by $-q$ is equivalent to 
$(-q^2;q^2)_\infty\times (\ref{t-ljslist}.20)$.
  \begin{gather*}
    \sum_{h=0}^\infty \sum_{k=0}^\infty \frac{(-1)^k q^{(h+k)^2 + h}}
      {(q^2;q^2)_k (q^2;q^2)_h} \\ =
    \prod_{j=1}^\infty \frac{ (1-q^{10j-1})(1-q^{10j-9})(1-q^{10j}) 
    (1-q^{20j-8})(1-q^{20j-12})}
     {(1-q^{2j})} \tag{\ref{dual}.17}
    \end{gather*}
\end{id*}

\begin{obs*}{\ref{dual}.18-even} 
 The dual of (18 even) is $(79)$.
\end{obs*}

\begin{obs*}{\ref{dual}.18-odd} 
 The dual of (18 odd) is $(94)$. 
\end{obs*} 

\begin{id*}[Dual of \ref{t-SLfin}.20]{\ref{dual}.20f}
 \begin{gather*}
   \sum_{h\geqq 0} \sum_{k\geqq 0} q^{(h+k)^2}
    \gp{n-h-1}{k}{q^2} \gp{n-k}{h}{q}
      \\=
     \sum_{j=-\infty}^\infty 
      (-1)^j q^{j(15j-1)}          \Trb{n}  {5j  }{5j  }{q^2}
     +(-1)^j q^{j(15j-1) + 2n - 1} \Trb{n-1}{5j+1}{5j+1}{q^2}\\
     -(-1)^j q^{j(15j+19) + 6    } \Trb{n}  {5j+3}{5j+3}{q^2}
     -(-1)^j q^{j(15j+19) +2n  +5} \Trb{n-1}{5j+3}{5j+3}{q^2} 
  \tag{\ref{dual}.20f} 
 \end{gather*}
\end{id*}

\begin{id*}[Dual of \ref{t-ljslist}.20]{\ref{dual}.20} 
Note:  The dual of (\ref{t-ljslist}.20) with $q$ replaced by $-q$
is equivalent to  $(q;q^2)_\infty \times (\ref{t-ljslist}.79)$.
 \begin{gather*}
   \sum_{h=0}^\infty\sum_{k=0}^\infty \frac{q^{(h+k)^2}}
   { (q^2;q^2)_h (q^2;q^2)_k} \\ =
     \prod_{j=1}^\infty \frac{ (1-q^{10j-2}) (1-q^{10j-8}) (1-q^{10j})
    (1-q^{20j-14})(1-q^{20j-6})} {(1-q^{2j})} \tag{\ref{dual}.20} 
 \end{gather*} 
\end{id*}

\begin{id*}[Dual of \ref{t-SLfin}.25]{\ref{dual}.25f}
\begin{gather*}
  \sum_{h\geqq 0} \sum_{i\geqq 0} \sum_{k\geqq 0}  
   (-1)^k q^{2i(i+h+k) + (h+k)^2}
    \gp{n-i-h-k}{i}{q^2} \gp{n-i-h}{k}{q^2} \gp{n-i-k}{h}{q^2} \\ =
  \sum_{j=-\infty}^\infty (-1)^j q^{6j^2 }         \Trb{n}{3j}{3j}{q^2} +
  \sum_{j=-\infty}^\infty (-1)^j q^{6j^2 + 6j + 1} \Trb{n}{3j+1}{3j+1}{q^2}
  \tag{\ref{dual}.25f} 
\end{gather*}
\end{id*}

\begin{id*}[Dual of \ref{t-ljslist}.25]{\ref{dual}.25} 
\begin{gather*}
 \sum_{h\geqq 0} \sum_{i\geqq 0} \sum_{k\geqq 0}
    \frac{  (-1)^k q^{2i(h+i+k) + (h+k)^2 }}
    {(q^2;q^2)_h (q^2;q^2)_i (q^2;q^2)_k }
    =  \prod_{j=1}^\infty
  \frac{ (1-q^{12j-6})^2 (1-q^{12j})}{ (1-q^{2j})} \tag{\ref{dual}.25} \\
  \prod_{j=1}^\infty (1+q^{2j-1})
  \sum_{h=0}^\infty\sum_{i=0}^\infty \frac { (-1)^k q^{(h+i)^2 + i^2}}
    { (q^2;q^2)_i (q^2;q^2)_h (-q;q^2)_{h+i} } = 
  \prod_{j=1}^\infty
  \frac{ (1-q^{12j-6})^2 (1-q^{12j})}{ (1-q^{2j})} \tag{\ref{dual}.$25'$} 
\end{gather*}
\end{id*}

\begin{id*}[Dual of \ref{t-SLfin}.27 even]{\ref{dual}.27f-even}
\begin{gather*}
  \sum_{i\geqq 0}\sum_{J\geqq 0}\sum_{k\geqq 0}\sum_{l\geqq 0}
    (-1)^l q^{2J^2 + 2J + i^2 -k}
    \gp{n-J}{i}{q^2} \gp{n-J+k}{k}{q^2} \gp{n-J+l-1}{l}{q^2}\\ \times
    \gp{n+J-i-k-l}{2j-i-k-l}{q^2} \\ =
   \sum_{j=-\infty}^\infty q^{2j(3j+1)} \gp{2m+1}{m+3j+1}{q^2}
                          -q^{2j(3j+2) + 2m + 1}\gp{2m}{m+3j+1}{q^2}
   \tag{\ref{dual}.27f-even}
\end{gather*}
\end{id*}

\begin{id*}[Dual of \ref{t-ljslist}.27 even]{\ref{dual}.27-even}
\begin{gather*}
  \sum_{i=0}^\infty\sum_{J=0}^\infty\sum_{k=0}^\infty\sum_{l=0}^\infty
    \frac{ (-1)^l q^{2J^2 + 2J + i^2 -k} }
    {(q^2;q^2)_i (q^2;q^2)_k (q^2;q^2)_l (q^2;q^2)_{2j-i-k-l}}\\ =
   \prod_{j=1}^\infty \frac{ (1+q^{12j-4})(1+q^{12j-8})(1-q^{12j})}
     {(1-q^{2j})}
   \tag{\ref{dual}.27-even}
\end{gather*}
\end{id*}

\begin{id*}[Dual of \ref{t-SLfin}.28]{\ref{dual}.28f}
\begin{gather*}
   \sum_{h\geqq 0} \sum_{i\geqq 0} 
      q^{(h+i)^2 + i^2} \gp{n-i-h}{i}{q^2} \gp{2n-2i-h+1}{h}{q} \\ =
  \sum_{j=-\infty}^\infty (-1)^j q^{j(6j+1)}\Trb{n+1}{3j}{3j+1}{q^2}
  \tag{\ref{dual}.28f} 
\end{gather*}
\end{id*}

\begin{id*}[Dual of \ref{t-ljslist}.28]{\ref{dual}.28}
\begin{gather*}
  \sum_{h=0}^\infty \sum_{i=0}^\infty
   \frac { q^{i^2 + (h+i)^2}}
     { (q^2;q^2)_i (q;q)_h }\\ = 
   \prod_{j=1}^\infty \frac{ (1+q^{3j-1}) (1+q^{3j-2}) (1-q^{3j})} 
  {(1-q^{2j})} \tag{\ref{dual}.28} 
\end{gather*}
\end{id*}

\begin{id*}[Dual of \ref{t-SLfin}.29]{\ref{dual}.29f}
\begin{gather*} \sum_{h\geqq 0} \sum_{i\geqq 0}
  q^{ (h+i)^2 + i^2 } \gp{n-h-i}{i}{q^2} \gp{2n-2i-h}{h}{q} \\ =
  \sum_{k=-\infty}^\infty q^{6k^2 -k} \Trb{n}{3k}{3k}{q^2}
                         +q^{6k^2+5k+1}\Trb{n}{3k+1}{3k+1}{q^2}
\tag{\ref{dual}.29f} 
\end{gather*}
\end{id*} 

\begin{id*}[Dual of \ref{t-ljslist}.29]{\ref{dual}.29} 
  \begin{gather*}
   \sum_{h=0}^\infty \sum_{i=0}^\infty \frac{ q^{i^2 + (i+h)^2}}
      {(q^2;q^2)_i (q;q)_h } =
  \prod_{j=1}^\infty \frac{ (1+q^{3j-1}) (1+q^{3j-2}) (1-q^{3j})} 
  {(1-q^{2j})}
\tag{\ref{dual}.29}
\end{gather*}
\end{id*} 

\begin{id*}[Dual of \ref{t-SLfin}.31 even]{\ref{dual}.31f-even}
  \begin{gather*}
    \sum_{J,k,L\geqq 0} (-1)^{k+L} q^{2J^2 + 2J -k} \gp{m-J+k}{k}{q^2}
      \gp{m-J+L-1}{L}{q^2} \gp{m+J-k-L}{m-J}{q^2} \\
     =\sum_{k=-\infty}^\infty 
       q^{42k^2 +  4k    }\gp{2m+1}{m+7k+1}{q^2}
     - q^{42k^2 + 32k + 6}\gp{2m+1}{m+7k+3}{q^2}\\
     + q^{42k^2 + 60k +21} 
      \left(\gp{2m+1}{m+7k+5}{q^2}-\gp{2m+1}{m+7k+6}{q^2}\right)
  \tag{\ref{dual}.31f-even}
  \end{gather*}
\end{id*}

\begin{id*}[Dual of \ref{t-ljslist}.31 even]{\ref{dual}.31-even}
  \begin{gather*}
    \sum_{J=0}^\infty \sum_{k=0}^\infty \sum_{L=0}^\infty
    \frac{(-1)^{k+L} q^{2J^2 + 2J - k}} { (q^2;q^2)_k (q^2;q^2)_L
       (q^2;q^2)_{2J-k-L}} \\    
=\prod_{j=1}^\infty \frac{ (1-q^{28j-6}) (1-q^{28j-22}) (1-q^{28j})
    (1-q^{56j-16}) (1-q^{56j-40}) } {(1-q^{2j})}
  \tag{\ref{dual}.31-even}
  \end{gather*}
\end{id*}

\begin{id*}[Dual of \ref{t-SLfin}.31 odd]{\ref{dual}.31f-odd}
  \begin{gather*}
    \sum_{J,k,L\geqq 0} (-1)^{k+L} q^{2J^2 + 4J -k+1} \gp{m-J+k}{k}{q^2}
      \gp{m-J+L-1}{L}{q^2} \gp{m+J-k-L+1}{m-J}{q^2} \\
     =\sum_{k=-\infty}^\infty
     - q^{42k^2 - 10k    }\gp{2m+2}{m+7k}{q^2}
     + q^{42k^2 - 38k + 8}\gp{2m+2}{m+7k-2}{q^2}\\
     + q^{42k^2 + 18k +1}
      \left(\gp{2m+2}{m+7k+2}{q^2}-\gp{2m+2}{m+7k+3}{q^2}\right)
  \tag{\ref{dual}.31f-odd}
  \end{gather*}
\end{id*}

\begin{id*}[Dual of \ref{t-ljslist}.31 odd]{\ref{dual}.31-odd}
  \begin{gather*}
    \sum_{J=0}^\infty \sum_{k=0}^\infty \sum_{L=0}^\infty
    \frac{(-1)^{k+L} q^{2J^2 + 4J - k+1}}{(q^2;q^2)_k (q^2;q^2)_L
       (q^2;q^2)_{2J-k-L}} \\
=\prod_{j=1}^\infty \frac{ (1-q^{28j-8}) (1-q^{28j-20}) (1-q^{28j})
    (1-q^{56j-28}) (1-q^{56j-44}) } {(1-q^{2j})}
  \tag{\ref{dual}.31-odd}
  \end{gather*}
\end{id*}

\begin{id*}[Dual of \ref{t-SLfin}.32 even]{\ref{dual}.32f-even}
  \begin{gather*}
    \sum_{J,k,L\geqq 0} (-1)^{k+L} q^{2J^2 + 2J +k} \gp{m-J+k-1}{k}{q^2}
      \gp{m-J+L-1}{L}{q^2} \gp{m+J-k-L}{m-J}{q^2} \\
     =\sum_{k=-\infty}^\infty
       q^{42k^2 +  8k    }\gp{2m+1}{m+7k+1}{q^2}
     - q^{42k^2 - 20k + 2}\gp{2m+1}{m+7k+6}{q^2}\\
     + q^{42k^2 + 22k + 2n + 3}\gp{2m}{m+7k+2}{q^2} 
     - q^{42k^2 + 50k + 2n + 15}\gp{2m}{m+7k+4}{q^2}
  \tag{\ref{dual}.32f-even}
  \end{gather*}
\end{id*}

\begin{id*}[Dual of \ref{t-ljslist}.32 even]{\ref{dual}.32-even}
  \begin{gather*}
    \sum_{J=0}^\infty \sum_{k=0}^\infty \sum_{L=0}^\infty
    \frac{(-1)^{k+L} q^{2J^2 + 2J + k}}{(q^2;q^2)_k (q^2;q^2)_L
       (q^2;q^2)_{2J-k-L}} \\
=\prod_{j=1}^\infty \frac{ (1-q^{28j-2}) (1-q^{28j-26}) (1-q^{28j})
    (1-q^{56j-28}) (1-q^{56j-32}) } {(1-q^{2j})}
  \tag{\ref{dual}.32-even}
  \end{gather*}
\end{id*}

\begin{id*}[Dual of \ref{t-SLfin}.32 odd]{\ref{dual}.32f-odd}
  \begin{gather*}
    \sum_{J,k,L\geqq 0} (-1)^{k+L} q^{2J^2 + 4J +k-1} \gp{m-J+k-1}{k}{q^2}
      \gp{m-J+L-1}{L}{q^2}\\ \times \gp{m+J-k-L+1}{m-J}{q^2} \\
     =\sum_{k=-\infty}^\infty
     - q^{42k^2 - 22k    }\gp{2m+2}{m+7k+3}{q^2}
     + q^{42k^2 - 34k + 4}\gp{2m+2}{m+7k-2}{q^2}\\
     + q^{42k^2 +  8k 2m-1}\gp{2m+1}{m+7k+1}{q^2}
     - q^{42k^2 + 64k + 2m+23}\gp{2m+1}{m+7k+6}{q^2}
  \tag{\ref{dual}.32f-odd}
  \end{gather*}
\end{id*}

\begin{id*}[Dual of \ref{t-ljslist}.32 odd]{\ref{dual}.32-odd}
  \begin{gather*}
    \sum_{J=0}^\infty \sum_{k=0}^\infty \sum_{L=0}^\infty
    \frac{(-1)^{k+L} q^{2J^2 + 4J +k-1}}{(q^2;q^2)_k (q^2;q^2)_L
       (q^2;q^2)_{2J-k-L+1}} \\
=\prod_{j=1}^\infty \frac{ (1-q^{28j-4}) (1-q^{28j-24}) (1-q^{28j})
    (1-q^{56j-20}) (1-q^{56j-36}) } {(1-q^{2j})}
  \tag{\ref{dual}.32-odd}
  \end{gather*}
\end{id*}

\begin{id*}[Dual of \ref{t-SLfin}.33 even]{\ref{dual}.33f-even}
  \begin{gather*}
    \sum_{J,k,L\geqq 0} (-1)^{k+L} q^{2J^2 +k} \gp{m-J+k-1}{k}{q^2}
      \gp{m-J+L-1}{L}{q^2} \gp{m+J-k-L}{m-J}{q^2} \\
     =\sum_{k=-\infty}^\infty
       q^{42k^2 -  2k    }\gp{2m}{m+7k}{q^2}
     - q^{42k^2 + 26k + 4}\gp{2m}{m+7k-2}{q^2}\\
     + q^{42k^2 + 30k +5}
      \left(\gp{2m}{m+7k+2}{q^2}-\gp{2m}{m+7k+3}{q^2}\right)
  \tag{\ref{dual}.33f-even}
  \end{gather*}
\end{id*}

\begin{id*}[Dual of \ref{t-ljslist}.33 even]{\ref{dual}.33-even}
  \begin{gather*}
    \sum_{J=0}^\infty \sum_{k=0}^\infty \sum_{L=0}^\infty
    \frac{(-1)^{k+L} q^{2J^2 + k}}{(q^2;q^2)_k (q^2;q^2)_L
       (q^2;q^2)_{2J-k-L}} \\
=\prod_{j=1}^\infty \frac{ (1-q^{28j-4}) (1-q^{28j-24}) (1-q^{28j})
    (1-q^{56j-20}) (1-q^{56j-36}) } {(1-q^{2j})}
  \tag{\ref{dual}.33-even}
  \end{gather*}
\end{id*}

\begin{obs*}{\ref{dual}.34} Identity 34 is self-dual.\end{obs*}

\begin{id*}[Dual of \ref{t-SLfin}.35]{\ref{dual}.35f}
  \begin{gather*}
    \sum_{h\geqq 0} \sum_{i\geqq 0} \sum_{k\geqq 0}
      q^{ 2h(k+i) + h(h+3)/2 + 4ik + 2k(k+1) + 3i(i+1)}
      \gp{n-2i-h-2k}{i}{q^2} \\ \times \gp{n-2i-h-k}{k}{q^2}
         \gp{n-2i-2k}{h}{q} \\ =
      \sum_{j=-\infty}^\infty (-1)^j q^{n + 4j^2 + 3j + 1} 
    \Trb{n+1}{4j+1}{4j+2}{q} \\ +
  \sum_{j=-\infty}^\infty (-1)^j q^{4j^2 + 3j} \Trb{n+1}{4j+1}{4j+1}{q}
   \tag{\ref{dual}.35f}
\end{gather*} 
\end{id*}
    
\begin{id*}[Dual of \ref{t-ljslist}.35]{\ref{dual}.35}
\begin{gather*}
  \sum_{h=0}^\infty \sum_{i=0}^\infty \sum_{k=0}^\infty
   \frac{ q^{ 2h(k+i) + h(h+3)/2 + 4ik  + 2k(k+1) + 3i(i+1)}}
   {(q^2;q^2)_i (q^2;q^2)_k (q;q)_h} \\ =
  \prod_{j=1}^\infty \frac{ (1-q^{8j-1})(1-q^{8j-7}) (1-q^{8j})}{(1-q^j)}
 \tag{\ref{dual}.35}
\end{gather*}
\end{id*}

\begin{obs*}{\ref{dual}.36} Identity 36 is self-dual.\end{obs*}

\begin{id*}[Dual of \ref{t-SLfin}.37]{\ref{dual}.37f}
\begin{gather*}
  \sum_{h\geqq 0} \sum_{i\geqq 0} \sum_{k\geqq 0}
    q^{ 2hk + h(h+1)/2 + 4ik + 2ih + 2k^2 + 3i^2 + i}
   \gp{n-2i-h-2k}{i}{q^2} \gp{n-2i-h-k}{k}{q^2} \\ \times\gp{n-2i-2k}{h}{q} \\ =
  \sum_{j=-\infty}^\infty (-1)^j q^{n + 4j^2 + j} 
    \Trb{n}{4j}{4j+1}{q} +
  \sum_{j=-\infty}^\infty (-1)^j q^{4j^2 + j} \Trb{n}{4j}{4j}{q}
   \tag{\ref{dual}.37f}
\end{gather*} 
\end{id*}

\begin{id*}[Dual of \ref{t-ljslist}.37]{\ref{dual}.37}
Note: This identity is equivalent to $(-q;q)_\infty \times 
(\ref{t-ljslist}.39)$ with $q$ replaced by $-q$.
\begin{gather*}
  \sum_{h=0}^\infty \sum_{i=0}^\infty \sum_{k=0}^\infty
   \frac{ q^{ 2hk + h(h+1)/2 + 4ik + 2ih + 2k^2 + 3i^2 + i}}
   {(q^2;q^2)_i (q^2;q^2)_k (q;q)_h} =
  \prod_{j=1}^\infty \frac{ (1-q^{8j-3})(1-q^{8j-5}) (1-q^{8j})}{(1-q^j)}
 \tag{\ref{dual}.37}
\end{gather*}
\end{id*}

\begin{obs*}{\ref{dual}.38-even} Identity (38 even) is the dual of (39 odd). \end{obs*}
\begin{obs*}{\ref{dual}.38-odd}  Identity (38 odd)  is self-dual. \end{obs*}
\begin{obs*}{\ref{dual}.39-even} Identity (39 even) is self-dual. \end{obs*}
\begin{obs*}{\ref{dual}.39-odd}  Identity (39 odd)  is the dual of (38 even).\end{obs*}

\begin{id*}[Dual of \ref{t-ljslist}.45]{\ref{dual}.45}
 \begin{gather*}
   \sum_{h=0}^\infty \sum_{i=0}^\infty \sum_{k=0}^\infty
   \frac{q^{2k(i+h) + ih + h(h+1)/2 + 2k^2 + i^2}}
   {(q^2;q^2)_i (q^2;q^2)_k (q;q)_h} \\ = 
  \prod_{j=1}^\infty 
    \frac{ (1+q^{5j-1})(1+q^{5j-4})(1-q^{10j-3})(1-q^{10j-7})(1-q^{5j})}
    {(1-q^j)}
 \end{gather*}
\end{id*}

\begin{id*}[Dual of \ref{t-SLfin}.50]{\ref{dual}.50f}
\begin{gather*}
  \sum_{h\geqq 0} \sum_{i\geqq 0} 
   q^{2i^2 + 2ih + h^2 + 2i + h}
    \gp{n-i-h}{i}{q^2} \gp{2n-2i-h+1}{h}{q} \\ =
  \sum_{j=-\infty}^\infty (-1)^j q^{3j^2 + 2j} \gp{2n+2}{n+3j+2}{q}
 \tag{\ref{dual}.50f}
\end{gather*}
\end{id*}

\begin{id*}[Dual of \ref{t-ljslist}.50]{\ref{dual}.50}
\begin{gather*}
  \sum_{h=0}^\infty \sum_{i=0}^\infty
  \frac { q^{(h+i)^2 + i(i+2) + h}}
    { (q^2;q^2)_i (q;q)_h } = 
  \prod_{j=1}^\infty\frac{(1-q^{6j-1}) (1-q^{6j-5}) (1-q^{6j}) }
  {(1-q^{j}) } \tag{\ref{dual}.50}
\end{gather*}
\end{id*}

\begin{id*}[Dual of \ref{t-SLfin}.59 even]{\ref{dual}.59f-even}
\begin{gather*}
  \sum_{j\geqq 0} \sum_{k\geqq 0}
    q^{(j+k)^2 +2j +k} \gp{n-j}{k}{q^2} \gp{n+j-k}{2j}{q}\\ =
  \sum_{k=-\infty}^\infty q^{21 k^2 + 4k   } \gp{2m+2}{m+7k+2}{q}
                         -q^{21 k^2 +32k+ 4} \gp{2m+2}{m+7k+6}{q}
 \tag{\ref{dual}.59f-even}
\end{gather*}
\end{id*}

\begin{id*}[Dual of \ref{t-ljslist}.59 even]{\ref{dual}.59 even}
Note: This identity is equivalent to $(-q;q^2)_\infty \times
 (\ref{t-ljslist}.118)$
\begin{gather*}
  \sum_{j=0}^\infty \sum_{k=0}^\infty
    \frac{ q^{ (j+k)^2 + 2j + k} } {(q^2;q^2)_k (q;q)_{2j}}
 \\ =  \prod_{j=1}^\infty
   \frac{(1-q^{14j-1})(1-q^{14j-13})(1-q^{28j-12})(1-q^{28j-16})
(1-q^{14j})} {(1-q^j)}
\tag{\ref{dual}.59-even}
\end{gather*}
\end{id*}

\begin{id*}[Dual of \ref{t-SLfin}.59 odd]{\ref{dual}.59f odd}
\begin{gather*}
  \sum_{j\geqq 0} \sum_{k\geqq 0}
    q^{(j+k)^2 +3j + 2k} \gp{n-j}{k}{q^2} \gp{n+j-k+1}{2j+1}{q}\\ =
  \sum_{k=-\infty}^\infty q^{21 k^2 +11k   } \gp{2m+3}{m+7k+3}{q}
                         -q^{21 k^2 +25k+ 6} \gp{2m+3}{m+7k+6}{q}
 \tag{\ref{dual}.59f-odd}
\end{gather*}
\end{id*}

\begin{id*}[Dual of \ref{t-ljslist}.59 odd]{\ref{dual}.59 odd}
Note: This identity is equivalent to $(-q^2;q^2)_\infty \times 
(\ref{t-ljslist}.82)$
with $q$ replaced by $q^2$.
\begin{gather*}
  \sum_{j=0}^\infty \sum_{k=0}^\infty
    \frac{ q^{ (j+k)^2 + 3j + 2k} } {(q^2;q^2)_k (q;q)_{2j+1}}
  \\ = \prod_{j=1}^\infty
   \frac{(1-q^{14j-6})(1-q^{14j-8})(1-q^{28j-2})(1-q^{28j-26})
(1-q^{14j})} {(1-q^j)}
\tag{\ref{dual}.59-odd}
\end{gather*}
\end{id*}

\begin{id*}[Dual of \ref{t-SLfin}.60 even]{\ref{dual}.60f-even}
\begin{gather*}
  \sum_{j\geqq 0} \sum_{k\geqq 0}
    q^{(j+k)^2 + j} \gp{n-j}{k}{q^2} \gp{n+j-k}{2j}{q}\\ =
  \sum_{k=-\infty}^\infty q^{21 k^2 +  k   } \gp{2m+1}{m+7k+1}{q}
                         -q^{21 k^2 +29k+10} \gp{2m+1}{m+7k+5}{q}
 \tag{\ref{dual}.60f-even}
\end{gather*}
\end{id*}

\begin{id*}[Dual of \ref{t-ljslist}.60 even]{\ref{dual}.60-even}
Note: This identity is equivalent to $(-q;q^2)_\infty \times 
(\ref{t-ljslist}.81)$
with $q$ replaced by $q^2$.
\begin{gather*}
  \sum_{j=0}^\infty \sum_{k=0}^\infty
    \frac{ q^{ (j+k)^2 + j} } {(q^2;q^2)_k (q;q)_{2j}}
 \\ =  \prod_{j=1}^\infty
   \frac{(1-q^{14j-2})(1-q^{14j-12})(1-q^{28j-10})(1-q^{28j-18})
(1-q^{14j})} {(1-q^j)}
\tag{\ref{dual}.60-even}
\end{gather*}
\end{id*}

\begin{id*}[Dual of \ref{t-SLfin}.60 odd]{\ref{dual}.60f-odd}
\begin{gather*}
  \sum_{j\geqq 0} \sum_{k\geqq 0}
    q^{(j+k)^2 + 2j + k} \gp{n-j}{k}{q^2} \gp{n+j-k+1}{2j+1}{q}\\ =
  \sum_{k=-\infty}^\infty q^{21 k^2 + 8k   } \gp{2m+2}{m+7k+2}{q}
                         -q^{21 k^2 +22k+ 5} \gp{2m+2}{m+7k+5}{q}
 \tag{\ref{dual}.60f-odd}
\end{gather*}
\end{id*}

\begin{id*}[Dual of \ref{t-ljslist}.60 odd]{\ref{dual}.60-odd}
Note: This identity is equivalent to $(-q;q^2)_\infty \times 
(\ref{t-ljslist}.119)$
\begin{gather*}
  \sum_{j=0}^\infty \sum_{k=0}^\infty
    \frac{ q^{ (j+k)^2 + 2j + k} } {(q^2;q^2)_k (q;q)_{2j+1}}
 = \\ \prod_{j=1}^\infty
   \frac{(1-q^{14j-5})(1-q^{14j-9})(1-q^{28j-4})(1-q^{28j-28})
(1-q^{14j})} {(1-q^j)}
\tag{\ref{dual}.60-odd}
\end{gather*}
\end{id*}

\begin{id*}[Dual of \ref{t-SLfin}.61 even]{\ref{dual}.61f-even}
\begin{gather*}
  \sum_{j\geqq 0} \sum_{k\geqq 0}
    q^{(j+k)^2 + k} \gp{n-j-1}{k}{q^2} \gp{n+j-k}{2j}{q}\\ =
  \sum_{k=-\infty}^\infty q^{21 k^2 + 2k   } \gp{2m}{m+7k  }{q}
                         -q^{21 k^2 +26k +8} \gp{2m}{m+7k+4}{q}
 \tag{\ref{dual}.61f-even}
\end{gather*}
\end{id*}

\begin{id*}[Dual of \ref{t-ljslist}.61 even]{\ref{dual}.61-even}
Note: This identity is equivalent to $(-q^2;q^2)_\infty \times 
(\ref{t-ljslist}.117)$. 
\begin{gather*}
  \sum_{j=0}^\infty \sum_{k=0}^\infty
    \frac{ q^{ (j+k)^2 + k} } {(q^2;q^2)_k (q;q)_{2j}}
 \\ =  \prod_{j=1}^\infty 
   \frac{(1-q^{14j-3})(1-q^{14j-11})(1-q^{28j-8})(1-q^{28j-20})
(1-q^{14j})} {(1-q^j)}
\tag{\ref{dual}.61-even}
\end{gather*}
\end{id*}

\begin{id*}[Dual of \ref{t-SLfin}.61 odd]{\ref{dual}.61f-odd}
\begin{gather*}
  \sum_{j\geqq 0} \sum_{k\geqq 0}
    q^{(j+k)^2 + j +2k} \gp{n-j-1}{k}{q^2} \gp{n+j-k+1}{2j+1}{q}\\ =
  \sum_{k=-\infty}^\infty q^{21 k^2 + 5k   } \gp{2m}{m+7k+1}{q}
                         -q^{21 k^2 +19k +4} \gp{2m}{m+7k+4}{q}
 \tag{\ref{dual}.61f-odd}
\end{gather*}
\end{id*}

\begin{id*}[Dual of \ref{t-ljslist}.61 odd]{\ref{dual}.61-odd}
Note: This identity is equivalent to $(-q;q^2)_\infty \times 
(\ref{t-ljslist}.80)$ with
$q$ replaced by $q^2$.
\begin{gather*}
  \sum_{j=0}^\infty \sum_{k=0}^\infty
    \frac{ q^{ (j+k)^2 + j +2k} } {(q^2;q^2)_k (q;q)_{2j+1}}
 \\ =  \prod_{j=1}^\infty
   \frac{(1-q^{14j-4})(1-q^{14j-10})(1-q^{28j-6})(1-q^{28j-22})
(1-q^{14j})} {(1-q^j)}
\tag{\ref{dual}.61-odd}
\end{gather*}
\end{id*}

\begin{id*}[Dual of \ref{t-SLfin}.68]{\ref{dual}.68f}
\begin{gather*}
  \sum_{h\geqq 0} \sum_{i\geqq 0} \sum_{k\geqq 0} \sum_{l\geqq 0}
    (-1)^k q^{2h + 2i(2h+2k+2l+1+3i +(h+k)^2}
     \gp{n-2i-h-k-l}{i}{q^4} \\ \times \gp{n-2i-h-l}{k}{q^2}
      \times \gp{n-2i-h-k}{l}{q^2} \gp{n-2i-k-l}{h}{q^2}\\ =
  \sum_{j=-\infty}^\infty (-1)^j q^{8j^2 + 2j   } \Trb{n+1}{4j+1}{4j+1}{q^2}
                         -(-1)^j q^{8j^2 +10j +3} \Trb{n+1}{4j+2}{4j+2}{q^2}
 \tag{\ref{dual}.68}
\end{gather*}
\end{id*}

\begin{id*}[Dual of \ref{t-ljslist}.68]{\ref{dual}.68}
\begin{gather*}
  \sum_{h=0}^\infty \sum_{i=0}^\infty \sum_{k=0}^\infty \sum_{l=0}^\infty
    \frac{ (-1)^k q^{2h + 2i(2h+2k+2l+1+3i +(h+k)^2}}{ (q^4;q^4)_i 
    (q^2;q^2)_k (q^2;q^2)_l (q^2;q^2)_h}
 \\= \frac{(q,-q^3,-q^4;-q^4)_\infty}{(q^2;q^2)_\infty}
\tag{\ref{dual}.68}
\end{gather*}
\end{id*}

\begin{id*}[Dual of \ref{t-SLfin}.69]{\ref{dual}.69f}
\begin{gather*}
  \sum_{h\geqq 0} \sum_{i\geqq 0} \sum_{k\geqq 0} \sum_{l\geqq 0}
    (-1)^i q^{2h + 2i(2h+2k+2l+1+3i) +(h+k)^2}
     \gp{n-2i-h-k-l}{i}{q^4}\\ \times \gp{n-2i-h-l}{k}{q^2}  
      \times \gp{n-2i-h-k}{l}{q^2} \gp{n-2i-k-l}{h}{q^2}\\ =
  \sum_{j=-\infty}^\infty (-1)^j q^{8j^2 + 2j   } \Trb{n+1}{4j+1}{4j+1}{q^2}
                         +(-1)^j q^{8j^2 +10j +3} \Trb{n+1}{4j+2}{4j+2}{q^2}
 \tag{\ref{dual}.69f}
\end{gather*}
\end{id*}

\begin{id*}[Dual of \ref{t-ljslist}.69]{\ref{dual}.69}
\begin{gather*}
  \sum_{h=0}^\infty \sum_{i=0}^\infty \sum_{k=0}^\infty \sum_{l=0}^\infty
    \frac{ (-1)^i q^{2j + 2i(2h+2k+2l+1+3i) +(h+k)^2}}{ (q^4;q^4)_i 
    (q^2;q^2)_k (q^2;q^2)_l (q^2;q^2)_h}
 =  \prod_{j=1}^\infty {(1+q^j) }
\tag{\ref{dual}.69}
\end{gather*}
\end{id*}

\begin{obs*}{\ref{dual}.79} 
 The dual of (79) is (18 even).
\end{obs*}

\begin{id*}[Dual of \ref{t-SLfin}.80]{\ref{dual}.80f}
\begin{gather*}
  \sum_{h\geqq 0} \sum_{k\geqq 0} 
    q^{ h(h+1)/2 + 2k(h+k) } \gp{n-h-k}{k}{q^2} \gp{n-2k}{h}{q}
 \\ =
  \sum_{k=-\infty}^\infty 
     q^{14k^2 +   k    } \Trb{n+1}{7k  }{7k+1}{q}
   - q^{14k^2 + 13k + 3} \Trb{n+1}{7k+3}{7k+4}{q} 
\tag{\ref{dual}.80f}
\end{gather*}
\end{id*}

\begin{id*}[Dual of \ref{t-ljslist}.80]{\ref{dual}.80}
Equivalent to Andrews~\cite[p. 331, eqn. (1.1)]{mod11}, and to 
$(-q;q)_\infty
\times (\ref{t-ljslist}.33)$.
\begin{gather*}
  \sum_{h=0}^\infty \sum_{k=0}^\infty
   \frac{ q^{ h(h+1)/2 + 2k(h+k)} } {(q^2;q^2)_k (q;q)_h }
 = \prod_{j=1}^\infty \frac{ (1-q^{7j-3}) (1-q^{7j-4}) (1-q^{7j})} {(1-q^j)}
\tag{\ref{dual}.80}
\end{gather*}
\end{id*}

\begin{id*}[Dual of \ref{t-SLfin}.81]{\ref{dual}.81f}
\begin{gather*}
  \sum_{h\geqq 0} \sum_{k\geqq 0}
    q^{ h(h+1)/2 + 2k(h+k+1) } \gp{n-h-k-1}{k}{q^2} \gp{n-2k}{h}{q}
 \\ =
  \sum_{k=-\infty}^\infty
     q^{14k^2 -  4k    } \Trb{n+1}{7k-1}{7k  }{q}
   - q^{14k^2 +  4k    } \Trb{n+1}{7k+1}{7k+2}{q}
\tag{\ref{dual}.81f}
\end{gather*}
\end{id*}

\begin{id*}[Dual of \ref{t-ljslist}.81]{\ref{dual}.81}
Equivalent to Andrews~\cite[p. 331, eqn. (1.2)]{mod11}, and to
$(-q;q)_\infty
\times (\ref{t-ljslist}.32)$.
\begin{gather*}
  \sum_{h=0}^\infty \sum_{k=0}^\infty
   \frac{ q^{ h(h+1)/2 + 2k(h+k+1)} } {(q^2;q^2)_k (q;q)_h }
 = \prod_{j=1}^\infty \frac{ (1-q^{7j-2}) (1-q^{7j-5}) (1-q^{7j})} {(1-q^j)}
\tag{\ref{dual}.81}
\end{gather*}
\end{id*}
 
\begin{id*}[Dual of \ref{t-SLfin}.82]{\ref{dual}.82f}
\begin{gather*}
  \sum_{h\geqq 0} \sum_{k\geqq 0}
    q^{ h(h+3)/2 + 2k(h+k+1) } \gp{n-h-k}{k}{q^2} \gp{n-2k}{h}{q}
 \\ =
  \sum_{k=-\infty}^\infty
     q^{14k^2 +  5k    } \Trb{n+2}{7k+1}{7k+2}{q}
   - q^{14k^2 +  9k + 1} \Trb{n+2}{7k+2}{7k+3}{q}
\tag{\ref{dual}.82f}
\end{gather*}
\end{id*}

\begin{id*}[Dual of \ref{t-ljslist}.82]{\ref{dual}.82}
Equivalent to Andrews~\cite[p. 331, eqn. (1.3)]{mod11}, and to
$(-q;q)_\infty
\times (\ref{t-ljslist}.31)$.
\begin{gather*}
  \sum_{h=0}^\infty \sum_{k=0}^\infty
   \frac{ q^{ h(h+3)/2 + 2k(h+k+1)} } {(q^2;q^2)_k (q;q)_h }
 = \prod_{j=1}^\infty \frac{ (1-q^{7j-1}) (1-q^{7j-6}) (1-q^{7j})} {(1-q^j)}
\tag{\ref{dual}.82}
\end{gather*}
\end{id*}

\begin{id*}[Dual of \ref{t-SLfin}.90-even]{\ref{dual}.90f-even}
\begin{gather*}
  \sum_{i=0}^\infty \sum_{J=0}^\infty \sum_{k=0}^\infty \sum_{L=0}^\infty
  (-1)^i q^{J^2 + 3J + 3i(i-1)/2 - k - 2L} \gp{n-J}{i}{q^3} \\ \times
  \gp{n-J+k}{k}{q^2} \gp{n-J+L}{L}{q^2} \gp{n+J-3i-2k-2L}{n-J}{q} \\ =
  \sum_{k=-\infty}^\infty q^{27k^2 + 6k} \gp{2m+3}{m+9k+3}{q} -
     q^{27k^2 + 42k + 16} \gp{2m+3}{m+9k+8}{q}
\tag{\ref{dual}.90f-even}
\end{gather*}
\end{id*}

\begin{id*}[Dual of \ref{t-ljslist}.90-even]{\ref{dual}.90-even}
\begin{gather*}
  \sum_{i=0}^\infty \sum_{J=0}^\infty \sum_{k=0}^\infty \sum_{L=0}^\infty
   \frac{ (-1)^i q^{J^2 + 3J + 3i(i-1)/2 - k - 2L} } {(q^3;q^3)_i (q^2;q^2)_k
    (q^2;q^2)_L (q;q)_{2J-3i-2k-2L} }\\
 = \prod_{j=1}^\infty 
     \frac{ (1-q^{18j-1}) (1-q^{18j-17}) (1-q^{18j}) (1-q^{36j-16}) 
     (1-q^{36j-20})} {(1-q^j)}
\tag{\ref{dual}.90-even}
\end{gather*}
\end{id*}

\begin{id*}[Dual of \ref{t-SLfin}.90-odd]{\ref{dual}.90f-odd}
\begin{gather*}
  \sum_{i=0}^\infty \sum_{J=0}^\infty \sum_{k=0}^\infty \sum_{L=0}^\infty
  (-1)^i q^{J^2 + 4J + 3i(i-1)/2 - k - 2L} \gp{n-J}{i}{q^3} \\ \times
  \gp{n-J+k}{k}{q^2} \gp{n-J+L}{L}{q^2} \gp{n+J-3i-2k-2L}{n-J}{q} \\ =
  \sum_{k=-\infty}^\infty q^{27k^2 + 15k} \gp{2m+4}{m+9k+4}{q} -
     q^{27k^2 + 33k + 8} \gp{2m+4}{m+9k+8}{q}
\tag{\ref{dual}.90f-odd}
\end{gather*}
\end{id*}

\begin{id*}[Dual of \ref{t-ljslist}.90-odd]{\ref{dual}.90-odd}
\begin{gather*}
  \sum_{i=0}^\infty \sum_{J=0}^\infty \sum_{k=0}^\infty \sum_{L=0}^\infty
   \frac{ (-1)^i q^{J^2 + 4J + 3i(i-1)/2 - k - 2L} } {(q^3;q^3)_i (q^2;q^2)_k
    (q^2;q^2)_L (q;q)_{2J-3i-2k-2L} }\\
 = \prod_{j=1}^\infty 
     \frac{ (1-q^{18j-8}) (1-q^{18j-10}) (1-q^{18j}) (1-q^{36j-2}) 
     (1-q^{36j-34})} {(1-q^j)}
\tag{\ref{dual}.90-odd}
\end{gather*}
\end{id*}

\begin{id*}[Dual of \ref{t-SLfin}.91-even]{\ref{dual}.91f-even}
\begin{gather*}
  \sum_{i=0}^\infty \sum_{J=0}^\infty \sum_{k=0}^\infty \sum_{L=0}^\infty
  (-1)^i q^{J^2 + 2J + 3i(i-1)/2 - k - 2L} \gp{n-J}{i}{q^3} \\ \times
  \gp{n-J+k}{k}{q^2} \gp{n-J+L}{L}{q^2} \gp{n+J-3i-2k-2L}{n-J}{q} \\ =
  \sum_{k=-\infty}^\infty q^{27k^2 + 3k} \gp{2m+2}{m+9k+2}{q} -
     q^{27k^2 + 39k + 14} \gp{2m+2}{m+9k+7}{q}
\tag{\ref{dual}.91f-even}
\end{gather*}
\end{id*}

\begin{id*}[Dual of \ref{t-ljslist}.91-even]{\ref{dual}.91-even}
\begin{gather*}
  \sum_{i=0}^\infty \sum_{J=0}^\infty \sum_{k=0}^\infty \sum_{L=0}^\infty
   \frac{ (-1)^i q^{J^2 + 2J + 3i(i-1)/2 - k - 2L} } {(q^3;q^3)_i (q^2;q^2)_k
    (q^2;q^2)_L (q;q)_{2J-3i-2k-2L} }\\
 = \prod_{j=1}^\infty 
     \frac{ (1-q^{18j-2}) (1-q^{18j-16}) (1-q^{18j}) (1-q^{36j-14}) 
     (1-q^{36j-22})} {(1-q^j)}
\tag{\ref{dual}.91-even}
\end{gather*}
\end{id*}

\begin{id*}[Dual of \ref{t-SLfin}.91-odd]{\ref{dual}.91f-odd}
\begin{gather*}
  \sum_{i=0}^\infty \sum_{J=0}^\infty \sum_{k=0}^\infty \sum_{L=0}^\infty
  (-1)^i q^{J^2 + 3J + 3i(i-1)/2 - k - 2L} \gp{n-J}{i}{q^3} \\ \times
  \gp{n-J+k}{k}{q^2} \gp{n-J+L}{L}{q^2} \gp{n+J+1-3i-2k-2L}{n-J}{q} \\ =
  \sum_{k=-\infty}^\infty q^{27k^2 + 12k} \gp{2m+3}{m+9k+3}{q} -
     q^{27k^2 + 30k + 7} \gp{2m+3}{m+9k+7}{q}
\tag{\ref{dual}.91f-odd}
\end{gather*}
\end{id*}

\begin{id*}[Dual of \ref{t-ljslist}.91-odd]{\ref{dual}.91-odd}
\begin{gather*}
  \sum_{i=0}^\infty \sum_{J=0}^\infty \sum_{k=0}^\infty \sum_{L=0}^\infty
   \frac{ (-1)^i q^{J^2 + 3J + 3i(i-1)/2 - k - 2L} } {(q^3;q^3)_i (q^2;q^2)_k
    (q^2;q^2)_L (q;q)_{2J-3i-2k-2L} }\\
 = \prod_{j=1}^\infty 
     \frac{ (1-q^{18j-7}) (1-q^{18j-11}) (1-q^{18j}) (1-q^{36j-4}) 
     (1-q^{36j-32})} {(1-q^j)}
\tag{\ref{dual}.91-odd}
\end{gather*}
\end{id*}

\begin{id*}[Dual of \ref{t-SLfin}.92-even]{\ref{dual}.92f-even}
\begin{gather*}
  \sum_{i=0}^\infty \sum_{J=0}^\infty \sum_{k=0}^\infty \sum_{L=0}^\infty
  (-1)^i q^{J^2 + J + 3i(i-1)/2 - k} \gp{n-J}{i}{q^3} \\ \times
  \gp{n-J+k}{k}{q^2} \gp{n-J+L-1}{L}{q^2} \gp{n+J-3i-2k-2L}{n-J}{q} \\ =
  \sum_{k=-\infty}^\infty q^{27k^2 + 3k} \gp{2m+2}{m+9k+2}{q} -
     q^{27k^2 + 39k + 14} \gp{2m+2}{m+9k+7}{q}
\tag{\ref{dual}.92f-even}
\end{gather*}
\end{id*}

\begin{id*}[Dual of \ref{t-ljslist}.92-even]{\ref{dual}.92-even}
\begin{gather*}
  \sum_{i=0}^\infty \sum_{J=0}^\infty \sum_{k=0}^\infty \sum_{L=0}^\infty
   \frac{ (-1)^i q^{J^2 + J + 3i(i-1)/2 - k} } {(q^3;q^3)_i (q^2;q^2)_k
    (q^2;q^2)_L (q;q)_{2J-3i-2k-2L} }\\
 = \prod_{j=1}^\infty 
     \frac{ (1-q^{18j-3}) (1-q^{18j-15}) (1-q^{18j}) (1-q^{36j-12}) 
     (1-q^{36j-24})} {(1-q^j)}
\tag{\ref{dual}.92-even}
\end{gather*}
\end{id*}

\begin{id*}[Dual of \ref{t-SLfin}.92-odd]{\ref{dual}.92f-odd}
\begin{gather*}
  \sum_{i=0}^\infty \sum_{J=0}^\infty \sum_{k=0}^\infty \sum_{L=0}^\infty
  (-1)^i q^{J^2 + 2J + 3i(i-1)/2 - k} \gp{n-J}{i}{q^3} \\ \times
  \gp{n-J+k}{k}{q^2} \gp{n-J+L-1}{L}{q^2} \gp{n+J+1-3i-2k-2L}{n-J}{q} \\ =
  \sum_{k=-\infty}^\infty q^{27k^2 + 9k} \gp{2m+2}{m+9k+2}{q} -
     q^{27k^2 + 27k + 6} \gp{2m+2}{m+9k+6}{q}
\tag{\ref{dual}.92f-odd}
\end{gather*}
\end{id*}

\begin{id*}[Dual of \ref{t-ljslist}.92-odd]{\ref{dual}.92-odd}
\begin{gather*}
  \sum_{i=0}^\infty \sum_{J=0}^\infty \sum_{k=0}^\infty \sum_{L=0}^\infty
   \frac{ (-1)^i q^{J^2 + 2J + 3i(i-1)/2 - k} } {(q^3;q^3)_i (q^2;q^2)_k
    (q^2;q^2)_L (q;q)_{2J-3i-2k-2L} }\\
 = \prod_{j=1}^\infty 
     \frac{ (1-q^{18j-6}) (1-q^{18j-12}) (1-q^{18j}) (1-q^{36j-6}) 
     (1-q^{36j-30})} {(1-q^j)}
\tag{\ref{dual}.92-odd}
\end{gather*}
\end{id*}

\begin{obs*}{\ref{dual}.94} 
 The dual of (94) is (18 odd).
\end{obs*}

\begin{obs*}{\ref{dual}.96} 
 The dual of (96) is (14 odd).
\end{obs*}

\begin{obs*}{\ref{dual}.99} 
 The dual of (99) is (14 even).
\end{obs*}

\begin{id*}[Dual of \ref{t-SLfin}.120]{\ref{dual}.120f}
\begin{gather*}
\sum_{h\geqq 0} \sum_{i\geqq 0} \sum_{k\geqq 0} q^{(h+k)^2 + 2i(h+k+i) + 
h + 2(k+i)} \gp{n-i-h-k-1}{i}{q^2} \gp{n-i-h-1}{k}{q^2}\\ \times \gp{n-i-k}{h}{q^2}
\\ =
\sum_{j=-\infty}^\infty q^{12j^2 + 4j} \gp{2n+1}{n+6j+1}{q} - q^{12j^2 + 
8j +1}\gp{2n+1}{n+6j+2}{q} \\
\tag{\ref{dual}.120f}
\end{gather*}
\end{id*}

\begin{id*}[Dual of \ref{t-ljslist}.120]{\ref{dual}.120}
\begin{gather*}
\sum_{h\geqq 0} \sum_{i\geqq 0} \sum_{k\geqq 0} 
 \frac{ q^{(h+k)^2 + 2i(h+k+i) + 
h + 2(k+i)}} {(q^2;q^2)_h (q^2;q^2)_i (q^2;q^2)_k}  =
\prod_{j=1}^\infty \frac{ (1-q^{6j-1})(1-q^{6j-5})(1-q^{6j})} {1-q^j} \\
\tag{\ref{dual}.120}
\end{gather*}
\end{id*}

\begin{id*}[Dual of \ref{t-SLfin}.130]{\ref{dual}.130f}
\begin{gather*}
  \sum_{h=0}^\infty \sum_{i=0}^\infty \sum_{k=0}^\infty \sum_{l=0}^\infty 
  (-1)^l q^{4ik + 4ih + 4il + h^2 + 2hk + 2hl + l^2 + 2kl + k^2 -k + 
  6i^2} \gp{n-2i-h-k-l}{i}{q^4} \\ \times \gp{n-2i-h-l}{k}{q^2} 
  \gp{n-2i-h-j-1}{l}{q^2} \gp{n-2i-k-l}{h}{q^2} \\ =
  \sum_{j=-\infty}^\infty (-1)^j q^{8j^2 + 4j} \Trb{n}{4j+1}{4j+1}{q^2}
 +\sum_{j=-\infty}^\infty (-1)^j q^{8j^2 + 4j+2n-1} 
 \Trb{n-1}{4j+1}{4j+1}{q^2}\\
+\sum_{j=-\infty}^\infty (-1)^j q^{8j^2} \Trb{n}{4j}{4j}{q^2}
+\sum_{j=-\infty}^\infty (-1)^j q^{8j^2 + 2n -1} \Trb{n-1}{4j}{4j}{q^2}
\tag{\ref{dual}.130f}
\end{gather*}
\end{id*}

\begin{id*}[Dual of \ref{t-ljslist}.130]{\ref{dual}.130}
\begin{gather*}
  \sum_{h=0}^\infty \sum_{i=0}^\infty \sum_{k=0}^\infty \sum_{l=0}^\infty 
  (-1)^l \frac{ q^{4ik + 4ih + 4il + h^2 + 2hk + 2hl + l^2 + 2kl + k^2 -k + 
  6i^2}} { (q^4;q^4)_i (q^2;q^2)_k (q^2;q^2)_l (q^2;q^2)_h} \\ =
  \prod_{j=1}^\infty \frac{ (1-q^{16j-8})^2 (1-q^{16j}) + 
  (1-q^{16j-4})(1-q^{16j-12})(1-q^{16j}) }{ (1-q^{2j})}
\tag{\ref{dual}.130}
\end{gather*}
\end{id*}
\index{duality, reciprocal|)}

\section{More Finitizations: Bressoud Type Polynomial Identities} \label{bressoud}

\subsection{On Bressoud's Identities}
The process of finitization is certainly not unique.  There may be many 
different polynomial sequences which converge to a given series.  So far, 
we have only considered the polynomial identites which arise when the 
two-parameter generalization $f(q,t)$ of the series $\phi(q)$ satisfies a 
first order,
nonhomogeneous $q$-difference equation.  
For example, if we let
\[ \phi(q) = \sum_{j=0}^\infty \frac{q^{j^2}}{(q;q)_j}, \] 
the series side of the first Rogers-Ramanujan identity,  
\index{Rogers-Ramanujan identities}equation (A.18),
then \begin{eqnarray*}
 f(q,t) & = & \sum_{j=0}^\infty \frac{t^{2j} q^{j^2}}{(t;q)_{j+1}} \\
        & = & \sum_{j=0}^\infty t^{2j} q^{j^2} \sum_{k=0}^\infty 
        \gp{j+k}{k}{q} t^k \mbox{\qquad\qquad (by (\ref{qbc2}))} \\
        & = & \sum_{j=0}^\infty \sum_{k=0}^\infty t^{2j+k} q^{j^2} 
        \gp{j+k}{k}{q} \\
        & = & \sum_{n=0}^\infty t^n \sum_{k=0}^\infty q^{j^2} \gp{n-j}{k}{q} 
        \mbox{\qquad\qquad (by letting $n=2j+k$)}.
      \end{eqnarray*}
So $f(q,t)$ is the generating function for polynomials $D_n(q)$ which are given by the 
identity
\begin{gather*}
       \sum_{j\geqq 0} q^{j^2} \gp{n-j}{j}{q}
     = \sum_{j=-\infty}^\infty (-1)^j q^{j(5j+1)/2} \gp{n}{\lfloor 
       \frac{n+5j+1}{2} \rfloor}{q}. 
\end{gather*}

 Let us examine what can happen if we drop the requirement that $f(q,t)$ satisfy 
the a first order nonhomogeneous $q$-difference equation, 
\index{q-difference equation}but retain the 
other two requirements of Conditions~\ref{cond}. 
We modify $f(q,t)$ slightly to obtain
 \begin{eqnarray*}
 \tilde{f}(q,t) & = & \sum_{j=0}^\infty \frac{t^{j} q^{j^2}}{(t;q)_{j+1}} \\
        & = & \sum_{j=0}^\infty t^{j} q^{j^2} \sum_{k=0}^\infty 
        \gp{j+k}{k}{q} t^k \mbox{\qquad\qquad (by (\ref{qbc2}))} \\
        & = & \sum_{j=0}^\infty \sum_{k=0}^\infty t^{j+k} q^{j^2} 
        \gp{j+k}{k}{q} \\
        & = & \sum_{n=0}^\infty t^n \sum_{k=0}^\infty q^{j^2} \gp{n}{k}{q} 
        \mbox{\qquad\qquad (by letting $n=j+k$)}.
      \end{eqnarray*}
It turns out that $\tilde{f}(q,t)$ is the generating function for 
polynomials $B_n(q)$ which satisfy the polynomial identity
 \begin{equation} \label{db1}
   \sum_{j\geqq 0} q^{j^2} \gp{n}{j}{q} = \sum_{j=-\infty}^\infty (-1)^j 
 q^{j(5j+1)/2} \gp{2n}{n+2j}{q}, 
 \end{equation}
an identity which was discovered by David Bressoud~\cite[p. 211, eqn. 
(1.1)]{dmb} (by a different method), 
\index{Bressoud, David M.}
\index{Rogers-Ramanujan identities!Bressoud finitization}
and is easily seen to converge to the first Rogers-Ramanujan 
identity (A.18).  Note well that (\ref{db1}) is not merely an alternate 
representation of the MacMahon-Schur finitization (3.18) of the first
Rogers-Ramanujan identity (A.18), but rather a different sequence of 
polynomials which also converges to (A.18).

   We can perform the analagous calculation with the second Rogers-Ramanujan 
identity and arrive at 
    \begin{equation} \label{db2}
   \sum_{j\geqq 0} q^{j(j+1)} \gp{n}{j}{q} = \sum_{j=-\infty}^\infty (-1)^j 
 q^{j(5j+3)/2} \gp{2n+1}{n+2j+1}{q}, 
 \end{equation}
which is equivalent to Bressoud~\cite[p. 212, eqn. (1.3)]{dmb}.

  The next identity to appear in Bressoud's paper,
   \begin{equation} \label{db3}
    \sum_{j\geqq 0} q^{jn} \gp{n}{j}{q} = \sum_{j=-\infty}^\infty (-1)^j 
      q^{j(3j+1)/2} \gp{2n}{n+2j}{q},
   \end{equation}
(Bressoud~\cite[p. 212, (1.5)]{dmb}) 
\index{Bressoud, David M.}
is a finitization of Euler's \index{Euler, L.}
Pentagonal Number Theorem (\ref{t-ljslist}.1).  
\index{Pentagonal Number Theorem!finite}It is interesting to note 
that this identity is the reciprocal dual of both (\ref{db1}) and 
(\ref{db2}).  \index{duality, reciprocal}
This exemplifies the fact that duality can only be discussed 
within the context of a particular finitization.  Recall from 
\S~\ref{dual} that the dual of the First Rogers-Ramanujan Identity is 
the pair of identities (\ref{t-ljslist}.79) and (\ref{t-ljslist}.94).  We 
see here, however, that under the Bressoud finitization of the First 
Rogers-Ramanujan Identity, its dual is Euler's Pentagonal Number Theorem.  
Thus, a reciprocal duality theory for Rogers-Ramanujan type identities only 
makes sense with respect to a given method of finitization.

\subsection{New Bressoud Type Polynomial Identities}
By altering the exponent of $t$ in the $f(q,t)$ associated with a given 
Rogers-Ramanujan type identity, we can always produce a fermionic 
representation for a sequence of polynomials $P_n(q)$ which converges to the series 
side of the original identity.  In some cases, we may be lucky enough to 
find a nice bosonic representation for $P_n(q)$, and when this is the case, 
we have found what I will refer to as a {\em polynomial identity of the 
Bressoud Type}.  Note, however, that since we have forfeited the 
generalized first order $q$-difference equation property, we no longer 
immediately obtain a 
recurrence relation for the polynomials generated.  However, 
Riese's Mathematica packages can find a recurrence which 
the polyomials satisfy.\index{Riese, Axel}

I found the following alternate finitizations of some of the identities on 
Slater's list.

\begin{id*}[Alternate finitization of 
(\ref{t-ljslist}.79)]{\ref{bressoud}.79} This identity is explicitly a 
special case of Warnaar~\cite[equation (4.7)]{sow:qti}.\index{Warnaar, S. Ole}  It is likely
that the identities below can also be derived by the methods of    
Warnaar~\cite{sow:qti}. 
 \begin{gather*}
   \sum_{j\geqq 0} q^{j^2} \gp{n}{2j}{q} = \sum_{k=-\infty}^\infty 
   q^{15k^2 + k} \Trb{n}{6k}{6k}{q} - q^{15k^2 + 19k + 6} 
   \Trb{n}{6k+4}{6k+4}{q} \tag{\ref{bressoud}.79}
 \end{gather*}
\end{id*}

\begin{id*}[Alternative finitization of 
(\ref{t-ljslist}.94)]{\ref{bressoud}.94}
 \begin{gather*}
   \sum_{j\geqq 0} q^{j(j+1)} \gp{n+1}{2j+1}{q} \\= \sum_{k=-\infty}^\infty 
   q^{15k^2 + 4k} \Trb{n+1}{6k+1}{6k+1}{q} - q^{15k^2 + 16k + 4} 
   \Trb{n+1}{6k+3}{6k+3}{q} \tag{\ref{bressoud}.94}
 \end{gather*}
\end{id*}

\begin{id*}[Alternative finitization of 
(\ref{t-ljslist}.96)]{\ref{bressoud}.96}
 \begin{gather*}
   \sum_{j\geqq 0} q^{j(j+2)} \gp{n+1}{2j+1}{q} \\ = \sum_{k=-\infty}^\infty 
   q^{15k^2 + 7k} \Trb{n+1}{6k+1}{6k+1}{q} - q^{15k^2 + 17k + 4} 
   \Trb{n+1}{6k+3}{6k+3}{q} \tag{\ref{bressoud}.96}
 \end{gather*}
\end{id*}

\begin{id*}[Alternative finitization of 
(\ref{t-ljslist}.99)]{\ref{bressoud}.99}
 \begin{gather*}
   \sum_{j\geqq 0} q^{j(j+1)} \gp{n}{2j}{q} = \sum_{k=-\infty}^\infty 
   q^{15k^2 + 2k} \Trb{n}{6k}{6k}{q} - q^{15k^2 + 22k + 8} 
   \Trb{n}{6k+4}{6k+4}{q} \tag{\ref{bressoud}.99}
 \end{gather*}
\end{id*}

\begin{proof}
We shall prove (\ref{bressoud}.79) and (\ref{bressoud}.96) simultaneously.
Let $S_n(q)$ represent the polynomial on the LHS of (\ref{bressoud}.79), and 
$T_n(q)$ present the polynomial on the LHS of (\ref{bressoud}.96).  Then 
\begin{gather*}
  T_n(q) =  \sum_{j=0}^\infty q^{j(j+2)} \gp{n+1}{2j+1}{q} \\
  = \sum_{j=0}^\infty q^{j(j+2)} \gp{n}{2j+1}{q} + 
    \sum_{j=0}^\infty q^{j^2 + n} \gp{n}{2j}{q} \mbox{\quad (by~\ref{qpt1})}
  \\= T_{n-1}(q) + q^n S_n(q)
\end{gather*} 
Thus, we have
\begin{equation} 
  T_n(q) - T_{n-1}(q) - q^n S_n(q) = 0 \mbox{ for $n\geqq 1$.} \label{firstrec}
\end{equation}

Also, 
\begin{gather*}
  S_n(q) = \sum_{j=0}^\infty q^{j^2} \gp{n}{2j}{q} \\
 = \sum_{j=0}^\infty q^{j^2} \gp{n-1}{2j}{q} +
   \sum_{j=1}^\infty q^{j^2 + n - 2j} \gp{n-1}{2j-1}{q} 
  \mbox{\quad(by~\ref{qpt1})} \\
= \sum_{j=0}^\infty q^{j^2} \gp{n-1}{2j}{q} +
   q^{n-1}\sum_{j=0}^\infty q^{j^2 } \gp{n-1}{2j+1}{q}\\ 
=\sum_{j=0}^\infty q^{j^2} \gp{n-1}{2j}{q} +
   q^{n-1}\sum_{j=0}^\infty q^{j^2 } \gp{n-2}{2j}{q} 
 + q^n \sum_{j=0}^\infty q^{j(j-2)}\gp{n-2}{2j+1}{q} \\
= S_{n-1}(q) + q^{n-1} S_{n-2}(q) + q^n T_{n-3}(q)
\end{gather*}

Thus,
\begin{equation}
  S_n(q) - S_{n-1}(q) - q^{n-1}S_{n-2}(q) - q^n T_{n-3}(q) = 0
\mbox{ for $n\geqq 2$.}\label{secrec}
\end{equation}

So, by verifying that the RHS of (\ref{bressoud}.79) and (\ref{bressoud}.96) 
satisfy 
(\ref{firstrec}) and (\ref{secrec}) along with the initial conditions
\[ S_0(q) = S_1(q) = T_0(q) = 1, \]
the identities will be proved.

We show that the RHS of (\ref{bressoud}.79) and (\ref{bressoud}.96) satisfy 
(\ref{firstrec}):
\begin{gather*}
  \sum_{j=-\infty}^\infty q^{15j^2 + 7j} \Trb{n+1}{6j+1}{6j+1}{q}
 -\sum_{j=-\infty}^\infty q^{15j^2 + 7j}\Trb{n}{6j+1}{6j+1}{q}\\
-\sum_{j=-\infty}^\infty q^{15j^2 + 17j + 4}\Trb{n+1}{6j+3}{6j+3}{q}
+\sum_{j=-\infty}^\infty q^{15j^2 + 17j + 4}\Trb{n}{6j+3}{6j+3}{q}\\ 
-\sum_{j=-\infty}^\infty q^{15j^2 + j + n} \Trb{n}{6j}{6j}{q}
+\sum_{j=-\infty}^\infty q^{15j^2 + 19j + 6 + n}\Trb{n}{6j+4}{6j+4}{q}  
\end{gather*}
Apply (\ref{ETrb28}) to the first and third terms to obtain
\begin{gather*}
= \sum_{j=-\infty}^\infty q^{15j^2 + 7j + n} \Trb{n}{6j+1}{6j+2}{q}
 -\sum_{j=-\infty}^\infty q^{15j^2 + 17j + 4 + n}\Trb{n}{6j+}{6j+4}{q}\\
 -\sum_{j=-\infty}^\infty q^{15j^2 + 11j + 2 + n}\Trb{n}{6j+2}{6j+2}{q}
 +\sum_{j=-\infty}^\infty q^{15j^2 - 11j + 2 + n}\Trb{n}{6j-2}{6j-2}{q}
\end{gather*}
In the above, the third and fourth terms cancel after replacing
$j$ by $-j$ in the fourth and applying (\ref{Trbsym}).
To see that  the first and second terms cancel, replace 
$j$ by $1-j$ in the second term, and apply (\ref{Trbsym}) to it. 
 
\[ =0. \]

Now, we shall see that the RHS's of (\ref{bressoud}.79) and 
(\ref{bressoud}.96) satisfy
(\ref{secrec}):

\begin{gather*} 
 \sum_{j=-\infty}^\infty q^{15j^2 + j} \Trb{n}{6j}{6j}{q}
-\sum_{j=-\infty}^\infty q^{15j^2 + 19j + 6} \Trb{n}{6j+4}{6j+4}{q}\\
-\sum_{j=-\infty}^\infty q^{15j^2 + j} \Trb{n-1}{6j}{6j}{q} 
+\sum_{j=-\infty}^\infty q^{15j^2 + 19j + 6} \Trb{n-1}{6j+4}{6j+4}{q}\\
-\sum_{j=-\infty}^\infty q^{15j^2 + j+n-1} \Trb{n-2}{6j}{6j}{q} 
+\sum_{j=-\infty}^\infty q^{15j^2 + 19j + 5+n} \Trb{n-2}{6j+4}{6j+4}{q}\\ 
-\sum_{j=-\infty}^\infty q^{15j^2 +7j+n} \Trb{n-2}{6j+1}{6j+1}{q}
+\sum_{j=-\infty}^\infty q^{15j^2 +17j+n+4}\Trb{n-2}{6j+3}{6j+3}{q}
\end{gather*}
Apply (\ref{ETrb28}) to the first term to obtain
\begin{gather*}  
=\sum_{j=-\infty}^\infty q^{15j^2 + j+n-1} \Trb{n-1}{6j}{6j}{q} 
+\sum_{j=-\infty}^\infty q^{15j^2 - 5j + n}\Trb{n-1}{6j-1}{6j-1}{q}\\
-\sum_{j=-\infty}^\infty q^{15j^2 + 19j + 6} \Trb{n}{6j+4}{6j+4}{q} 
+\sum_{j=-\infty}^\infty q^{15j^2 + 19j + 6} \Trb{n-1}{6j+4}{6j+4}{q}\\
-\sum_{j=-\infty}^\infty q^{15j^2 + j+n-1} \Trb{n-2}{6j}{6j}{q}  
+\sum_{j=-\infty}^\infty q^{15j^2 + 19j + 5+n} \Trb{n-2}{6j+4}{6j+4}{q}\\ 
-\sum_{j=-\infty}^\infty q^{15j^2 +7j+n} \Trb{n-2}{6j+1}{6j+1}{q}
+\sum_{j=-\infty}^\infty q^{15j^2 +17j+n+4}\Trb{n-2}{6j+3}{6j+3}{q}
\end{gather*}
Next, apply (\ref{ETrb1}) to the first term, (\ref{ETrb28}) to
the third term, and apply (\ref{ETrb00})
to the seventh term, to wind up with
\begin{gather*}  
=\sum_{j=-\infty}^\infty q^{15j^2 - 5j + n}\Trb{n-1}{6j-1}{6j-1}{q}
-\sum_{j=-\infty}^\infty q^{15j^2 + 19j+n+5}\Trb{n-1}{6j+4}{6j+5}{q}\\
-\sum_{j=-\infty}^\infty q^{15j^2 + 13j + 2 + n}\Trb{n-1}{6j+3}{6j+3}{q}
+\sum_{j=-\infty}^\infty q^{15j^2 + 19j + 5+n} \Trb{n-2}{6j+4}{6j+4}{q}\\ 
+\sum_{j=-\infty}^\infty q^{15j^2 + 7j+2n-2} \Trb{n-2}{6j+1}{6j+2}{q}
+\sum_{j=-\infty}^\infty q^{15j^2 +17j+n+4}\Trb{n-2}{6j+3}{6j+3}{q}.
\end{gather*}
Next, apply (\ref{ETrb1}) to the second term:
\begin{gather*}  
=\sum_{j=-\infty}^\infty q^{15j^2 - 5j + n}\Trb{n-1}{6j-1}{6j-1}{q}
-\sum_{j=-\infty}^\infty q^{15j^2 + 19j+2n+3}\Trb{n-2}{6j+4}{6j+5}{q}\\
-\sum_{j=-\infty}^\infty q^{15j^2 + 25j+10 + n} \Trb{n-2}{6j+6}{6j+6}{q}
-\sum_{j=-\infty}^\infty q^{15j^2 + 13j + 2 + n}\Trb{n-1}{6j+3}{6j+3}{q}\\
+\sum_{j=-\infty}^\infty q^{15j^2 + 7j+2n-2} \Trb{n-2}{6j+1}{6j+2}{q}
+\sum_{j=-\infty}^\infty q^{15j^2 +17j+n+4}\Trb{n-2}{6j+3}{6j+3}{q}.
\end{gather*}
Next, shift $j$ to $j-1$ in the second, third, and fourth terms, 
apply (\ref{ETrb0}) to the first to obtain the following:
\begin{gather*}  
=\sum_{j=-\infty}^\infty q^{15j^2 - 11j + 2n}\Trb{n-2}{6j-2}{6j-2}{q}
-\sum_{j=-\infty}^\infty q^{15j^2 + 13j+n+2}\Trb{n-1}{6j+3}{6j+3}{q}\\
+\sum_{j=-\infty}^\infty q^{15j^2 + 7j+2n-2} \Trb{n-2}{6j+1}{6j+2}{q}
+\sum_{j=-\infty}^\infty q^{15j^2 -13j+n+2}\Trb{n-2}{6j-3}{6j-3}{q}.
\end{gather*}
Apply (\ref{ETrb29}) to the second term:
\begin{gather*}  
=\sum_{j=-\infty}^\infty q^{15j^2 - 11j + 2n}\Trb{n-2}{6j-2}{6j-2}{q}
-\sum_{j=-\infty}^\infty q^{15j^2 + 13j+n+2}\Trb{n-2}{6j+3}{6j+3}{q}\\
-\sum_{j=-\infty}^\infty q^{15j^2 +19j+2n+4} \Trb{n-2}{6j+4}{6j+4}{q}
+\sum_{j=-\infty}^\infty q^{15j^2 -13j+n+2}\Trb{n-2}{6j-3}{6j-3}{q}.
\end{gather*}
The first and third terms cancel after shifting $j$ to $j-1$ in the third.
The second and fourth terms can be seen to cancel after replacing $j$
by $-j$ in the fourth, and applying (\ref{Trbsym}) to the fourth.
\[ = 0 .\]

Upon checking the initial conditions, the proof is complete. 
\end{proof}

  Identities (\ref{bressoud}.94) and 
(\ref{bressoud}.99) can be proved together in an analogous
fashion.

\appendix
\section{Annotated Slater List }\label{t-ljslist}

\index{Slater, Lucy J.|(}
\index{Slater's list of Rogers-Ramanujan type identities|(}
Lucy Slater's list of 130 identities of the Rogers-Ramanujan type appeared 
in \cite[p. 152-167]{ljs}.  While studying these identities, I found that 
some identities appeared more than once on the list, usually in slightly 
different forms.  Also, I found misprints in identities 
6, 10, 40, 41, 42, 54, 88, 89, 97, 100, 108, 110, 115, and 128.  
Some identities are simple 
sums of others on the list; a few are the sum of an identity on the list 
with another which does not appear but is easily deduced.  Not all of the 
identities originate with Slater.  Some date back as far as Euler; many 
were known to Rogers.  I have 
attempted to provide detailed references for all of the identities which 
predate Slater. 

In all cases, $|q|<1$.
\begin{id}[Euler, 1748] The Pentagonal Number Theorem. 
Euler~\cite[p. 274, sec. 323]{euler}.
\index{Euler, L.}
\index{Pentagonal Number Theorem}
  \begin{gather*}
    \sum_{n=-\infty}^\infty (-1)^n q^{3n(n+1)/2}  =
    \prod_{n=1}^\infty (1-q^n) \tag{A.1}
  \end{gather*} 
\end{id}

\begin{id}[Euler, 1748] Euler~\cite{euler}.  
\index{Euler, L.}
\index{Andrews, George E.}
See also 
Andrews~\cite[p. 19, eqn (2.2.6) with $t=q$]{top}.
Note: This identity is the same as $(7)$ with $q$ replaced by $\sqrt{q}$.
 \begin{gather*}
   \sum_{n=0}^\infty \frac{q^{n(n+1)/2}} {(q;q)_n} =
   \prod_{n=1}^\infty (1+q^n)  \tag{A.2}
 \end{gather*} 
\end{id}

\begin{id}[Euler, 1748] Euler~\cite{euler}. 
\index{Euler, L.}
\index{Andrews, George E.}
See also
Andrews~\cite[p. 19, eqn (2.2.6) with $q$ replaced by $q^2$ and $t=-q$]{top}.
Note: This identity is the same as $(23)$.
 \begin{gather*}
   \sum_{n=0}^\infty \frac{(-1)^n q^{n^2}}{(q^2;q^2)_n} =
   \prod_{n=1}^\infty (1-q^{2n-1}) \tag{A.3}
 \end{gather*} 
\end{id}

\begin{id}
 \begin{gather*}
   \sum_{n=0}^\infty \frac{(-1)^n (-q;q^2)_n q^{n^2}}{(q^4;q^4)_n} =
   \prod_{n=1}^\infty (1-q^{2n-1})(1-q^{4n-2}) \tag{A.4}
 \end{gather*} 
\end{id}

\begin{obs}
Note: Identity $(5)$ (with $q$ replaced by $-q$) is the same as $(9)$
 and $(84)$. 
\end{obs}

\begin{id}
 \begin{gather*}
   \sum_{n=0}^\infty \frac{(-1;q)_n q^{n^2}}{(q;q)_n (q;q^2)_n} =
   \prod_{n=1}^\infty \frac{(1+q^{3n-1})(1+q^{3n-2})(1-q^{3n})}{1-q^n} \tag{A.6}
 \end{gather*} 
\end{id}

\begin{obs}
Identity $(7)$ is the same as $(2)$ with $q$ replaced by $q^2$. 
\end{obs}

\begin{id}[Gauss-Lebesgue] 
\index{Gauss, K. F.}
\index{Lebesgue, V. A.}
This is a special case $(z= -q)$ of an identity
appearing in Lebesgue~\cite[pp. 44-47]{val}.  See also 
Andrews~\cite{2color}.  
Lebesgue attributes the identity to 
Gauss~\cite{kfg}, however I have been unable to find it there.
 \begin{gather*}
   \sum_{n=0}^\infty \frac{(-q;q)_n q^{n(n+1)/2}}{(q;q)_n} =
   \prod_{n=1}^\infty 
   \frac{1-q^{4n}}{1-q^n} \tag{A.8}
 \end{gather*} 
\end{id}

\begin{id}[Jackson, 1928] Jackson~\cite[p. 179, 3 lines from bottom]{fhj}
Note: Identity $(9)$ is the same as $(84)$ and equivalent to
$(5)$ with $q$ replaced by $-q$.
\index{Jackson, F. H.}
\begin{gather*}
   \sum_{n=0}^\infty \frac{q^{n(2n+1)}}{(q;q)_{2n+1}} =
   \prod_{n=1}^\infty (1+q^n) \tag{A.5}
 \end{gather*} 
\end{id}

\begin{id} Note: This identity is the same as $(47)$.
 \begin{gather*}
   \sum_{n=0}^\infty \frac{(-1;q)_{2n} q^{n^2}}{(q^2;q^2)_{n} (q^2;q^4)_n} =
   \prod_{n=1}^\infty (1+q^{2n-1})(1+q^n)  \tag{A.10}
 \end{gather*}
\end{id}

\begin{id}
Note: This identity is the same as $(51)$ and $(64)$.
 \begin{gather*}
   \sum_{n=0}^\infty \frac{(-q;q)_{2n} q^{n(n+1)}}{(q;q^2)_{n+1} (q^4;q^4)_n} =
   \prod_{n=1}^\infty (1+q^n)(1+q^{2n})  \tag{A.11}
 \end{gather*}
\end{id}

\begin{id}[Gauss-Lebesgue] This is a special case $(z= -1)$ of an identity
appearing in Lebesgue~\cite[pp. 44-47]{val}.  
\index{Gauss, K. F.}
\index{Lebesgue, V. A.}
\index{Andrews, George E.}
Lebesgue attributes the identity to 
Gauss~\cite{kfg}, however I have been unable to find it there.  See also 
Andrews~\cite{2color}. 
 \begin{gather*}
  \sum_{n=0}^\infty \frac{(-1;q)_n q^{n(n+1)/2}} {(q;q)_n} =
   \prod_{n=0}^\infty \frac{1+q^{2n+1}}{1-q^{2n+1}} \tag{A.12}
 \end{gather*} 
\end{id}

\begin{id} 
 \begin{gather*}
  \sum_{n=0}^\infty \frac{(-q;q)_n q^{n(n-1)/2}} {(q;q)_n} =
   \prod_{n=1}^\infty \frac{1-q^{4n}}{1-q^{n}} + \prod_{n=1}^\infty 
   \frac{1+q^{2n-1}}{1-q^{2n-1}}  \tag{A.13}
 \end{gather*}
\end{id}

\begin{id}[Rogers, 1894] The Second Rogers-Ramanujan Identity.
\index{Rogers-Ramanujan identities}
\index{Rogers, L. J.}
Rogers~\cite[p. 329 (2)]{ljr:mem2}.
 \begin{gather*}
  \sum_{n=0}^\infty \frac{q^{n(n+1)}} {(q;q)_n} =
   \prod_{n=1}^\infty \frac{1}{(1-q^{5n-2})(1-q^{5n-3})} \tag{A.14}
 \end{gather*} 
\end{id}

\begin{id}[Rogers, 1917] Rogers~\cite[p. 330 (5)]{ljr:1917}. 
\index{Rogers, L. J.}
\index{Jackson, F. H.}
\index{Bailey, W. N.}
This identity appears with a misprint in 
Jackson~\cite[p. 170, third eqn.]{fhj}.  Bailey notes and  corrects 
this error in~\cite[p. 426, eqn 6.3]{wnb:comb}.  It appears that neither
Jackson nor Bailey noticed that it had already appeared (correctly)
in Rogers~\cite{ljr:1917}.
 \begin{gather*}
  \sum_{n=0}^\infty \frac{(-1)^n q^{n(3n-2)}} {(-q;q^2)_n (q^4;q^4)_n} =
   \prod_{n=1}^\infty \frac{(1-q^{5n-1})(1-q^{5n-4})(1-q^{5n})}{(1-q^{2n})}
  \tag{A.15}
 \end{gather*} 
\end{id}

\begin{id}[Rogers, 1894]\index{Rogers, L. J.}
Rogers~\cite[p. 331, immediately preceding eqn. (7)]{ljr:mem2}. 
 \begin{gather*}
  \sum_{n=0}^\infty \frac{q^{n(n+2)}} {(q^4;q^4)_n} =
    \prod_{n=1}^\infty \frac{1}{(1-q^{5n-2})(1-q^{5n-3})(1+q^{2n})}
  \tag{A.16}
 \end{gather*} 
\end{id}

\begin{id} Note: This identity is equivalent to (A.94) with $q$ replaced
by $-q$. 
 \begin{gather*}
  \sum_{n=0}^\infty \frac{q^{n(n+1)}} {(q^2;q^2)_n (-q;q^2)_{n+1}} =
   \prod_{n=1}^\infty \frac{(1-q^{5n-1})(1-q^{5n-4})(1-q^{5n})(1+q^{2n})}
   {(1-q^{2n})} \tag{A.17}
 \end{gather*} 
\end{id}

\begin{id}[Rogers, 1894] The First Rogers-Ramanujan Identity.
Rogers~\cite[p. 328 (2)]{ljr:mem2}.
\index{Rogers, L. J.}
\index{Rogers-Ramanujan identities}
 \begin{gather*}
  \sum_{n=0}^\infty \frac{q^{n^2}} {(q;q)_n} =
   \prod_{n=1}^\infty \frac{1}{(1-q^{5n-1})(1-q^{5n-4})} \tag{A.18}
 \end{gather*} 
\end{id}

\begin{id}[Rogers, 1894] 
\index{Rogers, L. J.}
Rogers~\cite[p. 339, immediately preceeding
``Example 3"]{ljr:mem2} 
 \begin{gather*}
  \sum_{n=0}^\infty \frac{(-1)^n q^{3n^2}} {(-q;q^2)_n (q^4;q^4)_n} =
   \prod_{n=1}^\infty \frac{(1-q^{5n-2})(1-q^{5n-3})(1-q^{5n})}{(1-q^{2n})}
  \tag{A.19}
 \end{gather*} 
\end{id}

\begin{id} [Rogers, 1894]
\index{Rogers, L. J.} 
Rogers~\cite[p. 330, last eqn.]{ljr:mem2}. 
\begin{gather*}
  \sum_{n=0}^\infty \frac{q^{n^2}} {(q^4;q^4)_n} =
    \prod_{n=1}^\infty \frac{1}{(1-q^{5n-1})(1-q^{5n-4})(1+q^{2n})}
  \tag{A.20}
 \end{gather*} 
\end{id}

\begin{id} 
 \begin{gather*}
  \sum_{n=0}^\infty \frac{(-1)^n (q;q^2)_n q^{n^2}} {(-q;q^2)_n (q^4;q^4)_n} =
   \prod_{n=1}^\infty \frac{(1+q^{5n-2})(1+q^{5n-3})(1-q^{5n})(1-q^{2n-1})}{(1-q^{2n})}
\tag{A.21} 
\end{gather*} 
\end{id}

\begin{id} 
 \begin{gather*}
  \sum_{n=0}^\infty \frac{(-q;q)_n q^{n(n+1)}} {(q;q^2)_{n+1} (q;q)_n} =
   \prod_{n=1}^\infty \frac{(1-q^{6n-1})(1-q^{6n-5})(1-q^{6n})(1+q^{n})}{(1-q^{n})}
\tag{A.22} 
\end{gather*} 
\end{id}

\begin{obs} Identity $(23)$ is the same as $(3)$. \end{obs}

\begin{id} 
 \begin{gather*}
  \sum_{n=0}^\infty \frac{(-1;q)_{2n} q^n} {(q^2;q^2)_n} =
   \prod_{n=1}^\infty \frac{(1-q^{6n-3})^2 (1-q^{6n})(1+q^{n})}{(1-q^{n})}
 \tag{A.24}
 \end{gather*} 
\end{id}

\begin{id} 
 \begin{gather*}
  \sum_{n=0}^\infty \frac{(-q;q^2)_n q^{n^2}} {(q^4;q^4)_n} =
   \prod_{n=1}^\infty \frac{(1-q^{6n-3})^2 (1-q^{6n})(1+q^{2n-1})}{(1-q^{2n})}
 \tag{A.25} 
\end{gather*} 
\end{id}

\begin{id} 
 \begin{gather*}
  \sum_{n=0}^\infty \frac{(-q;q)_{n} q^{n^2}} {(q;q^2)_{n+1} (q;q)_n} =
   \prod_{n=1}^\infty \frac{(1-q^{6n-3})^2 (1-q^{6n})(1+q^{n})}{(1-q^{n})}
\tag{A.26}
 \end{gather*} 
\end{id}

\begin{id} Note: This identity is the same as $(87)$.
 \begin{gather*}
  \sum_{n=0}^\infty \frac{(-q;q^2)_{n} q^{2n(n+1)}} {(q;q^2)_{n+1} (q^4;q^4)_n} =
   \prod_{n=1}^\infty \frac{(1+q^{6n-1})(1+q^{6n-5}) (1-q^{6n})}{(1-q^{2n})}
\tag{A.27}
 \end{gather*} 
\end{id}

\begin{id} 
 \begin{gather*}
  \sum_{n=0}^\infty \frac{(-q^2;q^2)_{n} q^{n(n+1)}} {(q;q)_{2n+1}} =
   \prod_{n=1}^\infty \frac{(1+q^{6n-1})(1+q^{6n-5}) (1-q^{6n}) (1+q^{2n})}{(1-q^{2n})}
\tag{A.28}
 \end{gather*} 
\end{id}

\begin{id} 
 \begin{gather*}
  \sum_{n=0}^\infty \frac{(-q;q^2)_{n} q^{n^2}} {(q;q)_{2n}} =
   \prod_{n=1}^\infty 
   \frac{(1+q^{6n-2})(1+q^{6n-4}) (1-q^{6n}) (1+q^{2n-1})}{(1-q^{2n})}
 \tag{A.29}
 \end{gather*}
\end{id}

\begin{obs} 
Identity $(30)$ is the same as $(24)$ with $q$ replaced by $-q$.
\end{obs}

\begin{id}[Rogers, 1917] The Third Rogers-Selberg Identity.
\index{Rogers, L. J.}
\index{Selberg, A.}
\index{Rogers-Selberg identities|(} 
Rogers~\cite[p. 331 (6), line 2]{ljr:1917}.  Independently rediscovered by 
Selberg~\cite{as}.
 \begin{gather*}
  \sum_{n=0}^\infty   \frac{ q^{2n(n+1)} } {(q^2;q^2)_n (-q;q)_{2n+1}} =
   \prod_{n=1}^\infty \frac{(1-q^{7n-1})(1-q^{7n-6}) (1-q^{7n}) }{(1-q^{2n})}
   \tag{A.31}
 \end{gather*} 
\end{id}

\begin{id}[Rogers, 1894]  The  Second Rogers-Selberg Identity.
\index{Rogers, L. J.}
\index{Selberg, A.}  
Rogers~\cite[bottom of p. 342]{ljr:mem2}.  Independently rediscovered by 
Selberg~\cite{as}.
 \begin{gather*}
  \sum_{n=0}^\infty   \frac{ q^{2n(n+1)} } {(q^2;q^2)_n (-q;q)_{2n}} =
   \prod_{n=1}^\infty \frac{(1-q^{7n-2})(1-q^{7n-5}) (1-q^{7n}) }{(1-q^{2n})}
   \tag{A.32}
 \end{gather*}
\end{id}

\begin{id}[Rogers, 1894]  The First Rogers-Selberg Identity. 
\index{Rogers, L. J.}
\index{Selberg, A.}
\index{Rogers-Selberg identities|)}
Rogers~\cite[bottom of p. 339]{ljr:mem2}. Independently rediscovered by 
Selberg~\cite{as}.
 \begin{gather*}
  \sum_{n=0}^\infty   \frac{ q^{2n^2} } {(q^2;q^2)_n (-q;q)_{2n}} =
   \prod_{n=1}^\infty \frac{(1-q^{7n-3})(1-q^{7n-4}) (1-q^{7n}) }{(1-q^{2n})}
  \tag{A.33}
 \end{gather*}
\end{id}

\begin{id}[Slater, 1952] The analytic version of the 
second G\"ollnitz-Gordon partition identity.
\index{G\"ollnitz, H.}
\index{Gordon, Basil}
\index{G\"ollnitz-Gordon identities}
 \begin{gather*}
  \sum_{n=0}^\infty   \frac{(-q;q^2)_n q^{n(n+2)} } {(q^2;q^2)_n } =
   \prod_{n=1}^\infty \frac{1}{(1-q^{8n-3})(1-q^{8n-4})(1-q^{8n-5})}
  \tag{A.34}
 \end{gather*}
\end{id}

\begin{id} Note: This identity is the same as $(106)$.
 \begin{gather*}
  \sum_{n=0}^\infty   \frac{ (-q;q^2)_n (-q;q)_n q^{n(n+3)/2} } 
  {(q;q)_{2n+1}} =
   \prod_{n=1}^\infty \frac{(1-q^{8n-1})(1-q^{8n-7}) (1-q^{8n}) (1+q^n) } {(1-q^n)}
 \tag{A.35}
\end{gather*}
\end{id}

\begin{id}[Slater, 1952] The analytic version of the 
first G\"ollnitz-Gordon partition identity.
\index{G\"ollnitz, H.}
\index{Gordon, Basil}
\index{G\"ollnitz-Gordon identities}
 \begin{gather*}
  \sum_{n=0}^\infty   \frac{(-q;q^2)_n q^{n^2} } {(q^2;q^2)_n } =
   \prod_{n=1}^\infty \frac{1}{(1-q^{8n-1})(1-q^{8n-4})(1-q^{8n-7})}  \tag{A.36}
 \end{gather*}  
\end{id}

\begin{id} Note: This identity is the same as $(105)$.
 \begin{gather*}
  \sum_{n=0}^\infty   \frac{ (-q;q^2)_n (-q;q)_n q^{n(n+1)/2} } 
  {(q;q)_{2n+1}} =
   \prod_{n=1}^\infty \frac{(1-q^{8n-3})(1-q^{8n-5}) (1-q^{8n}) (1+q^n) } {(1-q^n)}
  \tag{A.37}
\end{gather*}
\end{id}

\begin{id} Note: This identity is the same as $(86)$.  
 \begin{gather*}
  \sum_{n=0}^\infty   \frac{q^{2n(n+1)} } {(q;q)_{2n+1}} =
   \prod_{n=1}^\infty \frac{(1-q^{8n-3})(1-q^{8n-5}) (1-q^{8n}) 
   (1-q^{16n-2})(1-q^{16n-14}) } {(1-q^n)}
  \tag{A.38}
 \end{gather*}
\end{id}

\begin{id}[Jackson, 1928] Note: This identity is the same as $(83)$. 
\index{Jackson, F. H.} 
 Jackson~\cite[p. 170, 5th eqn.]{fhj}. 
 \begin{gather*}
  \sum_{n=0}^\infty   \frac{q^{2n^2} } {(q;q)_{2n}} =
   \prod_{n=1}^\infty \frac{(1-q^{8n-1})(1-q^{8n-7}) (1-q^{8n}) 
   (1-q^{16n-6})(1-q^{16n-10}) } {(1-q^n)}
   \tag{A.39}
 \end{gather*}
\end{id}

\begin{id}[Bailey, 1947] Bailey~\cite[p. 422, eqn. (1.7)]{wnb:comb}.
\index{Bailey, W. N.}
\index{Bailey's mod 9 identities|(}
 \begin{gather*}
  \sum_{n=0}^\infty   \frac{(q;q)_{3n+1} q^{3n(n+1)} } {(q^3;q^3)_n 
  (q^3;q^3)_{2n+1}} =
   \prod_{n=1}^\infty \frac{(1-q^{9n-1})(1-q^{9n-8}) (1-q^{9n}) } {(1-q^{3n})}
  \tag{A.40}
 \end{gather*}
\end{id}

\begin{id}[Bailey, 1947] Bailey~\cite[p. 422, eqn. (1.8)]{wnb:comb}.
\index{Bailey, W. N.}
 \begin{gather*}
  \sum_{n=0}^\infty   \frac{(q;q)_{3n} (1-q^{3n+2}) q^{3n(n+1)} } {(q^3;q^3)_n 
  (q^3;q^3)_{2n+1}} =
   \prod_{n=1}^\infty \frac{(1-q^{9n-2})(1-q^{9n-7}) (1-q^{9n}) } {(1-q^{3n})}
  \tag{A.41}
 \end{gather*}
\end{id}

\begin{id}[Bailey, 1947] Bailey~\cite[p. 422, eqn. (1.6)]{wnb:comb}.
\index{Bailey, W. N.}
 \begin{gather*}
  \sum_{n=0}^\infty   \frac{(q;q)_{3n}  q^{3n^2} } {(q^3;q^3)_n 
  (q^3;q^3)_{2n}} =
   \prod_{n=1}^\infty \frac{(1-q^{9n-4})(1-q^{9n-5}) (1-q^{9n}) } {(1-q^{3n})}
   \tag{A.42}
 \end{gather*}
\end{id}
\index{Bailey's mod 9 identities|)}

\begin{id}
 \begin{gather*}
  \sum_{n=0}^\infty   \frac{(-q;q)_{n}  q^{n(n+3)/2} } {(q;q^2)_{n+1} 
  (q;q)_{n}} =
   \prod_{n=1}^\infty \frac{(1-q^{10n-1})(1-q^{10n-9}) (1-q^{10n}) (1+q^n) } {(1-q^n)}
  \tag{A.43}
\end{gather*}
\end{id}

\begin{id}[Rogers, 1917]  Rogers~\cite[p. 330 (2), line 2]{ljr:1917} 
\index{Rogers, L. J.}
Note: this identity is the same as $(63)$.
 \begin{gather*}
  \sum_{n=0}^\infty   \frac{ q^{3n(n+1)/2} } {(q;q^2)_{n+1} (q;q)_{n}} =
   \prod_{n=1}^\infty \frac{(1-q^{10n-2})(1-q^{10n-8}) (1-q^{10n})  } {(1-q^n)}
  \tag{A.44}
 \end{gather*}
\end{id}

\begin{id}
 \begin{gather*}
  \sum_{n=0}^\infty   \frac{(-q;q)_{n}  q^{n(n+1)/2} } 
   {(q;q^2)_{n+1} (q;q)_{n}} =
  \prod_{n=1}^\infty \frac{(1-q^{10n-3})(1-q^{10n-7}) (1-q^{10n}) (1+q^n) } 
  {(1-q^n)}
  \tag{A.45}
 \end{gather*}
\end{id}

\begin{id}[Rogers, 1917]  Rogers~\cite[p. 330 (2) line 3]{ljr:1917}.  
This identity is equivalent to (62): See Andrews~\cite[p. 20, eqn. 
(8.7)]{comb}
\index{Rogers, L. J.}
\index{Andrews, George E.}
 \begin{gather*}
  \sum_{n=0}^\infty   \frac{ q^{n(3n-1)/2} } {(q;q^2)_{n} (q;q)_{n}} =
  \prod_{n=1}^\infty \frac{(1-q^{10n-4})(1-q^{10n-6}) (1-q^{10n})  } {(1-q^n)}
   \tag{A.46}
 \end{gather*}
\end{id}

\begin{id} Note: This identity equivalent to $(10)$, and to $(54) + q\times(49)$.
 \begin{gather*}
  \sum_{n=0}^\infty   \frac{(-1;q^2)_{n}  q^{n^2} } {(q;q)_{2n}} =
  \prod_{n=1}^\infty \frac{1+q^{2n-1}}{1-q^{2n-1}}
  \tag{A.47}
 \end{gather*}
\end{id}

\begin{id} Note: This identity is $(54) -q\times(49)$.
 \begin{gather*}
   \sum_{n=0}^\infty   \frac{(-1;q^2)_{n}  q^{n(n+1)} } {(q;q)_{2n}} \\ =
  \prod_{n=1}^\infty \frac{ (1-q^{12n-5})(1-q^{12n-7})(1-q^{12n})}{1-q^n}-q 
     \frac{ (1-q^{12n-1})(1-q^{12n-11})(1-q^{12n})}{1-q^n}
   \tag{A.48}
 \end{gather*}
\end{id}

\begin{id}
 \begin{gather*}
  \sum_{n=0}^\infty   \frac{(-q^2;q^2)_{n}  (1-q^{n+1}) q^{n(n+2)} } 
  {(q;q)_{2n+2}} =
  \prod_{n=1}^\infty \frac{(1-q^{12n-1})(1-q^{12n-11})(1-q^{12n})}{(1-q^n)}
   \tag{A.49}
 \end{gather*}
\end{id}

\begin{id}
 \begin{gather*}
  \sum_{n=0}^\infty   \frac{(-q;q^2)_{n}  q^{n(n+2)} } 
  {(q;q)_{2n+1}} =
  \prod_{n=1}^\infty \frac{(1-q^{12n-2})(1-q^{12n-10})(1-q^{12n})}{(1-q^n)}
  \tag{A.50}
 \end{gather*}
\end{id}

\begin{obs} Identity $(51)$ is the same as $(11)$ and $(64)$. \end{obs}

\begin{id} Note: This identity is the same as $(85)$.
 \begin{gather*}
  \sum_{n=0}^\infty   \frac{ q^{n(2n-1)} } {(q;q)_{2n}}  =
  \prod_{n=1}^\infty (1+q^{n}) \tag{A.52}
 \end{gather*}
\end{id}

\begin{id}
 \begin{gather*}
  \sum_{n=0}^\infty   \frac{(q;q^2)_{2n} q^{4n^2} } {(q^4;q^4)_{2n}} =
  \prod_{n=1}^\infty  \frac{(1-q^{12n-5})(1-q^{12n-7})(1-q^{12n})} {(1-q^{4n})}
  \tag{A.53}
 \end{gather*}
\end{id}

\begin{id}
 \begin{gather*}
  \sum_{n=0}^\infty   \frac{(-q^2;q^2)_{n-1} (1+q^n) q^{n^2} } {(q;q)_{2n}} =
  \prod_{n=1}^\infty  \frac{(1-q^{12n-5})(1-q^{12n-7})(1-q^{12n})} {(1-q^n)}
  \tag{A.54}
 \end{gather*}
\end{id}

\begin{id}
 \begin{gather*}
  \sum_{n=0}^\infty   \frac{(q;q^2)_{2n+1}  q^{4n(n+1)} } {(q^4;q^4)_{2n+1}} =
  \prod_{n=1}^\infty  \frac{(1-q^{12n-1})(1-q^{12n-11})(1-q^{12n})} 
  {(1-q^{4n})} \tag{A.55}
 \end{gather*}
\end{id}

\begin{id}
 \begin{gather*}
  \sum_{n=0}^\infty   \frac{(-q;q)_{n}  q^{n(n+2)} } {(q;q^2)_{n+1} 
  (q;q)_{n+1} } =
  \prod_{n=1}^\infty  \frac{(1+q^{12n-1})(1+q^{12n-11})(1-q^{12n})} 
  {(1-q^{n})}
  \tag{A.56}
 \end{gather*}
\end{id}

\begin{id} 
Note: This identity is the same as (55) with $q$ replaced by $-q$.
 \begin{gather*}
  \sum_{n=0}^\infty   \frac{(-q;q^2)_{2n+1}  q^{4n(n+1)} } {(q^4;q^4)_{2n+1}} =
  \prod_{n=1}^\infty  \frac{(1+q^{12n-1})(1+q^{12n-11})(1-q^{12n})} 
  {(1-q^{4n})} \tag{A.57}
 \end{gather*}
\end{id}

\begin{id}
 \begin{gather*}
  1 + \sum_{n=1}^\infty   \frac{(-q;q)_{n-1}  q^{n^2} } {(q;q^2)_{n} 
  (q;q)_{n} } =
  \prod_{n=1}^\infty  \frac{(1+q^{12n-5})(1+q^{12n-7})(1-q^{12n})} 
  {(1-q^{n})} \tag{A.58}
 \end{gather*}
\end{id}

\begin{id}[Rogers, 1917]  Rogers~\cite[p. 329 (1), line 3]{ljr:1917} 
\index{Rogers, L. J.}
 \begin{gather*}
  \sum_{n=0}^\infty   \frac{ q^{n(n+2)} } {(q;q^2)_{n+1} (q;q)_{n} } =
  \prod_{n=1}^\infty  \frac{(1-q^{14n-2})(1-q^{14n-12})(1-q^{14n})} 
  {(1-q^{n})} \tag{A.59}
 \end{gather*}
\end{id}

\begin{id}[Rogers, 1917]  Rogers~\cite[p. 329 (1), line 2]{ljr:1917}
\index{Rogers, L. J.} 
 \begin{gather*}
  \sum_{n=0}^\infty   \frac{ q^{n(n+1)} } {(q;q^2)_{n+1} (q;q)_{n} } =
  \prod_{n=1}^\infty  \frac{(1-q^{14n-4})(1-q^{14n-10})(1-q^{14n})} 
  {(1-q^{n})} \tag{A.60}
 \end{gather*}
\end{id}

\begin{id}[Rogers, 1894]  Rogers~\cite[p. 341, ex. 2]{ljr:mem2} 
\index{Rogers, L. J.}
 \begin{gather*}
  \sum_{n=0}^\infty   \frac{ q^{n^2} } {(q;q^2)_{n} (q;q)_{n} } =
  \prod_{n=1}^\infty  \frac{(1-q^{14n-6})(1-q^{14n-8})(1-q^{14n})} 
  {(1-q^{n})} \tag{A.61}
 \end{gather*}
\end{id}

\begin{id} This identity is equivalent to (46): See Andrews~\cite[p. 20, 
eqn. (8.7)]{comb}.
 \begin{gather*}
  \sum_{n=0}^\infty   \frac{(-q;q)_n q^{n(3n+1)/2} } {(q;q)_{2n+1}} =
  \prod_{n=1}^\infty  \frac{(1-q^{10n-4})(1-q^{10n-6})(1-q^{10n})} 
  {(1-q^{n})} \tag{A.62}
 \end{gather*}
\end{id}

\begin{obs}
 Identity $(63)$ is the same as $(44)$. 
\end{obs}

\begin{obs} Identity $(64)$ is the same as $(11)$ and $(51)$. 
\end{obs}

\begin{obs} Note: Identity $(65)$ is $(37) + \sqrt{q} \times (35)$, 
with $q$ replaced by $q^2$ throughout. 
\end{obs}

\begin{id} Note: Identity $(66)$ is $(71) + q\times (68)$
 \begin{gather*}
  \sum_{n=0}^\infty   \frac{(-1;q^4)_n (-q;q^2)_n q^{n^2}} {(q^2;q^2)_{2n}} 
  = \frac{(q^6,q^{10},q^{16};q^{16})_\infty + q(q^2,q^{14},q^{16};q^{16})_\infty}
  {(q,q^2)_\infty (q^4;q^4)_\infty} \tag{A.66}
 \end{gather*}
\end{id}

\begin{id} Note: Identity $(67)$ is $(71) - q\times (68)$
 \begin{gather*}
  \sum_{n=0}^\infty   \frac{(-1;q^4)_n (-q;q^2)_n q^{n(n+2)}} {(q^2;q^2)_{2n}} 
  = \frac{(q^6,q^{10},q^{16};q^{16})_\infty - q(q^2,q^{14},q^{16};q^{16})_\infty}
  {(q,q^2)_\infty (q^4;q^4)_\infty} \tag{A.67}
 \end{gather*}
\end{id}

\begin{id} 
 \begin{gather*}
  \sum_{n=0}^\infty   \frac{(-q^;q^2)_n (-q^4;q^4)_n q^{n(n+2)}} 
  {(-q^2;q^2)_{n+1} (q^2;q^4)_{n+1} (q^2;q^2)_n} =
   \prod_{n=1}^\infty \frac{ (1-q^{16n-2})(1-q^{16n-14})(1-q^{16n}) 
   (1+q^{2n-1})}{(1-q^{2n})}  \tag{A.68}
 \end{gather*}
\end{id}

\begin{id} 
 \begin{gather*}
  \sum_{n=0}^\infty   \frac{(-q^2;q^2)_n  q^{n(n+2)}} 
  {(q;q)_{2n+2} } =
   \prod_{n=1}^\infty  \left[  (1+q^{16n-2})(1+q^{16n-14})(1-q^{16n}) 
   \right]
   \frac{(1+q^{2n-1})}{(1-q^{2n})}  \tag{A.69}
 \end{gather*}
\end{id}

\begin{id} 
 \begin{gather*}
  \sum_{n=0}^\infty   \frac{(-q;q^2)_{n+1} (-q^2;q^4)_n q^{n(n+2)}} 
  {(q^2;q^4)_{n+1} (q^4;q^4)_n } =
   \prod_{n=1}^\infty  \frac{(1-q^{16n-4})(1-q^{16n-12})(1-q^{16n}) 
    (1+q^{2n-1})}{(1-q^{2n})} \tag{A.70}
 \end{gather*} 
\end{id}

\begin{id} 
 \begin{gather*}
  \sum_{n=0}^\infty   \frac{(-q^4;q^4)_{n-1} (-q;q^2)_n q^{n^2}} 
  {(q^2;q^4)_{n} (q^2;q^2)_n (-q^2;q^2)_{n-1} } =
   \prod_{n=1}^\infty  \frac{(1-q^{16n-6})(1-q^{16n-10})(1-q^{16n}) 
   (1+q^{2n-1})}{(1-q^{2n})} \tag{A.71}
 \end{gather*}
\end{id}

\begin{id} 
 \begin{gather*}
  1 + \sum_{n=1}^\infty   \frac{(-q^4;q^4)_{n-1} (-q;q^2)_n q^{n^2}} 
  {(q^2;q^4)_{n} (q^2;q^2)_n (q^2;q^2)_{n-1} } =
   \prod_{n=1}^\infty  \frac{ (1-q^{16n+6})(1-q^{16n+10})(1-q^{16n}) 
   (1+q^{2n-1})}{(1-q^{2n})}  \tag{A.72}
 \end{gather*}
\end{id}

\noindent\textbf{Identity 72-a.} \textit{Note: This identity does not appear 
explictly in Slater's list.  It appears implicitly as the sum of $(130)$ 
and $q\times (70)$. }
\begin{gather*}
   \sum_{n=0}^\infty \frac 
    {(-q^4; q^4)_{n} (-q;q^2)_{n+1} (1+q^{2n+1}) q^{n^2}}
    {(q^2;q^2)_{2n+2}  }  =
  \prod_{n=1}^\infty 
 \frac{ (1-q^{16n-8})^2 (1-q^{16n})(1+q^{2n-1})}{1-q^{2n}}
  \tag{A.72-a}
\end{gather*}

\begin{id} Note: Identity $(73)$ is the sum of $(77)$ and $(78)$.
\begin{gather*}
 1+\sum_{n=1}^\infty   \frac{(q^3;q^3)_{n-1} (-q;q)_{n} q^{n(n-1)/2}} 
  {(q;q)_{2n-1} (q;q)_{n} } = 
  \frac{(q^6,q^{12},q^{18};q^{18})_\infty + (q^9,q^9,q^{18};q^{18})_\infty}
   {(q;q)_\infty (q;q^2)_\infty}  \tag{A.73}
\end{gather*}
\end{id}

\begin{id} Note: Identity $(74)$ is $(77) + (78) - q\times (76)$.
\begin{gather*}
1+\sum_{n=1}^\infty   \frac{(q^3;q^3)_{n-1} (-q;q)_{n} q^{n(n-1)/2}} 
  {(q;q)_{2n} (q;q)_{n-1} }  \\ =
  \frac{(q^6,q^{12},q^{18};q^{18})_\infty + (q^9,q^9,q^{18};q^{18})_\infty -
   q(q^3,q^{15},q^{18};q^{18})_\infty}
  {(q;q)_\infty (q;q^2)_\infty} \tag{A.74}
\end{gather*}
\end{id}

\begin{obs}  Note: Identity $(75)$ is $(78) - q\times (76)$.
\begin{gather*}
1+\sum_{n=1}^\infty   \frac{(q^3;q^3)_{n-1} (-q;q)_{n} q^{n(n+1)/2}} 
  {(q;q)_{2n} (q;q)_{n-1} } 
  =\frac{(q^9,q^9,q^{18};q^{18})_\infty -
   q(q^3,q^{15},q^{18};q^{18})_\infty}
  {(q;q)_\infty (q;q^2)_\infty}   \tag{A.75}
\end{gather*}
\end{obs}

\begin{id}[Bailey-Dyson, 1947] Appears in Bailey~\cite[p. 434, (D1)]{wnb:comb}; 
Bailey credits its derivation to Dyson.
\index{Bailey, W. N.}
\index{Dyson, Freeman}
\begin{gather*} 
 \sum_{n=0}^\infty   \frac{(q^3;q^3)_{n} (-q;q)_{n+1} q^{n(n+3)/2}} 
  {(q;q)_{2n+2} (q;q)_n } =
   \prod_{n=1}^\infty  \frac{(1-q^{18n-3})(1-q^{18n-15})(1-q^{18n}) 
    (1+q^{n})}{(1-q^{n})}  \tag{A.76}
\end{gather*}
\end{id}

\begin{id}
\begin{gather*} 
 \sum_{n=0}^\infty   \frac{(q^3;q^3)_{n} (-q;q)_{n+1} (1-q^{n+1}) q^{n(n+1)/2}} 
  {(q;q)_{2n+2} (q;q)_{n} }  =
   \prod_{n=1}^\infty  \frac{ (1-q^{18n-6})(1-q^{18n-12})(1-q^{18n}) 
   (1+q^{n})}{(1-q^{n})}  \tag{A.77}
\end{gather*}
\end{id}

\begin{id}[Bailey-Dyson, 1947] Note: This identity is $(75) + q\times (76)$.
Appears in Bailey~\cite[p. 434, (D3)]{wnb:comb}; 
Bailey credits its derivation to Dyson.
\index{Bailey, W. N.}
\index{Dyson, Freeman}
\begin{gather*} 
 1 + \sum_{n=1}^\infty   \frac{(q^3;q^3)_{n-1} (-1;q)_{n+1} q^{n(n+1)/2}} 
  {(q;q)_{2n} (q;q)_{n-1} } =
   \prod_{n=1}^\infty  \frac{ (1-q^{18n-9})^2 (1-q^{18n}) 
   (1+q^{n})}{(1-q^{n})}  \tag{A.78}
\end{gather*}
\end{id}

\begin{id}[Rogers, 1894] Rogers~\cite[p. 330 second eqn. from top]{ljr:mem2}.
 Note: This identity is the same as $(98)$.  
\index{Rogers, L. J.}
\begin{gather*} 
 \sum_{n=0}^\infty   \frac{q^{n^2}} {(q;q)_{2n}} =
   \prod_{n=1}^\infty  \frac{ (1-q^{20n-8})(1-q^{20n-12})(1-q^{20n}) 
   (1+q^{2n-1})}{(1-q^{2n})} \tag{A.79}
\end{gather*}
\end{id}

\begin{id} [Rogers, 1917] Rogers~\cite[p. 331 (1), line 2]{ljr:1917}. 
\index{Rogers, L. J.}
\begin{gather*} 
 \sum_{n=0}^\infty   \frac{q^{n(n+1)/2}} 
  {(q;q^2)_{n+1} (q;q)_n  }  \\ =
   \prod_{n=1}^\infty  \Big[  
   (1-q^{7n-2})(1-q^{7n-5})(1-q^{7n})(1-q^{14n-3})(1-q^{14n-11}) 
   \Big]
   \frac{(1+q^{n})}{(1-q^{n})} \tag{A.80}
\end{gather*}
\end{id}

\begin{id} [Rogers, 1917] Rogers~\cite[p. 331 (1), line 1]{ljr:1917}. 
\index{Rogers, L. J.}
\begin{gather*} 
 \sum_{n=0}^\infty   \frac{ q^{n(n+1)/2}} 
  {(q;q^2)_{n} (q;q)_n  }  \\ =
   \prod_{n=1}^\infty  \Big[  
   (1-q^{7n-1})(1-q^{7n-6})(1-q^{7n})(1-q^{14n-5})(1-q^{14n-9}) 
   \Big]
   \frac{(1+q^{n})}{(1-q^{n})} \tag{A.81}
\end{gather*}
\end{id}

\begin{id} [Rogers, 1917] Rogers~\cite[p. 331 (1), line 3]{ljr:1917}. 
\index{Rogers, L. J.}
\begin{gather*} 
 \sum_{n=0}^\infty   \frac{ q^{n(n+3)/2}} 
  {(q;q^2)_{n+1} (q;q)_n } \\ =
   \prod_{n=1}^\infty  \Big[  
   (1-q^{7n-3})(1-q^{7n-4})(1-q^{7n})(1-q^{14n-1})(1-q^{14n-13}) 
   \Big]
   \frac{(1+q^{n})}{(1-q^{n})} \tag{A.82}
\end{gather*}
\end{id}

\begin{obs} Identity $(83)$ is the same as $(39)$. \end{obs}

\begin{obs} Identity $(84)$ is the same as $(9)$. \end{obs}

\begin{obs} Identity $(85)$ is the same as $(52)$. \end{obs}

\begin{obs} Identity $(86)$ is the same as $(38)$. \end{obs}

\begin{obs} Identity $(87)$ is the same as $(27)$. \end{obs}

\begin{id} Note: This identity is $(91) - q^2\times(90)$.
\begin{gather*} 
 \sum_{n=1}^\infty   \frac{(q^3;q^3)_{n-1}  (1-q^{n+1}) q^{n^2 - 1}} 
  { (q;q)_{2n} (q;q)_{n-1} } 
  = \frac{ (q^6,q^{21},q^{27};q^{27})_\infty -q^2(q^3,q^{24},q^{27};q^{27})_\infty}
  {(q;q)_\infty}\tag{A.88}
\end{gather*}
\end{id}

\begin{id}  Note: This identity is $(93) - q\times(91)$.
\begin{gather*} 
1 + \sum_{n=1}^\infty   \frac{(q^3;q^3)_{n-1}  q^{n(n+1)}} 
  { (q;q)_{2n} (q;q)_{n-1} } 
   = \frac{ (q^{12},q^{15},q^{27};q^{27})_\infty -q (q^6,q^{21},q^{27};q^{27})_\infty}
  {(q;q)_\infty}\tag{A.89}
\end{gather*}
\end{id}

\begin{id}[Bailey-Dyson, 1947] Appears in Bailey~\cite[p. 434, (B1)]{wnb:comb}; 
Bailey credits its derivation to Dyson.
\index{Bailey, W. N.}
\index{Dyson, Freeman}
\begin{gather*}
  \sum_{n=0}^\infty   \frac{(q^3;q^3)_{n}  q^{n(n+3)}} 
  { (q;q)_{2n+2} (q;q)_{n} }  =
   \prod_{n=1}^\infty \frac{
        (1-q^{27n-3})(1-q^{27n-24})(1-q^{27n})} {1-q^n} \tag{A.90}
\end{gather*}
\end{id}

\begin{id}[Bailey-Dyson, 1947] Appears in Bailey~\cite[p. 434, (B2)]{wnb:comb}; 
Bailey credits its derivation to Dyson.
\index{Bailey, W. N.}
\index{Dyson, Freeman}
\begin{gather*} 
  \sum_{n=0}^\infty   \frac{(q^3;q^3)_{n}  q^{n(n+2)}} 
  { (q;q)_{2n+2} (q;q)_{n} }  =
   \prod_{n=1}^\infty \frac{
        (1-q^{27n-6})(1-q^{27n-18})(1-q^{27n})} {1-q^n} \tag{A.91}
\end{gather*}
\end{id}

\begin{id}[Bailey-Dyson, 1947] Appears in Bailey~\cite[p. 434, (B3)]{wnb:comb}; 
Bailey credits its derivation to Dyson.
\index{Bailey, W. N.}
\index{Dyson, Freeman}
\begin{gather*} 
  \sum_{n=0}^\infty   \frac{(q^3;q^3)_{n}  q^{n(n+1)}} 
  { (q;q)_{2n+1} (q;q)_{n} }  =
   \prod_{n=1}^\infty \frac{1-q^{9n}} {1-q^n} \tag{A.92}
\end{gather*}
\end{id}

\begin{id}[Bailey-Dyson, 1947] Appears in Bailey~\cite[p. 434, (B4)]{wnb:comb}; 
Bailey credits its derivation to Dyson.
\index{Bailey, W. N.}
\index{Dyson, Freeman}
\begin{gather*} 
  1 + \sum_{n=1}^\infty   \frac{(q^3;q^3)_{n-1}  q^{n^2}} 
  { (q;q)_{2n-1} (q;q)_{n} }  =
   \prod_{n=1}^\infty \frac{
        (1-q^{27n-12})(1-q^{27n-15})(1-q^{27n})} {1-q^n} \tag{A.93}
\end{gather*}
\end{id}

\begin{id}[Rogers, 1894] 
Equivalent to Rogers~\cite[p. 331, eqn. (6)]{ljr:mem2}. 
\index{Rogers, L. J.}
 \begin{gather*} 
  \sum_{n=0}^\infty   \frac{q^{n(n+1)} } { (q;q)_{2n+1} }  =
   \prod_{n=1}^\infty \frac{
      (1-q^{10n-3})(1-q^{10n-7})(1-q^{10n})(1-q^{20n-4})(1-q^{20n-16})} 
{1-q^n} \tag{A.94}
\end{gather*}
\end{id} 

\begin{id} Note: This identity is equivalent to (97).
 \begin{gather*} 
  \sum_{n=0}^\infty   \frac{(-q;q^2)_n q^{n(3n-2)}} 
  { (q^2;q^2)_{2n} }  \\ =
   \prod_{n=1}^\infty 
      (1-q^{10n-3})(1-q^{10n-7})(1-q^{10n})(1-q^{20n-4})(1-q^{20n-16})  
\cdot\frac{1+q^{2n-1}}{1-q^{2n}} \tag{A.95}
\end{gather*}
\end{id} 

\begin{id}[Rogers, 1894] 
Equivalent to Rogers~\cite[p. 331, eqn. (7)]{ljr:mem2}.
\index{Rogers, L. J.}
 \begin{gather*} 
  \sum_{n=0}^\infty   \frac{q^{n(n+2)} } { (q;q)_{2n+1} }  =
   \prod_{n=1}^\infty \frac{
      (1-q^{10n-4})(1-q^{10n-6})(1-q^{10n})(1-q^{20n-2})(1-q^{20n-18})} 
{1-q^n} \tag{A.96}
\end{gather*}
\end{id} 

\begin{id} Note: This identity is equivalent to (95).
 \begin{gather*} 
  \sum_{n=0}^\infty   \frac{(-q;q^2)_{n+1} q^{n(3n+2)}} 
  { (q^2;q^2)_{2n+1} } \\ =
   \prod_{n=1}^\infty 
      (1-q^{10n-3})(1-q^{10n-7})(1-q^{10n})(1-q^{20n-4})(1-q^{20n-16})  
\cdot\frac{1+q^{2n-1}}{1-q^{2n}} \tag{A.97}
\end{gather*}
\end{id} 

\begin{obs} Identity $(98)$ is the same as $(79)$. \end{obs}

\begin{id}[Rogers, 1894] 
Equivalent to Rogers~\cite[p. 332, between eqns. (12) and (13)]{ljr:mem2}.
\index{Rogers, L. J.}
\begin{gather*} 
 \sum_{n=0}^\infty   \frac{q^{n(n+1)} } { (q;q)_{2n} }  =
   \prod_{n=1}^\infty \frac{
      (1-q^{10n-1})(1-q^{10n-9})(1-q^{10n})(1-q^{20n-8})(1-q^{20n-12})} 
{1-q^n} \tag{A.99}
\end{gather*}
\end{id} 

\begin{id}
 \begin{gather*} 
  \sum_{n=0}^\infty   \frac{(-q;q^2)_n q^{3n^2} } { (q^2;q^2)_{2n} } \\ =
   \prod_{n=1}^\infty 
      (1-q^{10n-1})(1-q^{10n-9})(1-q^{10n})(1-q^{20n-8})(1-q^{20n-12})
   \cdot\frac{1+q^{2n-1}}{1-q^{2n}} 
 \tag{A.100}
\end{gather*}
\end{id} 

\begin{id} Note: This identity is the sum of $(104)$ and $(105$-a$)$.
  \begin{gather*}
    1 + \sum_{n=1}^\infty \frac{ 
   (-q;q)_n (-q^2;q^2)_{n-1} q^{n(n-1)/2} } {(q;q)_{2n}}  \\
   = \frac{ (q^2,q^6,q^8;q^8)_\infty + (-q^{16},-q^{16},q^{32};q^{32})_\infty 
     -q(-q^8,-q^{24},q^{32};q^{32})_\infty}{(q;q)_\infty (q;q^2)_\infty}
  \tag{A.101}
\end{gather*}
\end{id}

\begin{id} Note: This identity is $(105$-a$) + q\times(103)$
  \begin{gather*}
    \sum_{n=0}^\infty \frac{ 
   (-q;q)_{n+1} (-q^2;q^2)_{n} q^{n(n+1)/2} } {(q,q)_{2n+2}}  \\ = 
  \frac{ (-q^{12},-q^{20},q^{32};q^{32})_\infty -q^2 (-q^{4},-q^{28},q^{32};q^{32})_\infty 
     +q(-q^8,-q^{24},q^{32};q^{32})_\infty -q^4(-1,-q^{32},q^{32};q^{32})_\infty}
     {(q;q)_\infty (q;q^2)_\infty}
  \tag{A.102}
\end{gather*}
\end{id}

\begin{id}
  \begin{gather*}
    \sum_{n=0}^\infty \frac{ 
   (-q;q)_{n+1} (-q^2;q^2)_{n} q^{n(n+3)/2} } {(q;q)_{2n+2}} 
   = \frac{(-q^8,-q^{24},q^{32};q^{32})_\infty -q^3(-1,-q^{32},q^{32};q^{32})_\infty}
      {(q;q)_\infty (q;q^2)_\infty}
  \tag{A.103}
\end{gather*}
\end{id}

\begin{id}
  \begin{gather*}
    1 + \sum_{n=1}^\infty \frac{ 
   (-q;q)_{n} (-q^2;q^2)_{n-1} q^{n(n+1)/2} } {(q;q)_{2n}}  
   = \frac{(-q^{16},-q^{16},q^{32};q^{32})_\infty -q(-q^8,-q^{24},q^{32};q^{32})_\infty}
      {(q;q)_\infty (q;q^2)_\infty}
  \tag{A.104}
\end{gather*}
\end{id}

\begin{obs} Identity $(105)$ is the same as $(37)$. \end{obs}

\noindent\textbf{Identity A.105-a.} \textit{This identity does not appear 
explicitly in Slater's list.  It appears implicilty as the difference of 
$(101)$ and $(104)$, and as the difference of $(102)$ and $q\times (103)$.}
  \begin{gather*}
    \sum_{n=0}^\infty \frac{ 
   (-q^2;q^2)_{n}  q^{n(n+1)/2} } {(q;q^2)_{n+1} (q;q)_n}  = 
  \prod_{n=1}^\infty 
    \frac{ (1-q^{8n-2})(1-q^{8n-6})(1-q^{8n}) (1+q^n) } {1-q^n}
  \tag{A.105-a}
\end{gather*}

\begin{obs} Identity $(106)$ is the same as $(35)$. \end{obs}

\begin{id} 
\begin{gather*}
 \sum_{n=0}^\infty \frac 
    {(q^3;q^6)_{n} (-q^2; q^2)_{n} q^{n(n+1)}}
    {(q^2;q^2)_{2n+1} (q;q^2)_n}  \\ =
 \prod_{n=1}^\infty \frac
 {(1+q^{12n-3})(1+q^{12n-9})(1-q^{12n})(1-q^{24n-6})(1-q^{24n-18})(1+q^{2n})}
     {1-q^{2n}}
  \tag{A.107}
\end{gather*}
\end{id}

\begin{id} Note: This identity is $(115) - q^2\times(116)$.
\begin{gather*}
 \sum_{n=0}^\infty \frac 
    {(q^6;q^6)_{n} (-q; q^2)_{n+1} (1-q^{2n+2}) q^{n(n+2)}}
    {(q^2;q^2)_{2n+2} (q^2;q^2)_n}  \\ =
 \prod_{n=1}^\infty \Big[
       (1+q^{12n- 5})(1+q^{12n-7})(1-q^{12n}) 
       (1-q^{24n- 2})(1-q^{24n-22}) \Big] 
\frac {(1+q^{2n-1})}{(1-q^{2n})}
  \tag{A.108}
\end{gather*}
\end{id}

\begin{id} 
\begin{gather*} 
 \sum_{n=0}^\infty \frac 
    {(q^3;q^6)_{n} (-q; q^2)_{n+1}  q^{n^2}}
    {(q^2;q^2)_{2n+1} (q;q^2)_n}  =
 \frac{(-q^2,-q^{10},q^{12};q^{12})_\infty (q^8,q^{16};q^{24})_\infty +q(q^{12};q^{12})_\infty }
 {(q^4;q^4)_\infty (q;q^2)_\infty}
  \tag{A.109}
\end{gather*}
\end{id}

\noindent\textbf{Identity A.109-a.} \textit{This identity does not appear 
explicitly in Slater's list.  It appears implicilty as the difference of 
$(109)$ and $q\times(110)$.}
\begin{gather*} 
 \sum_{n=0}^\infty \frac 
    {(q^3;q^6)_{n}  q^{n^2}}
    {(q^4;q^4)_{n} (q;q^2)_n^2} \\ =
 \prod_{n=1}^\infty 
\frac{ (1+q^{12n-2})(1+q^{12n-10}) (1-q^{12n})(1-q^{24n-8})(1-q^{24n-16}) 
(1+q^{2n-1})}{1-q^{2n}}
  \tag{A.109-a}
\end{gather*}

\begin{id} 
\begin{gather*}
 \sum_{n=0}^\infty \frac 
    {(q^3;q^6)_{n} (-q; q^2)_{n+1}  q^{n(n+2)}}
    {(q^2;q^2)_{2n+1} (q;q^2)_n}   =
 \prod_{n=1}^\infty 
  \frac{(1-q^{12n})(1+q^{2n-1})}{1-q^{2n}}
  \tag{A.110}
\end{gather*}
\end{id}

\begin{id} 
\begin{gather*}
 1 + \sum_{n=1}^\infty \frac 
    {(q^6;q^6)_{n-1} (-q; q^2)_{n}  q^{n(n+2)}}
    {(q^2;q^2)_{2n} (q^2;q^2)_{n-1}} 
    = \frac{(q^{15},q^{21},q^{36};q^{36})_\infty - q(q^9,q^{27},q^{36};q^{36})_\infty}
    {(q^4;q^4)_\infty (q;q^2)_\infty}
  \tag{A.111}
\end{gather*}
\end{id}

\begin{id} Note: This identity is $(115) + q^3\times (116)$. 
\begin{gather*}
 \sum_{n=0}^\infty \frac 
    {(q^6;q^6)_{n} (-q; q^2)_{n+2}  q^{n(n+2)}}
    {(q^2;q^2)_{2n+2} (q^2;q^2)_{n}} 
    = \frac{(q^{9},q^{27},q^{36};q^{36})_\infty + q^3(q^9,q^{3},q^{33};q^{36})_\infty}
    {(q^4;q^4)_\infty (q;q^2)_\infty}
  \tag{A.112}
\end{gather*}
\end{id}

\begin{id}  Note: This identity is $(115) - q^3\times (116)$.
\begin{gather*}
 1 + \sum_{n=1}^\infty \frac 
    {(q^6;q^6)_{n-1} (-q; q^2)_{n}  q^{n(n+2)}}
    {(q^2;q^2)_{2n-1} (q^2;q^2)_{n-1}}  
    = \frac{(q^{15},q^{21},q^{36};q^{36})_\infty - q^3(q^3,q^{33},q^{36};q^{36})_\infty}
    {(q^4;q^4)_\infty (q;q^2)_\infty}
  \tag{A.113}
\end{gather*}
\end{id}

\begin{id} 
\begin{gather*}
 1 + \sum_{n=1}^\infty \frac 
    {(q^6;q^6)_{n-1} (-q; q^2)_{n}  q^{n^2}}
    {(q^2;q^2)_{2n-1} (q^2;q^2)_{n}}  =
 \prod_{n=1}^\infty 
 \Big[ (1-q^{36n-15})(1-q^{36n-21})(1-q^{36n}) \Big]
\frac {1+q^{2n-1}}{1-q^{2n}}
  \tag{A.114}
\end{gather*}
\end{id}

\begin{id}[Bailey-Dyson] Bailey~\cite[p. 435, (C3)]{wnb:comb}.
\index{Bailey, W. N.}\index{Dyson, Freeman}
\begin{gather*}
  \sum_{n=0}^\infty \frac 
    {(q^6;q^6)_{n} (-q; q^2)_{n+1}  q^{n(n+2)}}
    {(q^2;q^2)_{2n+2} (q^2;q^2)_{n}}  =
 \prod_{n=1}^\infty 
 \Big[ (1-q^{36n-9})(1-q^{36n-27})(1-q^{36n}) \Big]
\frac {1+q^{2n-1}}{1-q^{2n}}
  \tag{A.115}
\end{gather*}
\end{id}

\begin{id}[Bailey-Dyson] Bailey~\cite[p. 435, (C1)]{wnb:comb}.
\index{Bailey, W. N.}\index{Dyson, Freeman}
\begin{gather*}
  \sum_{n=0}^\infty \frac 
    {(q^6;q^6)_{n} (-q; q^2)_{n+1}  q^{n(n+4)}}
    {(q^2;q^2)_{2n+2} (q^2;q^2)_{n}}  =
 \prod_{n=1}^\infty 
 \Big[ (1-q^{36n-3})(1-q^{36n-33})(1-q^{36n}) \Big]
\frac {1+q^{2n-1}}{1-q^{2n}}
  \tag{A.116}
\end{gather*}
\end{id}

\begin{id}  
\begin{gather*}
  \sum_{n=0}^\infty \frac 
    {(-q; q^2)_{n}  q^{n^2}}
    {(q^2;q^2)_{2n} }  \\ =
 \prod_{n=1}^\infty 
 \Big[ (1-q^{14n-3})(1-q^{14n-11})(1-q^{14n})(1-q^{28n-8})(1-q^{28n-20}) \Big]
\frac {1+q^{2n-1}}{1-q^{2n}}
  \tag{A.117}
\end{gather*}
\end{id}

\begin{id}
\begin{gather*}
  \sum_{n=0}^\infty \frac 
    {(-q; q^2)_{n}  q^{n(n+2)}}
    {(q^2;q^2)_{2n} }  \\ =
 \prod_{n=1}^\infty 
 \Big[ (1-q^{14n-1})(1-q^{14n-13})(1-q^{14n})(1-q^{28n-12})(1-q^{28n-16}) \Big]
\frac {1+q^{2n-1}}{1-q^{2n}}
  \tag{A.118}
\end{gather*}
\end{id}

\begin{id}
\begin{gather*}
  \sum_{n=0}^\infty \frac 
    {(-q; q^2)_{n+1}  q^{n(n+2)}}
    {(q^2;q^2)_{2n+1} }  \\ =
 \prod_{n=1}^\infty 
 \Big[ (1-q^{14n-5})(1-q^{14n-9})(1-q^{14n})(1-q^{28n-4})(1-q^{28n-24}) \Big]
\frac {1+q^{2n-1}}{1-q^{2n}}
  \tag{A.119}
\end{gather*}
\end{id}

\begin{id}
\begin{gather*}
  1 + \sum_{n=1}^\infty \frac 
    {(-q^2; q^2)_{n-1}  q^{n(n+1)}}
    {(q;q)_{2n} } 
  = \frac{(-q^{22},-q^{26},q^{48};q^{48})_\infty -q(-q^{14},-q^{34},q^{48};q^{48})_\infty}
  {(q;q)_\infty}
  \tag{A.120}
\end{gather*}
\end{id}

\begin{id}
\begin{gather*}
  1 + \sum_{n=1}^\infty \frac 
    {(-q^2; q^2)_{n-1}  q^{n^2}}
    {(q;q)_{2n} } \\ =
 \prod_{n=1}^\infty 
 \frac {(1-q^{16n-2})(1-q^{16n-14})(1-q^{16n}) 
        (1-q^{32n-12})(1-q^{32n-20})} {1-q^n}
  \tag{A.121}
\end{gather*}
\end{id}

\begin{id}
\begin{gather*}
  \sum_{n=0}^\infty \frac 
    {(-q^2; q^2)_{n}  q^{n(n+3)}}
    {(q;q)_{2n+2} } 
    =  \frac{(-q^{10},-q^{26},q^{38};q^{48})_\infty -q^3(-q^2,-q^{46},q^{48};q^{48})_\infty}
  {(q;q)_\infty}
  \tag{A.122}
\end{gather*}
\end{id}

\begin{id}
\begin{gather*}
  \sum_{n=0}^\infty \frac 
    {(-q^2; q^2)_{n}  q^{n(n+2)}}
    {(q;q)_{2n+2} } \\ =
 \prod_{n=1}^\infty 
 \frac {(1-q^{16n-6})(1-q^{16n-10})(1-q^{16n}) 
        (1-q^{32n-4})(1-q^{32n-28})} {1-q^n}
  \tag{A.123}
\end{gather*}
\end{id}

\begin{id}
\begin{gather*}
  \sum_{n=0}^\infty \frac 
    {(q^3; q^6)_{n}  q^{2n(n+1)}}
    {(q^2;q^2)_{2n+1} (q;q^2)_n } \\ =
 \prod_{n=1}^\infty 
 \frac {(1+q^{18n-5})(1+q^{18n-13})(1-q^{18n}) 
        (1-q^{36n-8})(1-q^{36n-28})} {1-q^{2n}}
  \tag{A.124}
\end{gather*}
\end{id}

\begin{id}
\begin{gather*}
  \sum_{n=0}^\infty \frac 
    {(q^3; q^6)_{n}  q^{2n(n+2)}}
    {(q^2;q^2)_{2n+1} (q;q^2)_n } \\ =
 \prod_{n=1}^\infty 
 \frac {(1+q^{18n-7})(1+q^{18n-11})(1-q^{18n}) 
        (1-q^{36n-4})(1-q^{36n-32})} {1-q^{2n}}
  \tag{A.125}
\end{gather*}
\end{id}

\begin{id} Note: This identity is equivalent to 
 $(71) + q\times(68) - q\times(128)$.
\begin{gather*}
  1 + \sum_{n=1}^\infty \frac 
    {(-q^4; q^4)_{n-1} (-q;q^2)_{n} q^{n^2}}
    {(q^2;q^2)_{2n}  } 
     = \frac{(-q^{28},-q^{36},q^{64};q^{64})_\infty -q^3(-q^{12},-q^{52},q^{64};q^{64})_\infty}
  {(q;q^2)_\infty (q^4;q^4)_\infty}
  \tag{A.126}
\end{gather*}
\end{id}

\begin{id}
Note: This identity is equivalent to 
 $(71) - q\times (128)$.
\begin{gather*}
  1 + \sum_{n=1}^\infty \frac 
    {(-q^4; q^4)_{n-1} (-q;q^2)_{n} q^{n(n+2)}}
    {(q^2;q^2)_{2n}  } 
     = \frac{(-q^{28},-q^{36},q^{64};q^{64})_\infty -q(-q^{20},-q^{44},q^{64};q^{64})_\infty}
  {(q;q^2)_\infty (q^4;q^4)_\infty}
  \tag{A.127}
\end{gather*}
\end{id}

\begin{id} 
\begin{gather*}
   \sum_{n=0}^\infty \frac 
    {(-q^4; q^4)_{n} (-q;q^2)_{n+1} q^{n(n+2)}}
    {(q^2;q^2)_{2n+2}  } 
    =  \frac{(-q^{20},-q^{44},q^{64};q^{64})_\infty -q^5(-q^{4},-q^{60},q^{64};q^{64})_\infty}
  {(q;q^2)_\infty (q^4;q^4)_\infty}
  \tag{A.128}
\end{gather*}
\end{id}

\begin{id} Note: This identity is equivalent to
$q^{-2}\times \Big( (128) - (68) \Big).$
\begin{gather*}
   \sum_{n=0}^\infty \frac 
    {(-q^4; q^4)_{n} (-q;q^2)_{n+1} q^{n(n+4)}}
    {(q^2;q^2)_{2n+2}  } 
     = \frac{(-q^{12},-q^{52},q^{64};q^{64})_\infty -q^3(-q^{4},-q^{60},q^{64};q^{64})_\infty}
  {(q;q^2)_\infty (q^4;q^4)_\infty}
  \tag{A.129}
\end{gather*}
\end{id}

\begin{id} Note: This identity is $(72$-a$) - q\times(70)$.
\begin{gather*}
   \sum_{n=0}^\infty \frac 
    {(-q^2; q^4)_{n} (-q;q^2)_{n+1} q^{n^2}}
    {(q^2;q^2)_{2n+1}  } 
     = \frac{(q^{8},q^{8},q^{16};q^{16})_\infty -q(q^{4},q^{12},q^{16};q^{16})_\infty}
  {(q;q^2)_\infty (q^4;q^4)_\infty}
  \tag{A.130}
\end{gather*}
\end{id}
\index{Slater, Lucy J.|)}
\index{Slater's list of Rogers-Ramanujan type identities|)}

\section*{Acknowledgements}
First and foremost, I thank my thesis advisor George E. Andrews \index{Andrews, George E.}.   
I am also grateful to Alexander Berkovich,
\index{Berkovich, Alexander}
Paul Eakin,
\index{Eakin, Paul}
Peter Paule, 
\index{Paule, Peter}
Avinash Sathaye, 
\index{Sathaye, Avinash}
and Doron Zeilberger 
\index{Zeilberger, Doron}
for for their support and encouragement of this project.

\bibliographystyle{plain}

\index{Jacobi's triple product identity|see{triple product identity}}
\index{q-binomial co\"efficient|see{Gaussian polynomial}}
\index{Euler's pentagonal number theorem|see{Pentagonal Number Theorem}}
\index{trinomial co\"efficeint!q-analog of|see{q-trinomial co\"efficients}}

\begin{theindex}

  \item Abel's lemma, 12
  \item Alder, H. L., 14
  \item Andrews, George E., 2, 7, 8, 10, 11, 14, 21, 25, 26, 31--33, 35, 
		43, 46, 47, 56, 57, 74, 98, 99, 103, 114

  \indexspace

  \item Bailey's mod 9 identities, 103
    \subitem finite, 33--34
  \item Bailey, W. N., 12, 14, 100, 103, 107--109, 112
  \item Baxter, Rodney, 2, 8, 10, 13, 14, 21
  \item Baxter, Rodney M., 2
  \item Berkovich, Alexander, 2, 4, 7, 14, 21--26, 30--32, 39, 114
  \item binomial theorem, 5
  \item bosonic representation, 13
  \item Bressoud, David M., 3, 93, 94

  \indexspace

  \item Cauchy, A.-L., 5
  \item creative symmetrization, 61--63
  \item creative telescoping, 59

  \indexspace

  \item duality, reciprocal, 74--92, 94
  \item Dyson, Freeman, 107--109, 112

  \indexspace

  \item Eakin, Paul, 114
  \item Ekhad, Shalosh, B., 59
  \item Euler's pentagonal number theorem, 
		\see{Pentagonal Number Theorem}{120}
  \item Euler, L., 16, 22, 94, 98

  \indexspace

  \item Fasenmyer, Sister Mary Celine, 58
  \item fermionic representation, 13
  \item Forrester, P. J., 2, 14

  \indexspace

  \item Gauss, K. F., 5, 99
  \item Gaussian polynomial, 4--5, 18
  \item Gordon, Basil, 14, 102
  \item G\"ollnitz, H., 14, 102
  \item G\"ollnitz-Gordon identities, 102
    \subitem finite, 31--32

  \indexspace

  \item Hardy, G. H., 14
  \item Heine's transformation, 12, 28
  \item Heine, E., 5, 12, 28

  \indexspace

  \item Jackson, F. H., 14, 99, 100, 103
  \item Jacobi's triple product identity, 
		\see{triple product identity}{120}
  \item Jacobi, K., 11

  \indexspace

  \item Koornwinder, T., 58

  \indexspace

  \item Lebesgue, V. A., 99
  \item Lepowski, J., 13

  \indexspace

  \item MacMahon, P. A., 13, 14, 25, 26
  \item McCoy, Barry M., 2, 4, 14, 21--26, 30--32, 39

  \indexspace

  \item operator
    \subitem annihilating, 64
  \item operator algebra, 63
  \item Orrick, William P., 4, 21, 22, 24, 25

  \indexspace

  \item partition (of an integer), 13
  \item Paule, Peter, 14, 59--62, 114
  \item Pentagonal Number Theorem, 98
    \subitem finite, 22, 94
  \item Petkov{\v s}ek, Marko, 58

  \indexspace

  \item q-binomial co\"efficient, \see{Gaussian polynomial}{120}
  \item q-binomial theorem, 5
  \item q-difference equation, 93
  \item q-difference equations, 14--17
  \item q-factorial, 4
  \item q-trinomial co\"efficients, 6--10
    \subitem asymptotics, 10
    \subitem definitions, 7
    \subitem recurrences, 7
  \item qMultiSum (Mathematica package), 59, 66
  \item qZeil (Mathematica package), 59, 60

  \indexspace

  \item Ramanujan, S., 13
  \item recurrence proof, 69--73
  \item Riese, Axel, 59, 60, 66, 94
  \item Rogers, L. J., 13, 14, 100, 102, 103, 105, 107--109
  \item Rogers-Ramanujan identities, 13--14, 92, 100
    \subitem Bressoud finitization, 93
    \subitem MacMahon-Schur finitization, 25, 26, 76
  \item Rogers-Selberg identities, 102
    \subitem finite, 30--31
  \item Rothe, H. A., 6
  \item RRtools (Maple package), 14, 16

  \indexspace

  \item Santos, J. P. O., 2, 22, 24, 26--47, 50--54, 56, 57, 73
  \item Sathaye, Avinash, 114
  \item Schilling, Anne, 2, 22
  \item Schur, I., 14, 25, 26
  \item Selberg, A., 102
  \item Sills, Andrew V., 14, 16, 73
  \item Slater's list of Rogers-Ramanujan type identities, 14, 98--114
    \subitem polynomial generalizations of, 21--55
  \item Slater, L. J., 14
  \item Slater, Lucy J., 2, 3, 11, 14, 16, 20--55, 98--114
  \item Steele Prize, 58

  \indexspace

  \item theta functions, 11
  \item trinomial co\"efficeint
    \subitem q-analog of, \see{q-trinomial co\"efficients}{120}
  \item trinomial co\"efficient
    \subitem ordinary, 6
  \item triple product identity, 10

  \indexspace

  \item Warnaar, S. Ole, 2, 3, 7, 14, 94
  \item Watson, G. N., 12
  \item Weierstra\ss, K., 12
  \item Wilf, Herbert S., 3, 58, 59
  \item Wilson, Robert L., 13
  \item WZ method, 58--69

  \indexspace

  \item Zeilberger's algorithm, 59
  \item Zeilberger, Doron, 3, 58--60, 114

\end{theindex}

\end{document}